\documentclass[12pt,a4paper,leqno]{amsart}

\usepackage{amsmath}
\usepackage{amssymb}
\usepackage{bbm} 
\usepackage{amsfonts}

\usepackage{amsthm}
\newtheorem{definition}{Definition}[section]
\newtheorem{theorem}[definition]{Theorem}
\newtheorem{proposition}[definition]{Proposition}
\newtheorem{remark}[definition]{Remark}
\newtheorem{corollary}[definition]{Corollary}
\newtheorem{lemma}[definition]{Lemma}
\newtheorem{example}[definition]{Example}

\usepackage{indentfirst}
\usepackage{multirow}
\numberwithin{equation}{section}
\allowdisplaybreaks[4] 

\usepackage{geometry}
\usepackage{tikz}

\usepackage[nonewpage]{imakeidx}
\makeindex[title=List of Symbols]

\usepackage{xcolor}

\definecolor{corr}{RGB}{0,150,255} 

\usepackage[foot]{amsaddr} 

\title{Cauchy Problem for Cylinder-like Capillary Jets}

\author{Haocheng Yang$^{1,2}$}
\date{\today}
\thanks{$^{1}$Ecole Normale Supérieure Paris-Saclay, CNRS Centre Borelli UMR9010, 4 Avenue des Sciences, F-91190 Gif-sur-Yvette}
\thanks{$^{2}$Universit{\'e} Paris XIII (Sorbonne Paris-Nord), LAGA, CNRS (UMR 7539), 99 Avenue J.-B. Cl{\'e}ment, F-93430 Villetaneuse}

\usepackage[backend=biber, isbn=false, url=false, doi=false, maxnames=5]{biblatex}
\renewbibmacro{in:}{}
\usepackage[colorlinks,linkcolor=blue,citecolor=red,urlcolor=black]{hyperref}

\usepackage{verbatim} 

\begin{document}
	
	\pagenumbering{arabic}
	
	\newcommand{\Op}[2][]{\operatorname{Op}^{#1}\left(#2\right)}
	\newcommand{\Supp}[1]{\operatorname{Supp} #1}
	\newcommand{\sgn}[1]{\operatorname{sgn}\left(#1\right)}
	\newcommand{\Real}{\operatorname{Re}}
	\newcommand{\Imag}{\operatorname{Im}}
	\newcommand{\diver}{\operatorname{div}}
	\newcommand{\curl}{\operatorname{curl}}
	\newcommand{\tr}[1]{\operatorname{tr}\left(#1\right)}
	
	\newcommand{\R}{\mathbb{R}}
	\newcommand{\N}{\mathbb{N}}
	\newcommand{\T}{\mathbb{T}}
	\newcommand{\D}{\mathbb{D}}
	\newcommand{\Z}{\mathbb{Z}}
	
	\renewcommand{\le}{\leqslant}
	\renewcommand{\ge}{\geqslant}
	
	\newcommand{\mar}[1]{\marginpar{\textcolor{red}{#1}}} 
	
	\begin{abstract}
		The motion of liquid jets plays an important role in physics and engineering, and needs rigorous mathematical investigations. Recently, Huang-Karakhanyan proved the first local well-posedness in Sobolev spaces for axisymmetric jets. In this paper, we will extend this result to general jets, namely without any axisymmetry condition.
	\end{abstract}
	
	\maketitle

	\tableofcontents

	\section{Introduction}\label{Sect:intro}
	
	\subsection{History of the problem}\label{subsect:history}
	
	The study of liquid jets has a long history in physics. It is until the beginning 19th century that researchers discovered that the singularity (namely, break-up) of jets is regardless of any exterior forces (gravity, for example), and should be due to surface tension, which is a nature of the fluid. In 1873, Plateau \cite{plateau1873experimental} firstly observed that this instability is related to the area of jets. Soon in 1879, Rayleigh \cite{rayleigh1879capillary} developed a linear stability method and explained this phenomenon from a theoretical point of view. Meanwhile, other physical nature of jets were also widely studied. We refer to \cite{eggers2008physics} for more histories and results on physics.
	
	In mathematics, the study of jets is concentrated on the steady case (system independent of time), especially the flow in nozzles. For the research on fixed infinitely long nozzles, we refer to \cite{bers1958mathematical,xie2007global,xie2009global,du2011subsonic}. In the case of semi-infinite nozzle (with half free boundary), the early study using hodograph transformation and conformal mappings can be found in \cite{birkhoff1990jets,gilbarg1960jets} and a stronger method via variational formulation is developed by Alt, Caffarelli, and Friedman \cite{caffarelli1981existence,alt1982asymmetric,alt1983axially}. Recent progress on this topic can be found in \cite{cheng2018existence,wang2019elliptic,li2024jet} and the references therein. Except for nozzle problem, there are also some other results on stationary jets \cite{groves2018spatial,gordon2020gelfand}.
	
	For the system depending on time, some formal construction of special solutions is established \cite{john1952example,andreev1992instability}. The first study of the Cauchy problem has recently been given by Huang-Karakhanyan \cite{huang2023wellposedness}, where the authors prove that the axisymmetric jets are locally well-posed in Sobolev spaces. In \cite{huang2024dirichlet}, the same authors studied in detail the Dirichlet-to-Neumann operator in the problem of Taylor-cone arising from break-up of jets. To our knowledge, further mathematical research on jets, such as the formation of break-up, is still blank. In this article, we will extend the well-posedness obtained in \cite{huang2023wellposedness} to the case without the axisymmetry condition.

	\subsection{Setting of the problem}\label{subsect:setting}
	
	\subsubsection{Cylinder-like jets}
	
	In this paper, we are interested in the 3D jets that are not necessarily invariant by rotation. The shape of such jets can be characterized by 
	\begin{equation}\label{eq-intro:domain}
		\Omega(t) = \{ (x,z)\in\mathbb{R}^2\times\mathbb{R} : x=0\ \text{or}\ |x|<\eta(t,x/|x|,z) \},
	\end{equation}\index{O@$\Omega(t)$ Domain of fluid}
	and its free boundary is given by
	\begin{equation}\label{eq-intro:interface}
		\Sigma(t) = \{ (x,z)\in\mathbb{R}^2\times\mathbb{R} : |x|=\eta(t,x/|x|,z) \},
	\end{equation}\index{S@$\Sigma(t)$ Free boundary}
	where $\eta$\index{e@$\eta$ Radius of free boundary} is a strictly positive function defined on $\mathbb{R}\times\mathbb{T}\times\mathbb{R}$. From physical background of the problem, we focus on two cases: steady interface with little perturbation and periodic interfaces. Rigorously, the following hypotheses on initial data $\eta_0 = \eta|_{t=0}$ will be made in the whole paper:
	\begin{align}
		&\exists C_0,c_0>0,\ \text{such that }c_0 < \eta_0 <C_0; \label{hyp-intro:bounds}\tag{H0}  \\[0.5ex]
		&\text{(Perturbative case) }\exists R>0,\ \text{such that }\eta_0-R \in H^s(\mathbb{T}\times\mathbb{R}),\ \text{for some proper }s; \label{hyp-intro:perturb}\tag{H1} \\[0.5ex]
		&\text{(Periodic case) }\exists L>0,\ \text{such that }\eta_0\ \text{is periodic in }z\ \text{of period }L. \label{hyp-intro:period}\tag{H1'}
	\end{align}\index{c@$c_0,C_0$ Lower and upper bounds of $\eta$}

	In perturbative case, \eqref{hyp-intro:perturb} indicates that, when $s>1$ and $z\rightarrow \pm\infty$, $\eta_0(\omega,z)$ tends to $R$ uniformly in $\omega$, which can be formally preserved for all time due to the governing equation \eqref{eq-intro:bdy-kinematic}, while \eqref{hyp-intro:perturb} and \eqref{hyp-intro:bounds} will hold for short time from our main result Theorem \ref{thm-intro:main}. In periodic case, one may check that \eqref{hyp-intro:period} is preserved for all time by using the fact that equations \eqref{eq-intro:Euler}, \eqref{eq-intro:bdy-kinematic}, and \eqref{eq-intro:bdy-pressure} are invariant by translations in $z$-direction. Therefore, $\eta$ should be viewed as a function on $\mathbb{R}\times\mathbb{T}^2$ (for simplicity, we may take $L=2\pi$). In what follows, we focus on the perturbative case, while all the results hold true in periodic case with normalizations eliminated.
	
	We also assume that the fluid is inviscid, incompressible, and irrotational. Consequently, the velocity field $u$ satisfies \textit{Euler equation} with divergence-free and rotation-free conditions
	\begin{equation}\label{eq-intro:Euler}
		\left\{\begin{array}{ll}
			\partial_t u + u\cdot\nabla_{x,z} u + \nabla_{x,y} P = 0, & \text{in }\Omega(t), \\ [0.5ex]
			\diver_{x,z}u = 0,\ \ \curl_{x,z}u=0, & \text{in }\Omega(t),
		\end{array}\right.
	\end{equation}
	where $P$ is the pressure driven only by surface tension.
	
	\subsubsection{Boundary conditions}
	
	In order to fully determine the motion of fluid and free surface, we introduce two boundary conditions on $\Sigma(t)$. The first one is the \textit{kinematic boundary condition},
	\begin{equation}\label{eq-intro:bdy-kinematic-ori}
		n \cdot \gamma_t = n \cdot u,\ \ \text{on }\Sigma(t),
	\end{equation}
	where, for all time $t$, $n= n(t)$ is the outward unit normal direction to $\Sigma(t)$ in $\R^3$ and $\gamma = \gamma(t)$ is the parametrization of $\Sigma(t)$,
	\begin{equation*}
		\begin{array}{lccc}
			\gamma(t) : &\mathbb{T}\times\mathbb{R} & \rightarrow & \mathbb{R}^3 \\[0.5ex]
			&(\omega,z) & \mapsto & \left( \eta(t,\omega,z) \cos{\omega}, \eta(t,\omega,z) \sin{\omega} ,z \right).
		\end{array}
	\end{equation*}
	In coordinate $(\omega,z)$, $n$ equals
	\begin{equation}\label{eq-intro:normal-dirction}
		n = \frac{1}{l} \left( e_r  - \frac{\eta_\omega}{\eta} e_\omega  - \eta_z e_z \right),
	\end{equation}\index{n@$n$ Conormal vector of free boundary}
	with
	\begin{equation*}
		e_r = (\cos{\omega},\sin{\omega},0),\ \ e_\omega = (-\sin{\omega},\cos{\omega},0),\ \ e_z = (0,0,1),
	\end{equation*}
	and
	\begin{equation}\label{eq-intro:def-l}
		l = \sqrt{1+\left(\frac{\eta_\omega}{\eta}\right)^2 + \eta_z^2}.
	\end{equation}
	As a result, \eqref{eq-intro:bdy-kinematic-ori} can be rewritten as
	\begin{equation}\label{eq-intro:bdy-kinematic}
		\eta_t = \left( e_r  - \frac{\eta_\omega}{\eta} e_\omega  - \eta_z e_z \right)\cdot u|_{\Sigma}.
	\end{equation}
	
	The second boundary condition in need is the balance of forces on the free boundary $\Sigma$:
	\begin{equation}\label{eq-intro:bdy-pressure}
		P|_\Sigma = \sigma \left(H-\frac{1}{2R}\right),
	\end{equation}
	where $\sigma>0$ is a constant and $H$ is the mean curvature. The normalization $1/(2R)$ is necessary in perturbative case, since the mean curvature
	\begin{equation}\label{eq-intro:mean-curv}
		H = \frac{1}{2} \left[ \frac{1}{\eta l} - \frac{\partial_\omega}{\eta}\left( \frac{\eta_\omega}{\eta l} \right) - \partial_z\left( \frac{\eta_z}{l} \right) \right]
	\end{equation}\index{H@$H$ Mean curvature of free boundary}
	tends to $1/(2R)$ as $z\rightarrow\pm\infty$ under the hypothesis \eqref{hyp-intro:perturb} with $s>3$. In periodic case, this normalization can be omitted.
	
	\subsubsection{Impact of gravity}\label{subsubsect:gravity}
	
	We emphasize that the effect of gravity is not important in the Cauchy problem. In fact, under gravity force, $P$ should be replaced by $P+gz$, where $g$ is the gravity acceleration. Let $(u^g,P^g,\eta^g)$ be any solution to the system with gravity. A simple calculus shows that
	\begin{align*}
		&u(t,x,z) = u^g(t,x,z-\frac{1}{2}gt^2) + (0,0,gt), \\
		&P(t,x,z) = P^g(t,x,z-\frac{1}{2}gt^2) + gz, \\
		&\eta(t,\omega,z) = \eta^g(t,\omega,z-\frac{1}{2}gt^2),
	\end{align*}
	is a solution to the problem \eqref{eq-intro:Euler}, \eqref{eq-intro:bdy-kinematic}, and \eqref{eq-intro:bdy-pressure} without gravity. This transformation between $(u,P,\eta)$ and $(u^g,P^g,\eta^g)$ is clearly invertible, meaning that the systems with and without gravity are equivalent.

	\subsection{Craig-Sulem-Zakharov formulation}\label{subsect:CSZ}
	
	Unlike standard 3D Euler equation, \eqref{eq-intro:Euler} is defined in a domain varying in time. To overcome this difficulty, we follow the idea from Zakharov \cite{zakharov1968stability} and Craig-Sulem \cite{craig1993numerical}, which can reduce the problem to equations on a fixed domain $\mathbb{T}\times\mathbb{R}$ (or $\mathbb{T}^2$ in periodic case).
	
	To begin with, we observe that the irrotational and incompressible conditions guarantee the existence of harmonic scalar potential $\phi$
	\begin{equation}\label{eq-intro:def-potential}
		\Delta_{x,z}\phi = 0,\ \ \nabla_{x,z}\phi = u,\ \  \text{in }\Omega(t).
	\end{equation}\index{p@$\phi$ Scalar potential}
	In terms of $\phi$, \eqref{eq-intro:Euler} can be written as
	\begin{equation*}
		\nabla_{x,z} \left(\partial_t \phi + \frac{|\nabla_{x,z}\phi|^2}{2} + P\right) = 0,
	\end{equation*}
	which yields the \textit{Bernoulli's equation}
	\begin{equation}\label{eq-intro:Bernoulli}
		\partial_t \phi + \frac{|\nabla_{x,z}\phi|^2}{2} + P = 0,\ \   \text{in }\Omega(t).
	\end{equation}
	Note that the right hand side should be a constant depending only on time, while, by absorbing this constant in $P$, we may take it to be zero for simplicity.
	
	Now, let us consider the trace $\psi$ of $\phi$ at the free surface $\Sigma(t)$, namely
	\begin{equation}\label{eq-intro:def-psi}
		\psi(t,\omega,z) = \phi\left( \eta(t,\omega,z) \cos{\omega}, \eta(t,\omega,z) \sin{\omega} ,z \right).
	\end{equation}\index{p@$\psi$ Trace of scalar potential}
	By definition, $\phi$ can be uniquely determined by $\psi$ (and $\eta$ implicitly) via the following linear elliptic equation
	\begin{equation}\label{eq-intro:ellip}
		\left\{\begin{array}{ll}
			\Delta_{x,z}\phi = 0, & \text{in }\Omega(t), \\[0.5ex]
			\phi|_{\Sigma(t)} = \psi. &
		\end{array}\right.
	\end{equation}
	Consequently, the right hand side of \eqref{eq-intro:bdy-kinematic} can be regarded as a function depending linearly on $\psi$ and implicitly on $\eta$, which gives rise to the (formal) definition of \textit{Dirichlet-to-Neumann operator},
	\begin{equation}\label{eq-intro:def-DtN}
		G(\eta)\psi := \left( e_r  - \frac{\eta_\omega}{\eta} e_\omega  - \eta_z e_z \right)\cdot \nabla_{x,z}\phi|_{r=\eta(\omega,z)}
	\end{equation}\index{G@$G(\eta)$ Dirichlet-to-Neumann operator}
	It is easy to check that $G(\eta)$ is linear, positive, and that $\eta G(\eta)$ is self-adjoint. A rigorous study of this operator will be given in Section \ref{Sect:pre} and \ref{Sect:paralin}. In terms of the quantities below
	\begin{align}
		&N = B V\cdot\left( \frac{\eta_\omega}{\eta}e_\omega + \eta_z e_z \right) + \frac{|V|^2 - B^2}{2}, \label{eq-intro:def-N} \\
		&B = e_r\cdot \nabla_{x,z}\phi|_{\Sigma}, \label{eq-intro:def-B} \\
		&V = e_\omega \left(e_\omega \cdot \nabla_{x,z}\phi\right)|_{\Sigma} + e_z \left(e_z \cdot \nabla_{x,z}\phi\right)|_{\Sigma}. \label{eq-intro:def-V}
	\end{align}\index{N@$N$ Nonlinear term}\index{B@$B$ Radial componant of velocity fluid at free boundary}\index{V@$V$ Angular componant of velocity fluid at free boundary}
	we have
	\begin{equation*}
		\eta_t = G(\eta)\psi = B - V\cdot\left( \frac{\eta_\omega}{\eta}e_\omega + \eta_z e_z \right),
	\end{equation*}
	and, from \eqref{eq-intro:Bernoulli} and \eqref{eq-intro:bdy-pressure},
	\begin{align*}
		\psi_t =& \partial_t \left( \phi\left(t, \eta \cos{\omega}, \eta \sin{\omega} ,z \right) \right)  =  \phi_t|_{\Sigma} + \eta_t e_r\cdot\nabla_{x,z}\phi|_{\Sigma} \\
		=& -\left.\left( \frac{|\nabla_{x,z}\phi|^2}{2} + P \right)\right|_{\Sigma} + \left( B - V\cdot\left( \frac{\eta_\omega}{\eta}e_\omega + \eta_z e_z \right) \right) B \\
		=& - \frac{|V|^2+B^2}{2} - \sigma \left(H-\frac{1}{2R}\right) + \left( B - V\cdot\left( \frac{\eta_\omega}{\eta}e_\omega + \eta_z e_z \right) \right) B \\
		=& - \sigma \left(H-\frac{1}{2R}\right) -N.
	\end{align*}
	
	In conclusion, the system \eqref{eq-intro:Euler}, \eqref{eq-intro:bdy-kinematic}, and \eqref{eq-intro:bdy-pressure} is reformulated as the following equations on $\mathbb{T}\times\mathbb{R}$ (or $\mathbb{T}^2$ in periodic case)
	\begin{equation}\label{eq-intro:WW}
		\left\{\begin{array}{l}
			\eta_t = G(\eta)\psi, \\[0.5ex]
			\psi_t + \sigma \left(H-\frac{1}{2R}\right) + N = 0.
		\end{array}\right.
	\end{equation}

	\subsection{Main results}
	
	The main result of this paper is that the system \eqref{eq-intro:WW} is locally well-posed in Sobolev spaces. More precisely,
	\begin{theorem}\label{thm-intro:main}
		Let $(\eta_0-R,\psi_0) \in H^{s+\frac{1}{2}}(\mathbb{T}\times\mathbb{R}) \times H^s(\mathbb{T}\times\mathbb{R})$ with $s>3$. Assume that $\eta_0$ satisfies the hypotheses \eqref{hyp-intro:bounds} and \eqref{hyp-intro:perturb}, or \eqref{hyp-intro:bounds} and \eqref{hyp-intro:period} for periodic case. Then there exists $T>0$, such that the system \eqref{eq-intro:WW} with initial data $(\eta_0,\psi_0)$ admits a unique solution
		\begin{equation*}
			(\eta-R,\psi) \in C\left( [0,T[; H^{s+\frac{1}{2}}(\mathbb{T}\times\mathbb{R}) \times H^s(\mathbb{T}\times\mathbb{R}) \right).
		\end{equation*}
		Moreover, the hypotheses \eqref{hyp-intro:bounds}, \eqref{hyp-intro:perturb} or \eqref{hyp-intro:period} are preserved for all $t\in[0,T[$.
	\end{theorem}
	
	Moreover, we shall prove that the flow map is continuous in the sense of following theorem:
	\begin{theorem}\label{thm-intro:flow-map}
		Under the hypotheses of Theorem \ref{thm-intro:main}, for all $0<r \ll 1$, we denote by $B_s\left(\eta_0,\psi_0;r\right)$ the collection of all $(\tilde{\zeta}_0,\tilde{\psi}_0)$ with
		\begin{equation*}
			\|(\eta_0 -R)- \tilde{\zeta}_0\|_{H^{s+\frac{1}{2}}(\T\times\R)} + \|\psi_0 - \tilde{\psi}_0\|_{H^{s}(\T\times\R)} < r,
		\end{equation*}
		which is a subset of $H^{s+\frac{1}{2}}(\T\times\R)\times H^s(\T\times\R)$. Then there exists $T_r>0$ (depending on $(\eta_0,\psi_0)$), such that the system \eqref{eq-intro:WW} with initial data in $B_s\left(\eta_0,\psi_0;r\right)+(R,0)$ admits a unique solution $(\tilde{\eta},\tilde{\psi})$ on $[0,T_r]$ with $(\tilde{\eta}-R,\tilde{\psi}) \in C\left( [0,T_r[; H^{s+\frac{1}{2}}(\mathbb{T}\times\mathbb{R}) \times H^s(\mathbb{T}\times\mathbb{R}) \right)$, and the hypotheses \eqref{hyp-intro:bounds}, \eqref{hyp-intro:perturb} or \eqref{hyp-intro:period} are preserved for all $t\in[0,T_r[$.
		
		Moreover, the flow map
		\begin{equation}\label{eq-intro:def-flow-map}
			\begin{array}{lccc}
				\mathfrak{F} : & B_s\left(\eta_0,\psi_0;r\right) & \rightarrow & L^\infty([0,T_r[;H^{s+\frac{1}{2}}(\T\times\R)\times H^s(\T\times\R)) \\[0.5ex]
				&(\tilde{\zeta}_0,\tilde{\psi}_0) & \mapsto & \left( \tilde{\eta}(t)-R,\tilde{\psi}(t) \right).
			\end{array}
		\end{equation}\index{F@$\mathfrak{F}$ Flow map}
		is continuous.
	\end{theorem}
	
	These well-posedness results have attracted lots of attention for planar water-wave \cite{wu1997well,wu1999well,lannes2005well,alvarez2008large,germain2012global,ionescu2014global,alazard2015golbal,ai2023two,ai2022two} in recent years. And the same problem with different geometric setting such as water drops \cite{beyer1998cauchy,christodoulou2000motion,lindblad2005well,shatah2008geometry,shatah2008apriori,shao2023cauchy,baldi2024liquid} has also been widely studied. \cite{huang2023wellposedness} provides the first result on cylinder, and the theorems above extend it by deleting the axisymmetric condition.

	\subsection{Idea of the proof}
	
	Formally speaking, we attempt to investigate the hyperbolic nature of water-wave equation \eqref{eq-intro:Euler} with \eqref{eq-intro:bdy-kinematic}, \eqref{eq-intro:bdy-pressure}, and reformulate the equivalent system \eqref{eq-intro:WW} as
	\begin{equation}\label{eq-intro:idea-main-para}
		\partial_t Y + A(Y)Y = F(Y),
	\end{equation}
	where $Y$ is a new variable defined as an elliptic operator acting on $(\eta,\psi)$ and $F(Y)$ is a (relatively) regular source term relying implicitly on $Y$. The high order linear operator $A(Y)$ also depends implicitly on $Y$ with
	$$ A(Y)^* = - A(Y). $$
	Then classical methods for hyperbolic equations allow us to construct a series of approximate solutions and prove their convergence to a solution to \eqref{eq-intro:idea-main-para} via energy estimates. With a similar argument, one may also show the uniqueness of the solution, which completes the well-posedness stated in Theorem \ref{thm-intro:main}.
	
	In the context of water-wave problem \eqref{eq-intro:WW}, there are several difficulties:\\
	
	1. How to rewrite the nonlinear terms as $A(Y)Y$ up to acceptible remainders ?
	
	2. What is the dependence of $A(Y)$ in $Y$ ? Or equivalently, how do $G(\eta)$, $N$, and $H$ depend on $(\eta,\psi)$ ?
	
	3. Does the unique solution depend continuously on time and initial data ?\\
	
	To solve these problems, we follow the idea of \cite{alazard2009paralinearization,alazard2011water} and apply the techniques of paradifferential calculus. More precisely, we shall write all the nonlinear terms as paradifferential operators acting on $(\eta,\psi)$, where the implicit dependence in $(\eta,\psi)$ will be reflected in the symbols to be calculated explicitly.
	
	\subsubsection{Paralinearization of $G(\eta)$.} 
	In order to write $G(\eta)$ as a paradifferential operator with symbol depending explicitly on $\eta$, from definition \eqref{eq-intro:def-DtN}, one can see that the core of this problem is to express normal derivative of potential $\phi$ as tangential ones, where the implicit dependence on $\eta$ is mainly hidden in Poisson's equation \eqref{eq-intro:ellip}. As in \cite{alazard2009paralinearization}, we will reformulate  \eqref{eq-intro:ellip} as an elliptic equation on a fixed domain (see \eqref{eq-pre:ellip}), where the Laplacian $\Delta_g$ relies on metric $g$, or equivalently $\eta$. Then, by decomposing the Laplacian in normal and tangential part, we can see that the difference of normal derivative of $\phi$ and a tangential operator acting on $\phi$ is governed by a hyperbolic equation, implying the desired paralinearization \eqref{eq-paralin:paralin-DtN}.
	
	A major difficulty in this step is that the coefficients are smooth in $\eta$, which has limited regularity. As a consequence, if one applies directly the calculus above, most of the remainders will have the same regularity as the principal terms. To deal with this lack of regularity, we replace $\psi$ by \textit{Alinhac's good unknown} $U = \psi - T_B\eta$ defined by \eqref{eq-paralin:good-unknown}, which is inspired by paracomposition (see Appendix \ref{subsect:compo} and \cite{alazard2009paralinearization}). Note that it is possible to recover $\psi$ from $U$ and $\eta$.
	
	\subsubsection{Paralinearization of $N$, $H$.} 
	The treatment of $N$ is easy since it can be expressed as a smooth function of $G(\eta)\psi$, $\eta$, $\psi$, and their derivatives. The desired paralinearization follows easily from classical formula \eqref{eq-para:paralin}. The same formula also works for mean curvature $H$ which is a smooth function of $\eta$ and its derivatives.
	
	\subsubsection{Symmetrization.} 
	Given the paralinearizations above, we can rewrite the main equation \eqref{eq-intro:WW} as \eqref{eq-paralin:WW}, which takes the form of \eqref{eq-intro:idea-main-para} with
	$$ A(Y) = T_V\cdot\bar{\nabla} + \mathcal{L}. $$
	The first part $T_V\cdot\bar{\nabla}$ can be roughly viewed as transport $V\cdot\nabla$ with a Lipschitzian vector $V$, and the classical arguments can be applied. The second part $\mathcal{L}$ is a matrix of paradifferential operators and the symbol of each entry is homogeneous in Fourier variable $\xi$. After simple linear transform, the matrix of symbols becomes anti-Hermitian (see Section \ref{Sect:sym}). To do the same symmetrization at operator level, it suffices to apply symbolic calculus as in Section 4 of \cite{alazard2011water}.
	
	\subsubsection{Cauchy problem.} 
	As mentioned above, we will first construct a series of approximate solutions governed by
	\begin{equation}\label{eq-intro:idea-appx-para}
		\partial_t Y_\epsilon + A_\epsilon(Y_\epsilon)Y_\epsilon = F(Y_\epsilon),
	\end{equation}
	where $A_\epsilon(Y)$ is a mollification of $A(Y)$. Note that, in order to maintain the property $A_\epsilon(Y)^* = - A_\epsilon(Y)$, this mollifier cannot be taken as a simple localization in Fourier variable. Instead, $A_\epsilon(Y)$ will be defined by inserting a well-chosen paradifferential operator depending on $\eta$ in $A(Y)$. 
	
	Another difficulty arsing from the convergence of approximate solutions $Y_\epsilon$ is that, to compare $Y_{\epsilon}$ with another $Y_{\epsilon'}$, it is necessary to calculate the Lipschitz norm of $F$. We will see in Section \ref{subsect:contin-in-eta} that the core of this problem is to study the derivative-in-$\eta$ of $G(\eta)\psi$, which is known as \textit{shape derivative}. A standard proof of this can be found in Section 3.3 of \cite{lannes2013water}, which does not work here due to some difficulties from the geometric nature of the cylinder. Inspired by \cite{beyer1998cauchy}, we first study the Hamiltonian formulation (Proposition \ref{prop-pre:Hamiltonian-formulation}) of the system \eqref{eq-intro:WW}, which then implies the desired shape derivative formula \eqref{eq-pre:shape-deri}.
	
	\subsubsection{Continuity of the solution.} It remains to check that the unique solution constructed as the limit of approximate solutions is continuous in time and initial data. It is possible to use the same strategy as in Section 6 of \cite{alazard2011water}, while, in this paper, we will use another method. In a recent paper \cite{alazard2024nonlinear}, the authors prove a nonlinear interpolation theorem (see Theorem \ref{thm-cauchy:nonlin-interpol}), which can be briefly stated as: if the flow map satisfies the contraction \eqref{eq-cauchy:nonlin-interpol-contraction} and tame estimate \eqref{eq-cauchy:nonlin-interpol-tame}, the solution will be continuous in time and initial data. The former one follows easily from the energy estimate, while the tame estimate holds due to the paradifferential calculus. In fact, all the estimates involving paradifferential operators require the index of regularity to belong to an open interval. Then it is harmless to replace these index by slightly smaller one, which gives the desired tame estimate.

	\subsection{Plan of the paper}
	
	In Section \ref{Sect:pre}, we will construct a proper change of variable to pull the Poisson's equation \eqref{eq-intro:ellip} back to a fixed domain and prove the elliptic regularity with details left to Appendix \ref{App:ellip}. Moreover, the Hamiltonian formulation will be given and, as a result, we will deduce the shape derivative formula. 
	
	In Section \ref{Sect:paralin}, we first introduce the homogeneous symbols to be studied in this paper and clarify the basic properties of paradifferential operator of these symbols (the general theory will be reviewed in Appendix \ref{App:para}). Then the paralinearization of $G(\eta)$, $N$, and $H$ will be given as well as the paralinearization of the system \eqref{eq-intro:WW}. In the end of this section, we will justify that all the remainders arising from the paralinearization are Lipschitzian in $(\eta,\psi)$.
	
	With the symmetrization proved in Section \ref{Sect:sym}, we are able to define the approximate system in Section \ref{Sect:cauchy}, where the energy estimate will also be proved. As a consequence, we deduce the convergence of approximate solutions and the uniqueness of the solution, whose continuity in time and initial data is left to the end of Section \ref{Sect:cauchy} via nonlinear interpolation.
	
	The Appendix \ref{App:sobo} is devoted to a review of Sobolev spaces defined on domains with smooth boundary which is used in the pull-back of the Poisson's equation \ref{eq-intro:ellip}. And Appendix \ref{App:var-metric} contains a direct (but complicated) variational calculus, which proves the most important identity in the Hamiltonian formulation of \eqref{eq-intro:WW}.
	
	\subsection{Notations and conventions}
	
	In what follows, we list the notations and conventions frequently used in the whole paper. And a list of symbols is available at the end of this paper.
	
	- Let $W\in C_b^\infty$ be a fixed function (usually it is a constant). We denote by $H^s_W$ the set of functions $f$ such that 
	$$ \|f\|_{H^s_W} := \|f-W\|_{H^s} < +\infty, $$
	where $H^s$ is the Sobolev norm. For example, all $\eta$ satisfying \eqref{hyp-intro:perturb} with $R>0$ belong to $H^s_R(\T\times\R)$. Furthermore, an operator $L:H^{s_0}_{W_0}\rightarrow H^{s_1}_{W_1}$ is said to be linear when $(f-W_0)\mapsto(Lf-W_1)$ is linear from $H^{s_0}$ to $H^{s_1}$.\index{H@$H^s$ Sobolev space} \index{H@$H^s_W$ Sobolev space with normalization $W$} \index{C@$C^\infty_b$ Bounded smooth functions with every derivative bounded}
	
	- We denote by $\D$ the open unit disk in 2D, namely $\{y\in\R^2:|y|<1\}$. Its boundary is denoted as $\mathbb{S}$, which will be identify as the 1D torus $\T$.\index{D@$\D$ 2D Unit disk} \index{S@$\mathbb{S}$ 1D circle} \index{T@$\T$ 1D Torus}
	
	- We say $\chi$ is a smooth truncation near $K\subset\mathbb{R}^d$, if $\chi\in C^\infty_b(\mathbb{R}^d)$ is supported in a (usually small) neighborhood of $K$ and is equal to $1$ in a smaller neighborhood of $K$.
	
	- All the large constants will be denoted as $C$ or $C_\alpha$ if this constant depends on some parameter $\alpha$. Moreover, we write $C(Q)$ for some quantity $Q>0$ when $C>0$ is a smooth increasing function of $Q$.
	
	
	- We use the Einstein summation convention: if an index appears twice, a summation in this index should be added automatically. For simplicity, theses summations will not be precised in the formulas.
	
	- We use double integral $\iint$ for the integrals in domain $\Omega(t)$ (or its pull-back by diffeomorphisms), and single integral $\int$ for those on interface $\Sigma(t)$ without precising the region of integration.
	
	- We say a linear operator $T$ is of order $m\in\R$, if it is bounded from $H^{s}$ to $H^{s-m}$ for all $s\in\R$. When an operator is of order $m$ for all $m\in\R$, it is said to be a smoothing operator.
	
	- For real number $s$, we write $s+$ for $s+\epsilon$, where $\epsilon>0$ is a small number, when the exact value of $\epsilon$ is not important. Similarly, $s-$ stands for $s-\epsilon$ with $0<\epsilon\ll 1$.

	\section{Preliminaries}\label{Sect:pre}
	
	In Section \ref{subsect:CSZ}, we have reformulated our problem as \eqref{eq-intro:WW} on fixed domain $\mathbb{T}\times\mathbb{R}$, while the implicit dependence on the shape of fluid, or equivalently $\eta$, is hidden in the definition of Dirichlet-to-Neumann operator $G(\eta)$ and the non-linear term $N$. In order to clarify this dependence, it is essential to rigorously investigate the relation between $\psi$ and $\phi$, or equivalently the Poisson's equation \eqref{eq-intro:ellip}, where the main difficulty is that the domain $\Omega(t)$ varies in time. To reduce the problem to a fixed domain, in axis-symmetric case \cite{huang2023wellposedness}, the authors use an explicit change of variable :
	\begin{equation*}
		[0,1]\times\T\times\R \rightarrow \Omega(t),\ \ (\rho,\theta,z) \mapsto (\rho\eta(t,z)\cos(\theta),\rho\eta(t,z)\sin(\theta),z),
	\end{equation*}
	where the dependence on angular variable $\theta$ can be omitted due to the symmetricity. In this way, one may rewrite the Poisson's equation \eqref{eq-intro:ellip} as an elliptic equation on $[0,1]\times\R$,
	\begin{equation}\label{eq-pre:ellip-ref}
		\left\{\begin{array}{ll}
			\nabla_{\rho,z}\cdot Q(\eta,\eta_z)\nabla_{\rho,z}\varphi = 0, & \text{in }[0,1]\times\R, \\[0.5ex]
			\varphi|_{\rho=1} = \psi, &
		\end{array}\right.
	\end{equation}
	where $Q$ is a smooth function of $\eta,\eta_z$ that can be calculated explicitly. At $\rho=0$, there is no boundary condition directly from \eqref{eq-intro:ellip}, while one may add the compatibility condition $\partial_\rho\varphi|_{\rho=0}=0$ to avoid singularity at the axis $\{\rho=0\}\subset\Omega(t)$. 
	
	Clearly, this boundary condition is not reasonable in general case (with dependence on angular variable). To overcome this difficulty, instead of looking for a generalized boundary condition, we attempt to flatten $\Omega(t)$ in an alternative way so that no extra boundary will be generated. In the study of water-drop, where $\Omega(t)$ is a perturbation of unit ball, Beyer-G{\"u}nther \cite{beyer1998cauchy} introduce a diffeomorphism from unit ball to the water-drop which equals identity near zero, up to multiple of positive constants. Inspired by this work, we shall extend $\eta$ defined on $\T\times\R$ to a positive function $\zeta = \zeta(y,z)$ on $\D\times\R$ (see \eqref{eq-pre:def-zeta}), such that in polar coordinate $(y,z) = (\rho\theta,z)$, $\zeta$ equals $\eta$ at $\rho=1$ and behaves like $\rho$ near $\rho=0$. In this way, the domain $\Omega(t)$ can be characterized by $x=\zeta(\rho\theta,z)\theta,z=z$, and there is no singularity at $\rho=0$ since $x = \zeta(\rho\theta,z)\theta$ behaves like $\rho\theta = y$ near $y=0$ (see Proposition \ref{prop-pre:chgt-of-var}).
	
	Via this change of coordinate, \eqref{eq-intro:ellip} can be reduced to an elliptic equation on $\D\times\R$, whose coefficients depend smoothly on $\eta$. The desired elliptic regularity (Proposition \ref{prop-pre:ellip-reg}) then follows from some classical arguments. As a corollary, we may obtain the high-order estimates for Dirichlet-to-Neumann operator $G(\eta)$ (Corollary \ref{cor-pre:bound-DtN}). 
	
	Furthermore, this change of variable gives a rigorous meaning of the variational calculus, which is required for Hamiltonian formulation (Proposition \ref{prop-pre:Hamiltonian-formulation}). In fact, from the definition \eqref{eq-pre:Hamiltonian} of energy Hamiltonian $\mathcal{H}$, it is clear that its variation in $\eta$ can be reduced to derivative-in-$\eta$ of metric, which is a smooth function of $\eta$ and its derivatives, due to our construction of coordinate $(x,z)=(\zeta(y,z)y,z)$. Hence, in consideration of the elliptic regularity (which guarantees all the integrands are integrable), all the derivatives-in-$\eta$ of potential and kinetic energy make sense. 
	
	As a result of Hamiltonian formulation (Proposition \ref{prop-pre:Hamiltonian-formulation}), we shall also deduce the shape derivative formula \eqref{eq-pre:shape-deri}, which provides a preliminary explanation of how Dirichlet-to-Neumann operator $G(\eta)$ depends on $\eta$ (while more delicate results will be given in Section \ref{Sect:paralin} via paradifferential calculus). 
	
	Note that, in planar water-wave case (where free surface is homotopic to hyperplane), it is easier to start by proving shape derivative formula (see Chapter 3 of \cite{lannes2013water} for instance), which implies the Hamiltonian formulation. To use this method, it is essential to find explicitly a harmonic extension of $B\delta\eta$ appearing in the right hand side of \eqref{eq-pre:shape-deri} so that one could calculate $G(\eta)(B\delta)$. However, this harmonic extension is not evident in our context, preventing us to repeat the argument for planar water-wave. That is why we begin with the Hamiltonian formulation and then prove the shape derivative formula. As a by-product, during the proof of Proposition \ref{prop-pre:shape-deri}, we shall see that the shape derivative formula is somehow equivalent to the variation of kinetic energy in $\eta$, which is the core of Hamiltonian formulation. 

	It is also worth mentioning the recent work by Baldi-Julin-Manna \cite{baldi2024liquid} on the study of liquid water drop. The authors develop an alternative method based on differential geometry to deal with the singularity at the origin. This method allows us to prove directly the shape derivative formula with little interior information involved (i.e. one can work solely on the boundary). And it could be expected to work well for jets and the planar water-wave with rough bottom (for example, the case studied in \cite{alazard2011water}). Nevertheless, these geometrical calculations require further justifications in low-regular context, which may end up with more technical details. In the mean time, the equivalence between shape derivative formula and Hamiltonian formulation is not revealed during the proof.
	
	In this section, the hypotheses \eqref{hyp-intro:bounds} and \eqref{hyp-intro:perturb} (or \eqref{hyp-intro:period}) play an important role, while we will not precise them in each statement for simplicity.
	
	\subsection{Change of coordinate}\label{subsect:chgt-of-var}
	
	As explained above, we aim to construct a diffeomorphism of the form
	\begin{equation}\label{eq-pre:chgt-of-var}
		\begin{array}{lccc}
			\iota(t) : &\mathbb{D}\times\mathbb{R} & \rightarrow & \Omega(t) \subset\mathbb{R}^3 \\[0.5ex]
			&(y,z) & \mapsto & \left( y\zeta(t,y,z), z \right)
		\end{array}
	\end{equation}\index{i@$\iota$ Diffeomorphism to straight cylinder}
	with boundary correspondence, where $\zeta$ is an explicit extension of $\eta$. To begin with, we fix a general extension $\tilde{\eta}$ of $\eta$ to $\mathbb{D}\times\mathbb{R}$, and construct $\zeta$ in a proper way so that the mapping above is a bijective and thus a diffeomorphism. For simplicity, time variable $t$ will be omitted in this part since each step below is independent of time. The construction of $\zeta$ has been given in \cite{beyer1998cauchy} for water droplets (free surface homotopic to $\mathbb{S}^d$). We will see in the sequel that it also works for jets.
	
	\subsubsection{Extension of $\eta$.} 
	Let $\eta$ be a function in $H^s_R(\mathbb{T}\times\mathbb{R})$ with $s>2$, satisfying \eqref{hyp-intro:bounds} and \eqref{hyp-intro:perturb} (or \eqref{hyp-intro:period}). Then we may define $\tilde{\eta}$, in polar coordinate $(y,z)=(\rho\theta,z)$, as
	\begin{equation}\label{eq-pre:def-tilde-eta}
		\tilde{\eta}(\rho\theta,z) := \chi_1(\rho) \hat{\chi}_0( (\rho-1)\langle D_{\theta,z} \rangle ) \left[\eta(\theta,z)-R\right] + R,
	\end{equation}
	where $\chi_0,\chi_1$ are smooth truncation near $0$ and $1$ respectively, $\chi_0$ is even and $\int\chi_0 = 1$. It is easy to check that this extension satisfies the following properties:
	\begin{lemma}\label{lem-pre:prop-tilde-eta}
		Let $s>2$. The scalar function $\tilde{\eta}$ defined by \eqref{eq-pre:def-tilde-eta} belongs to $H^{s+\frac{1}{2}}_R(\mathbb{D}\times\mathbb{R})$ with
		
		(1) $\tilde{\eta}|_{\rho=1} = \eta$;
		
		(2) $\partial_\rho \tilde{\eta}$ is uniformly continuous in $\mathbb{D}\times\mathbb{R}$ with $\partial_\rho \tilde{\eta}|_{\rho=1} = 0$;
		
		(3) The linear application $L_0: (\eta-R)\mapsto(\tilde{\eta}-R)$ is bounded from $H^{s}(\mathbb{T}\times\mathbb{R})$ to $H^{s+\frac{1}{2}}(\R^3)$, and from $C^{s-1}(\mathbb{T}\times\mathbb{R})$ to $C^{s-1}(\R^3)$;
		
		(4) There exists $0<\delta \ll 1$, such that, for all $\rho\in[1-\delta,1]$, $\tilde{\eta}>\frac{c_0}{2}$ and $|\partial_\rho\tilde{\eta}| < \frac{c_0}{4}$. Moreover, $\delta$ depends only on the $H^s$ norm of $\eta$.
		
		Recall that $c_0$ is the constant appearing in \eqref{hyp-intro:bounds}.
	\end{lemma}
	\begin{proof}
		These claims are trivial except for the boundedness in H{\"o}lder space (from $C^{s-1}(\mathbb{T}\times\mathbb{R})$ to $C^{s-1}(\R^3)$). To prove this, we first show that the multiplier $\chi_1(\rho) \hat{\chi}_0( (\rho-1)\langle D_{\theta,z} \rangle )$ is bounded from $C^0(\T\times\R)$ to $C^0(\R^3)$. The kernel of this multiplier is given by
		\begin{equation*}
			K(w,w';\rho) = \frac{1}{(2\pi)^2} \int e^{i(w-w')\cdot\xi} \chi_1(\rho) \hat{\chi}_0( (\rho-1)\langle \xi \rangle ) d\xi,
		\end{equation*}
		where $w = (\theta,z)$, $w'=(\theta',z')$, and $\xi$ is the Fourier variable associated to $(\theta,z)$. A simple calculus gives that
		\begin{equation*}
			|K(w,w';\rho)| \lesssim  \chi_1(\rho)|1-\rho|^{-2},
		\end{equation*}
		while
		\begin{align*}
			\left| (w-w') K(w,w';\rho) \right| =& \left| \frac{1}{(2\pi)^2} \int e^{i(w-w')\cdot\xi} \chi_1(\rho) \partial_\xi\left(\hat{\chi}_0( (\rho-1)\langle \xi \rangle ) \right) d\xi \right| \\
			\lesssim& \chi_1(\rho) |\rho-1| \left| \int \hat{\chi}_0'( (\rho-1)\langle \xi \rangle ) \frac{\xi}{\langle\xi\rangle} d\xi \right| \lesssim \chi_1(\rho)|1-\rho|^{-1}.
		\end{align*}
		Similarly, one gains $|1-\rho|$ from each product with $(w-w')$, and thus for all $N \gg 1$,
		\begin{equation}\label{eq-pre:tilde-eta-esti-kernel}
			\left| K(w,w';\rho) \right| \lesssim \chi_1(\rho)|1-\rho|^{-2} \left\langle \frac{w-w'}{|1-\rho|} \right\rangle^{-N},
		\end{equation}
		which proves the boundedness of $L_0$ from $L^\infty(\T\times\R)$ to $L^\infty(\R^3)$. Provided that $\eta\in C^0(\T\times\R)$, then continuity of $\tilde{\eta} = L_0(\eta-R)+R$ on $\{\rho\neq 1\}$ is obvious due to the continuity of $K$ in $(w,w';\rho)$ and the estimate \eqref{eq-pre:tilde-eta-esti-kernel} of $K$. To obtain the continuity at $\{\rho=1\}$, we notice that
		\begin{equation*}
			\int K(w,w';\rho) dw' = \chi_1(\rho),
		\end{equation*}
		which, together with \eqref{eq-pre:tilde-eta-esti-kernel}, guarantees that $\tilde{\eta}(\rho)\rightarrow\eta$ as $\rho\rightarrow 1$, locally uniformly in $(\theta,z)$ (see Section 1.2.4 of \cite{grafakos2008classical}).
		
		For high order regularity, we observe that the derivatives in $w=(\theta,z)$ is commutative with $L_0$, i.e.
		\begin{equation*}
			\|\partial_w^\alpha \tilde{\eta}\|_{L^\infty(\R^3)} = \|L_0(\partial_w^\alpha \eta) \|_{L^\infty(\R^3)} \lesssim \|\partial_w^\alpha \eta\|_{L^\infty(\T\times\R)},\ \ \forall \alpha\in\N^2.
		\end{equation*}
		As for the derivatives in $\rho$, one can see from the definition \eqref{eq-pre:def-tilde-eta} that each derivative in $\rho$ leads to a $\langle D_{w} \rangle$, reducing the problem to derivatives in $w$. In conclusion, we have proved that $L_0$ is bounded from $C^k(\T\times\R)$ to $C^k(\R^3)$ for all $k\in\N$ and the case of non-integer $s$ follows from the interpolation between H{\"o}lder spaces (see Section 1.1.1 of \cite{lunardi2018interpolation}).
	\end{proof}
	Here, we have managed to define $\tilde{\eta}$ on $\mathbb{R}^3$, while we only consider its restriction in $\mathbb{D}\times\mathbb{R}$. Note that this construction is not unique and the next step is valid for any $\tilde{\eta}$ verifying Lemma \ref{lem-pre:prop-tilde-eta}.
	
	\subsubsection{Construction of $\zeta$.} \index{z@$\zeta$ Extension of $\eta$}
	Let $\chi \in C^\infty_c(\mathbb{R})$ be an even function decreasing on $[0,+\infty[$, supported in $]-1,1[$, and equal to $1$ on $[-1+\delta,1-\delta]$ with $0<\delta \ll 1$. Then we define
	\begin{equation}\label{eq-pre:def-zeta}
		\zeta(r\theta,z) := (1-\chi)(\rho)\tilde{\eta}(\rho\theta,z) + \epsilon \chi(\rho) M\eta,
	\end{equation}
	where $0<\epsilon\ll 1$ is a constant depending on $c_0,C_0$ in \eqref{hyp-intro:bounds} and $M\eta\in\mathbb{R}$ is the mean of $\eta$. In perturbative case, when $s>2$, \eqref{hyp-intro:perturb} implies that
	\begin{equation*}
		M\eta := \lim_{l\rightarrow+\infty} \frac{1}{4\pi l} \int_{-l}^{l}\int_{0}^{2\pi} \eta(\theta,z) d\theta dz = R.
	\end{equation*}
	In the periodic case, $M\eta$ is the usual average of $\eta$ on $\mathbb{T}^2$. Note that \eqref{hyp-intro:bounds} guarantees $c_0<M\eta<C_0$.
	
	\begin{proposition}\label{prop-pre:chgt-of-var}
		Let $s>2$, $\delta$ be chosen as in (4) of Lemma \ref{lem-pre:prop-tilde-eta}, and $\epsilon = \frac{c_0}{4C_0}$. Then the mapping $\iota$ defined by \eqref{eq-pre:chgt-of-var} is bijective with boundary correspondence $\iota(\mathbb{S}\times\R) = \Sigma$. The Jacobian of $\iota$ is bounded from below by 
		\begin{equation}\label{eq-pre:jacobian-det-low-bd}
			\det J := \det(\partial_\alpha\iota^\beta) \ge \frac{c_0^3}{4C_0},
		\end{equation}\index{J@$J$ Jacobian matrix of $\iota$}
		where $c_0,C_0$ are constants in hypothesis \eqref{hyp-intro:bounds}. Moreover, the linear application
		\begin{equation*}
			L:(\eta-R)\mapsto\left(\zeta- R_\epsilon(\rho)\right),
		\end{equation*}
		is bounded from $H^{s}(\mathbb{T}\times\mathbb{R})$ to $H^{s+\frac{1}{2}}(\mathbb{D}\times\mathbb{R})$, and from $C^{s-1}(\mathbb{T}\times\mathbb{R})$ to $C^{s-1}(\mathbb{D}\times\mathbb{R})$, with $R_\epsilon$ defined by
		\begin{equation}\label{eq-pre:def-R-eps}
			R_\epsilon(\rho):= R -(1-\epsilon)\chi(\rho)R.
		\end{equation} 
	\end{proposition}
	\begin{proof}
		The boundedness of $L$ is a direct consequence of Lemma \ref{lem-pre:prop-tilde-eta}. To prove the bijectivity of $\iota$, from the definition \eqref{eq-pre:chgt-of-var}, it suffices to check that, in polar coordinate $(y,z) = (\rho\theta,z)$, $\partial_\rho(\rho\zeta)>0$ for all $0 \leqslant \rho\leqslant 1$. By definition \eqref{eq-pre:def-zeta} of $\zeta$, we have
		\begin{equation*}
			\partial_\rho(\rho\zeta) = (1-\chi)(\rho)\left( \tilde{\eta} + \rho\partial_\rho\tilde{\eta} \right) - \rho\chi'(\rho)(\tilde{\eta}-\epsilon M\eta) + \epsilon\chi(\rho) M\eta.
		\end{equation*}
		The first term is positive since it is supported for $1-\delta\leqslant r \leqslant 1$, where $\tilde{\eta} > c_0/2$ and $|\rho\partial_\rho\tilde{\eta}|<c_0/4$ due to (4) of Lemma \ref{lem-pre:prop-tilde-eta}. The positivity of the second term follows from the fact that $\epsilon M\eta \le\epsilon C_0 = c_0/4$ and $\tilde{\eta}>c_0/2$ on $\Supp{\chi'}\subset[1-\delta,1]$. The last term is trivially bounded from below by $\epsilon\chi(\rho)c_0$. To sum up, for all $\rho \ge 0$,
		\begin{equation}\label{eq-pre:low-bd-rhozeta}
			\partial_{\rho}(\rho\zeta) \ge \frac{c_0}{4}(1-\chi)(\rho) - \frac{c_0}{4}\rho\chi'(\rho) + \epsilon c_0\chi(\rho) \ge \frac{c_0^2}{4C_0}.
		\end{equation}
		Thus, by definition \eqref{eq-pre:chgt-of-var} of $\iota$, the Jacobian of $\iota$ admits the lower bound $\det J = \zeta\partial_\rho(\rho\zeta) \ge c_0^3/(4C_0)$.
	\end{proof}
	For simplicity, in the sequel, we shall ignore the normalizations and say that $\zeta=L\eta$ is linear as in periodic case. As the extension $\tilde{\eta}$ defined by \eqref{eq-pre:def-tilde-eta}, $\zeta$ is well-defined in $H^{s+\frac{1}{2}}(\mathbb{R}^3)$, up to $C^\infty_b(\R^3)$ normalization, whose restriction in $\mathbb{D}\times\mathbb{R}$ will be studied. Before entering the next section, we introduce some conventions and notations to be used frequently in the following sections.
	
	- We use Latin letters for the indices involving $y\in\mathbb{R}^2$ and Greek letters for those related to $(y,z)\in\mathbb{R}^3$.
	
	- We use $u\cdot v$ for the scalar product of vectors in Euclidean spaces, while $uv^T$ should be understood as the matrix product. Besides, all the vectors are columns if there is no further specifications.
	
	
	- Let $J=(a_{\alpha\beta})$ be the Jacobian matrix of $\iota$ defined by \eqref{eq-pre:chgt-of-var}. It is easy to see that
	\begin{equation}\label{eq-pre:jacobian-mat}
		J = \left(\begin{array}{cc}
			J_0 & J_1 \\[0.5ex]
			0 & 1
		\end{array}\right)
		=\left(\begin{array}{cc}
			\zeta + y\nabla_y^T\zeta & y\zeta_z \\[0.5ex]
			0 & 1
		\end{array}\right),
	\end{equation}
	whose inverse is denoted by
	\begin{equation}\label{eq-pre:jacobian-mat-inverse}
		J^{-1} = (a^{\alpha\beta}) = \left(\begin{array}{cc}
			J_0^{-1} & -J_0^{-1}J_1 \\[0.5ex]
			0 & 1
		\end{array}\right).
	\end{equation}
	
	- Let $(g_{\alpha\beta}) := J^T J$\index{g@$g_{\alpha\beta}$ Metric tensor of $\iota$} be the metric tensor whose inverse is denoted by $(g^{\alpha\beta}) = J^{-1}J^{-T}$\index{g@$g^{\alpha\beta}$ Inverse of $(g_{\alpha\beta})$}. We use $g$\index{g@$g$ Determinant of $(g_{\alpha\beta})$} to represent the determinant of $(g_{\alpha\beta})$, which is equal to
	\begin{equation}\label{eq-pre:jacobian-det}
		g := \det(g_{\alpha\beta}) = \left(\det J\right)^2 = \left(\det J_0\right)^2 = \left[ \zeta\left(\zeta+y\cdot\nabla_y\zeta\right) \right]^2 = \zeta^2\left(\partial_\rho(\rho\zeta)\right)^2 >0.
	\end{equation}
	Note that its restriction at $\rho=1$ is $\eta^4$ since $\partial_\rho\zeta|_{\rho=1}=0$ due to our construction.
	
	- We denote by $\nabla_g$ \index{n@$\nabla_g$ Pull-back of $\nabla_{x,z}$ by $\iota$} the pull-back of $\nabla_{x,z}$ by $\iota$, namely
	\begin{equation}\label{eq-pre:def-nabla-g}
		\nabla_g = J^{-T}\nabla_{y,z}.
	\end{equation}\index{D@$\Delta_g$ Pull-back of $\Delta_{x,z}$ by $\iota$}
	Moreover the pull-back of Laplacian $\Delta_{x,z}$ reads
	\begin{equation}\label{eq-pre:def-lap-g}
		\Delta_g = \frac{1}{\sqrt{g}} \partial_\alpha \left( \sqrt{g}g^{\alpha\beta} \partial_\beta \right),
	\end{equation}
	which is clearly uniform elliptic on $\D\times\R$.
	
	- Let $n_g$ be the pull-back of the conormal vector $n$ defined by \eqref{eq-intro:normal-dirction}. One may check that
	\begin{equation}\label{eq-pre:def-n-g}
		n_g = J^{-T}Y,
	\end{equation}\index{n@$n_g$ Pull-back of the conormal vector $n$}
	where $Y=(y_1,y_2,0)^{T}$.
	
	- In polar coordinate $(y,z) = (\rho\theta,z)$, the restriction of $\nabla_{g}$ at $\{\rho=1\}$ is 
	\begin{equation}\label{eq-pre:nabla-g-polar}
		\nabla_g = e_\rho \frac{1}{\eta}\partial_\rho + e_{\theta}\left( \frac{1}{\eta}\partial_{\theta} - \frac{\eta_\theta}{\eta^2}\partial_\rho \right) + e_{z}\left( \partial_{z} - \frac{\eta_z}{\eta}\partial_\rho \right),
	\end{equation}
	where $(e_\rho,e_\theta,e_z)$ is the orthogonormal basis associated to polar variables $(\rho,\theta,z)$. Moreover, $n_g$ can be expressed as
	\begin{equation}\label{eq-pre:n-g-polar}
		n_g = \frac{1}{\eta} \left( e_\rho - \frac{\eta_\theta}{\eta}e_\theta - \eta_z e_z \right),
	\end{equation}
	while $B,V,N$ defined in \eqref{eq-intro:def-B}, \eqref{eq-intro:def-V}, \eqref{eq-intro:def-N} read
	\begin{align}
		B = B(\psi) =& \frac{1}{\eta} \varphi_\rho|_{\rho=1}, \label{eq-pre:B-polar} \\
		V = V(\psi) =& \left.\left[ \frac{1}{\eta}\left(\varphi_\theta - \frac{\eta_\theta}{\eta}\varphi_\rho\right)e_\theta + \left(\varphi_z - \frac{\eta_z}{\eta}\varphi_\rho \right) e_z \right]\right|_{\rho=1}, \label{eq-pre:V-polar} \\
		N = N(\psi) =& BV\cdot\left(\frac{\eta_\theta}{\eta}e_\theta + \eta_z e_z\right) + \frac{|V|^2-B^2}{2}, \label{eq-pre:N-polar}
	\end{align}
	where $\varphi = \phi\circ\iota$. As a result, the Dirichlet-to-Neumann operator defined by \eqref{eq-intro:def-DtN} can be expressed as
	\begin{equation}\label{eq-pre:DtN-chgt-of-var}
		G(\eta)\psi = \left. \eta n_g\cdot\nabla_g\varphi \right|_{\rho=1} = \left. \eta (J^{-T}Y)\cdot\nabla_g\varphi \right|_{r=1} = B - V\cdot\left( \frac{\eta_\theta}{\eta},\eta_z \right).
	\end{equation}
    In addition, a simple calculation gives
	\begin{equation}\label{eq-pre:deri-of-psi}
		\left(e_\theta \frac{\partial_\theta}{\eta}+e_z \partial_z\right)\psi = V + B \left(e_\theta \frac{\eta_\theta}{\eta}+e_z \eta_z\right).
	\end{equation}

	\subsection{Elliptic regularity}\label{subsect:ellip-reg}
	
	In this section, we shall solve \eqref{eq-intro:ellip} via variational formulation and illustrate the regularity of $\phi$ from those of $(\eta,\psi)$ with the help of diffeomorphism constructed in \eqref{eq-pre:chgt-of-var}. As a consequence, a rigorous definition of the Dirichlet-to-Neumann operator $G(\eta)$, which is defined formally by \eqref{eq-intro:def-DtN}, and some high order estimates will also be established.
		
	In the rest of this paper, we denote by $\varphi$ the pull-back of $\phi$ by change of variable $\iota$. By definition \eqref{eq-intro:def-psi} of $\psi$, it is clear that in polar coordinate $(y,z)=(\rho\theta,z)$,
	\begin{equation*}
		\psi(\theta,z) = \varphi|_{\rho=1} = \varphi(\theta,z).
	\end{equation*}
	From \eqref{eq-intro:ellip}, $\varphi$ is the solution to
	\begin{equation}\label{eq-pre:ellip}
		\left\{\begin{array}{l}
			\Delta_g\varphi = 0, \\[0.5ex]
			\varphi|_{\rho=1} = \psi \in H^{s_0}(\T\times\R).
		\end{array}\right.
	\end{equation}
	Let $\Psi\in H^{s_0+\frac{1}{2}}(\D\times\R)$ be an extension of $\psi$ supported near $\rho=1$. One may repeat the procedure in previous sections (construction of $\tilde{\eta}$ from $\eta$) to obtain the desired $\Psi$. Then, by writing $\varphi = \Phi' + \Psi$, the equation above can be restated as the following variational problem
	\begin{equation}\label{eq-pre:ellip-variation}
		\iint \nabla_g\Phi' \cdot \nabla_g H \sqrt{g} dydz = - \iint \nabla_g\Psi \cdot \nabla_g H \sqrt{g} dydz,\ \ \forall H\in H^1_0(\D\times\R).
	\end{equation}
	Provided that $\eta \in H^{s+\frac{1}{2}-}_R(\T\times\R)$ with $s>\frac{3}{2}$, we have $\zeta\in H^{s+1-}_{R_\epsilon}(\D\times\R)$ from Proposition \ref{prop-pre:chgt-of-var} and thus $g_{\alpha\beta},g^{\alpha\beta},g \in L^\infty(\D\times\R)$. A classical argument by Lions–Lax–Milgram theorem (see for example \cite{brezis2011functional}) guarantees a unique solution $\Phi' \in H^{1}_0(\D\times\R)$ to the variational problem \eqref{eq-pre:ellip-variation}, which yields $\varphi\in H^{1}(\D\times\R)$. Moreover, by choosing $H=\Phi'$, we obtain the estimate 
	\begin{equation*}
		c_1 \inf_{\D\times\R}\sqrt{g} \|\Phi'\|_{H^1_0(\D\times\R)}^2 \le C_1 \sup_{\D\times\R}\sqrt{g} \|\Phi'\|_{H^1_0(\D\times\R)} \|\Psi\|_{H^1(\D\times\R)},
	\end{equation*}
	where $0<c_1<C_1$ satisfies $c_1 < (g^{\alpha\beta}) < C_1$. Under the hypothesis \eqref{hyp-intro:bounds}, by using the definition $(g^{\alpha\beta}) = J^{-1} J^{-T}$ with $J$ defined by \eqref{eq-pre:jacobian-mat} and lower bound \eqref{eq-pre:jacobian-det-low-bd}, we have
	\begin{equation*}
		\|\Phi'\|_{H^1_0(\D\times\R)} \le C_1' \|\Psi\|_{H^1(\D\times\R)},
	\end{equation*}
	for some $C_1'$ depending only on $c_0, C_0$ appearing in the hypothesis \eqref{hyp-intro:bounds}.
	
	Consequently, we are able to define $G(\eta)$ by duality:
	\begin{equation}\label{eq-pre:def-DtN-rig}
		\int G(\eta)\psi h d\theta dz := \iint \nabla_g\varphi \cdot \nabla_g H \sqrt{g} dydz,\ \ \forall H\in H^1(\D\times\R),\ h=H|_{\rho=1}.
	\end{equation}
	An application of the trace theorem and duality of Sobolev spaces implies the following estimate
	\begin{equation}\label{eq-pre:bound-DtN-low-order}
		\|G(\eta)\psi\|_{H^{-\frac{1}{2}}(\T\times\R)} \leqslant C\left(\|\eta\|_{H^{s+\frac{1}{2}-}_{R}(\T\times\R)}\right) \|\psi\|_{H^{\frac{1}{2}}(\T\times\R)},\ \ \forall s>\frac{3}{2}.
	\end{equation}
	
	In order to obtain high order estimates, we need to prove the regularity of the elliptic equation \eqref{eq-pre:ellip} or its variational form \eqref{eq-pre:ellip-variation}. 
	\begin{proposition}\label{prop-pre:ellip-reg}
		Let $\eta\in H^{s+\frac{1}{2}-}_R(\T\times\R)$ and $\psi\in H^{s_0}(\T\times\R)$ with s>$\frac{3}{2}$ and $\frac{1}{2}\le s_0 \le s$. Then the unique solution $\varphi = \Phi'+\Psi$ to \eqref{eq-pre:ellip-variation} verifies
		\begin{equation}\label{eq-pre:ellip-reg}
			\|\varphi\|_{H^{s_0+\frac{1}{2}}(\D\times\R)} \le C\left(\|\eta\|_{H^{s+\frac{1}{2}-}_R(\T\times\R)}\right) \|\psi\|_{H^{s_0}(\T\times\R)},
		\end{equation}
		where $C>0$ is a smooth increasing function.
	\end{proposition}
	
	The proof of this proposition is classic and we leave it to Appendix \ref{App:ellip}. As a corollary, we are able to generalize the estimate \eqref{eq-pre:bound-DtN-low-order} for high order Sobolev norms,
	\begin{corollary}\label{cor-pre:bound-DtN}
		Let $\eta\in H^{s+\frac{1}{2}-}_R(\T\times\R)$ $s>\frac{3}{2}$. Then for all $s_0\in[\frac{1}{2},s]$, we have
		\begin{equation}\label{eq-pre:bound-DtN}
			\left\|G(\eta)\right\|_{\mathcal{L}(H^{s_0}(\T\times\R);H^{s_0-1}(\T\times\R))} \leqslant C\left(\|\eta\|_{H^{s+\frac{1}{2}-}_R}\right),
		\end{equation}
		where $C>0$ is a smooth increasing function.
	\end{corollary}

	\subsection{Hamiltonian structure}\label{subsect:Ham}
	
	In this section, we investigate the Hamiltonian structure of \eqref{eq-intro:WW}. Recall that for standard water-wave equation where the interface is given by $\{(x,y)\in\R^{d+1}:y=\eta(x)\}$, the Hamiltonian $\mathcal{H}$ is taken as the total energy, which is preserved in time and can be written in terms of $\eta$ and $\psi$ (restriction of scalar potential), and the Hamiltonian formulation reads
	\begin{equation*}
		\left\{\begin{array}{l}
			\eta_t= \frac{\delta\mathcal{H}}{\delta\psi}, \\[0.5ex]
			\psi_t = -\frac{\delta\mathcal{H}}{\delta\eta}.
		\end{array}\right.
	\end{equation*}\index{d@$\frac{\delta}{\delta\eta}$ Variation in $\eta$}
	This formulation has been firstly discovered by Zakharov \cite{zakharov1968stability}, while the case of two-phase fluid is studied by Benjamin-Bridges \cite{benjamin1997reappraisal}. For fluids with nonzero vorticity, Castro-Lannes \cite{castro2015hamiltonian} formally prove the Hamiltonian equation for the triple $(\eta,\psi,\omega)$, where $\psi$ can be defined by projection and $\omega$ is the vorticity. 
	
	In the case of sphere-like interface, Beyer-G{\"u}nther \cite{beyer1998cauchy} pointed out that one should modify $\psi$ by $p=\eta^{d}\psi$, where $d$ is the dimension of spherical interface ( perturbation of unit sphere $\mathbb{S}^d$). Inspired by this, we turn to the new variable $(\eta,p) := (\eta,\eta\psi)$ satisfying
	\begin{equation}\label{eq-pre:WW-alt}
		\left\{\begin{array}{l}
			\eta_t = G(\eta)\psi, \\[0.5ex]
			-p_t = -\psi G(\eta)\psi + \eta \left( \sigma \left(H-\frac{1}{2R}\right) + N \right).
		\end{array}\right.
	\end{equation}
	due to equation \eqref{eq-intro:WW}.
	
	The total energy of the system \eqref{eq-intro:WW} (or \eqref{eq-pre:WW-alt}) is
	\begin{equation}\label{eq-pre:Hamiltonian}
		\mathcal{H} := E_k + E_p.
	\end{equation}\index{E@$E_k$ Kinetic energy}\index{E@$E_p$ Potential energy}\index{H@$\mathcal{H}$ Hamiltonian}
	Here $E_k$ is kinetic energy
	\begin{equation}\label{eq-pre:def-energy-kin}
		E_k := \frac{1}{2}\iint_{\Omega(t)}\left|\nabla_{x,z}\phi\right|^2dxdz = \frac{1}{2}\iint_{D\times\R}\left|\nabla_{g}\varphi\right|^2\sqrt{g}dydz = \frac{1}{2}\int_{\T\times\R} \psi \left(\eta G(\eta)\right) \psi d\theta dz,
	\end{equation}
	where the last equality follows from divergence theorem and $\Delta_g\varphi=0$ (see \eqref{eq-pre:def-DtN-rig}). $E_p$ is potential energy 
	\begin{equation}\label{eq-pre:def-energy-pot}
		E_p := \frac{\sigma}{2} A = \frac{\sigma}{2} \int_{\T\times\R} \left[ \eta\left( \sqrt{1+\left(\frac{\eta_\theta}{\eta}\right)^2+\eta_z^2} -1 \right) - \frac{|\eta-R|^2}{2R} \right] d\theta dz,
	\end{equation}
	where $A$ is the (normalized) area of interface $\Sigma(t)$ and $R$ is the constant from \eqref{hyp-intro:perturb}. In the periodic case, these normalization can be omitted.
	
	By regarding $\mathcal{H}$ as a function on $\eta$ and $p=\eta\psi$, we have the following Hamiltonian formulation.
	\begin{proposition}\label{prop-pre:Hamiltonian-formulation}
		Let $1\le s_0 \le s$ and $s>\frac{3}{2}$. For all $\eta\in H^{s+\frac{1}{2}}_R(\T\times\R)$, $\delta\eta\in H^{s+\frac{1}{2}}(\T\times\R)$ and $p,\delta p\in H^{s_0}(\T\times\R)$, we have
		\begin{align}
			& \left.\frac{d}{d\epsilon}\right|_{\epsilon=0} \mathcal{H}(p+\epsilon\delta p,\eta) = \int \eta_t \delta p d\theta dz, \label{eq-pre:Hamiltonian-var-p} \\
			& \left.\frac{d}{d\epsilon}\right|_{\epsilon=0} \mathcal{H}(p,\eta+\epsilon\delta\eta) = \int (-p_t) \delta\eta d\theta dz, \label{eq-pre:Hamiltonian-var-eta}
		\end{align}
		where $\eta_t$ and $-p_t$ should be understood as the right hand side of \eqref{eq-pre:WW-alt}.
	\end{proposition}
	
	Note that the results in previous sections imply that $g,(g_{\alpha\beta}),(g^{\alpha\beta})\in C^{s-\frac{3}{2}}(\D\times\R)$ and $\varphi\in H^{s_0+\frac{1}{2}}(\D\times\R)$, so as their variations in $p$ or $\eta$. These ensure that all the integrands below are integrable.
	
	\subsubsection{Variation in $p$.} 
	We first prove \eqref{eq-pre:Hamiltonian-var-p}. In this paragraph, we fix $\eta$ and use the notation
	\begin{equation*}
		\delta Q := \left.\frac{d}{d\epsilon}\right|_{\epsilon=0} Q(p+\epsilon\delta p)
	\end{equation*}
	for any quantity depending on $p$. Since the potential energy \eqref{eq-pre:def-energy-pot} depends only on $\eta$, the variation can be applied merely for kinetic energy \eqref{eq-pre:def-energy-kin},
	\begin{align*}
		\delta\mathcal{H} = \delta E_k =& \iint \nabla_g^T\varphi \nabla_g \delta\varphi \sqrt{g} dy dz \\
		=& \iint \left( \sqrt{g} J^{-1}\nabla_g\varphi \right) \cdot \nabla_{y,z} \delta\varphi dy dz \\
		=& \int Y \cdot\left( \sqrt{g} J^{-1}\nabla_g\varphi \right) \delta\psi d\theta dz,
	\end{align*}
	where $Y = (y_1,y_2,0)^T$ and the last equality follows from divergence theorem with $\delta \psi$ the restriction of $\delta\varphi$. By \eqref{eq-pre:jacobian-det} and \eqref{eq-pre:DtN-chgt-of-var}, we have
	\begin{equation*}
		Y \cdot\left( \sqrt{g} J^{-1}\nabla_g\varphi \right) \delta\psi  = (J^{-T}Y)\cdot\nabla_g\varphi \eta^2 \delta\psi = G(\eta)\psi \eta \delta\psi = G(\eta)\psi \delta p,
	\end{equation*}
	Note that the last equality follows from $\psi = p/\eta$ and the fact that $\eta$ is fixed. This completes the proof of \eqref{eq-pre:Hamiltonian-var-p}.
	
	\subsubsection{Variation in $\eta$.} 
	Let us fix $p$ and denote
	\begin{equation*}
		\delta Q := \left.\frac{d}{d\epsilon}\right|_{\epsilon=0} Q(\eta+\epsilon\delta\eta)
	\end{equation*}
	for any quantity depending on $\eta$. The variation of potential part $E_p$ can be calculated directly,
	\begin{equation}\label{eq-pre:var-energy-pot}
		\delta E_p = \frac{\sigma}{2} \delta A = \sigma \int \eta \left(H-\frac{1}{2R}\right) \delta\eta d\theta dz,
	\end{equation}
	where $H$ is the mean curvature defined in \eqref{eq-intro:mean-curv}. For the kinetic part, we have
	\begin{lemma}\label{lem-pre:var-of-kin-energy}
		Let $1\le s_0 \le s$ and $s>\frac{3}{2}$. For all $\eta,\delta\eta\in C^{s-\frac{1}{2}}(\T\times\R)$ and $p,\delta p\in H^{s_0}(\T\times\R)$, we have
		\begin{equation}\label{eq-pre:var-kin-energy}
			\delta E_k := \left.\frac{d}{d\epsilon}\right|_{\epsilon=0}E_k(p,\eta+\epsilon\delta\eta) = \int \left(-\psi G(\eta)\psi + \eta N\right) \delta\eta d\theta dz.
		\end{equation}
	\end{lemma}
	Note that, by Sobolev embedding, \eqref{eq-pre:var-kin-energy} and \eqref{eq-pre:var-energy-pot} implies \eqref{eq-pre:Hamiltonian-var-eta}. Here we may temporarily ignore the hypothesis \eqref{hyp-intro:perturb} since all the integrals below make sense whenever $\eta,\delta\eta$ (and $\zeta,\delta\zeta$) have Lipschitz regularity, which is ensured by the condition $\eta,\delta\eta\in C^{s-\frac{1}{2}}(\T\times\R)$, $s>\frac{3}{2}$.
	\begin{proof}
		By definition \eqref{eq-pre:def-energy-kin} of kinetic energy $E_k$, we have
		\begin{align}
			\delta E_k =& I_1 + I_2, \\
			I_1 =& \frac{1}{2} \iint \partial_\alpha\varphi \delta(\sqrt{g}g^{\alpha\beta}) \partial_\beta\varphi dydz, \label{eq-pre:def-I-1} \\
			I_2 =& \iint \nabla_g\varphi\cdot\nabla_g\delta\varphi \sqrt{g} dydz. \label{eq-pre:def-I-2}
		\end{align}
		As before, by divergence theorem, $I_2$ equals
		\begin{equation}\label{eq-pre:var-I-2}
			I_2 = \int \left(G(\eta)\psi\right) \eta \delta\psi d\theta dz = - \int \psi \left(G(\eta)\psi\right) \delta\eta d\theta dz,
		\end{equation}
		where the last equality follows from $\delta \psi = \delta (p/\eta) = -\delta\eta p/(\eta^2) = -\psi\delta\eta/\eta$.
		
		\begin{remark}
			Our final goal is to establish the Hamiltonian formulation w.r.t. variable $(\eta,p)$. Therefore, during the calculation of variation in $\eta$ above, the kinetic energy $E_k$, which is a function of $(\eta,\psi)$ due to definition \eqref{eq-pre:def-energy-kin}, should be regarded as a function of $(\eta,p)$, namely
			\begin{equation*}
				E_k = E_k\left(\eta,\frac{p}{\eta}\right) = \frac{1}{2} \int p G(\eta)\left(\frac{p}{\eta}\right) dx.
			\end{equation*}
			And the variation in $\eta$ should be understood for fixed $p$,
			\begin{equation*}
				\delta E_k = \left.\frac{d}{d\epsilon}\right|_{\epsilon=0} \frac{1}{2} \int p G(\eta+\epsilon\delta\eta)\left(\frac{p}{\eta+\epsilon\delta\eta}\right) dx
			\end{equation*}
			
			However, in most parts of this paper, we are interested in variable $(\eta,\psi)$ instead of $(p,\psi)$. Then, the derivative in $\eta$ should be defined with fixed $\psi$ instead of fixed $p$ (in such case, $p=\eta\psi$ may vary in $\eta$). The main difference is that, when $\psi$ is fixed, $I_2$ is automatically zero due to the first equality of \eqref{eq-pre:var-I-2}. This is the case of shape derivative to be studied in Section \ref{subsect:shape-deri}.
		\end{remark}
		
		As for $I_1$, a direct calculus gives that
		\begin{equation}\label{eq-pre:var-in-metric}
			\begin{aligned}
				\frac{1}{2}\delta(\sqrt{g}g^{\alpha\beta})\partial_\alpha\varphi\partial_\beta\varphi =& \frac{1}{2} \nabla_{y,z}^T \left( J^{-1}Y \left|\nabla_g\varphi\right|^2 \sqrt{g}\delta\zeta \right) - \nabla_{y,z}^T\left( \sqrt{g}J^{-1}\nabla_g\varphi\ Y^T\nabla_g\varphi \delta\zeta \right) \\
				& + \sqrt{g}\Delta_g\varphi Y\cdot\nabla_g\varphi \delta\zeta,
			\end{aligned} 
		\end{equation}
		the proof of this is left to Lemma \ref{lem-tech:variation-of-metric}. The technical identity above implies, via divergence theorem and $\Delta_g\varphi=0$, that
		\begin{equation*}
			I_1 = \frac{1}{2}\int Y^{T}J^{-1}Y \left|\nabla_g\varphi\right|^2 \eta^2\delta\eta d\theta dz - \int Y^TJ^{-1}\nabla_g\varphi\ Y^T\nabla_g\varphi \eta^2\delta\eta  d\theta dz.
		\end{equation*}
		Note that, at boundary $\rho=1$, we have $\sqrt{g}=\eta^2$, as it has been noticed after \eqref{eq-pre:jacobian-det}. By formula \eqref{eq-pre:jacobian-mat-inverse} of $J^{-1}$, definition \eqref{eq-pre:jacobian-mat} of $J_0$, and $Y=(y,0)^T$, we have,
		\begin{equation*}
			Y^{T}J^{-1}Y = \frac{\zeta}{\sqrt{g}}|y|^2.
		\end{equation*}
		The expression \eqref{eq-pre:nabla-g-polar} yields that, as $\rho=1$,
		\begin{equation*}
			\left|\nabla_g\varphi\right|^2 = B^2+|V|^2,\ Y^T\nabla_g\varphi = B.
		\end{equation*}
		In addition, we have seen in \eqref{eq-pre:DtN-chgt-of-var} that the trace of $Y^TJ^{-1}\nabla_g\varphi$ is $G(\eta)\psi/\eta$. Consequently,
		\begin{equation}\label{eq-pre:var-I-1}
			I_1 = \frac{1}{2}\int \eta (B^2+|V|^2) \delta\eta d\theta dz - \int \eta G(\eta)\psi B\delta\eta  d\theta dz = \int \eta N \delta\eta d\theta dz,
		\end{equation}
		thanks to \eqref{eq-pre:B-polar}, \eqref{eq-pre:V-polar}, \eqref{eq-pre:N-polar}, and \eqref{eq-pre:DtN-chgt-of-var}.
	\end{proof}
	
	In conclusion, the variational identity \eqref{eq-pre:Hamiltonian-var-eta} follows from \eqref{eq-pre:var-energy-pot} and \eqref{eq-pre:var-kin-energy}. Formally, we may write \eqref{eq-pre:Hamiltonian-var-p} and $\eqref{eq-pre:Hamiltonian-var-eta}$ in classical form
	\begin{equation*}
		\left\{\begin{array}{l}
			\eta_t= \frac{\delta\mathcal{H}}{\delta p}, \\[1ex]
			p_t = -\frac{\delta\mathcal{H}}{\delta\eta}.
		\end{array}\right.
	\end{equation*}
	As a consequence of this Hamiltonian formulation, the total energy $\mathcal{H}$ is preserved in time. Note that for Cauchy problem, there is no difference between $(\eta,\psi)$ and $(\eta,p)$, since \eqref{hyp-intro:bounds} together with $\eta \in H^{s+\frac{1}{2}}_R(\T\times\R)$, $s>3/2$ guarantees that $\psi \in H^{s}(\T\times\R)$ if and only if $p\in H^s(\T\times\R)$. In the rest of this paper, we shall still work on $(\eta,\psi)$ for the simplicity of notation.

	\subsection{Shape derivative}\label{subsect:shape-deri}
	
	We have proved in Corollary \ref{cor-pre:bound-DtN} that $G(\eta)$ is a $1$-order operator for any fixed $\eta\in H^{s+\frac{1}{2}}_R(\T\times\R)$ with $s>\frac{3}{2}$. In this section, we shall show that, for fixed $\psi$, $G(\eta)\psi$ is derivable w.r.t. $\eta\in H^{s+\frac{1}{2}}_R(\T\times\R)$ and this derivative can be calculated explicitly, which is known as \textit{shape derivative}.
	
	In the case of planar water-wave ($\Sigma$ as a perturbation of $\R^d$), we refer to \cite{lannes2013water}, Section 3.3 for a detailed study of shape derivative. The axis-symmetric version (system independent of $\theta$) of shape derivative \eqref{eq-pre:shape-deri} is also proved in \cite{huang2023wellposedness}. In these references, the authors calculate directly the derivative in $\eta$ of \eqref{eq-pre:DtN-chgt-of-var}. The main difficulty of this method is to represent $G(\eta)(B\delta\eta)$ in terms of $\eta$, $\varphi$, and $\delta\varphi$. To do so, one needs to apply variational calculus for the elliptic equation \eqref{eq-pre:ellip} and then determine the harmonic extension of $B\delta\eta$, which cannot be obtained simply due to the (relatively complicated) change of coordinate \eqref{eq-pre:chgt-of-var}. Nevertheless, one will see in the proof of Proposition \ref{prop-pre:shape-deri} that the variational calculus \eqref{eq-pre:var-kin-energy} of kinetic energy is equivalent to \eqref{eq-pre:shape-deri}, where the former one has already been proved in previous section without using shape derivative.

	\begin{proposition}\label{prop-pre:shape-deri}
		Let $s>\frac{3}{2}$ and $1 \le s_0 \le s$. Then for all fixed $\psi\in H^{s_0}$, the map
		\begin{equation*}
			\begin{array}{ccc}
				H^{s+\frac{1}{2}}_R(\T\times\R) & \rightarrow & H^{s_0-1}(\T\times\R) \\[0.5ex]
				\eta & \mapsto & G(\eta)\psi
			\end{array}
		\end{equation*}
		is $C^1$ in the sense that, for all $\delta\eta \in H^{s+\frac{1}{2}}(\T\times\R)$ or $\delta\eta \in C^{s+\frac{1}{2}}(\T\times\R)$,
		\begin{equation}\label{eq-pre:shape-deri}
			\left.\frac{d}{d\epsilon}\right|_{\epsilon=0} G(\eta+\epsilon\delta\eta)\psi = -G(\eta)(B\delta\eta) - \nabla_{\theta,z}\cdot\left( \left( \frac{V^\theta}{\eta}e_\theta + V^ze_z \right) \delta\eta \right) - \frac{B\delta\eta}{\eta},
		\end{equation}
		where $V^\theta,V^z$ are defined by orthogonal decomposition $V = (V^\theta,V^z) =V^\theta e_\theta+V^ze_z$. Recall that the second term on the right hand side should be understood as
		\begin{equation*}
			\nabla_{\theta,z}\cdot\left( \left( \frac{V^\theta}{\eta}e_\theta + V^ze_z \right) \delta\eta \right) = \partial_\theta\left( \frac{V^\theta}{\eta}\delta\eta \right) + \partial_z\left( V^z \delta\eta \right).
		\end{equation*}
	\end{proposition}
	\begin{proof}
		As in previous section, $\delta$ denotes the derivative in $\eta$. To begin with, we write \eqref{eq-pre:shape-deri} into integrated form, namely, for all $f\in C^\infty_c(\T\times\R)$,
		\begin{align*}
			&\int \eta \delta\left(G(\eta)\psi\right) f d\theta dz \\
			&= -\int \eta G(\eta)(B\delta\eta) f d\theta dz - \int \eta \nabla_{\theta,z}\cdot\left( \left( \frac{V^\theta}{\eta}e_\theta + V^ze_z \right) \delta\eta \right) f d\theta dz - \int B\delta\eta f d\theta dz \\
			&= -\int B \eta G(\eta)f \delta\eta d\theta dz -  \int B f\delta\eta d\theta dz + \int \left( V^\theta e_\theta + V^ze_z \right)\cdot \left(e_\theta\frac{\partial_\theta}{\eta} + e_z \partial_z\right)(\eta f) \delta\eta d\theta dz \\
			&= -\int B G(\eta)f  \eta\delta\eta d\theta dz -  \int G(\eta)\psi f\delta\eta d\theta dz + \int \left( V^\theta \frac{f_\theta}{\eta} + V^z f_z \right) \eta \delta\eta d\theta dz.
		\end{align*}
		Let $F$ be the harmonic extension of $f$, namely $\Delta_{g}F = 0$ and $F|_{\rho=1}=f$. Since $f$ is regular, $F$ is well-defined as well as $B(f)$ and $V(f)$. By applying \eqref{eq-pre:DtN-chgt-of-var} and \eqref{eq-pre:deri-of-psi}, we may write the equality above as
		\begin{align*}
			 &\int\eta \delta\left(G(\eta)\psi\right) f d\theta dz + \int G(\eta)\psi f\delta\eta d\theta dz \\
			 &= \int \left( B(\psi)V(f)\cdot\left(\frac{\eta_\theta}{\eta},\eta_z\right) + B(f)V(\psi)\left(\frac{\eta_\theta}{\eta}, \eta_z\right) - B(\psi)B(f) + V(\psi)\cdot V(f) \right) \eta\delta\eta d\theta dz,
		\end{align*}
		from which, one can see that the desired result \eqref{eq-pre:shape-deri} is equivalent to: for all $f\in C^\infty_c(\T\times\R)$,
		\begin{equation}\label{eq-pre:shape-deri-int}
			\begin{aligned}
				&\delta\left( \int \eta G(\eta)\psi f d\theta dz \right) \\
				&= \int \left( B(\psi)V(f)\cdot\left(\frac{\eta_\theta}{\eta},\eta_z\right) + B(f)V(\psi)\left(\frac{\eta_\theta}{\eta}, \eta_z\right) - B(\psi)B(f) + V(\psi)\cdot V(f) \right) \eta\delta\eta d\theta dz.
			\end{aligned}
		\end{equation}
		When $\psi$ belongs to the same class $C_c^\infty(\T\times\R)$ as $f$, due to the fact that both sides are symmetric quadratic form of $(\psi,f)$, it suffices to check the case $\psi=f$. According to \eqref{eq-pre:def-energy-kin}, the left hand side becomes
		\begin{equation*}
			\delta\left( \int \eta G(\eta)\psi \psi d\theta dz \right) = 2 \delta E_k.
		\end{equation*}
		Note that, unlike in the proof of Proposition \ref{prop-pre:Hamiltonian-formulation}, $\psi$ is required to be fixed, instead of $p$. Thus the terms involving $\delta\psi$ are automatically zero. By applying the same calculation as in Lemma \ref{lem-pre:var-of-kin-energy}, we can conclude that
		\begin{equation*}
			2 \delta E_k = 2 (I_1+I_2) = 2 \int \eta N\delta\eta d\theta dz,
		\end{equation*}
		where $I_1,I_2$ are defined in \eqref{eq-pre:def-I-1} and \eqref{eq-pre:def-I-2} with $I_2=0$ as explained above. From \eqref{eq-pre:N-polar}, when $\psi=f$, the right hand side of \eqref{eq-pre:shape-deri-int} equals
		\begin{equation*}
			2 \int \left(BV\cdot\left(\frac{\eta_\theta}{\eta},\eta_z\right) + \frac{|V|^2-B^2}{2}\right) \eta\delta\eta d\theta dz = 2 \int \eta N\delta\eta d\theta dz.
		\end{equation*}
		Till now, we have proved \eqref{eq-pre:shape-deri} for $\psi\in C_c^\infty(\T\times\R)$. Since $C_c^\infty(\T\times\R)$ is dense in $H^{s_0}(\T\times\R)$, it remains to check that the left hand side of \eqref{eq-pre:shape-deri} satisfies
		\begin{equation}\label{eq-pre:shape-deri-esti}
			\left\| \left.\frac{d}{d\epsilon}\right|_{\epsilon=0} G(\eta+\epsilon\delta\eta)\psi \right\|_{H^{s_0-1}} \le C\left(\|\eta\|_{H_R^{s+\frac{1}{2}}},\|\delta\eta\|_{C^{s-\frac{1}{2}}}\right) \|\psi\|_{H^{s_0-1}}.
		\end{equation}
		To prove this, we apply the variation in $\eta$ to the equation \eqref{eq-pre:ellip},
		\begin{equation*}
			\left\{\begin{array}{l}
				\Delta_g\delta\varphi = -\frac{1}{\sqrt{g}}\partial_\alpha\left(\delta\left(\sqrt{g}g^{\alpha\beta}\right)\partial_\beta\varphi\right), \\[0.5ex]
				\delta\varphi|_{\rho=1} = 0.
			\end{array}\right.
		\end{equation*}
		Recall that $\delta$ should be understood as the derivative in $\epsilon$ at $\epsilon=0$. It follows immediately from \eqref{eq-pre:jacobian-mat}-\eqref{eq-pre:jacobian-det} that $\delta\left(\sqrt{g}g^{\alpha\beta}\right)$ is a smooth function of $\zeta$ and $\nabla_{y,z}\zeta$, depending linearly in $\delta\eta$ and $\nabla_{y,z}\delta\eta$. Consequently, $\varphi\in H^{s_0+\frac{1}{2}}(\D\times\R)$ implies that
		\begin{equation*}
			\left\| \frac{1}{\sqrt{g}}\partial_\alpha\left(\delta\left(\sqrt{g}g^{\alpha\beta}\right)\partial_\beta\varphi\right) \right\|_{H^{s_0-\frac{3}{2}}(\D\times\R)} \le C\left(\|\eta\|_{H^{s+\frac{1}{2}}_R(\T\times\R)},\|\delta\eta\|_{C^{s-\frac{1}{2}}(\T\times\R)}\right) \|\varphi\|_{H^{s_0+\frac{1}{2}}(\D\times\R)}.
		\end{equation*}
		In fact, $\delta\left(\sqrt{g}g^{\alpha\beta}\right)$ can be written in terms of $F_1(\zeta,\nabla_{\rho,\theta,z}\zeta)\delta\zeta$ and $F_2(\zeta,\nabla_{\rho,\theta,z}\zeta)\nabla_{\rho,\theta,z}\delta\zeta$, where $F_1$ and $F_2$ are smooth functions with $F_1(0)=F_2(0)=0$. Since $\zeta\in H^{s+1}_{R_\epsilon}(\D\times\R)$ with $R_\epsilon$ defined by \eqref{eq-pre:def-R-eps}, due to Proposition \ref{prop-pre:chgt-of-var} and the assumption $\eta\in H^{s+\frac{1}{2}}_R(\T\times\R)$, the coefficients before $\delta\zeta$ and $\nabla_{\rho,\theta,z}\delta\zeta$ belong to $H^s(\D\times\R)$, up to some $C^\infty_b$ normalizations (see Proposition \ref{prop-para:paralin}). In the mean time, when $\delta\eta\in H^{s+\frac{1}{2}}_R(\T\times\R)$, $\delta\zeta\in H^{s+1}_{W}(\D\times\R)$ and $\nabla_{\rho,\theta,z}\delta\zeta\in H^{s}_{W'}(\D\times\R)$ for some $W,W'\in C^\infty_b(\D\times\R)$ due to Proposition \ref{prop-pre:chgt-of-var}, which also guarantees that, as $\delta\eta \in C^{s+\frac{1}{2}}(\T\times\R)$, $\delta\zeta\in C^{s+\frac{1}{2}}(\D\times\R)$ and $\nabla_{\rho,\theta,z}\delta\zeta\in C^{s-\frac{1}{2}}(\D\times\R)$. Therefore, the desired inequality follows from the estimate of products (Proposition \ref{prop-para:paraprod-bound}). As a consequence, an application of elliptic regularity \eqref{eq-ellip:ellip-alt-reg} gives that
		\begin{equation}\label{eq-pre:shape-deri-esti-in-del-varphi}
			\left\| \delta\varphi \right\|_{H^{s_0+\frac{1}{2}}(\D\times\R)} \le C\left(\|\eta\|_{H^{s+\frac{1}{2}}_R(\T\times\R)},\|\delta\eta\|_{C^{s-\frac{1}{2}}(\T\times\R)}\right) \|\psi\|_{H^{s_0}(\T\times\R)} .
		\end{equation}
		
		By expression \eqref{eq-pre:DtN-chgt-of-var}, the variation of
		\begin{equation*}
			G(\eta)\psi = \zeta Y^\alpha g^{\alpha\beta} \partial_{\beta}\varphi|_{\rho=1}
		\end{equation*} 
		is composed by two parts, the variation in coefficients (depending only on $\zeta$) and in $\varphi$, namely,
		\begin{align*}
			\delta \left(G(\eta)\psi\right) =& \left.\delta \left( \zeta Y^\alpha g^{\alpha\beta} \partial_{\beta}\varphi \right) \right|_{\rho=1} \\
			=& \left.\left( \delta(\zeta Y^\alpha g^{\alpha\beta}) \partial_{\beta}\varphi \right)\right|_{\rho=1} + \left.\left( \zeta Y^\alpha g^{\alpha\beta} \partial_{\beta}\delta\varphi \right)\right|_{\rho=1}
		\end{align*}
		We can deduce from \eqref{eq-pre:shape-deri-esti-in-del-varphi} that $\partial_\beta\delta\varphi|_{\rho=1}\in H^{s_0-1}(\T\times\R)$, while, up to $C^\infty_b(\D\times\R)$ normalizations, $\zeta Y^\alpha g^{\alpha\beta}$ has $H^{s}(\D\times\R)$ regularity, the trace at $\rho=1$ of which belongs to $H^{s-\frac{1}{2}}(\T\times\R)$. By applying Corollary \ref{cor-para:product-law}, we are able to conclude that the second part (variation in $\varphi$) lies in $H^{s_0-1}(\T\times\R)$. As for the first part, we observe that $\zeta Y^\alpha g^{\alpha\beta}$ can be written as a smooth function of $\zeta$, $\nabla_{y,z}\zeta$, and $Y$, i.e.
		\begin{equation*}
			\zeta Y^\alpha g^{\alpha\beta} = G(\zeta,\nabla_{y,z}\zeta,Y),
		\end{equation*}
		for some smooth function $G$. The the variation in $\eta$ reads
		\begin{equation*}
			\delta \left( \zeta Y^\alpha g^{\alpha\beta} \right) = \partial_1 G(\zeta,\nabla_{y,z}\zeta,Y) \delta\zeta + \partial_2 G(\zeta,\nabla_{y,z}\zeta,Y) \nabla_{y,z}\delta\zeta,
		\end{equation*}
		whose trace at $\rho=0$ belongs to $H^{s-\frac{1}{2}}(\T\times\R)$, thanks to the same argument as above. Then, by Corollary \ref{cor-para:product-law}, its product with $\partial_\beta\varphi|_{\rho=1}\in H^{s_0-1}(\T\times\R)$ remains in $H^{s_0-1}(\T\times\R)$, which completes the proof.
	\end{proof}
	
	In the formula \eqref{eq-pre:shape-deri}, the left hand side belongs to $H^{s_0-1}(\T\times\R)$ while the first two terms on the right hand side are only in $H^{s_0-2}(\T\times\R)$. This indicates that there exists an implicit cancellation between these terms. In fact, by choosing $\delta\eta=1 \in C^{s+\frac{1}{2}}(\T\times\R)$, we can deduce from $B/\eta\in H^{s_0-1}(\T\times\R)$, which is a consequence of Corollary \ref{cor-para:product-law} and \ref{prop-para:paralin}, that
	\begin{proposition}\label{prop-pre:cancellation}
		Let $s>\frac{3}{2}$ and $1 \le s_0 < s$. Then for all $\eta\in H^{s+\frac{1}{2}}_R(\T\times\R)$ and $\psi\in H^{s_0}(\T\times\R)$, we have
		\begin{equation}\label{eq-pre:cancellation}
			\left\| G(\eta)B + \left( \frac{e_\theta\partial_\theta }{\eta}+ e_z\partial_z \right)\cdot V \right\|_{H^{s_0-1}(\T\times\R)} \le C\left(\|\eta\|_{H^{s+\frac{1}{2}}_R(\T\times\R)}\right) \|\psi\|_{H^{s_0}(\T\times\R)},
		\end{equation}
		where $C>0$ is an increasing smooth function. 
	\end{proposition}
	\begin{proof}
		It remains to check that
		\begin{equation*}
			\left\| \left( \frac{e_\theta\partial_\theta }{\eta}+ e_z\partial_z \right)\cdot V - \nabla_{\theta,z}\cdot\left( \frac{V^\theta}{\eta}e_\theta + V^z e_z \right) \right\|_{H^{s_0-1}(\T\times\R)} \le C\left(\|\eta\|_{H^{s+\frac{1}{2}}_R(\T\times\R)}\right) \|\psi\|_{H^{s_0}(\T\times\R)},
		\end{equation*}
		which is equivalent to
		\begin{equation*}
			\left\| \partial_\theta\eta^{-1} V^\theta \right\|_{H^{s_0-1}(\T\times\R)} \le C\left(\|\eta\|_{H^{s+\frac{1}{2}}_R(\T\times\R)}\right) \|\psi\|_{H^{s_0}(\T\times\R)}.
		\end{equation*}
		We have seen that $\eta\in H^{s+\frac{1}{2}}_R(\T\times\R)$, which implies $\partial_\theta\eta^{-1}\in H^{s-\frac{1}{2}}_{R^{-1}}(\T\times\R)$, thanks to Proposition \ref{prop-para:paralin}. Since $V^\theta\in H^{s_0-1}$, the desired result follows from Corollary \ref{cor-para:product-law}.
	\end{proof}

	\section{Paralinearization of the system}\label{Sect:paralin}
	
	This section is devoted to an explicit formulation of Dirichlet-to-Neumann operator $G(\eta)$ as paradifferential operator, up to some remainder of lower order. And, as a result, we shall write the system \eqref{eq-intro:WW} in paralinear form. Namely, all the nonlinear terms will be replaced by paradifferential operators acting on $(\eta,\psi)$, up to proper regular remainders. And the symbol of these paradifferential operators is a smooth function of $\eta$ and its derivatives and homogeneous in Fourier variable $\xi$. We shall follow the strategy used in \cite{alazard2009paralinearization} and \cite{alazard2011water} (see also \cite{alazard2014cauchy}).
	
	To achieve this, we first observe that, by \eqref{eq-pre:DtN-chgt-of-var},
	\begin{equation}\label{eq-paralin:DtN-main}
		\begin{aligned}
			G(\eta)\psi =& \frac{1}{\eta}\left(1 + \left(\frac{\eta_\theta}{\eta}\right)^2 + \eta_z^2\right) \partial_\rho \varphi |_{\rho=1} - \frac{\eta_\theta}{\eta}\frac{\varphi_\theta |_{\rho=1}}{\eta} - \eta_z \varphi_z |_{\rho=1} \\
			=& \frac{1}{\eta}\left(1 + \left(\frac{\eta_\theta}{\eta}\right)^2 + \eta_z^2\right) \partial_\rho \varphi |_{\rho=1} - \frac{\eta_\theta}{\eta}\frac{\psi_\theta}{\eta} - \eta_z \psi_z,
		\end{aligned}
	\end{equation}
	where the only implicit term is $\partial_\rho\varphi|_{\rho=1}$, which is linear in $\psi$ and depends implicitly on $\eta$. In order to represent normal derivative $\partial_\rho|_{\rho=1}$ as tangential ones, we follow the idea in \cite{alazard2009paralinearization} and return to the Poisson's equation \eqref{eq-pre:ellip}. Formally speaking, we shall rewrite Laplacian $\Delta_g$ as 
	\begin{equation*}
		\Delta_g = A_0(\partial_\rho + A_1')(\partial_\rho - A_1) + \text{remainders},
	\end{equation*}
	where $A_0$, $A_1$, and $A_1'$ are elliptic operators involving only tangential derivatives. This decomposition will be proved in Lemma \ref{lem-paralin:decomp-of-P}, while the tangential operators should be understood as paradifferential operator since the coefficients are smooth functions of $\eta$ and its derivatives and thus have limited regularity. As a result, the Poisson's equation \eqref{eq-pre:ellip} ensures that $v:=(\partial_\rho - A_1)\varphi$ solves the parabolic equation
	\begin{equation*}
		(\partial_\rho + A_1')v = \text{regular terms},
	\end{equation*}
	whose smoothing effect leads to a higher regularity of $v|_{\rho=1}$ than $\varphi$ (or $\psi$). Consequently, one may replace $\partial_\rho|_{\rho=1}$ by $A_1|_{\rho=1}$ and then deduce the paralinearization of \eqref{eq-paralin:DtN-main} (see Proposition \ref{prop-paralin:paralin-DtN}).
	
	With paralinear formula \eqref{eq-paralin:paralin-DtN} of $G(\eta)\psi$, we are able to write the nonlinear terms in \eqref{eq-intro:WW} in a similar form. In fact, these terms (quadratic term $N$ and mean curvature $H$) can be represented as a function of $\eta$, $\psi$, and $G(\eta)\psi$, together with their derivatives. Hence, one may simply apply the paralinear formula \eqref{eq-para:paralin} and the paralinearization of $G(\eta)\psi$, which will be proved in Section \ref{subsect:paralin-nonlin}.
	
	In the end of this section, we shall check that all the remainders appearing in the normalization of the system are Lipschitzian w.r.t. $(\eta,\psi)$ in proper functional spaces. This is essential to prove convergence of approximate solution and uniqueness of the solution (which will be studied in Section \ref{subsect:conv-uni}).
	
	Inside this section, we always assume that $(\eta,\psi)\in H^{s+\frac{1}{2}}_R \times H^s$ with $s>3$, which is the regularity required in the main theorem \ref{thm-intro:main}. One should pay attention to the fact that, thanks to the paradifferential calculus (reviewed in Appendix \ref{App:para}), all the estimates to be proved are \textit{tame}, in the sense that they can be written as 
	\begin{equation*}
		Q \le K \left( \|\eta\|_{H^{s+\frac{1}{2}}_R} + \|\psi\|_{H^{s}} \right),
	\end{equation*}
	where $Q$ stands for quantities to be studied, and the constant $K= K\left( \|\eta\|_{H^{s+\frac{1}{2}-}_R}, \|\psi\|_{H^{s-}} \right)$ depends only on $\|\eta\|_{H^{s+\frac{1}{2}-}_R}$, $\|\psi\|_{H^{s-}}$ instead of $\|\eta\|_{H^{s+\frac{1}{2}}_R}, \|\psi\|_{H^{s}}$. In Section \ref{Sect:cauchy}, this observation will lead to the \textit{tame estimate} \eqref{eq-cauchy:nonlin-interpol-tame}, implying the continuity of the solution to \eqref{eq-intro:WW} in time and initial data via a nonlinear interpolation (Theorem \ref{thm-cauchy:nonlin-interpol}). \\
	
	To simplify the computations, instead of the change of variable $r=\rho\zeta(\rho\theta,z), \omega=\theta, z=z$ defined in Section \ref{Sect:pre}, where $(x,z)=(r\omega,z)$ is the polar coordinate, we shall use the same change of variable as in \cite{huang2023wellposedness},
	\begin{equation}\label{eq-paralin:chgt-of-var}
		r=\bar{\rho}\eta(\bar{\theta},\bar{z}),\ \omega=\bar{\theta},\ z=\bar{z}.
	\end{equation}
	Recall that this change of variable is easy to calculate and behaves well near the interface $r=\eta$, while singularities exists at $r=0$, which is out of concern in this section. By definition, the relation between $(\bar{\rho},\bar{\theta},\bar{z})$ and the variable $(\rho,\theta,z)$ studied in previous section is given by
	\begin{equation}\label{eq-paralin:alt-var}
		(\bar{\rho},\bar{\theta},\bar{z}) = \bar{\iota}(\rho,\theta,z) := \left( \frac{\rho\zeta(\rho\theta,z)}{\eta(\theta,z)},\theta,z \right).
	\end{equation}
	We denote by $\bar{\varphi}$ the potential $\phi$ in coordinate $(\bar{\rho},\bar{\theta},\bar{z})$ (equivalently, $\bar{\varphi}\circ\bar{\iota}=\varphi$). It can be checked that $\bar{\varphi}$ has the same regularity as $\varphi$ when $\bar{\rho}$ is close to $1$. The proof of lemma below is left to the end of Appendix \ref{subsect:ellip-alt-reg}. 
	\begin{lemma}\label{lem-paralin:reg-of-pot}
		Let $s>3$ and $\delta>0$ be small enough. If $\psi\in H^{s_0}(\T\times\R)$ with $\frac{3}{2}<s_0\le s$ and $\eta\in H^{s+\frac{1}{2}}_R(\T\times\R)$, the potential $\bar{\varphi}$ in new coordinate $(\bar{\rho},\bar{\theta},\bar{z})$ satisfies
		\begin{equation}\label{eq-paralin:reg-of-pot-loc}
			\|\partial_{\bar{\rho}}^l\bar{\varphi}\|_{C^0([1-\delta,1];H^{s_0-l}(\T\times\R))} \le C\left(\|\eta\|_{H_R^{s+\frac{1}{2}-}}\right) \|\psi\|_{H^{s_0}},\ \ l=0,1,2,3.
		\end{equation}
	\end{lemma}

	From now on, we will no more use the coordinate introduced in Section \ref{Sect:pre} and, for simplicity of notation, we omit the bar over $(\bar{\rho},\bar{\theta},\bar{z})$ and $\bar{\varphi}$. Moreover, all the functions to be studied should be regarded as functions defined on $\T\times\R$ with parameter $\rho\in[1-\delta,1]$. By definition \eqref{eq-paralin:chgt-of-var}, the potential $\varphi$ satisfies
	\begin{equation}\label{eq-paralin:ellip}
		\left\{ \begin{array}{ll}
			L\varphi:= \left(\alpha\partial_\rho^2 + \beta\cdot\nabla_{\theta,z}\partial_\rho + \gamma\partial_\rho + \frac{1}{\rho^2\eta^2}\partial_{\theta}^2 + \partial_z^2 \right) \varphi = 0, & \forall 1-\delta\le\rho\le 1, \\ [0.5ex]
			\varphi|_{\rho=1} = \psi.
		\end{array}\right.
	\end{equation}
	where
	\begin{align}
		\alpha =& \frac{1}{\eta^2} \left( 1 + \left(\frac{\eta_\theta}{\eta}\right)^2 + \rho^2\eta_z^2 \right), \label{eq-paralin:def-alpha}\\
		\beta =& \left( -\frac{2\eta_\theta}{\rho\eta^3} , -\frac{2\rho\eta_z}{\eta} \right), \label{eq-paralin:def-beta}\\
		\gamma =& - \frac{1}{\rho\eta}\left(\frac{\eta_\theta}{\eta^2}\right)_\theta - \rho\eta \left(\frac{\eta_z}{\eta^2}\right)_z + \frac{1}{\rho\eta^2}. \label{eq-paralin:def-gamma}
	\end{align}
	 Note that by construction \eqref{eq-pre:def-zeta} of $\zeta$, we have $\partial_\rho\varphi|_{\rho=1} = \partial_{\bar{\rho}}\bar{\varphi}|_{\bar{\rho}=1}$. Thus, the formula \eqref{eq-paralin:DtN-main} for Dirichlet-to-Neumann operator as well as formulation \eqref{eq-pre:B-polar}, \eqref{eq-pre:V-polar}, \eqref{eq-pre:N-polar} of $B,V,N$ remain unchanged. Before entering the next part, we introduce the following estimates for $\alpha$, $\beta$, and $\gamma$, which will be frequently used in this section.
	 \begin{lemma}\label{lem-paralin:esti-alpha-beta-gamma}
	 	Let $\eta\in H^{s+\frac{1}{2}-}_R$ with $s>3$. Then we have
	 	\begin{equation}\label{eq-paralin:esti-alpha-beta-gamma}
	 		\begin{aligned}
	 			\|\alpha\|_{C^0_\rho H^{s-\frac{1}{2}-}_{R^{-2}}} + \|\beta\|_{C^0_\rho H^{s-\frac{1}{2}-}} + \|\gamma\|_{C^0_\rho H^{s-\frac{1}{2}-}_{\rho^{-1}R^{-2}}} \le C\left( \|\eta\|_{H^{s+\frac{1}{2}-}_R} \right) \|\eta\|_{H^{s+\frac{1}{2}-}_R}.
	 		\end{aligned}
	 	\end{equation}
	 \end{lemma}
	 The proof of this Lemma is no more than an application of Proposition \ref{prop-para:paralin}.

	\subsection{Preliminaries in paradifferential calculus}\label{subsect:paralin-pre}
	
	During the whole paper, the involved paradifferential operators have symbol homogeneous in Fourier variable $\xi$ and depending smoothly in $\eta$ and its derivatives (see Definition \ref{def-paralin:homo-sym} below). In this case, the general theory of paradifferential calculus reviewed in Appendix \ref{App:para} can be refined with simpler formulas, which are collected in this section. Note that most of these technical results have been presented in Section 4.1 of \cite{alazard2011water}, where one may find detailed proofs.
	
	\begin{definition}\label{def-paralin:homo-sym}
		Given $m\in\R$, we denote by $\Sigma^m$\index{S@$\Sigma^m$ Homogeneous symbols depending on $\eta$} the collection of symbols on $\T\times\R$ that take the form
		\begin{equation}\label{eq-paralin:homo-sym}
			a(w,\xi) = a^{(m)}(w,\xi) + a^{(m-1)}(w,\xi),
		\end{equation}
		where $a^{(m)}$ and $a^{(m-1)}$ takes the form
		\begin{align*}
			a^{(m)}(w,\xi) =& F(\eta,\nabla_{\theta,z}\eta;\xi), \\
			a^{(m-1)}(w,\xi) =& \sum_{|\alpha|\le 2} G_{\alpha}(\eta,\nabla_{\theta,z}\eta;\xi)\partial_{\theta,z}^\alpha\eta
		\end{align*}
		Here $F,G_\alpha$ are smooth functions and homogeneous of degree $m,m-1$ in $\xi$, respectively.
		
		Under the assumption that $\eta\in H_R^{s+\frac{1}{2}}$ with $s>3$ (true in the rest of this paper), we have $a^{(m)}\in \Gamma_{3/2+}^m$ and $a^{(m-1)}\in\Gamma^{m-1}_{1/2+}$ (see Definition \ref{def-para:paradiff-symbol-class}). In the sequel, we shall regard $a\in\Sigma^m$ as an element in $\Gamma_{3/2+}^m + \Gamma_{1/2+}^{m-1}$.
	\end{definition}
	
	As a convention, we shall write all the equations in polar coordinate $(y,z) = (\rho\theta,z)$ and denote the Fourier variable associated to $w=(\theta,z)\in\T\times\R$ as $\xi = (\xi_\theta,\xi_z)\in\R^2$ (see Section \ref{subsect:PDO} for Fourier variables on torus). Furthermore, all the estimates in this section are uniform in time, the dependence on which will be omitted for simplicity.
	
	For symbols $a,b\in\Sigma^m$, we write
	\begin{equation*}
		T_a \approx T_b
	\end{equation*}\index{e@$\approx$ Equivalence between paralinear operators}
	if $T_a-T_b$ is of order $m-\frac{3}{2}-$, i.e. maps $H^{s}$ to $H^{s-m+\frac{3}{2}+}$ for all $s\in\R$, with operator norm controlled by $C\left( \|\eta\|_{H^{\frac{7}{2}+}_R} \right)$. Then an application of Proposition \ref{prop-para:paradiff-cal-sym} gives that
	\begin{proposition}\label{prop-paralin:homo-sym-cal}
		Let $\eta \in H_R^{s+\frac{1}{2}-}$ with $s>3$. For symbols $a = a^{(m)}+a^{(m-1)}\in\Sigma^m$ and $b = b^{(m')}+ b^{(m'-1)}\in\Sigma^{m'}$ with $m,m'\in\R$, we have
		\begin{align}
			&T_aT_b \approx T_{a \sharp b}, \label{eq-paralin:homo-sym-comp} \\
			&T_{a}^*\approx T_{a^*}, \label{eq-paralin:homo-sym-adj}
		\end{align}
		where
		\begin{align}
			a \sharp b =& a^{(m)}b^{(m')} + \left(\partial_\xi a^{(m)}\cdot D_w b^{(m')} + a^{(m)} b^{(m'-1)} + a^{(m-1)} b^{(m')}\right) \in \Sigma^{m+m'}, \label{eq-paralin:homo-sym-comp-formula} \\
			a^* =& \overline{a^{(m)}} + \left( D_w\cdot\partial_\xi\overline{a^{(m)}} + \overline{a^{(m-1)}} \right) \in \Sigma^m. \label{eq-paralin:homo-sym-adj-formula}
		\end{align}
	\end{proposition}
	Recall that we have regarded $\Sigma^m$ as a subclass of $\Gamma_{3/2+}^m + \Gamma_{1/2+}^{m-1}$. Thus the parameter $\rho$ in Proposition \ref{prop-para:paradiff-cal-sym} should be taken as $3/2+$ or $1/2+$. Moreover, all the involved operator norms can actually be controlled by a positive smooth increasing function of $\|\eta\|_{H_R^{s+\frac{1}{2}-}}$. For simplicity, we shall not precise these in the sequel.
	
	As a consequence, we are able to calculate the symbol of commutator,
	\begin{corollary}\label{cor-paralin:homo-sym-commu}
		Let $\eta \in H_R^{s+\frac{1}{2}-}$ with $s>3$. For symbols $a = a^{(m)}+a^{(m-1)}\in\Sigma^m$ and $b = b^{(m')}+ b^{(m'-1)}\in\Sigma^{m'}$ with $m,m'\in\R$, their commutator $[T_a, T_b]$ is of order at most $m+m'-1$. More precisely, $[T_a, T_b]\approx T_{-i\{a^{(m)},b^{(m')}\}}$, where $\{\cdot,\cdot\}$ is Poisson bracket,
		\begin{equation*}
			\{ a,b \} := \partial_\xi a^{(m)}\cdot\partial_w b^{(m')} - \partial_w a^{(m)}\cdot\partial_\xi b^{(m')} \in \Gamma^{m+m'-1}_{1/2+}.
		\end{equation*}
	\end{corollary}
	
	\begin{definition}\label{def-paralin:homo-sym-ellip}
		Let symbol $a = a^{(m)}+a^{(m-1)}\in \Sigma^m$ with $m\in\R$. We say that $a$ is \textit{elliptic}, if there exists $0<c<C$ such that
		\begin{equation*}
			c|\xi|^m \leqslant \Real a^{(m)}(w,\xi) \leqslant C|\xi|^m,\ \ \forall w\in\T\times\R,\ \xi\neq 0,
		\end{equation*}
	\end{definition}
	
	\begin{proposition}\label{prop-paralin:homo-sym-ellip-inverse}
		Let $\eta \in H_R^{s+\frac{1}{2}-}$ with $s>3$. For any elliptic symbol $a\in\Sigma^m$ with $m\in\R$, we can construct another elliptic symbol $\tilde{a}\in \Sigma^{-m}$ such that
		\begin{equation}\label{eq-paralin:homo-sym-ellip-inverse}
			T_aT_{\tilde{a}} \approx T_{\tilde{a}}T_a \approx id.
		\end{equation}
	\end{proposition}
	\begin{proof}
		It suffices to take $\tilde{a}^{(-m)}a^{(m)}=1$ and $\tilde{a}^{(-m-1)}$ as the solution to
		\begin{equation*}
			\partial_\xi a^{(m)} \cdot D_w \frac{1}{a^{(m)}} + a^{(m)} \tilde{a}^{(-m-1)} + a^{(m-1)} \tilde{a}^{(-m)} = 0.
		\end{equation*}
		The ellipticity of $\tilde{a}$ is obvious from this definition. The desired estimate \eqref{eq-paralin:homo-sym-ellip-inverse} can be checked directly from Proposition \ref{prop-paralin:homo-sym-cal}. More precisely, we have $T_a T_{\tilde{a}} \sim T_{a\sharp\tilde{a}}$ where
		\begin{align*}
			a\sharp\tilde{a} =& a^{(m)}\tilde{a}^{(m)} + \left(\partial_\xi a^{(m)}\cdot D_w \tilde{a}^{(-m)} + a^{(m)} \tilde{a}^{(-m-1)} + a^{(m-1)} \tilde{a}^{(-m)}\right) \\
			=& 1 + \left(\partial_\xi a^{(m)}\cdot D_w \frac{1}{a^{(m)}} + a^{(m)} \tilde{a}^{(-m-1)} + a^{(m-1)} \tilde{a}^{(-m)}\right)=1.
		\end{align*}
		Therefore, $T_a T_{\tilde{a}} \approx T_1$, which equals identity up to a smoothing operator. The proof of $T_{\tilde{a}}T_a\approx id$ is similar. It suffices to observe that
		\begin{equation*}
			\partial_\xi \tilde{a}^{(-m)}\cdot D_w a^{(m)} = \partial_\xi \frac{1}{a^{(m)}}\cdot D_w a^{(m)} = \frac{i\partial_\xi a^{(m)}\cdot \partial_w a^{(m)}}{\left(a^{(m)}\right)^2} = \partial_\xi a^{(m)}\cdot D_w \frac{1}{a^{(m)}} = \partial_\xi a^{(m)}\cdot D_w \tilde{a}^{(-m)},
		\end{equation*}
		which implies $\tilde{a}\sharp a = a\sharp\tilde{a} = 1$.
	\end{proof}
	
	\begin{proposition}\label{prop-paralin:homo-sym-ellip-esti}
		Let $\eta\in H^{2+}_R$. For any elliptic symbol $a\in\Sigma^m$ with $m\in\R$, we have the following estimate,
		\begin{equation}\label{eq-paralin:homo-sym-ellip-esti}
			\|u\|_{H^{r}} \le C\left(\|\eta\|_{H^{2+}_R}\right) \left( \|T_a u\|_{H^{r-m}} + \|u\|_{L^2} \right),\ \ \forall r\in\R,
		\end{equation}
		where $C>0$ is a smooth increasing function and the $L^2$-norm of $u$ can be replaced with any other Sobolev norms.
	\end{proposition}
	If $\eta\in H^{s+\frac{1}{2}-}_R$ with $s>3$, \eqref{eq-paralin:homo-sym-ellip-esti} is a consequence of Proposition \ref{prop-paralin:homo-sym-ellip-inverse}. For the refined version $\eta\in H^{2+}_R$, we refer to Proposition 4.6 of \cite{alazard2011water}.
	
	\begin{proposition}\label{prop-paralin:homo-sym-cond-aa}
		Let $\eta \in H_R^{s+\frac{1}{2}-}$ with $s>3$. For any $a = a^{(m)}+a^{(m-1)}\in\Sigma^m$ with $m\in\R$ and $\Imag a^{(m)}=0$, we have the following equivalence,
		\begin{equation}\label{eq-paralin:homo-sym-cond-aa}
			T_a \approx T_a^* \Leftrightarrow (T_a - T_a^*)\text{ is of order }(m-1-) \Leftrightarrow \Imag a^{(m-1)} = -\frac{1}{2}\partial_w\cdot\partial_\xi a^{(m)}.
		\end{equation}
	\end{proposition}
	\begin{proof}
		The first assertion implying the second is easy. By Proposition \ref{prop-paralin:homo-sym-cal}, we have $T_a^*\approx T_{a^*}$, ensuring that $T_a - T_a^*$ is of order $(m-3/2-)$ and thus of order $(m-1-)$.
		
		To show that the second assertion implies the third one, we observe that, since $T_a^*\approx T_{a^*}$, $T_{a-a^*}$ is also of order $(m-1-)$. Moreover, $a,a^*$ belong to the class $\Sigma^m$, meaning that $a-a^*$ is the sum of a symbol homogeneous in $\xi$ of degree $m$ and another one of degree $m-1$. Once the order of $T_{a-a^*}$ is strictly smaller than $m-1$, both components vanish and consequently $a=a^*$, while, due to the expression \eqref{eq-paralin:homo-sym-adj-formula} of $a^*$,
		\begin{equation*}
			a-a^* = a^{(m)} - \overline{a^{(m)}} - D_w\cdot\partial_\xi\overline{a^{(m)}} + a^{(m-1)} - \overline{a^{(m-1)}}
			= i\partial_w\cdot\partial_\xi a^{(m)} + 2i \Imag a^{(m-1)},
		\end{equation*}
		which gives the third assertion.
		
		If the third assertion holds true, the calculation above ensures that $a=a^*$ and the first assertion follows from $T_a^* \approx T_{a^*} = T_a$.
	\end{proof}

	\subsection{Paralinearization of normal derivative and good unknown}\label{subsect:paralin-ellip}
	
	In this part, we shall separate Poisson's equation \eqref{eq-pre:ellip} in normal and tangential directions near boundary $\rho=1$ and deduce that $\partial_\rho\varphi|_{\rho=1}$ equals some one order tangential operator acting on $\psi$, up to some remainders. To do so, we use the strategy in \cite{alazard2009paralinearization} (see also \cite{alazard2011water}) where the authors apply paradifferential calculus in tangential direction and reduce the problem to a paralinear parabolic equation which yields the regularity of remainders. The main result of this section is 
	\begin{proposition}\label{prop-paralin:normal-to-tangential}
		Let $(\eta,\psi)\in H_R^{s+\frac{1}{2}}\times H^{s_0}$ with $s>3$ and $\frac{3}{2}<s_0\le s$. Then there exists $\tau \in \Sigma^1$ such that
		\begin{equation}\label{eq-paralin:normal-to-tangential}
			\|\partial_\rho\varphi|_{\rho=1} - T_\tau U\|_{H^{s_0+\frac{1}{2}}} \le C\left(\|\eta\|_{H_R^{s+\frac{1}{2}-}}\right) \left( \|\psi\|_{H^{s_0}} + \|\eta\|_{H_R^{s+\frac{1}{2}}}\|\psi\|_{H^{s_0-}} \right),
		\end{equation}
		where
		\begin{equation}\label{eq-paralin:good-unknown}
			U:= \psi - T_B\eta \in H^{s_0}(\T\times\R).
		\end{equation}
		Moreover, $\tau = A|_{\rho=0}$, where $A(\rho) = A^{(1)}(\rho) + A^{(0)}(\rho)$ is defined by \eqref{eq-paralin:def-A-1} and \eqref{eq-paralin:def-A-0} below.
	\end{proposition}
	
	Note that, under the assumption of Proposition \ref{eq-paralin:normal-to-tangential}, due to Proposition \ref{prop-para:paraprod-bound} and Lemma \ref{lem-paralin:esti-B-V-psi}, the new variable $U$ satisfies
	\begin{equation}\label{eq-paralin:esti-good-unknown}
		\|U\|_{H^{s_0}} \le \|\psi\|_{H^{s_0}} + \|T_B\eta\|_{H^{s_0}} \le C\left(\|\eta\|_{H_R^{s+\frac{1}{2}-}}\right) \|\psi\|_{H^{s_0}},\ \ \forall s>3,\ \frac{3}{2}<s_0\le s.
	\end{equation}
	
	As indicated in \cite{alazard2009paralinearization}, the use of $U$ instead of $\psi$ is essential, which will be explained briefly below. Let us consider the paralinearized version of Laplacian operator $L$ appearing in \eqref{eq-paralin:ellip},
	\begin{equation}\label{eq-paralin:def-para-Lap}
		P:= T_\alpha\partial_\rho^2 + T_{i\beta\cdot\xi+\gamma}\partial_\rho - T_{\rho^{-2}\eta^{-2}\xi_\theta^2+\xi_z^2}.
	\end{equation}\index{P@$P$ Paralinearized Laplacian}
	One may expect that, $P\varphi - L\varphi$ is much more regular than $\psi$, which is true if the coefficients $\alpha,\beta,\gamma,\eta^{-2}$ are smooth. However, under the assumption of Proposition \ref{prop-paralin:normal-to-tangential}, these coefficients have only Sobolev regularity and then the paraproducts $T_{\partial_\rho^2\varphi}\alpha, T_{\nabla_{\theta,z}\partial_\rho\varphi}\cdot\beta$ appearing in the remainders lie in $C^0_\rho H^{s_0-\frac{1}{2}}$, which is not enough to conclude \eqref{eq-paralin:normal-to-tangential} (we need $C^0_\rho H^{s_0-\frac{1}{2}+}$). This lack of regularity is actually a consequence of paracomposition, the general theory of which is sketched in Appendix \ref{subsect:compo}. Let $\chi$ represent the diffeomorphism from $(\rho,\theta,z)\in[1-\delta,1]\times\T\times\R$ to $(x,z)\in\Omega(t)$ and $X^*$ be the paracomposition operator associated to $\chi$. A formal application of Proposition \ref{prop-para:paracomp-conj} gives that
	\begin{equation*}
		0 = X^*\Delta_{x,z}\phi = T_{a} X^*\phi + R,
	\end{equation*}
	where the symbol $a$ is the pull-back of symbol of Laplacian operator and remainder $R$ has the desired regularity (see Lemma \ref{lem-paralin:paralin-of-ellip} below). Note that $\Delta_{x,z}$ should be understood as the paralinearized Laplacian operator differing from $\Delta_{x,z}$ only in low frequency. Therefore, this difference only generates smooth remainders and can be omitted. Formally, one may identify $T_a$ as $P$ defined by \eqref{eq-paralin:def-para-Lap}, since Laplacian $\Delta_{x,z}$ is a differential operator. To sum up, we have $PX^*\phi \in C^0_\rho H^{s_0-\frac{1}{2}+}$. Meanwhile, by Proposition \ref{prop-para:paracomp-exist}, up to remainders, the term $X^*\phi$ equal
	\begin{equation*}
		X^*\phi = \varphi - T_{\phi'\circ\chi}\chi = \varphi - T_{\eta^{-1}\partial_\rho\varphi}\rho\eta,
	\end{equation*}
	which suggests us to turn to \textit{Alinhac's good unknown}
	\begin{equation}\label{eq-paralin:good-unknown-ext}
		\Phi := \varphi - T_{\eta^{-1}\partial_\rho\varphi}\rho\eta,
	\end{equation}\index{P@$\Phi$ Alinhac's good unknown}
	whose trace at $\rho=1$ is equal to $U$ defined by \eqref{eq-paralin:good-unknown}.
	\begin{lemma}\label{lem-paralin:paralin-of-ellip}
		Let $(\eta,\psi)\in H^{s+\frac{1}{2}}_R\times H^{s_0}$ with $s>3$ and $\frac{3}{2} < s_0\le s$. Then 
		\begin{equation}\label{eq-paralin:paralin-of-ellip}
			P\Phi = r_1.
		\end{equation}
		where $\Phi$ is defined by \eqref{eq-paralin:good-unknown-ext} and $r_1$ satisfies
		\begin{equation}\label{eq-paralin:paralin-of-ellip-varphi-rem}
			\|r_1\|_{C^0_\rho H^{s_0-\frac{1}{2}+\epsilon}} \le C\left(\|\eta\|_{H^{s+\frac{1}{2}-}_R}\right) \|\psi\|_{H^{s_0}}.
		\end{equation} 
	\end{lemma}
	During the proof of Lemma \ref{lem-paralin:paralin-of-ellip}, we shall use the equivalence: for $u,v$ defined on $[1-\delta]\times\T\times\R$,
	\begin{equation}\label{eq-paralin:equi-paralin-of-ellip}
		u\sim v \Leftrightarrow \|u-v\|_{C^0_\rho H^{s_0-\frac{1}{2}+\epsilon}} \le C\left(\|\eta\|_{H_R^{s+\frac{1}{2}-}}\right) \|\psi\|_{H^{s_0}},
	\end{equation}\index{e@$\sim$ Equivalence between functions}
	with $0<\epsilon \ll 1$ to be determined later.
	
	To begin with, we check that
	\begin{lemma}\label{lem-paralin:paralin-of-ellip-S1}
		Under the hypothesis of Lemma \ref{lem-paralin:paralin-of-ellip}, we have,
		\begin{equation}\label{eq-paralin:paralin-of-ellip-S1}
			P\Phi \sim - T_{\partial_\rho^2\varphi} \alpha - T_{\nabla_w\partial_\rho\varphi}\cdot\beta - T_{\partial_\rho\varphi}\gamma - PT_{\eta^{-1}\partial_\rho\varphi}\rho\eta,
		\end{equation}
		in the sense of \eqref{eq-paralin:equi-paralin-of-ellip}. Recall that $w$ stands for the couple of variable $(\theta,z)$.
	\end{lemma}
	\begin{proof}
		By definition \eqref{eq-paralin:good-unknown-ext} of the good unknown $\Phi$, we have
		\begin{align*}
			P\Phi =& P\varphi - PT_{\eta^{-1}\partial_\rho\varphi}\rho\eta \\
			=& T_\alpha\partial_\rho^2\varphi + T_{i\beta\cdot\xi+\gamma}\partial_\rho\varphi - T_{\rho^{-2}\eta^{-2}\xi_\theta^2+\xi_z^2}\varphi - PT_{\eta^{-1}\partial_\rho\varphi}\rho\eta \\
			=& T_\alpha\partial_\rho^2\varphi + T_{\beta}\cdot\nabla_w\partial_\rho\varphi+T_{\gamma}\partial_\rho\varphi + T_{\rho^{-2}\eta^{-2}}\partial_\theta^2\varphi + \partial_z^2\varphi - PT_{\eta^{-1}\partial_\rho\varphi}\rho\eta + (-T_{\xi_z^2}\varphi - \partial_z^2\varphi) \\
			=& L\varphi - T_{\partial_\rho^2\varphi} \alpha - T_{\nabla_w\partial_\rho\varphi}\cdot\beta - T_{\partial_\rho\varphi}\gamma -  T_{\partial_\theta^2\varphi}\rho^{-2}\eta^{-2} - PT_{\eta^{-1}\partial_\rho\varphi}\rho\eta \\
			& -R(\partial_\rho^2\varphi,\alpha) - R(\nabla_w\partial_\rho\varphi, \beta)  - R(\partial_\rho\varphi,\gamma) - R(\partial_\theta^2\varphi,\rho^{-2}\eta^{-2}) + (-T_{\xi_z^2}\varphi - \partial_z^2\varphi)\\
			\sim& L\varphi - T_{\partial_\rho^2\varphi} \alpha - T_{\nabla_w\partial_\rho\varphi}\cdot\beta - T_{\partial_\rho\varphi}\gamma - PT_{\eta^{-1}\partial_\rho\varphi}\rho\eta \\
			=& - T_{\partial_\rho^2\varphi} \alpha - T_{\nabla_w\partial_\rho\varphi}\cdot\beta - T_{\partial_\rho\varphi}\gamma - PT_{\eta^{-1}\partial_\rho\varphi}\rho\eta.
		\end{align*}
		Recall that $L$ is the Laplacian in coordinate $(\rho,\theta,z)$ defined in \eqref{eq-paralin:ellip}. To check this equivalence, we first observe that, thanks to Lemma \ref{lem-paralin:reg-of-pot}, $\nabla_{\rho,w}^2\varphi\in C^0_\rho H^{s_0-2-}$ and $\partial_{\rho}\varphi\in C^0_\rho H^{s_0-1-}$. By Proposition \ref{prop-para:paraprod-bound}, Remark \ref{rmk-para:add-constant}, and estimate \eqref{eq-paralin:esti-alpha-beta-gamma}, we have the following estimates:
		\begin{align*}
			&\|T_{\partial_\theta^2\varphi}\rho^{-2}\eta^{-2}\|_{C^0_\rho H^{s_0-\frac{1}{2}+\epsilon}} + \|R(\partial_\theta^2\varphi,\rho^{-2}\eta^{-2})\|_{C^0_\rho H^{s_0-\frac{1}{2}+\epsilon}} \\
			&\le \|T_{\partial_\theta^2\varphi}(\eta^{-2}-R^{-2})\|_{C^0_\rho H^{s_0-\frac{1}{2}+\epsilon}} + \|R(\partial_\theta^2\varphi,\eta^{-2})\|_{C^0_\rho H^{s_0-\frac{1}{2}+\epsilon}} \\
			&\lesssim \left(\|\eta^{-2}\|_{H^{\max(\frac{5}{2},s_0-\frac{1}{2})+\epsilon+}_{R^{-2}}} + 1 \right) \|\partial_\theta^2\varphi\|_{C^0_\rho H^{s_0-2-}} \le C\left(\|\eta\|_{H_R^{s+\frac{1}{2}-}}\right) \|\psi\|_{H^{s_0}}; \\
			&\| R(\partial_\rho^2\varphi,\alpha) + R(\nabla_w\partial_\rho\varphi, \beta) \|_{C^0_\rho H^{s_0-\frac{1}{2}+\epsilon}} \\
			&\lesssim \left(\|\alpha\|_{C^0_\rho H^{\frac{5}{2}+\epsilon+}_{R^{-2}}} + 1 \right) \|\partial_\rho^2\varphi\|_{C^0_\rho H^{s_0-2-}} + \|\beta\|_{C^0_\rho H^{\frac{5}{2}+\epsilon+}} \|\nabla_w \partial_\rho\varphi\|_{C^0_\rho H^{s_0-2-}} \\
			&\le C\left(\|\eta\|_{H_R^{s+\frac{1}{2}-}}\right) \|\psi\|_{H^{s_0}}; \\			
			&\|R(\partial_\rho\varphi,\gamma)\|_{C^0_\rho H^{s_0-\frac{1}{2}+\epsilon}} \lesssim \left(\|\gamma\|_{C^0_\rho H^{\frac{3}{2}+\epsilon+}_{\rho^{-1}R^{-2}}} + 1 \right) \|\partial_\rho\varphi\|_{C^0_\rho H^{s_0-1-}}\le C\left(\|\eta\|_{H_R^{s+\frac{1}{2}-}}\right) \|\psi\|_{H^{s_0}}.
		\end{align*}
		As for the remaining term $-T_{\xi_z^2}\varphi - \partial_z^2\varphi$, we observe that, since the symbol of $\partial_z^2$ equals $-\xi_z^2$, the difference $-T_{\xi_z^2} - \partial_z^2$ is a smoothing operator. Namely, for all $M \gg 1$, we have
		\begin{equation*}
			\|-T_{\xi_z^2}\varphi - \partial_z^2\varphi\|_{C^0H^{s_0-\frac{1}{2}+\epsilon}} \lesssim \|\varphi\|_{C^0_\rho H^{-M}} \le C\left(\|\eta\|_{H_R^{s+\frac{1}{2}-}}\right) \|\psi\|_{H^{s_0}}.
		\end{equation*}
	\end{proof}
	
	Let us proceed with the proof of Lemma \ref{lem-paralin:paralin-of-ellip}. We observe that $r_1 = P\Phi$ is linear in $\psi$. Therefore, an interpolation argument allows us to reduce to the endpoints $s_0=\frac{3}{2}+\delta$ and $s_0=s$. If $s_0=\frac{3}{2}+\delta$ with $\delta>0$ fixed and small enough, the regularity of $\varphi$ reads $\nabla_{\rho,w}^2\varphi\in C^0_\rho H^{-\frac{1}{2}+\delta}$ and $\partial_{\rho}\varphi\in C^0_\rho H^{\frac{1}{2}+\delta}$ (see Lemma \ref{lem-paralin:reg-of-pot}). As a result, by Proposition \ref{prop-para:paraprod-bound} and \eqref{eq-paralin:esti-alpha-beta-gamma}, we have
	\begin{align*}
		&\|T_{\partial_\rho^2\varphi} \alpha + T_{\nabla_w\partial_\rho\varphi}\cdot\beta + T_{\partial_\rho\varphi}\gamma\|_{C^0_\rho H^{1+\delta+\epsilon}} \\
		&\lesssim \|\partial_\rho^2\varphi\|_{C^0_\rho H^{-\frac{1}{2}+\delta}} \|\alpha\|_{C^0_\rho H^{\frac{5}{2}+\epsilon}_{R^{-2}}} + \|\nabla_w\partial_\rho\varphi\|_{C^0_\rho H^{-\frac{1}{2}+\delta}} \|\beta\|_{C^0_\rho H^{\frac{5}{2}+\epsilon}} + \|\partial_\rho\varphi\|_{C^0_\rho H^{\frac{1}{2}+\delta}} \|\gamma\|_{C^0_\rho H^{\frac{3}{2}+\epsilon}_{\rho^{-1}R^{-2}}} \\
		&\lesssim C\left(\|\eta\|_{H_R^{s+\frac{1}{2}-}}\right) \|\psi\|_{H^{\frac{3}{2}+\delta}}.
	\end{align*}
	Meanwhile, Proposition \ref{prop-para:paraprod-bound} and Corollary \ref{cor-para:product-law} give
	\begin{equation*}
		\|T_{\eta^{-1}\partial_\rho\varphi}\rho\eta\|_{C^0_\rho H^{3+\delta+\epsilon}} \lesssim \|\eta^{-1}\partial_\rho\varphi\|_{C^0_\rho H^{\frac{1}{2}+\delta}} \|\rho\eta\|_{C^0_\rho H^{\frac{7}{2}+\epsilon}_{R}} \lesssim C\left(\|\eta\|_{H_R^{s+\frac{1}{2}-}}\right) \|\psi\|_{H^{\frac{3}{2}+\delta}}.
	\end{equation*}
	Since $P$ is an operator of order $2$, we can conclude 
	\begin{equation*}
		P\Phi \sim - T_{\partial_\rho^2\varphi} \alpha - T_{\nabla_w\partial_\rho\varphi}\cdot\beta - T_{\partial_\rho\varphi}\gamma - PT_{\eta^{-1}\partial_\rho\varphi}\rho\eta \sim 0.
	\end{equation*}
	Note that the extra term $T_{\eta^{-1}\partial_\rho\varphi}\rho\eta$ from \textit{Alinhac's good unknown} is useless in this case $s_0=3/2+\delta$, while it will be crucial in the case $s_0=s$. In fact, when $s_0=s$, the estimates above are no more correct. For example, we can only prove $T_{\partial_\rho^2\varphi} \alpha \in C^0_\rho H^{s-\frac{1}{2}} = C^0_\rho H^{s_0-\frac{1}{2}}$ instead of $C^0_\rho H^{s_0-\frac{1}{2}+\epsilon}$.
	
	From now on, we take $s_0=s$ and check that the extra term $PT_{\eta^{-1}\partial_\rho\varphi}\rho\eta$ is able to cancel the low regular components of  $T_{\partial_\rho^2\varphi} \alpha$, $T_{\nabla_w\partial_\rho\varphi}\cdot\beta$, and $T_{\partial_\rho\varphi}\gamma$. To begin with, by definition \eqref{eq-paralin:def-para-Lap} of $P$,
	\begin{equation*}
		PT_{\eta^{-1}\partial_\rho\varphi}\rho\eta = T_\alpha\partial_\rho^2 T_{\eta^{-1}\partial_\rho\varphi}\rho\eta + T_{i\beta\cdot\xi}\partial_\rho T_{\eta^{-1}\partial_\rho\varphi}\rho\eta + T_{\gamma}\partial_\rho T_{\eta^{-1}\partial_\rho\varphi}\rho\eta - T_{\rho^{-2}\eta^{-2}\xi_\theta^2+\xi_z^2}T_{\eta^{-1}\partial_\rho\varphi}\rho\eta,
	\end{equation*}
	and we will separately study each terms on the right hand side. As a corollary of the following estimates, we are able to deduce that
	\begin{equation}\label{eq-paralin:paralin-of-ellip-S1'}
		\begin{aligned}
			P\Phi \sim& - T_{\partial_\rho^2\varphi} \alpha - T_{\nabla_w\partial_\rho\varphi}\cdot\beta - T_{\partial_\rho\varphi}\gamma - T_\beta\cdot T_{\eta^{-1}\partial_\rho^2\varphi}\rho\nabla_w\eta - T_\beta\cdot T_{\eta^{-1}\partial_\rho\varphi}\nabla_w\eta \\
			&- 2\rho^{-1} T_{\eta^{-2}}T_{\partial_\theta\left(\eta^{-1}\partial_\rho\varphi\right)}\eta_\theta - \rho^{-1}T_{\eta^{-2}}T_{\eta^{-1}\partial_\rho\varphi}\eta_{\theta\theta} - 2\rho T_{\partial_z\left(\eta^{-1}\partial_\rho\varphi\right)}\eta_z - \rho T_{\eta^{-1}\partial_\rho\varphi}\eta_{zz},
		\end{aligned}
	\end{equation}
	when $s_0=s$.
	
	\begin{lemma}\label{lem-paralin:paralin-of-ellip-S2}
		Under the hypotheses of Lemma \ref{lem-paralin:paralin-of-ellip}, we have the following equivalences:
		\begin{align}
			 T_\alpha\partial_\rho^2 T_{\eta^{-1}\partial_\rho\varphi}\rho\eta \sim& 0, \label{eq-paralin:paralin-of-ellip-S2-1} \\
			 T_{i\beta\cdot\xi}\partial_\rho T_{\eta^{-1}\partial_\rho\varphi}\rho\eta \sim&  T_\beta\cdot T_{\eta^{-1}\partial_\rho^2\varphi}\rho\nabla_w\eta + T_\beta\cdot T_{\eta^{-1}\partial_\rho\varphi}\nabla_w\eta, \label{eq-paralin:paralin-of-ellip-S2-2} \\
			 T_{\gamma}\partial_\rho T_{\eta^{-1}\partial_\rho\varphi}\rho\eta \sim& 0, \label{eq-paralin:paralin-of-ellip-S2-3} \\
			 - T_{\rho^{-2}\eta^{-2}\xi_\theta^2+\xi_z^2}T_{\eta^{-1}\partial_\rho\varphi}\rho\eta \sim& 2\rho^{-1} T_{\eta^{-2}}T_{\partial_\theta\left(\eta^{-1}\partial_\rho\varphi\right)}\eta_\theta + \rho^{-1}T_{\eta^{-2}}T_{\eta^{-1}\partial_\rho\varphi}\eta_{\theta\theta} \label{eq-paralin:paralin-of-ellip-S2-4} \\
			 &+ 2\rho T_{\partial_z\left(\eta^{-1}\partial_\rho\varphi\right)}\eta_z + \rho T_{\eta^{-1}\partial_\rho\varphi}\eta_{zz}, \nonumber
		\end{align}
		in the sense of \eqref{eq-paralin:equi-paralin-of-ellip} with $s_0=s$.
	\end{lemma}
	\begin{proof}
		As before, we observe that $\eta^{-1}\partial_\rho\varphi \in C^0_\rho H^{s-1}$ and thus $T_{\eta^{-1}\partial_\rho\varphi}\rho\eta\in C^0_\rho H^{s+\frac{1}{2}}$. To prove the first equivalence \eqref{eq-paralin:paralin-of-ellip-S2-1}, we apply Proposition \ref{prop-para:paraprod-bound}, Corollary \ref{cor-para:product-law}, together with \eqref{eq-paralin:esti-alpha-beta-gamma} and obtain
		\begin{align*}
			&\|T_\alpha\partial_\rho^2 T_{\eta^{-1}\partial_\rho\varphi}\rho\eta\|_{C^0_\rho H^{s-\frac{1}{2}+\epsilon}} = \|T_\alpha T_{\eta^{-1}\partial_\rho^3\varphi}\rho\eta+ 2T_\alpha T_{\eta^{-1}\partial_\rho^2\varphi}\eta\|_{C^0_\rho H^{s-\frac{1}{2}+\epsilon}} \\
			\lesssim& \left(\|\alpha\|_{C^0_\rho H^{1+}_{R^{-2}}} + 1 \right) \left( \| T_{\eta^{-1}\partial_\rho^3\varphi}\eta\|_{C^0_\rho H^{s-\frac{1}{2}+\epsilon}} + \| T_{\eta^{-1}\partial_\rho^2\varphi}\eta\|_{C^0_\rho H^{s-\frac{1}{2}+\epsilon}} \right) \\
			\le& C\left(\|\eta\|_{H_R^{s+\frac{1}{2}-}}\right) \left( \|\eta^{-1}\partial_\rho^3\varphi\|_{C^0_\rho H^{s-3-}} \|\eta\|_{H^{\max(\frac{7}{2},s-\frac{1}{2})+\epsilon+}_R} + \|\eta^{-1}\partial_\rho^2\varphi\|_{C^0_\rho H^{s-2-}} \|\eta\|_{H^{\max(\frac{5}{2},s-\frac{1}{2})+\epsilon+}_R} \right) \\
			\le& C\left(\|\eta\|_{H_R^{s+\frac{1}{2}-}}\right) \left( \|\eta^{-1}\partial_\rho^3\varphi\|_{C^0_\rho H^{s-3-}}  + \|\eta^{-1}\partial_\rho^2\varphi\|_{C^0_\rho H^{s-2-}} \right) \\
			\le& C\left(\|\eta\|_{H_R^{s+\frac{1}{2}-}}\right) \left(\|\eta^{-1}\|_{H^{\max(1+,s-2)}_{R^{-1}}} + 1 \right) \left( \|\partial_\rho^3\varphi\|_{C^0_\rho H^{s-3-}}  + \|\partial_\rho^2\varphi\|_{C^0_\rho H^{s-2-}} \right) \\
			\le& C\left(\|\eta\|_{H_R^{s+\frac{1}{2}-}}\right) \|\psi\|_{H^{s}}.
		\end{align*}
		Recall that $\partial_\rho^3\varphi \in C^0_\rho H^{s-3}$ thanks to \eqref{eq-paralin:reg-of-pot-loc}. The term $T_{i\beta\cdot\xi}\partial_\rho T_{\eta^{-1}\partial_\rho\varphi}\rho\eta$ in the second equivalence \eqref{eq-paralin:paralin-of-ellip-S2-2} can be written as
		\begin{equation*}
			T_\beta\cdot T_{\nabla_w\left(\eta^{-1}\partial_\rho^2\varphi\right)}\rho\eta + T_\beta\cdot T_{\eta^{-1}\partial_\rho^2\varphi}\rho\nabla_w\eta + T_\beta\cdot T_{\nabla_w\left(\eta^{-1}\partial_\rho\varphi\right)}\eta + T_\beta\cdot T_{\eta^{-1}\partial_\rho\varphi}\nabla_w\eta,
		\end{equation*}
		with
		\begin{align*}
			&\| T_\beta\cdot T_{\nabla_w\left(\eta^{-1}\partial_\rho^2\varphi\right)}\rho\eta + T_\beta\cdot T_{\nabla_w\left(\eta^{-1}\partial_\rho\varphi\right)}\eta \|_{C^0_\rho H^{s-\frac{1}{2}+\epsilon}} \\
			\lesssim& \|\beta\|_{C^0_\rho H^{1+}} \| T_{\nabla_w\left(\eta^{-1}\partial_\rho^2\varphi\right)}\rho\eta +  T_{\nabla_w\left(\eta^{-1}\partial_\rho\varphi\right)}\eta \|_{C^0_\rho H^{s-\frac{1}{2}+\epsilon}} \\
			\le& C\left(\|\eta\|_{H_R^{s+\frac{1}{2}-}}\right) \left( \|\nabla_w\left(\eta^{-1}\partial_\rho^2\varphi\right)\|_{C^0_\rho H^{s-3-}} + \|\nabla_w\left(\eta^{-1}\partial_\rho\varphi\right)\|_{C^0_\rho H^{s-2-}} \right) \|\eta\|_{C^0_\rho H^{\max(\frac{7}{2},s-\frac{1}{2})+\epsilon+}} \\
			\le& C\left(\|\eta\|_{H_R^{s+\frac{1}{2}-}}\right) \|\psi\|_{H^{s}},
		\end{align*}
		which proves \eqref{eq-paralin:paralin-of-ellip-S2-2}. The third equivalence \eqref{eq-paralin:paralin-of-ellip-S2-3} is due to the estimate
		\begin{align*}
			&\| T_{\gamma}\partial_\rho T_{\eta^{-1}\partial_\rho\varphi}\rho\eta \|_{C^0_\rho H^{s-\frac{1}{2}+\epsilon}} = \| T_{\gamma} T_{\eta^{-1}\partial_\rho^2\varphi}\rho\eta + T_{\gamma}T_{\eta^{-1}\partial_\rho\varphi}\eta \|_{C^0_\rho H^{s-\frac{1}{2}+\epsilon}}  \\
			\lesssim& \left(\|\gamma\|_{C^0_\rho H^{1+}_{\rho^{-1}R^{-2}}} + 1 \right) \| T_{\eta^{-1}\partial_\rho^2\varphi}\rho\eta + T_{\eta^{-1}\partial_\rho\varphi}\eta \|_{C^0_\rho H^{s-\frac{1}{2}+\epsilon}} \le C\left(\|\eta\|_{H_R^{s+\frac{1}{2}-}}\right) \|\psi\|_{H^{s}},
		\end{align*}
		where one may apply Proposition \ref{prop-para:paralin} to deal with $\eta^{-1}$. The last equivalence can be treated as follow,
		\begin{align*}
			- T_{\rho^{-2}\eta^{-2}\xi_\theta^2}T_{\eta^{-1}\partial_\rho\varphi}\rho\eta =& \rho^{-1}T_{\eta^{-2}}T_{\partial_\theta^2\left(\eta^{-1}\partial_\rho\varphi\right)}\eta + 2\rho^{-1} T_{\eta^{-2}}T_{\partial_\theta\left(\eta^{-1}\partial_\rho\varphi\right)}\eta_\theta + \rho^{-1}T_{\eta^{-2}}T_{\eta^{-1}\partial_\rho\varphi}\eta_{\theta\theta}, \\
			- T_{\xi_z^2}T_{\eta^{-1}\partial_\rho\varphi}\rho\eta =& \rho T_{\partial_z^2\left(\eta^{-1}\partial_\rho\varphi\right)}\eta + 2\rho T_{\partial_z\left(\eta^{-1}\partial_\rho\varphi\right)}\eta_z + \rho T_{\eta^{-1}\partial_\rho\varphi}\eta_{zz} \\
			&+ \left( - T_{\xi_z^2} - \partial_z^2 \right)T_{\eta^{-1}\partial_\rho\varphi}\rho\eta,
		\end{align*}
		where similar estimates as above can be derived:
		\begin{align*}
			\| \rho^{-1}T_{\eta^{-2}}T_{\partial_\theta^2\left(\eta^{-1}\partial_\rho\varphi\right)}\eta \|_{C^0_\rho H^{s-\frac{1}{2}+\epsilon}} \lesssim& \left( \|\eta^{-2}\|_{H^{1+}_{R^{-2}}} + 1 \right) \|\partial_\theta^2\left(\eta^{-1}\partial_\rho\varphi\right)\|_{C^0_\rho H^{s-3-}} \|\eta\|_{H^{\max(\frac{7}{2},s-\frac{1}{2})+\epsilon+}_R} \\
			\le& C\left(\|\eta\|_{H_R^{s+\frac{1}{2}-}}\right) \|\eta^{-1}\partial_\rho\varphi\|_{C^0_\rho H^{s-1-}} \le C\left(\|\eta\|_{H_R^{s+\frac{1}{2}-}}\right) \|\psi\|_{H^{s}},
		\end{align*}
		and
		\begin{align*}
			\| \rho T_{\partial_z^2\left(\eta^{-1}\partial_\rho\varphi\right)}\eta \|_{C^0_\rho H^{s-\frac{1}{2}+\epsilon}} \lesssim& \|\partial_z^2\left(\eta^{-1}\partial_\rho\varphi\right)\|_{C^0_\rho H^{s-3-}} \|\eta\|_{H^{\max(\frac{7}{2},s-\frac{1}{2})+\epsilon+}_R} \\
			\le& C\left(\|\eta\|_{H_R^{s+\frac{1}{2}-}}\right) \|\eta^{-1}\partial_\rho\varphi\|_{C^0_\rho H^{s-1-}} \le C\left(\|\eta\|_{H_R^{s+\frac{1}{2}-}}\right) \|\psi\|_{H^{s}}.
		\end{align*}
		The desired result \eqref{eq-paralin:paralin-of-ellip-S2-4} follows from $\left( - T_{\xi_z^2} - \partial_z^2 \right)T_{\eta^{-1}\partial_\rho\varphi}\rho\eta\sim 0$ since $\left( - T_{\xi_z^2} - \partial_z^2 \right)$ is a smoothing operator. 
	\end{proof}
	
	\begin{proof}[Proof of Lemma \ref{lem-paralin:paralin-of-ellip}]
		Till now, we have managed to prove that, in the sense of \eqref{eq-paralin:equi-paralin-of-ellip},
		\begin{equation*}
			P\Phi \sim 0
		\end{equation*}
		when $s_0 = \frac{3}{2}+\delta$ with $0<\delta\ll 1$, and moreover, when $s_0=s$, \eqref{eq-paralin:paralin-of-ellip-S1'} holds. It remains to check that the right hand side of \eqref{eq-paralin:paralin-of-ellip-S1'} is equivalent to zero, namely 
		\begin{align*}
			&- T_{\partial_\rho^2\varphi} \alpha - T_{\nabla_w\partial_\rho\varphi}\cdot\beta - T_{\partial_\rho\varphi}\gamma - T_\beta\cdot T_{\eta^{-1}\partial_\rho^2\varphi}\rho\nabla_w\eta - T_\beta\cdot T_{\eta^{-1}\partial_\rho\varphi}\nabla_w\eta \\
			&- 2\rho^{-1} T_{\eta^{-2}}T_{\partial_\theta\left(\eta^{-1}\partial_\rho\varphi\right)}\eta_\theta - \rho^{-1}T_{\eta^{-2}}T_{\eta^{-1}\partial_\rho\varphi}\eta_{\theta\theta} - 2\rho T_{\partial_z\left(\eta^{-1}\partial_\rho\varphi\right)}\eta_z - \rho T_{\eta^{-1}\partial_\rho\varphi}\eta_{zz} \sim 0.
		\end{align*}
		Firstly, we claim that 
		\begin{equation*}
			- T_{\partial_\rho^2\varphi} \alpha - T_\beta\cdot T_{\eta^{-1}\partial_\rho^2\varphi}\rho\nabla_w\eta \sim 0.
		\end{equation*}
		In fact, by symbolic calculus (Proposition \ref{prop-paralin:homo-sym-cal}) and commutator estimate (Corollary \ref{cor-para:commu-esti}), we have
		\begin{align*}
			T_\beta\cdot T_{\eta^{-1}\partial_\rho^2\varphi}\rho\nabla_w\eta \sim& T_{\eta^{-1}\partial_\rho^2\varphi} T_{\beta}\cdot\rho\nabla_w\eta \sim T_{\partial_\rho^2\varphi} T_{\eta^{-1}} T_{\beta}\cdot\rho\nabla_w\eta \\
			\sim& T_{\partial_\rho^2\varphi} T_{\eta^{-1}\beta}\cdot\rho\nabla_w\eta,
		\end{align*}
		where we use the fact that $\beta,\eta^{-1}\in C^0_\rho H^{s-\frac{1}{2}}\subset \Gamma^0_{0+}$ and $\partial_\rho^2\varphi\in C^0_\rho H^{s-2}\subset \Gamma^0_{0+}$. As a consequence, in each step above, the error can be written as a $(0-)$-order operator acting on $\nabla_w \eta$, which lies in $C^0_\rho H^{s-\frac{1}{2}+}$. Meanwhile, by definition \eqref{eq-paralin:def-alpha} of $\alpha$,
		\begin{align*}
			T_{\partial_\rho^2\varphi} \alpha \sim T_{\partial_\rho^2\varphi} \left( \eta^{-4}\eta_\theta^2 + \rho^2\eta^{-2}\eta_z^2 \right) \sim T_{\partial_\rho^2\varphi} \left( T_{2\eta^{-4}\eta_\theta}\eta_\theta + T_{2\rho^2\eta^{-2}\eta_z}\eta_z \right) = T_{\partial_\rho^2\varphi} T_{-\rho\eta^{-1}\beta} \cdot\nabla_w\eta,
		\end{align*}
		where the first equivalence is due to
		\begin{align*}
			\| T_{\partial_\rho^2\varphi}\eta^{-2} \|_{C^0_\rho H^{s-\frac{1}{2}+\epsilon}} \lesssim \|\partial_\rho^2\varphi\|_{C^0_\rho H^{s-2-}} \|\eta^{-2}\|_{H^{\max(\frac{5}{2},s-\frac{1}{2})+\epsilon+}_{R^{-2}}} \lesssim C\left(\|\eta\|_{H_R^{s+\frac{1}{2}-}}\right) \|\psi\|_{H^{s}},
		\end{align*}
		and the second equivalence is a consequence of Corollary \ref{cor-para:quadra-prod}. 
		
		Now our problem is reduced to 
		\begin{align*}
			P\Phi \sim& - T_{\nabla_w\partial_\rho\varphi}\cdot\beta - T_{\partial_\rho\varphi}\gamma - T_\beta\cdot T_{\eta^{-1}\partial_\rho\varphi}\nabla_w\eta \\
			&- 2\rho^{-1} T_{\eta^{-2}}T_{\partial_\theta\left(\eta^{-1}\partial_\rho\varphi\right)}\eta_\theta - \rho^{-1}T_{\eta^{-2}}T_{\eta^{-1}\partial_\rho\varphi}\eta_{\theta\theta} - 2\rho T_{\partial_z\left(\eta^{-1}\partial_\rho\varphi\right)}\eta_z - \rho T_{\eta^{-1}\partial_\rho\varphi}\eta_{zz}.
		\end{align*}
		We observe that
		\begin{align*}
			2\rho^{-1} T_{\eta^{-2}}T_{\partial_\theta\left(\eta^{-1}\partial_\rho\varphi\right)}\eta_\theta =& 2\rho^{-1} T_{\eta^{-2}}T_{\partial_\theta\left(\eta^{-1}\right)\partial_\rho\varphi}\eta_\theta + 2\rho^{-1} T_{\eta^{-2}}T_{\eta^{-1}\partial_\theta\partial_\rho\varphi}\eta_\theta \\
			\sim& 2\rho^{-1} T_{\eta^{-2}}T_{\partial_\rho\varphi}T_{\partial_\theta\left(\eta^{-1}\right)}\eta_\theta + 2\rho^{-1} T_{\eta^{-2}}T_{\partial_\theta\partial_\rho\varphi}T_{\eta^{-1}}\eta_\theta \\
			\sim& 2\rho^{-1} T_{\partial_\rho\varphi}T_{\eta^{-2}}T_{\partial_\theta\left(\eta^{-1}\right)}\eta_\theta + 2\rho^{-1} T_{\partial_\theta\partial_\rho\varphi}T_{\eta^{-2}}T_{\eta^{-1}}\eta_\theta \\
			\sim& 2\rho^{-1} T_{\partial_\rho\varphi}T_{\eta^{-2}\partial_\theta\left(\eta^{-1}\right)}\eta_\theta + 2\rho^{-1} T_{\partial_\theta\partial_\rho\varphi}T_{\eta^{-3}}\eta_\theta
		\end{align*}
		These equivalences are consequences of symbolic calculus (Proposition \ref{prop-para:paradiff-cal-sym} and Corollary \ref{cor-para:commu-esti}) and commutator estimate (Corollary \ref{cor-para:commu-esti}), as well as the fact that all the involved symbols belong to $C^0_\rho H^{s-2-} \subset \Gamma^0_{1/2+}$, which guarantees that the error in each step takes the form of $T\eta_\theta$ with $T$ linear and of order $-\frac{1}{2}-$. Moreover, an application of Proposition \ref{prop-para:paraprod-bound} gives
		\begin{align*}
			\|T_{\eta^{-3}}\eta_\theta - \eta^{-3}\eta_\theta\|_{H^{s-\frac{1}{2}+}} \le& \|T_{\eta_\theta}\eta^{-3}+ R(\eta^{-3},\eta_\theta)\|_{H^{s-\frac{1}{2}+}} \\
			\lesssim& \|\eta_\theta\|_{H^{1+}} \|\eta^{-3} \|_{H^{s-\frac{1}{2}+\epsilon}_{R^{-3}}} \lesssim C\left( \|\eta\|_{H^{s+\frac{1}{2}-}_R} \right),
		\end{align*}
		and thus
		\begin{equation*}
			2\rho^{-1} T_{\eta^{-2}}T_{\partial_\theta\left(\eta^{-1}\partial_\rho\varphi\right)}\eta_\theta \sim 2\rho^{-1} T_{\partial_\rho\varphi}T_{\eta^{-2}\partial_\theta\left(\eta^{-1}\right)}\eta_\theta + 2\rho^{-1} T_{\partial_\theta\partial_\rho\varphi}\left(\eta^{-3}\eta_\theta\right).
		\end{equation*}
		Similarly, we have
		\begin{equation*}
			2\rho T_{\partial_z\left(\eta^{-1}\partial_\rho\varphi\right)}\eta_z \sim 2\rho T_{\partial_\rho\varphi}T_{\partial_z(\eta^{-1})}\eta_z + 2\rho T_{\partial_z\partial_\rho\varphi}\left(\eta^{-1}\eta_z\right).
		\end{equation*}
		By observing that
		\begin{equation*}
			T_{\nabla_w\partial_\rho\varphi}\cdot\beta = -2\rho^{-1} T_{\partial_\theta\partial_\rho\varphi}\eta^{-3}\eta_\theta - 2\rho T_{\partial_z\partial_\rho\varphi}\eta^{-1}\eta_z,
		\end{equation*}
		we can further reduce our problem to
		\begin{align*}
			P\Phi \sim& - T_{\partial_\rho\varphi}\gamma - T_\beta\cdot T_{\eta^{-1}\partial_\rho\varphi}\nabla_w\eta \\
			& -2\rho^{-1} T_{\partial_\rho\varphi}T_{\eta^{-2}\partial_\theta\left(\eta^{-1}\right)}\eta_\theta - \rho^{-1}T_{\eta^{-2}}T_{\eta^{-1}\partial_\rho\varphi}\eta_{\theta\theta} - 2\rho T_{\partial_\rho\varphi}T_{\partial_z(\eta^{-1})}\eta_z- \rho T_{\eta^{-1}\partial_\rho\varphi}\eta_{zz} \\
			\sim& - T_{\partial_\rho\varphi}\gamma - T_{\partial_\rho\varphi} T_{\eta^{-1}\beta}\cdot \nabla_w\eta \\
			& -2\rho^{-1} T_{\partial_\rho\varphi}T_{\eta^{-2}\partial_\theta\left(\eta^{-1}\right)}\eta_\theta - \rho^{-1}T_{\partial_\rho\varphi}T_{\eta^{-3}}\eta_{\theta\theta} - 2\rho T_{\partial_\rho\varphi}T_{\partial_z(\eta^{-1})}\eta_z- \rho T_{\partial_\rho\varphi}T_{\eta^{-1}}\eta_{zz},
		\end{align*}
		due to symbolic calculus (Proposition \ref{prop-para:paradiff-cal-sym} and Corollary \ref{cor-para:commu-esti}). From definition \eqref{eq-paralin:def-beta} and \eqref{eq-paralin:def-gamma} of $\beta$ and $\gamma$, respectively, it is easy to calculate that
		\begin{align*}
			-\gamma =& - \eta^{-1}\beta\cdot \nabla_w\eta + \rho^{-1}\eta^{-3}\eta_{\theta\theta}  + \rho \eta^{-1}\eta_{zz} + \rho^{-1}\eta^{-2}, \\
			\eta^{-1}\beta =& \left( 2\rho^{-1}\eta^{-2}\partial_\theta(\eta^{-1}), 2\rho \partial_z(\eta^{-1}) \right).
		\end{align*}
		As a result,
		\begin{align*}
			P\Phi \sim& - \rho^{-1} T_{\partial_\rho\varphi}\eta^{-2} + T_{\partial_\rho\varphi} \left( \eta^{-1}\beta\cdot \nabla_w\eta - 2T_{\eta^{-1}\beta}\cdot \nabla_w\eta \right) \\
			&+ \rho^{-1}T_{\partial_\rho\varphi} \left( \eta^{-3}\eta_{\theta\theta} - T_{\eta^{-3}}\eta_{\theta\theta} \right) + \rho T_{\partial_\rho\varphi} \left( \eta^{-1}\eta_{zz} - T_{\eta^{-1}}\eta_{zz} \right),
		\end{align*}
		the right hand side of which is equivalent to zero in the sense of \eqref{eq-paralin:equi-paralin-of-ellip}, due to the estimates
		\begin{align*}
			\|T_{\partial_\rho\varphi}\eta^{-2}\|_{H^{s-\frac{1}{2}+\epsilon}} \lesssim& \|\partial_\rho\varphi\|_{H^{1+}} \|\eta^{-2}\|_{H^{s-\frac{1}{2}+\epsilon}_{R^{-2}}} \le C\left(\|\eta\|_{H_R^{s+\frac{1}{2}-}}\right) \|\psi\|_{H^{s}}, \\
			\|T_{\partial_\rho\varphi} \left( \eta^{-3}\eta_{\theta\theta} - T_{\eta^{-3}}\eta_{\theta\theta} \right)\|_{H^{s-\frac{1}{2}+\epsilon}} \le&  \|T_{\partial_\rho\varphi} T_{\eta_{\theta\theta}}\eta^{-3}\|_{H^{s-\frac{1}{2}+\epsilon}} + \|T_{\partial_\rho\varphi} R(\eta^{-3},\eta_{\theta\theta}) \|_{H^{s-\frac{1}{2}+\epsilon}} \\
			\lesssim& \|\partial_\rho\varphi\|_{H^{1+}} \|\eta_{\theta\theta}\|_{H^{1+}} \|\eta^{-3}\|_{H^{s-\frac{1}{2}+\epsilon}_{R^{-3}}} \\
			\le& C\left(\|\eta\|_{H_R^{s+\frac{1}{2}-}}\right) \|\psi\|_{H^{s}}, \\
			\|T_{\partial_\rho\varphi} \left( \eta^{-1}\eta_{zz} - T_{\eta^{-1}}\eta_{zz} \right)\|_{H^{s-\frac{1}{2}+\epsilon}} \le&  \|T_{\partial_\rho\varphi} T_{\eta_{zz}}\eta^{-1}\|_{H^{s-\frac{1}{2}+\epsilon}} + \|T_{\partial_\rho\varphi} R(\eta^{-1},\eta_{zz}) \|_{H^{s-\frac{1}{2}+\epsilon}} \\
			\lesssim& \|\partial_\rho\varphi\|_{H^{1+}} \|\eta_{zz}\|_{H^{1+}} \|\eta^{-1}\|_{H^{s-\frac{1}{2}+\epsilon}_{R^{-1}}} \\
			\le& C\left(\|\eta\|_{H_R^{s+\frac{1}{2}-}}\right) \|\psi\|_{H^{s}},
		\end{align*}
		as well as
		\begin{align*}
			&\|T_{\partial_\rho\varphi} \left( \eta^{-1}\beta\cdot \nabla_w\eta - 2T_{\eta^{-1}\beta}\cdot \nabla_w\eta \right)\|_{H^{s-\frac{1}{2}+\epsilon}} \\
			\lesssim& \|\partial_\rho\varphi\|_{H^{1+}} \| \eta^{-1}\beta\cdot \nabla_w\eta - 2T_{\eta^{-1}\beta}\cdot \nabla_w\eta \|_{H^{s-\frac{1}{2}+\epsilon}} \\
			\le& C\left(\|\eta\|_{H_R^{s+\frac{1}{2}-}}\right) \|\psi\|_{H^{s}} \left( \| \eta^{-4}\eta_\theta^2 - 2T_{\eta^{-4}\eta_\theta}\eta_\theta \|_{H^{s-\frac{1}{2}+\epsilon}} + \| \eta^{-2}\eta_z^2 - 2T_{\eta^{-2}\eta_z}\eta_z \|_{H^{s-\frac{1}{2}+\epsilon}} \right) \\
			\le& C\left(\|\eta\|_{H_R^{s+\frac{1}{2}-}}\right) \|\psi\|_{H^{s}},
		\end{align*}
		which is a consequence of Corollary \ref{cor-para:quadra-prod}.
	\end{proof}

	Now we are in a position to separate the normal and tangential derivatives. 
	\begin{lemma}\label{lem-paralin:decomp-of-P}
		Let $\eta\in H_R^{s+\frac{1}{2}}$ with $s>3$. There exists elliptic symbols $a,A\in\Sigma^{1}$, such that
		\begin{equation}\label{eq-paralin:decom-of-P}
			P = T_\alpha(\partial_\rho + T_a)(\partial_\rho - T_A) + R_0 + R_1\partial_\rho,
		\end{equation}
		where $R_0$ is of order $\frac{1}{2}-$, and $R_1$ is of order $-\frac{1}{2}-$, satisfying
		\begin{equation*}
			\|R_0\|_{\mathcal{L}(H^{t};H^{t-\frac{1}{2}+})} + \|R_1\|_{\mathcal{L}(H^{t};H^{t+\frac{1}{2}+})} \le C\left( \|\eta\|_{H^{s+\frac{1}{2}-}_R} \right),\ \ \forall t\in\R.
		\end{equation*}
		
		Moreover, $a = a^{(1)}+a^{(0)}$ and $A = A^{(1)} + A^{(0)}$ can be computed explicitly,
		\begin{align}
			A^{(1)} =& \frac{1}{2\alpha}\sqrt{4\alpha\left(\frac{\xi_\theta^2}{\rho^2\eta^2}+\xi_z^2\right) - (\beta\cdot\xi)^2} - \frac{i\beta\cdot\xi}{2\alpha}, \label{eq-paralin:def-A-1}\\
			a^{(1)} =& \frac{1}{2\alpha}\sqrt{4\alpha\left(\frac{\xi_\theta^2}{\rho^2\eta^2}+\xi_z^2\right) - (\beta\cdot\xi)^2} + \frac{i\beta\cdot\xi}{2\alpha}, \label{eq-paralin:def-a-1}
		\end{align}
		and
		\begin{align}
		A^{(0)} =& -\frac{1}{A^{(1)}+a^{(1)}} \left( \frac{A^{(1)}\gamma}{\alpha} + \partial_\rho A^{(1)} + \partial_\xi a^{(1)}\cdot D_w A^{(1)} \right), \label{eq-paralin:def-A-0}\\
		a^{(0)} =& -\frac{1}{A^{(1)}+a^{(1)}} \left( -\frac{a^{(1)}\gamma}{\alpha} + \partial_\rho A^{(1)} + \partial_\xi a^{(1)}\cdot D_w A^{(1)} \right). \label{eq-paralin:def-a-0}
		\end{align}
	\end{lemma}
	Note that all the operators are defined on $\T\times\R$ with $\rho\in[1-\delta,1]$ regarded as a parameter.
	\begin{proof}
		From \eqref{eq-paralin:decom-of-P}, we shall see what is the conditions required for $a$ and $A$. In fact, the right hand side of \eqref{eq-paralin:decom-of-P} equals
		\begin{align*}
			& T_\alpha(\partial_\rho + T_a)(\partial_\rho - T_A) + R_0 + R_1\partial_\rho \\
			=& T_\alpha \left( \partial_\rho^2 + T_{a-A}\partial_\rho - T_aT_A - T_{\partial_\rho A} \right) +R_0 + R_1\partial_\rho \\
			=& T_\alpha\partial_\rho^2 + \left( T_\alpha T_{a-A} + R_1 \right)\partial_\rho - T_\alpha \left( T_aT_A + T_{\partial_\rho A} \right) + R_0.
		\end{align*}
		By comparing it with definition \eqref{eq-paralin:def-para-Lap} of $P$, one could see that \eqref{eq-paralin:decom-of-P} holds if the construction of $a,A$ satisfies
		\begin{equation*}
			\left\{\begin{aligned}
				&T_\alpha T_{a-A} + R_1 = T_{i\beta\cdot\xi + \gamma},\\
				& T_\alpha \left( T_aT_A + T_{\partial_\rho A} \right) - R_0 = T_{\rho^{-2}\eta^{-2}\xi_\theta^2 + \xi_z^2}.
			\end{aligned}\right.
		\end{equation*}
		Recall that by \eqref{eq-paralin:def-alpha} and definition \ref{def-paralin:homo-sym}, $\alpha\in\Sigma^0$. Since $\eta\in H^{s+\frac{1}{2}}_{R}$, the end of definition \ref{def-paralin:homo-sym} shows in particular that $\alpha\in \Gamma^{0}_{3/2+} + \Gamma^{-1}_{1/2+}$. Moreover, $\alpha$ is elliptic and independent of $\xi$. For all $a,A\in\Sigma^1$, symbolic calculus (Proposition \ref{prop-paralin:homo-sym-cal}) gives that
		\begin{align*}
			T_\alpha T_{a-A} \approx& T_{\alpha\sharp(a-A)} = T_{\alpha(a-A)}, \\
			T_\alpha \left( T_aT_A + T_{\partial_\rho A} \right) \approx& T_\alpha T_{a\sharp A + \partial_\rho A} \approx T_{\alpha\sharp\left(a\sharp A + \partial_\rho A \right)} = T_{\alpha\left(a\sharp A + \partial_\rho A \right)} \approx T_{\alpha\left(a\sharp A + \partial_\rho A^{(1)} \right)},
		\end{align*}
		where the difference in the first line is of order $0+1-\frac{3}{2}- = -\frac{1}{2}-$ and those in the second line is of order $0+1+1-\frac{3}{2}- = \frac{1}{2}-$, allowing us to define
		\begin{align*}
			R_1 :=& T_{\alpha(a-A)} - T_\alpha T_{a-A}, \\
			R_0 :=& T_{\alpha\left(a\sharp A + \partial_\rho A \right)} - T_\alpha \left( T_aT_A + T_{\partial_\rho A^{(1)}} \right).
		\end{align*}
		Therefore, it suffices to construct $A = A^{(1)} + A^{(0)}$ and $a = a^{(1)} + a^{(0)}$ verifying
		\begin{equation*}
			\left\{\begin{aligned}
				&\alpha(a-A) = i\beta\cdot\xi + \gamma, \\
				&\alpha\left(a\sharp A + \partial_\rho A^{(1)} \right) = \rho^{-2}\eta^{-2}\xi_\theta^2 + \xi_z^2,
			\end{aligned}\right.
		\end{equation*}
		From \eqref{eq-paralin:homo-sym-comp-formula} and correspondence of order (in $\xi$), we only need to solve
		\begin{equation*}
			\left\{\begin{aligned}
				&\alpha(a^{(1)}-A^{(1)}) = i\beta\cdot\xi, \\
				&\alpha(a^{(0)}-A^{(0)}) = \gamma, \\
				&\alpha a^{(1)} A^{(1)} = \rho^{-2}\eta^{-2}\xi_\theta^2 + \xi_z^2, \\
				&\partial_\xi a^{(1)} \cdot D_w A^{(1)} + a^{(0)} A^{(1)} + a^{(1)} A^{(0)} + \partial_\rho A^{(1)} =0,
			\end{aligned}\right.
		\end{equation*}
		which allows us to obtain \eqref{eq-paralin:def-A-1}-\eqref{eq-paralin:def-a-0}.
	\end{proof}
	
	By combining Lemma \ref{lem-paralin:paralin-of-ellip} and \ref{lem-paralin:decomp-of-P}, we have
	\begin{equation}\label{eq-paralin:decomp-ellip}
		T_\alpha(\partial_\rho + T_a)(\partial_\rho - T_A)\Phi = -R_0\Phi - R_1\partial_\rho\Phi + r_1.
	\end{equation}
	Under the hypothesis of Proposition \ref{prop-paralin:normal-to-tangential}, we have
	\begin{equation}\label{eq-paralin:decomp-esti-remainder}
		\|-R_0\Phi - R_1\partial_\rho\Phi + r_1\|_{C^0_\rho H^{s_0-\frac{1}{2}+\epsilon}} \le C\left(\|\eta\|_{H_R^{s+\frac{1}{2}-}}\right) \|\psi\|_{H^{s_0}}.
	\end{equation}
	In fact, the estimate for $r_1$ is a consequence of Lemma \ref{lem-paralin:paralin-of-ellip}, and the definition \eqref{eq-paralin:good-unknown-ext} of $\Phi$, together with Lemma \ref{lem-paralin:reg-of-pot} and Proposition \ref{prop-para:paraprod-bound}, implies that
	\begin{equation}\label{eq-paralin:reg-good-unknown-ext}
		\|\Phi\|_{C^0_\rho H^{s_0}} + \|\partial_\rho\Phi\|_{C^0_\rho H^{s_0-1}} + \|\partial^2_\rho\Phi\|_{C^0_\rho H^{s_0-2}} \le C\left(\|\eta\|_{H_R^{s+\frac{1}{2}-}}\right) \|\psi\|_{H^{s_0}},
	\end{equation}
	which implies \eqref{eq-paralin:decomp-esti-remainder}.
	
	By definition \eqref{eq-paralin:def-alpha} of $\alpha$, under the hypothesis of Proposition \ref{prop-paralin:normal-to-tangential}, $\alpha\in H_{R^{-2}}^{s-\frac{1}{2}-}$, which is a symbol in $\Gamma^0_{s-3/2-}\subset \Gamma^0_{3/2+}$. Since $\alpha$ is strictly positive due to hypothesis \eqref{hyp-intro:bounds}, $1/\alpha$ belongs to the same class of symbol. Consequently, Proposition \ref{prop-para:paraprod-bound} guarantees that
	\begin{equation*}
		T_{\frac{1}{\alpha}}T_\alpha = id + R_2,
	\end{equation*}
	where $R_2$ is of order $\frac{3}{2}+$. Now we may apply $T_{1/\alpha}$ on both hands of \eqref{eq-paralin:decomp-ellip} and obtain that
	\begin{equation*}
		(\partial_\rho + T_a)(\partial_\rho - T_A)\Phi =  T_{\frac{1}{\alpha}}\left(-R_0\Phi - R_1\partial_\rho\Phi + r_1\right) - R_2(\partial_\rho + T_a)(\partial_\rho - T_A)\Phi.
	\end{equation*}
	We claim that the $C^0_\rho H^{s_0-\frac{1}{2}+\epsilon}$ norm of the right hand side is bounded by $C(\|\eta\|_{H_R^{s+\frac{1}{2}-}}) \|\psi\|_{H^{s_0}}$. In fact, the estimate of the first term follows from \eqref{eq-paralin:decomp-esti-remainder} and the fact that $T_{1/\alpha}$ is of order zero. As for the second term, since $T_a$ and $T_A$ are of order $1$, we have
	\begin{equation*}
		\|R_2(\partial_\rho + T_a)(\partial_\rho - T_A)\Phi\|_{C^0_\rho H^{s_0-\frac{1}{2}+\epsilon}} \lesssim \|(\partial_\rho + T_a)(\partial_\rho - T_A)\Phi\|_{C^0_\rho H^{s_0-2}} \lesssim C\left(\|\eta\|_{H_R^{s+\frac{1}{2}-}}\right) \|\psi\|_{H^{s_0}},
	\end{equation*}
	where the last inequality is a consequence of \eqref{eq-paralin:reg-good-unknown-ext}.
	
	To sum up, under the hypothesis of Proposition \ref{prop-paralin:normal-to-tangential}, we have seen that $\Phi$ defined by \eqref{eq-paralin:good-unknown-ext} satisfies 
	\begin{equation*}
		(\partial_\rho + T_{a^{(1)}})(\partial_\rho - T_A)\Phi = T_{-a^{(0)}}(\partial_\rho - T_A)\Phi + r_2,
	\end{equation*}
	with
	\begin{equation*}
		\| r_2 \|_{C^0_\rho H^{s_0-\frac{1}{2}+\epsilon}} \le C\left(\|\eta\|_{H_R^{s+\frac{1}{2}}}\right) \|\psi\|_{H^{s_0}}.
	\end{equation*}
	
	\begin{lemma}[Proposition 3.19 of \cite{alazard2011water}]
		Let $b_1\in \Gamma^1_1$ and $b_0\in\Gamma^0_0$ with 
		$$ \Real b_1(w,\xi) \ge c|\xi| $$
		for some $c>0$. If $v\in C^1_\rho H^{M}$ for some $M\in\R$ that solves the equation
		$$ \partial_{\rho}v + T_{b_1}v = T_{b_0}v + f, $$
		with $f\in C^0_\rho H^r$ for some $r\in\R$, then, for all $\epsilon'>0$, we have
		$$ v|_{\rho=1} \in H^{r+1-\epsilon'}. $$
	\end{lemma}
	One may apply this lemma for $b_1=a^{(1)}\in \Gamma^1_{3/2+}\subset\Gamma^1_1$, $b_0=-a^{(0)}\in \Gamma^0_{1/2+}\subset\Gamma^0_0$, and $v=(\partial_\rho-T_A)\Phi$ with $M=s_0-1$, $r=s_0-\frac{1}{2}+\epsilon$, and $\epsilon'=\frac{\epsilon}{2}$ to conclude that
	\begin{equation*}
		\left\|\left.\left((\partial_\rho - T_A)\Phi\right)\right|_{\rho=1} \right\|_{H^{s_0+\frac{1}{2}}} \le C\left( \|\eta\|_{H^{s+\frac{1}{2}-}_R} \right) \|\psi\|_{H^{s_0}}.
	\end{equation*}
	In order to recover $\partial_\rho\varphi$ (appearing on the left hand side of \eqref{eq-paralin:normal-to-tangential}) from $\partial_\rho\Phi$, we use the definition \eqref{eq-paralin:good-unknown-ext} of $\Phi$,
	\begin{equation*}
		\left.\left((\partial_\rho - T_A)\Phi\right)\right|_{\rho=1} = \partial_\rho\varphi|_{\rho=1} - T_\tau U - T_{\eta^{-1}\partial_\rho^2\varphi}\eta|_{\rho=1} - T_{\eta^{-1}\partial_\rho\varphi}\eta|_{\rho=1},
	\end{equation*}
	where $U$ is defined by \eqref{eq-paralin:good-unknown} with $B$ given in \eqref{eq-pre:B-polar}, $\tau = A|_{\rho=1}$, and the remaining terms verify
	\begin{align*}
		\|T_{\eta^{-1}\partial_\rho^2\varphi}\eta\|_{C^0_\rho H^{s_0+\frac{1}{2}}} \lesssim& \|\eta^{-1}\partial_\rho^2\varphi\|_{C^0_\rho H^{s_0-2-}}  \|\eta\|_{H^{s+\frac{1}{2}}_R} \le C\left( \|\eta\|_{H^{s+\frac{1}{2}-}_R} \right) \|\psi\|_{H^{s_0-}}\|\eta\|_{H^{s+\frac{1}{2}}_R}, \\
		\|T_{\eta^{-1}\partial_\rho\varphi}\eta\|_{C^0_\rho H^{s_0+\frac{1}{2}}} \lesssim& \|\eta^{-1}\partial_\rho\varphi\|_{C^0_\rho H^{s_0-1-}}  \|\eta\|_{H^{s+\frac{1}{2}}_R}\le C\left( \|\eta\|_{H^{s+\frac{1}{2}-}_R} \right) \|\psi\|_{H^{s_0-}}\|\eta\|_{H^{s+\frac{1}{2}}_R},
	\end{align*}
	thanks to Proposition \ref{prop-para:paraprod-bound} and Corollary \ref{cor-para:product-law} together with Lemma \ref{lem-paralin:reg-of-pot}. Note that, when $s_0-2>1$, the estimate follows from $s_0\le s$, while, when $s_0-2\le\frac{d}{2}$, we use the fact that $(s_0-)-2+s+\frac{1}{2}-1>s_0+\frac{1}{2}$ since $s>3$.

	\subsection{Paralinearization of Dirichlet-to-Neumann operator}\label{subsect:paralin-DtN}
	
	Thanks to Proposition \ref{prop-paralin:normal-to-tangential}, we are now able to write the Dirichlet-to-Neumann operator \eqref{eq-paralin:DtN-main} in terms of tangential derivatives. We introduce the modified gradient (in $w$)
	\begin{equation}\label{eq-paralin:def-deri}
		\bar{\nabla} := \frac{e_\theta\partial_\theta}{\eta} + e_z\partial_z,
	\end{equation}\index{n@$\bar{\nabla}$ Modified gradient}
	which allows us to do the following calculation from \eqref{eq-pre:B-polar}-\eqref{eq-pre:deri-of-psi} (attention that we abuse the notation $(\rho,\theta,z)$ which has different meaning in Section \ref{Sect:pre} and here, whose relation is given by \eqref{eq-paralin:alt-var}),
	\begin{align}
		B =& \frac{\partial_\rho\varphi|_{\rho=1}}{\eta}, \label{eq-paralin:def-B} \\
		V =& \bar{\nabla}\varphi|_{\rho=1} - \frac{\bar{\nabla}\eta}{\eta}\partial_\rho\varphi|_{\rho=1}, \label{eq-paralin:def-V}\\
		N =& BV\cdot\bar{\nabla}\eta + \frac{|V|^2 - B^2}{2}, \label{eq-paralin:def-N}\\
		G(\eta)\psi =& B- V\cdot\bar{\nabla}\eta, \label{eq-paralin:DtN-formula}\\
		\bar{\nabla}\psi =& V + B\bar{\nabla}\eta. \label{eq-paralin:deriv-of-psi}
	\end{align}
	Note that these formulas also imply
	\begin{align}
		B &= \frac{G(\eta)\psi + \bar{\nabla}\psi\cdot\bar{\nabla}\eta}{1+ \left|\bar{\nabla}\eta\right|^2}, \label{eq-paralin:B-formula}\\
		V &= \bar{\nabla}\psi - B \bar{\nabla}\eta, \label{eq-paralin:V-formula}\\
		N &= \frac{1}{2}\frac{\left( G(\eta)\psi + \bar{\nabla}\psi\cdot\bar{\nabla}\eta \right)^2}{1+ \left|\bar{\nabla}\eta\right|^2}. \label{eq-paralin:N-formula}
	\end{align}
	The following estimate for $(B,V)$ is no more than a consequence of Lemma \ref{lem-paralin:reg-of-pot} and formula \eqref{eq-paralin:def-B}, together with Corollary \ref{cor-para:product-law} and Proposition \ref{prop-para:paralin}.
	\begin{lemma}\label{lem-paralin:esti-B-V-psi}
		Let $(\eta,\psi)\in H_R^{s+\frac{1}{2}}\times H^{s_0}$ with $s>3$ and $\frac{3}{2}<s_0\le s$. Then, by expressing $B,V$ as linear operators acting on $\psi$,
		\begin{equation}\label{eq-paralin:def-B-V-op}
			B=\mathcal{B}(\eta)\psi,\ \ V=\mathcal{V}(\eta)\psi,
		\end{equation}
		we have
		\begin{equation}\label{eq-paralin:esti-B-V-op}
			\|\mathcal{B}(\eta)\|_{\mathcal{L}(H^{s_0};H^{s_0-1})} + \|\mathcal{V}(\eta)\|_{\mathcal{L}(H^{s_0};H^{s_0-1})} \le C\left(\|\eta\|_{H_R^{s+\frac{1}{2}-}}\right),
		\end{equation}\index{B@$\mathcal{B}(\eta)$ Trace of radial derivative of harmonic extension}\index{V@$\mathcal{V}(\eta)$ Trace of anular derivative of harmonic extension}
		and, in particular,
		\begin{equation}\label{eq-paralin:esti-B-V-psi}
			\|B\|_{H^{s_0-1}} + \|V\|_{H^{s_0-1}} \le C\left(\|\eta\|_{H_R^{s+\frac{1}{2}-}}\right) \|\psi\|_{H^{s_0}}.
		\end{equation}
	\end{lemma}

	We can now state the main result of this section as
	\begin{proposition}\label{prop-paralin:paralin-DtN}
		Let $(\eta,\psi)\in H_R^{s+\frac{1}{2}}\times H^{s_0}$ with $s>3$ and $\frac{3}{2}<s_0\le s$. Let $U=\psi-T_B\eta$ be the good unknown as defined in \eqref{eq-paralin:good-unknown}. Then there exists elliptic symbol $\lambda\in\Sigma^1$, such that
		\begin{equation}\label{eq-paralin:paralin-DtN}
			\begin{aligned}
				G(\eta)\psi =& T_\lambda U - T_V\cdot\bar{\nabla}\eta + f_1 \\
				=& T_\lambda \left( \psi - T_{\mathcal{B}(\eta)\psi}\eta \right) - T_{\mathcal{V}(\eta)\psi}\cdot\bar{\nabla}\eta + f_1,
			\end{aligned}
		\end{equation}\index{f@$f_1$ Error in paralinearization of Dirichlet-to-Neumann operator}
		where $\mathcal{B}(\eta),\mathcal{V}(\eta)$ are defined by \eqref{eq-paralin:def-B-V-op} above and $f_1 = f_1(\eta,\psi)$ is linear in $\psi$ with
		\begin{equation}\label{eq-paralin:esti-f-1}
			\|f_1\|_{H^{s_0+\frac{1}{2}}} \le C\left(\|\eta\|_{H^{s+\frac{1}{2}-}_R}\right) \left( \|\psi\|_{H^{s_0}} + \|\eta\|_{H_R^{s+\frac{1}{2}}}\|\psi\|_{H^{s_0-}} \right).
		\end{equation}
		
		Moreover, $\lambda = \lambda^{(1)} + \lambda^{(0)}\in \Sigma^1$ can be calculated explicitly,
		\begin{equation}\label{eq-paralin:def-lambda}
			\lambda^{(1)} = \sqrt{\left(\frac{\xi_\theta^2}{\eta^2}+\xi_z^2\right) + \left(\frac{\xi_\theta}{\eta}\eta_z - \xi_z\frac{\eta_\theta}{\eta}\right)^2 },\ \ \lambda^{(0)} = \frac{l^2}{\eta} A^{(0)}|_{\rho=1}, 
		\end{equation}\index{l@$\lambda$ Symbol of Dirichlet-to-Nemann operator}
		where $A^{(0)}$ is defined in \eqref{eq-paralin:def-A-0}, and $l$ is as defined in \eqref{eq-intro:def-l},
		$$ l = \sqrt{1 + \left(\frac{\eta_\theta}{\eta}\right)^2 + \eta_z^2}. $$
	\end{proposition}
	
	\begin{proof}
		During the proof of Proposition \ref{prop-paralin:paralin-DtN}, we shall use the equivalence: for $u,v$ defined on $\T\times\R$,
		\begin{equation*}
			u\sim v \Leftrightarrow \|u-v\|_{H^{s_0+\frac{1}{2}}} \le C\left(\|\eta\|_{H_R^{s+\frac{1}{2}-}}\right) \|\psi\|_{H^{s_0}}.
		\end{equation*}
		
		Recall that, from \eqref{eq-paralin:DtN-main}, $G(\eta)\psi$ can be written as
		\begin{align*}
			G(\eta)\psi =& \frac{l^2}{\eta}\partial_\rho\varphi|_{\rho=1} - \frac{\eta_\theta}{\eta}\frac{\psi_\theta}{\eta} - \eta_z \psi_z  \\
			=& T_{\eta^{-1}l^2} \partial_\rho\varphi|_{\rho=1} + T_{\partial_\rho\varphi|_{\rho=1}}\eta^{-1}l^2 + R(\eta^{-1}l^2, \partial_\rho\varphi|_{\rho=1}) - T_{\eta^{-1}\eta_\theta}\eta^{-1}\psi_\theta - T_{\eta^{-1}\psi_\theta}\eta^{-1}\eta_\theta \\
			& - R(\eta^{-1}\eta_\theta,\eta^{-1}\psi_\theta) - T_{\eta_z}\psi_z - T_{\psi_z}\eta_z - R(\eta_z,\psi_z) \\
			\sim& T_{\eta^{-1}l^2} \partial_\rho\varphi|_{\rho=1} + T_{\partial_\rho\varphi|_{\rho=1}}\eta^{-1}l^2 - T_{\eta^{-1}\eta_\theta}\eta^{-1}\psi_\theta - T_{\eta^{-1}\psi_\theta}\eta^{-1}\eta_\theta - T_{\eta_z}\psi_z - T_{\psi_z}\eta_z \\
			=& T_{\eta^{-1}l^2} \partial_\rho\varphi|_{\rho=1} + T_{\partial_\rho\varphi|_{\rho=1}}\eta^{-1}l^2 - T_{\bar{\nabla}\eta}\cdot\bar{\nabla}\psi - T_{\bar{\nabla}\psi}\cdot\bar{\nabla}\eta \\
			=& T_{\eta^{-1}l^2} \partial_\rho\varphi|_{\rho=1} + T_{\partial_\rho\varphi|_{\rho=1}}\eta^{-1}l^2 - T_{\bar{\nabla}\eta}\cdot\bar{\nabla}\psi - T_{B\bar{\nabla}\eta}\cdot\bar{\nabla}\eta - T_{V}\cdot\bar{\nabla}\eta.
		\end{align*}
		where we use the formula \eqref{eq-paralin:deriv-of-psi}. To prove the equivalence, we observe that $\partial_\rho\varphi|_{\rho=1}$, $\eta^{-1}\psi_\theta$ and $\psi_z$ belong to $H^{s_0-1}$ with norm bounded by $C\left(\|\eta\|_{H_R^{s+\frac{1}{2}-}}\right) \|\psi\|_{H^{s_0}}$ using \eqref{eq-paralin:reg-of-pot-loc}. Thus, an application of Proposition \ref{prop-para:paralin} and Corollary \ref{cor-para:product-law} gives
		\begin{align*}
			&\| R(\eta^{-1}l^2, \partial_\rho\varphi|_{\rho=1}) + R(\eta^{-1}\eta_\theta,\eta^{-1}\psi_\theta) + R(\eta_z,\psi_z) \|_{H^{s_0+\frac{1}{2}}} \\
			\le& \left( \|\eta^{-1}l^2\|_{H^{\frac{5}{2}+}_{R^{-1}}} + \|\eta^{-1}\eta_\theta\|_{H^{\frac{5}{2}+}} + \|\eta_z\|_{H^{\frac{5}{2}+}} \right) C\left(\|\eta\|_{H_R^{s+\frac{1}{2}-}}\right) \|\psi\|_{H^{s_0}} \\
			\le& C\left(\|\eta\|_{H_R^{s+\frac{1}{2}-}}\right) \|\psi\|_{H^{s_0}}.
		\end{align*}
		Thanks to definition \eqref{eq-paralin:good-unknown} of good unknown $U$, formula \eqref{eq-paralin:def-B}, and Proposition \ref{prop-paralin:normal-to-tangential}, we have
		\begin{align*}
			G(\eta)\psi \sim& T_{\eta^{-1}l^2} \partial_\rho\varphi|_{\rho=1} + T_{\partial_\rho\varphi|_{\rho=1}}\eta^{-1}l^2 - T_{\bar{\nabla}\eta}\cdot\bar{\nabla}\psi - T_{B\bar{\nabla}\eta}\cdot\bar{\nabla}\eta - T_{V}\cdot\bar{\nabla}\eta \\
			\sim& T_{\eta^{-1}l^2}T_{\tau} U + T_{\partial_\rho\varphi|_{\rho=1}}\eta^{-1}l^2 - T_{\bar{\nabla}\eta}\cdot\bar{\nabla}\left(U+T_B \eta\right) - T_{B\bar{\nabla}\eta}\cdot\bar{\nabla}\eta - T_{V}\cdot\bar{\nabla}\eta \\
			=& T_{\eta^{-1}l^2}T_{\tau} U - T_{\bar{\nabla}\eta}\cdot\bar{\nabla}U - T_{V}\cdot\bar{\nabla}\eta + T_{\eta B}\eta^{-1}l^2 - T_{\bar{\nabla}\eta}\cdot\bar{\nabla}T_B \eta - T_{B\bar{\nabla}\eta}\cdot\bar{\nabla}\eta,
		\end{align*}
		where the equivalence is due to estimate \eqref{eq-paralin:normal-to-tangential} and $\eta^{-1}l^2 \in H^{s-\frac{1}{2}-}_{R^{-1}}$, ensuring that $T_{\eta^{-1}l^2}$ is of order $0$.
		
		We claim that 
		\begin{align}
			&T_{\eta^{-1}l^2}T_{\tau} U \sim T_{\eta^{-1}l^2\tau} U, \label{eq-paralin:paralin-DtN-1} \\
			&T_{\bar{\nabla}\eta}\cdot\bar{\nabla}U \sim -\frac{1}{2}T_{i\eta\xi\cdot\beta|_{\rho=1}} U, \label{eq-paralin:paralin-DtN-2} \\
			&T_{\eta B}\eta^{-1}l^2 - T_{\bar{\nabla}\eta}\cdot\bar{\nabla}T_B \eta - T_{B\bar{\nabla}\eta}\cdot\bar{\nabla}\eta \sim 0. \label{eq-paralin:paralin-DtN-3}
		\end{align}
		With these equivalences (whose proof will be given later), we may conclude that
		\begin{equation*}
			G(\eta)\psi \sim T_{\lambda} U - T_{V}\cdot\bar{\nabla}\eta,\ \ \lambda = \eta^{-1}l^2\tau + \frac{i\eta\xi\cdot\beta|_{\rho=1}}{2},
		\end{equation*}
		which coincides with \eqref{eq-paralin:def-lambda}.
	\end{proof}
	
	\begin{proof}[Proof of \eqref{eq-paralin:paralin-DtN-1}]
		In Proposition \ref{prop-paralin:normal-to-tangential} and Lemma \ref{lem-paralin:decomp-of-P}, we have seen that $\tau = A|_{\rho=1}\in \Sigma^1$. By applying Proposition \ref{prop-paralin:homo-sym-cal} and using the fact that $\eta^{-1}l^2 \in H^{s-\frac{1}{2}-}_{R^{-1}} \subset \Sigma^0$ (since $s-\frac{1}{2}->\frac{3}{2}+1$ due to $s>3$), we have
		\begin{equation*}
			\|T_{\eta^{-1}l^2}T_{\tau} U - T_{\left(\eta^{-1}l^2\right) \sharp \tau} U\|_{H^{s_0+\frac{1}{2}}} \lesssim C\left(\|\eta\|_{H_R^{s+\frac{1}{2}-}}\right) \|U\|_{H^{s_0}} \le  C\left(\|\eta\|_{H_R^{s+\frac{1}{2}-}}\right) \|\psi\|_{H^{s_0}}.
		\end{equation*}
		Note that, since $\eta^{-1}l^2 \in \Sigma^0$ is independent of $\xi$ and has no $\Gamma^{-1}_{1/2+}$ components, the formula \eqref{eq-paralin:homo-sym-comp-formula} yields
		\begin{equation*}
			\left(\eta^{-1}l^2\right) \sharp \tau = \eta^{-1}l^2\tau.
		\end{equation*}
	\end{proof}
	
	\begin{proof}[Proof of \eqref{eq-paralin:paralin-DtN-2}]
		By definition \eqref{eq-paralin:def-beta}, it is easy to see that
		\begin{equation*}
			-\frac{i\eta\xi\cdot\beta|_{\rho=1}}{2} = i\xi_\theta \frac{\eta_\theta}{\eta^2} + i\xi_z \eta_z.
		\end{equation*}
		We may decompose the left hand side of \eqref{eq-paralin:paralin-DtN-2} as
		\begin{align*}
			T_{\bar{\nabla}\eta}\cdot\bar{\nabla}U =& T_{\eta^{-1}\eta_\theta}\left(\eta^{-1}\partial_\theta U\right) + T_{\eta_z} \partial_z U \\
			=&T_{\eta^{-1}\eta_\theta}T_{\eta^{-1}}\partial_\theta U + T_{\eta^{-1}\eta_\theta}T_{\partial_\theta U}\eta^{-1} + T_{\eta^{-1}\eta_\theta}R(\eta^{-1},\partial_\theta U) + T_{i\xi_z \eta_z} U.
		\end{align*}
		Since $\eta^{-1}\eta_\theta, \eta^{-1} \in \Gamma^{0}_{3/2+}$, an application of Proposition \ref{prop-para:paradiff-cal-sym} and \eqref{eq-paralin:esti-good-unknown} gives that
		\begin{align*}
			\|\left( T_{\left(\eta^{-1}\eta_\theta\right)\sharp\left(\eta^{-1}\right)} - T_{\eta^{-1}\eta_\theta}T_{\eta^{-1}}\right) \partial_\theta U\|_{H^{s_0+\frac{1}{2}}} \le& C\left(\|\eta\|_{H_R^{s+\frac{1}{2}-}}\right) \|\partial_\theta U\|_{H^{s_0-1}} \\
			\le& C\left(\|\eta\|_{H_R^{s+\frac{1}{2}-}}\right) \|\psi\|_{H^{s_0}},
		\end{align*}
		with $\left(\eta^{-1}\eta_\theta\right)\sharp\left(\eta^{-1}\right) = \eta^{-2}\eta_\theta$, from which one may deduce that
		\begin{equation*}
			T_{\eta^{-1}\eta_\theta}T_{\eta^{-1}}\partial_\theta U \sim T_{\eta^{-2}\eta_\theta} \partial_\theta U = T_{i\xi_\theta\eta^{-2}\eta_\theta} U,
		\end{equation*}
		and the desired estimate \eqref{eq-paralin:paralin-DtN-2} can be reduced to
		\begin{equation*}
			T_{\eta^{-1}\eta_\theta}T_{\partial_\theta U}\eta^{-1} + T_{\eta^{-1}\eta_\theta}R(\eta^{-1},\partial_\theta U) \sim 0.
		\end{equation*}
		In fact, thanks to Proposition \ref{prop-para:paraprod-bound} and Corollary \ref{cor-para:product-law}, we have
		\begin{align*}
			\|T_{\eta^{-1}\eta_\theta}T_{\partial_\theta U}\eta^{-1}\|_{H^{s_0+\frac{1}{2}}} \lesssim& \|\eta^{-1}\eta_\theta\|_{H^{1+}} \|T_{\partial_\theta U}\eta^{-1}\|_{H^{s_0+\frac{1}{2}}} \\
			\le& C\left(\|\eta\|_{H_R^{s+\frac{1}{2}-}}\right) \|\partial_\theta U\|_{H^{s_0-1}} \|\eta^{-1}\|_{H^{\max(\frac{5}{2}+,s_0+\frac{1}{2})}_{R^{-1}}} \\
			\le& C\left(\|\eta\|_{H_R^{s+\frac{1}{2}-}}\right) \|\psi\|_{H^{s_0}}, \\
			\|T_{\eta^{-1}\eta_\theta}R(\eta^{-1},\partial_\theta U)\|_{H^{s_0+\frac{1}{2}}} \lesssim& \|\eta^{-1}\eta_\theta\|_{H^{1+}} \|R(\eta^{-1},\partial_\theta U)\|_{H^{s_0+\frac{1}{2}}} \\
			\le& C\left(\|\eta\|_{H_R^{s+\frac{1}{2}-}}\right) \|\eta^{-1}\|_{H^{\frac{5}{2}+}_{R^{-1}}} \|\partial_\theta U\|_{H^{s_0-1}} \\
			\le& C\left(\|\eta\|_{H_R^{s+\frac{1}{2}-}}\right) \|\psi\|_{H^{s_0}}.
		\end{align*}
		which completes the proof.
	\end{proof}
	
	\begin{proof}[Proof of \eqref{eq-paralin:paralin-DtN-3}]
		To prove \eqref{eq-paralin:paralin-DtN-3}, as in the proof of Lemma \ref{lem-paralin:paralin-of-ellip}, we only need to consider the case $s_0=\frac{3}{2}+\delta$ with $0<\delta\ll 1$ and the case $s_0=s$. In the first case, each term on the left hand side of \eqref{eq-paralin:paralin-DtN-3} is equivalent to zero. In fact, Proposition \ref{prop-para:paraprod-bound} ensures that
		\begin{align*}
			\|T_{\eta B}\eta^{-1}l^2\|_{H^{2+\delta}} \le& \|\eta B\|_{H^{\frac{1}{2}+\delta}} \|\eta^{-1}l^2\|_{H^{\frac{5}{2}}_{R^{-1}}} \le  C\left(\|\eta\|_{H_R^{s+\frac{1}{2}-}}\right) \|\psi\|_{H^{s_0}}, \\
			\|T_{\bar{\nabla}\eta}\cdot\bar{\nabla}T_B \eta\|_{H^{2+\delta}} \le& \|\bar{\nabla}\eta\|_{H^{1+}} \|\bar{\nabla}T_B \eta\|_{H^{2+\delta}} \le  C\left(\|\eta\|_{H_R^{s+\frac{1}{2}-}}\right) \|T_B\eta\|_{H^{3+\delta}} \\
			\le& C\left(\|\eta\|_{H_R^{s+\frac{1}{2}-}}\right) \|B\|_{H^{\frac{1}{2}+\delta}} \|\eta\|_{H^{\frac{7}{2}}_R} \le C\left(\|\eta\|_{H_R^{s+\frac{1}{2}-}}\right) \|\psi\|_{H^{s_0}}, \\
			\|T_{B\bar{\nabla}\eta}\cdot\bar{\nabla}\eta\|_{H^{2+\delta}} \lesssim& \|B\bar{\nabla}\eta\|_{H^{\frac{1}{2}+\delta}} \|\bar{\nabla}\eta\|_{H^{\frac{5}{2}}} \le C\left(\|\eta\|_{H_R^{s+\frac{1}{2}-}}\right) \|B\|_{H^{\frac{1}{2}+\delta}} \\
			\le& C\left(\|\eta\|_{H_R^{s+\frac{1}{2}-}}\right) \|\psi\|_{H^{s_0}}.
		\end{align*}
		Note that the estimate for $B$ has been given by \eqref{eq-paralin:esti-B-V-psi}.
		
		If $s_0=s>3$, we are able to apply symbolic calculus (Proposition \ref{prop-para:paradiff-cal-sym}), since all the involved symbols has at least H{\"o}lder regularity in $w$. Recall that $\eta\in H^{s+\frac{1}{2}-}_{R}$ and $B\in H^{s-1-}$. Then we have		
		\begin{align*}
			&T_{\eta B}\eta^{-1}l^2 - T_{\bar{\nabla}\eta}\cdot\bar{\nabla}T_B \eta - T_{B\bar{\nabla}\eta}\cdot\bar{\nabla}\eta \\
			=& T_{\eta B}\eta^{-1} + T_{\eta B}\left(\eta^{-1}|\bar{\nabla}\eta|^2\right) - T_{\eta^{-1}\eta_\theta}\left(\eta^{-1}\partial_\theta \left(T_B \eta\right)\right) - T_{\eta_z}\partial_z\left( T_B\eta \right) \\
			&- T_{\eta^{-1}\eta_\theta B}\left( \eta^{-1}\eta_\theta \right) - T_{\eta_z B}\eta_z \\
			\sim& T_{\eta B}\left(\eta^{-1}|\bar{\nabla}\eta|^2\right) - T_{\eta^{-1}\eta_\theta}\left(\eta^{-1}T_B \eta_\theta\right) - T_{\eta_z}T_B\eta_z - T_{\eta^{-1}\eta_\theta B}\left( \eta^{-1}\eta_\theta \right) - T_{\eta_z B}\eta_z \\
			\sim& T_{B}T_{\eta}\left(\eta^{-1}|\bar{\nabla}\eta|^2\right) - T_{\eta^{-1}\eta_\theta}T_{\eta^{-1}}T_B \eta_\theta - 2T_BT_{\eta_z}\eta_z - T_BT_{\eta^{-1}\eta_\theta}\left( \eta^{-1}\eta_\theta \right) \\
			\sim& T_{B}|\bar{\nabla}\eta|^2 - T_BT_{\eta^{-1}\eta_\theta}T_{\eta^{-1}} \eta_\theta - T_BT_{\eta^{-1}\eta_\theta}\left( \eta^{-1}\eta_\theta \right) - T_B\eta_z^2 \\
			\sim& T_{B}|\bar{\nabla}\eta|^2 - 2T_BT_{\eta^{-1}\eta_\theta}\left( \eta^{-1}\eta_\theta \right) - T_B\eta_z^2 \\
			\sim& T_{B}|\bar{\nabla}\eta|^2 - T_B\left( \eta^{-1}\eta_\theta \right)^2 - T_B\eta_z^2  =0.
		\end{align*}
	\end{proof}

	As a result of Proposition \ref{prop-paralin:paralin-DtN}, we can replace the first equation in \eqref{eq-intro:WW} by
	\begin{equation}\label{eq-paralin:eq-paralin-eta}
		\eta_t + T_V \cdot \bar\nabla\eta - T_\lambda U = f_1.
	\end{equation}

	Before entering to the paralinearization of nonlinear terms, we provide a calculation concerning the subprincipal part $\lambda^{(0)}$ of symbol $\lambda$ appearing in \eqref{eq-paralin:paralin-DtN}. We claim that 
	\begin{lemma}\label{lem-paralin:lambda-0-cal}
		The symbol $\lambda = \lambda^{(1)} + \lambda^{(0)}$ constructed in Proposition \ref{prop-paralin:paralin-DtN} satisfies
		\begin{equation}\label{eq-paralin:lambda-0-cal}
			\Imag\lambda^{(0)} = -\frac{1}{2}\partial_w\cdot\partial_\xi\lambda^{(1)} - \frac{1}{2}\frac{\partial_w\eta}{\eta}\cdot\partial_\xi\lambda^{(1)}.
		\end{equation}
	\end{lemma}
	\begin{proof}
		Set $a=\eta\lambda$. Since by \eqref{eq-paralin:def-lambda}, $\Imag \lambda^{(1)} = 0$, Proposition \ref{prop-paralin:homo-sym-cond-aa} shows that \eqref{eq-paralin:lambda-0-cal} will hold if we prove that the symbol $\eta\lambda$ verifies the second condition of \eqref{eq-paralin:homo-sym-cond-aa}, i.e. that $T_{\eta\lambda}-T_{\eta\lambda}^*$ is of order $0-$.
		
		To prove this, we assume that $\eta\in H^{+\infty}_R := \cap_{s\in\R}H^{s}_R$ and $\psi\in H^s$ for some fixed $s \gg 1$. Under this regularity assumption, from \eqref{eq-paralin:paralin-DtN} and \eqref{eq-paralin:good-unknown}, we have
		\begin{align*}
			\eta G(\eta)\psi =& \eta T_\lambda \psi - \eta T_\lambda T_B\eta - \eta T_V\cdot\bar{\nabla}\eta + f_1 \\
			=& T_\eta T_\lambda \psi + T_{T_\lambda \psi} \eta + R(\eta,T_\lambda \psi) - \eta T_\lambda T_B\eta - \eta T_V\cdot\bar{\nabla}\eta + f_1 \\
			=& T_{\eta\lambda} \psi + \left( T_\eta T_\lambda - T_{\eta\lambda} \right)\psi + T_{T_\lambda \psi} \eta + R(\eta,T_\lambda \psi) - \eta T_\lambda T_B\eta - \eta T_V\cdot\bar{\nabla}\eta + f_1.
		\end{align*}
		Since $\eta\in H^{+\infty}_R$, thanks to Proposition \ref{prop-para:paradiff-cal-sym}, \ref{prop-para:paraprod-bound}, and Corollary \ref{cor-para:product-law}, it is clear that the following terms are smoothing operator w.r.t. $\psi$,
		$$\left( T_\eta T_\lambda - T_{\eta\lambda} \right)\psi,\  T_{T_\lambda \psi} \eta,\  R(\eta,T_\lambda \psi),\  \eta T_\lambda T_B\eta,\ \eta T_V\cdot\bar{\nabla}\eta,$$
		while $f_1$ is of order $-\frac{1}{2}$ due to \eqref{eq-paralin:esti-f-1}. Consequently, we have that
		\begin{equation*}
			T_{\eta\lambda} - \eta G(\eta) \text{ is of order }-\frac{1}{2},
		\end{equation*}
		which, combined with the fact that $\eta G(\eta)$ is self-adjoint, yields that
		\begin{equation*}
			T_{\eta\lambda} - T_{\eta\lambda}^* = (T_{\eta\lambda} - \eta G(\eta)) - (T_{\eta\lambda} - \eta G(\eta))^* \text{ is of order }-\frac{1}{2},
		\end{equation*}
		which concludes the desired result \eqref{eq-paralin:lambda-0-cal} for regular enough $(\eta,\psi)$.
		
		We emphasize that the regularity assumption on $\eta$ and $\psi$ has no impact in the identity \eqref{eq-paralin:lambda-0-cal} since it can also be obtained by algebric calculus from \eqref{eq-paralin:def-lambda} (which is far more complicated).
	\end{proof}

	\subsection{Paralinearization of the nonlinear terms}\label{subsect:paralin-nonlin}
	In this section, we aim to rewrite the second equation of \eqref{eq-intro:WW} in the similar form as \eqref{eq-paralin:eq-paralin-eta}. The main difficulty is the paralinearization of the nonlinear terms, i.e.
	\begin{proposition}\label{prop-paralin:paralin-of-nonlin}
		Let $(\eta,\psi)\in H_R^{s+\frac{1}{2}}\times H^s$ with $s>3$. Then there exists $\mu\in\Sigma^2$ and $r_2,r_3\in H^{s+}$ such that
		\begin{align}
			&N = T_V\cdot\bar{\nabla}\psi - T_BG(\eta)\psi - T_B T_V\cdot\bar{\nabla}\eta + r_2, \label{eq-paralin:paralin-of-N} \\
			&H - \frac{1}{2R} = T_\mu\eta + r_3. \label{eq-paralin:paralin-of-H}
		\end{align}\index{m@$\mu$ Symbol of mean curvature}
		where $r_2 = r_2(\eta;\psi,\psi)$ is quadratic in $\psi$ and $r_3 = r_3(\eta)$ is independent of $\psi$ with, for all $\frac{3}{2}<s_0\le s$ and $\psi_1,\psi_2\in H^s$,
		\begin{align}
			&\|r_2(\eta;\psi_1,\psi_2)\|_{H^{s_0}} + \|r_2(\eta;\psi_2,\psi_1)\|_{H^{s_0}} \le C\left(\|\eta\|_{H_R^{s+\frac{1}{2}-}}\right) \|\psi_1\|_{H^{s-}} \|\psi_2\|_{H^{s_0}}, \label{eq-paralin:esti-r-2}\\
			&\|r_3(\eta)\|_{H^{s}} \le C\left(\|\eta\|_{H_R^{s+\frac{1}{2}-}}\right) \|\eta\|_{H_R^{s+\frac{1}{2}}}. \label{eq-paralin:esti-r-3}
		\end{align}\index{r@$r_2$ Error in paralinearization of nonlinear term}\index{r@$r_3$ Error in paralinearization of mean curvature}
		
		Moreover, $\mu = \mu^{(2)}+\mu^{(1)}\in \Sigma^2$ is elliptic with
		\begin{align}
			\mu^{(2)} =& \frac{1}{2l^3} \left[ \left(\frac{\xi_\theta}{\eta}\right)^2 + \xi_z^2 + \left(\frac{\xi_\theta}{\eta}\eta_z - \xi_z\frac{\eta_\theta}{\eta}\right)^2 \right]. \label{eq-paralin:def-mu-2}
		\end{align}
	\end{proposition}
	
	Once this proposition is proved, the second equation in \eqref{eq-intro:WW} can be replaced by
	\begin{equation}\label{eq-paralin:eq-paralin-psi}
		\psi_t + T_V\cdot\bar{\nabla}\psi + \sigma T_{\mu}\eta - T_BG(\eta)\psi - T_B T_V\cdot \bar{\nabla}\eta = -r_2 -\sigma r_3,
	\end{equation}
	which can be further written as
	\begin{equation*}
		\left(\partial_t + T_V\cdot\bar{\nabla}\right)\psi - T_B \left(\partial_t + T_V\cdot\bar{\nabla}\right)\eta + \sigma T_{\mu}\eta = -r_2 -\sigma r_3.
	\end{equation*}
	By combining this equation with \eqref{eq-paralin:eq-paralin-eta}, we obtain a reformulation of \eqref{eq-intro:WW},
	\begin{equation}\label{eq-paralin:eq-eta-psi}
		\left(\begin{array}{cc}
			I & 0 \\
			-T_B & I
		\end{array}\right)\left(\partial_t + T_V\cdot\bar{\nabla}\right)
		\left(\begin{array}{c}
			\eta \\
			\psi
		\end{array}\right) + 
		\left(\begin{array}{cc}
			0 & -T_\lambda \\
			\sigma T_\mu & 0
		\end{array}\right)
		\left(\begin{array}{cc}
			I & 0 \\
			-T_B & I
		\end{array}\right)
		\left(\begin{array}{c}
			\eta \\
			\psi
		\end{array}\right) = 
		\left(\begin{array}{c}
			f_1 \\
			f_2
		\end{array}\right),
	\end{equation}
	where
	\begin{equation}\label{eq-paralin:def-f-2}
		f_2 := f_2(\eta;\psi,\psi) = -r_2(\eta;\psi,\psi) -\sigma r_3(\eta).
	\end{equation}
	Note that 
	\begin{equation*}
		\left(\begin{array}{cc}
			I & 0 \\
			-T_B & I
		\end{array}\right)^{-1} =
		\left(\begin{array}{cc}
		I & 0 \\
		T_B & I
		\end{array}\right).
	\end{equation*}
	This identity allows us to recover from \eqref{eq-paralin:eq-eta-psi} an evolution equation for $(\eta,\psi)$,
	\begin{equation}\label{eq-paralin:WW}
		\left(\partial_t + T_V\cdot\bar{\nabla}\right)
		\left(\begin{array}{c}
			\eta \\
			\psi
		\end{array}\right) + 
		\mathcal{L}\left(\begin{array}{c}
			\eta \\
			\psi
		\end{array}\right) = f,
	\end{equation}
	where 
	\begin{equation}\label{eq-paralin:def-L}
		\mathcal{L} := \left(\begin{array}{cc}
			I & 0 \\
			T_B & I
		\end{array}\right)
		\left(\begin{array}{cc}
			0 & -T_\lambda \\
			\sigma T_\mu & 0
		\end{array}\right)
		\left(\begin{array}{cc}
			I & 0 \\
			-T_B & I
		\end{array}\right)
	\end{equation}\index{L@$\mathcal{L}$ Principal operator in paralinearization}
	and
	\begin{equation}\label{eq-paralin:def-f}
		f:= \left(\begin{array}{cc}
			I & 0 \\
			T_B & I
		\end{array}\right)
		\left(\begin{array}{c}
			f_1 \\
			f_2
		\end{array}\right)
		= \left(\begin{array}{c}
			f_1 \\
			T_B f_1 + f_2
		\end{array}\right).
	\end{equation}
	Since $\lambda$ and $\mu$ are elliptic, the operator $\mathcal{L}$ can be symmetrized, which is the purpose of Section \ref{Sect:sym}.
	
	\subsubsection{Paralinearization of $N$}\label{subsubsect:paralin-N}
	In this part, we shall prove \eqref{eq-paralin:paralin-of-N} and \eqref{eq-paralin:esti-r-2}. As before, during this proof, we shall use the equivalence: for $u,v$ defined on $\T\times\R$ and bilinear in $(\psi_1,\psi_2)$,
	\begin{equation}\label{eq-paralin:equi-paralin-N}
		u\sim v \Leftrightarrow \|(u-v)(\eta;\psi_1,\psi_2)\|_{H^{s_0}} + \|(u-v)(\eta;\psi_2,\psi_1)\|_{H^{s_0}} \le C\left(\|\eta\|_{H_R^{s+\frac{1}{2}-}}\right) \|\psi_1\|_{H^{s-}}\|\psi_2\|_{H^{s_0}}.
	\end{equation}
	
	To begin with, from formula \eqref{eq-paralin:def-N} for $N$, it is clear that $N$ is quadratic in $\psi$ as well as the three first terms on the right hand side of \eqref{eq-paralin:paralin-of-N}. As a result, $r_2$ is quadratic in $\psi$ (not necessarily symmetric).
	
	Due to \eqref{eq-paralin:def-N} and \eqref{eq-paralin:DtN-formula}, one may rewrite $N$ as follow,
	\begin{equation}\label{eq-paralin:N-first-paralin}
		\begin{aligned}
			N =& \frac{|V|^2 + B^2}{2} - BG(\eta)\psi \\
			=& T_V \cdot V + T_B B - T_B G(\eta)\psi - T_{B-V\cdot\bar{\nabla}\eta}B + R(V,V) + R(B,B) + R(B,G(\eta)\psi) \\
			\sim& T_V \cdot V + T_B B - T_B G(\eta)\psi - T_{B-V\cdot\bar{\nabla}\eta}B = T_V \cdot V - T_B G(\eta)\psi + T_{V\cdot\bar{\nabla}\eta}B
		\end{aligned}
	\end{equation}
	To prove the equivalence, it suffices to apply Proposition \ref{prop-para:paraprod-bound} and use the fact that $B,V,G(\eta)\psi$ are linear in $\psi$ with
	\begin{equation}\label{eq-paralin:esti-B-V-G-paralin}
		\|B\|_{H^{s'-1}} + \|V\|_{H^{s'-1}} + \|G(\eta)\psi\|_{H^{s'-1}} \le C\left(\|\eta\|_{H_R^{s+\frac{1}{2}-}}\right) \|\psi\|_{H^{s'}},\ \ \forall s'\in]\frac{3}{2},s].
	\end{equation}
	More precisely, we have
	\begin{align*}
		\|R(B(\psi_1),B(\psi_2))\|_{H^{s_0}} \lesssim& \|B(\psi_1)\|_{H^{2+}}\|B(\psi_2)\|_{H^{s_0-1-}} \\
		\le& C\left(\|\eta\|_{H_R^{s+\frac{1}{2}-}}\right) \|\psi_1\|_{H^{s-}}\|\psi_2\|_{H^{s_0}},
	\end{align*}
	and the similar estimate holds for $R(V,V)$ (note that $R(\cdot,\cdot)$ is symmetric), while 
	\begin{align*}
		&\|R(B(\psi_1),G(\eta)\psi_2)\|_{H^{s_0}} + \|R(B(\psi_2),G(\eta)\psi_1)\|_{H^{s_0}} \\
		\lesssim& \|B(\psi_1)\|_{H^{2+}}\|G(\eta)\psi_2\|_{H^{s_0-1-}} + \|G(\eta)\psi_1\|_{H^{2+}}\|B(\psi_2)\|_{H^{s_0-1-}} \\
		\le& C\left(\|\eta\|_{H_R^{s+\frac{1}{2}-}}\right) \|\psi_1\|_{H^{s-}}\|\psi_2\|_{H^{s_0}}.
	\end{align*}
	
	Inserting \eqref{eq-paralin:deriv-of-psi} in the right hand side of \eqref{eq-paralin:paralin-of-N}, we can write it, up to $r_2$, as
	\begin{equation*}
		T_V\cdot\bar{\nabla}\psi - T_BG(\eta)\psi - T_B T_V\cdot\bar{\nabla}\eta = T_V V + T_V\cdot(B\bar{\nabla}\eta) - T_B G(\eta)\psi - T_B T_V\cdot\bar{\nabla}\eta.
	\end{equation*}
	According to \eqref{eq-paralin:N-first-paralin}, this will be equivalent to $N$ if and only if
	\begin{equation*}
		T_{V\cdot\bar{\nabla}\eta}B \sim T_V\cdot(B\bar{\nabla}\eta) - T_B T_V\cdot\bar{\nabla}\eta.
	\end{equation*}
	The difference of the two sides can be expressed as
	\begin{align*}
		& T_V\cdot(B\bar{\nabla}\eta) - T_B T_V\cdot\bar{\nabla}\eta - T_{V\cdot\bar{\nabla}\eta}B \\
		=& T_V T_B \bar{\nabla}\eta + T_V\cdot T_{\bar{\nabla}\eta} B + T_V R(B,\bar{\nabla}\eta) - T_B T_V\cdot\bar{\nabla}\eta - T_{V\cdot\bar{\nabla}\eta}B \\
		=& [T_V,T_B]\bar{\nabla}\eta + \left( T_V\cdot T_{\bar{\nabla}\eta} - T_{V\cdot\bar{\nabla}\eta} \right) B + T_V R(B,\bar{\nabla}\eta).
	\end{align*}
	Therefore, it suffices to check the following equivalences:
	\begin{align}
		[T_V,T_B]\bar{\nabla}\eta \sim& 0, \label{eq-paralin:paralin-of-N-1} \\
		\left( T_V\cdot T_{\bar{\nabla}\eta} - T_{V\cdot\bar{\nabla}\eta} \right) B \sim& 0, \label{eq-paralin:paralin-of-N-2} \\
		T_V R(B,\bar{\nabla}\eta) \sim& 0, \label{eq-paralin:paralin-of-N-3}
	\end{align}
	in the sense of \eqref{eq-paralin:equi-paralin-N}. As before, the interpolation argument allows us to focus on the case $s_0=\frac{3}{2}+\delta$ with $0<\delta \ll 1$ and the case $s_0=s$. In the latter one, by observing that $B,V\in H^{s-1-} \subset \Gamma^0_{1+}$, $\bar{\nabla}\eta \in H^{s-\frac{1}{2}-} \subset \Gamma^0_{1+}$, the equivalences \eqref{eq-paralin:paralin-of-N-1}, \eqref{eq-paralin:paralin-of-N-2}, and \eqref{eq-paralin:paralin-of-N-3} follows from Corollary \ref{cor-para:commu-esti}, Proposition \ref{prop-para:paradiff-cal-sym}, and Proposition \ref{prop-para:paraprod-bound}, respectively, together with \eqref{eq-paralin:esti-B-V-G-paralin}.
	
	In the case $s_0=\frac{3}{2}+\delta$, we shall check that all the terms concerned are equivalent to zero. In fact, we have
	\begin{align*}
		&\|T_{V(\psi_1)}T_{B(\psi_2)}\bar{\nabla}\eta\|_{H^{\frac{3}{2}+\delta}} + \|T_{V(\psi_2)}T_{B(\psi_1)}\bar{\nabla}\eta\|_{H^{\frac{3}{2}+\delta}} \\
		\lesssim& \|V(\psi_1)\|_{H^{1+}} \|B(\psi_2)\|_{H^{\frac{1}{2}+\delta}} \|\bar{\nabla}\eta\|_{H^{2}} + \|V(\psi_2)\|_{H^{\frac{1}{2}+\delta}} \|B(\psi_1)\|_{H^{1+}} \|\bar{\nabla}\eta\|_{H^{2}} \\
		\le& C\left(\|\eta\|_{H_R^{s+\frac{1}{2}-}}\right) \|\psi_1\|_{H^{s-}}\|\psi_2\|_{H^{s_0}},
	\end{align*}
	which implies $T_V T_B \bar{\nabla}\eta \sim 0$ and $T_B T_V \bar{\nabla}\eta \sim 0$ in the same way, yielding \eqref{eq-paralin:paralin-of-N-1}. For \eqref{eq-paralin:paralin-of-N-2}, we apply Proposition \ref{prop-para:paraprod-bound},
	\begin{align*}
		&\|\left(T_{V(\psi_1)}\cdot T_{\bar{\nabla}\eta} - T_{V(\psi_1)\cdot\bar{\nabla}\eta}\right)B(\psi_2)\|_{H^{\frac{3}{2}+\delta}} + \|T_{V(\psi_2)}\cdot T_{\bar{\nabla}\eta} B(\psi_1)\|_{H^{\frac{3}{2}+\delta}}  \\
		&+ \|T_{V(\psi_2)\cdot\bar{\nabla}\eta} B(\psi_1)\|_{H^{\frac{3}{2}+\delta}} \\
		\le& \|V(\psi_1)\|_{H^{2+}} \|\bar{\nabla}\eta\|_{H^{2+}} \|B(\psi_2)\|_{H^{\frac{1}{2}+\delta}} + \|V(\psi_2)\|_{H^{\frac{1}{2}+\delta}} \|\bar{\nabla}\eta\|_{H^{1+}} \|B(\psi_2)\|_{H^{2}} \\
		\le& C\left(\|\eta\|_{H_R^{s+\frac{1}{2}-}}\right) \|\psi_1\|_{H^{s-}}\|\psi_2\|_{H^{s_0}}.
	\end{align*}
	The estimate corresponding to \eqref{eq-paralin:paralin-of-N-3} is due to Proposition \ref{prop-para:paraprod-bound}.
	\begin{align*}
		&\|T_{V(\psi_1)} R(B(\psi_2),\bar{\nabla}\eta)\|_{H^{\frac{3}{2}+\delta}} + \|T_{V(\psi_2)} R(B(\psi_1),\bar{\nabla}\eta)\|_{H^{\frac{3}{2}+\delta}} \\
		\le& \|V(\psi_1)\|_{H^{1+}} \|B(\psi_2)\|_{H^{\frac{1}{2}+\delta}} \|\bar{\nabla}\eta\|_{H^{2}} + \|V(\psi_2)\|_{H^{\frac{1}{2}+\delta}} \|\bar{\nabla}\eta\|_{H^{1}} \|B(\psi_2)\|_{H^{2}} \\
		\le& C\left(\|\eta\|_{H_R^{s+\frac{1}{2}-}}\right) \|\psi_1\|_{H^{s-}}\|\psi_2\|_{H^{s_0}}.
	\end{align*}

	\subsubsection{Paralinearization of $H$}\label{subsubsect:paralin-H}
	In this part we prove \eqref{eq-paralin:paralin-of-H} and \eqref{eq-paralin:esti-r-3}. As before, we shall use the equivalence: for $u,v$ functions of $\eta$,
	\begin{equation}\label{eq-paralin:equi-paralin-H}
		u\sim v \Leftrightarrow \|u(\eta) - v(\eta)\|_{H^{s}} \le C\left(\|\eta\|_{H_R^{s+\frac{1}{2}-}}\right) \|\eta\|_{H_R^{s+\frac{1}{2}}}.
	\end{equation} 
	By definition \eqref{eq-intro:mean-curv} of $H$ and expression \eqref{eq-intro:def-l} of $l$, it is easy to calculate that the left hand side of \eqref{eq-paralin:paralin-of-H} reads
	\begin{align*}
		H-\frac{1}{2R} =& \frac{1}{2} \left[ \frac{1}{\eta l} - \frac{\partial_\theta}{\eta}\left(\frac{\eta_\theta}{\eta l}\right) - \partial_z\left(\frac{\eta_z}{l}\right) \right] -\frac{1}{2R} \\
		=& \frac{1}{2\eta l} -\frac{1}{2R} + \frac{1}{2} \left[ \frac{1}{\eta}\frac{\eta_\theta}{l}\frac{\eta_\theta}{\eta^2} + \frac{1}{\eta}\frac{\eta_\theta}{\eta}\frac{l_\theta}{l^2} - \frac{1}{\eta^2l}\eta_{\theta\theta} \right] + \frac{1}{2} \left[ \eta_z\frac{l_z}{l^2} - \frac{1}{l}\eta_{zz} \right] \\
		=& \frac{1}{2\eta l} -\frac{1}{2R} + \frac{1}{2} \left[ \frac{\eta_\theta^2}{\eta^3l} + \frac{\eta_\theta}{\eta^2l^2}\frac{\frac{\eta_\theta}{\eta}\partial_\theta\left(\frac{\eta_\theta}{\eta}\right) + \eta_z\eta_{\theta z}}{l} - \frac{1}{\eta^2l}\eta_{\theta\theta} \right] \\
		&+ \frac{1}{2} \left[ \frac{\eta_z}{l^2}\frac{\frac{\eta_\theta}{\eta}\partial_z\left(\frac{\eta_\theta}{\eta}\right) + \eta_z\eta_{zz}}{l} - \frac{1}{l}\eta_{zz} \right] \\
		=& \frac{1}{2\eta l} -\frac{1}{2R} + \frac{1}{2} \left[ \frac{\eta_\theta^2}{\eta^3l} -\frac{\eta_\theta^4}{\eta^5l^3} + \frac{\eta_\theta \eta_z}{\eta^2l^3}\eta_{\theta z} - \left(\frac{1}{\eta^2l}-\frac{\eta_\theta^2}{\eta^4l^3}\right)\eta_{\theta\theta} \right] \\
		&+ \frac{1}{2} \left[ - \frac{\eta_\theta^2 \eta_z^2}{\eta^3l^3} + \frac{\eta_\theta\eta_z}{\eta^2l^3}\eta_{\theta z} - \left(\frac{1}{l}-\frac{\eta_z^2}{l^3}\right)\eta_{zz} \right] \\
		=& \frac{1}{2\eta l} -\frac{1}{2R} + \frac{\eta_\theta^2}{2\eta^3l^3} - \frac{1}{2l^3} \left[ \left( 1+\eta_z^2 \right)\frac{\eta_{\theta\theta}}{\eta^2} - 2\frac{\eta_\theta\eta_z}{\eta}\frac{\eta_{\theta z}}{\eta} + \left( 1+\left(\frac{\eta_\theta}{\eta}\right)^2 \right) \eta_{zz} \right] \\
		=& F(\eta,\eta_\theta,\eta_z) + \sum_{j,k\in\{\theta,z\}} G^{jk}(\eta,\eta_\theta,\eta_z)\eta_{jk},
	\end{align*}
	where $F$ and $G^{jk}$ are smooth functions defined by
	\begin{equation}\label{eq-paralin:def-FG-paralin-H}
		\begin{aligned}
			F(x,u,v) :=& \frac{1}{2xh} - \frac{1}{2R} + \frac{u^2}{2x^2h^3}, \\
			G^{\theta\theta}(x,u,v) :=& -\frac{1}{2h^3}\frac{1+v^2}{x^2},\ \ G^{zz}(x,u,v) := -\frac{1}{2h^3}\left( 1+\left(\frac{u}{x}\right)^2 \right), \\
			G^{\theta z}(x,u,v) =& G^{z\theta}(x,u,v) := \frac{1}{2h^3}\frac{uv}{x^2},
		\end{aligned}
	\end{equation}
	with $h(x,u,v):= \sqrt{1+(u/x)^2+v^2}$. Since $F(R,0) = G^{\theta z}(R,0) = G^{z \theta}(R,0) = 0$, $G^{\theta\theta}(R,0)=-(2R^2)^{-1}$, and $G^{zz}(R,0)=-1/2$, one may apply Proposition \ref{prop-para:paralin} and obtain
	\begin{equation}\label{eq-paralin:H-decomp}
		\begin{aligned}
			H - \frac{1}{2R} =& T_{F_x(\eta,\nabla_w\eta)}\eta + T_{\left( \nabla_{u,v}F \right)(\eta,\nabla_w\eta)}\cdot\nabla_w\eta + T_{G^{jk}(\eta,\nabla_w\eta)}\eta_{jk} \\
			& + T_{\eta_{jk}} T_{G^{jk}_x(\eta,\nabla_w\eta)} \eta + T_{\eta_{jk}} T_{\left( \nabla_{u,v}G^{jk} \right)(\eta,\nabla_w\eta)}\cdot \nabla_w\eta + R\left( \eta_{jk},G^{jk}(\eta,\nabla_w\eta) \right) + r_4,
		\end{aligned}
	\end{equation}
	where Einstein summation convention is applied for simplicity and the remainder $r_4$ is equal to
	\begin{equation}\label{eq-paralin:def-r-4}
		\begin{aligned}
			r_4 :=& F(\eta,\nabla_w\eta) - T_{\nabla_{x,u,v}F(\eta,\nabla_w\eta)}\cdot (\eta,\nabla_w\eta) \\
			&\hspace{6em} + T_{\eta_{jk}} \left( G^{jk}(\eta,\nabla_w\eta) - T_{\nabla_{x,u,v}G^{jk}(\eta,\nabla_w\eta)}\cdot (\eta,\nabla_w\eta) \right),
		\end{aligned}
	\end{equation}
	which lies in $H^{2(s-\frac{1}{2}-)-1} \subset H^{s}$ due to \eqref{eq-para:paralin} and Proposition \ref{prop-para:paraprod-bound}, together with the fact that $\eta\in H^{s+\frac{1}{2}-}_R$. We claim that
	\begin{lemma}\label{lem-paralin:paralin-of-H}
		Under the hypotheses of Proposition \ref{prop-paralin:paralin-of-nonlin}, we have the following equivalences in the sense of \eqref{eq-paralin:equi-paralin-H},
		\begin{align}
			&T_{F_x(\eta,\nabla_w\eta)}\eta \sim 0, \label{eq-paralin:paralin-of-H-2}\\
			&T_{\eta_{jk}} T_{G^{jk}_x(\eta,\nabla_w\eta)} \eta,\ R\left( \eta_{jk},G^{jk}(\eta,\nabla_w\eta) \right) \sim 0, \label{eq-paralin:paralin-of-H-3}\\
			&\left(T_{\eta_{jk}} T_{\left( \nabla_{u,v}G^{jk} \right)(\eta,\nabla_w\eta)} - T_{\eta_{jk}\left( \nabla_{u,v}G^{jk} \right)(\eta,\nabla_w\eta)}\right)\cdot \nabla_w\eta \sim 0. \label{eq-paralin:paralin-of-H-4}
		\end{align}
	\end{lemma}
	\begin{proof}
		From \eqref{eq-para:paralin-cor}, we know that the composition of smooth functions with $(\eta,\nabla_w\eta)$ belong to $H^{s-\frac{1}{2}}$. Thus, from Proposition \ref{prop-para:paraprod-bound}, we have
		\begin{align*}
			&\| T_{F_x(\eta,\nabla_w\eta)}\eta \|_{H^{s}} \lesssim \|F_x(\eta,\nabla_w\eta)\|_{H^{1+}} \|\eta\|_{H^{s}_R} \le C\left(\|\eta\|_{H_R^{s+\frac{1}{2}-}}\right)\|\eta\|_{H^{s+\frac{1}{2}}_R}, \\[1ex]
			&\| T_{\eta_{jk}} T_{G^{jk}_x(\eta,\nabla_w\eta)}\eta \|_{H^{s}} + \| R\left( \eta_{jk},G^{jk}(\eta,\nabla_w\eta) \right) \|_{H^{s}} \\
			\lesssim& \|\eta_{jk}\|_{H^{1+}} \|G^{jk}_x(\eta,\nabla_w\eta)\|_{H^{1+}} \|\eta\|_{H^{s}} + \|\eta_{jk}\|_{H^{\frac{3}{2}+}} \|G^{jk}(\eta,\nabla_w\eta)\|_{H^{s-\frac{1}{2}-}} \\
			\le&C\left(\|\eta\|_{H_R^{s+\frac{1}{2}-}}\right)\|\eta\|_{H^{s+\frac{1}{2}-}_R}, \\[1ex]
			&\| \left(T_{\eta_{jk}} T_{\left( \nabla_{u,v}G^{jk} \right)(\eta,\nabla_w\eta)} - T_{\eta_{jk}\left( \nabla_{u,v}G^{jk} \right)(\eta,\nabla_w\eta)}\right)\cdot \nabla_w\eta \|_{H^{s}} \\
			\le& C\left(\|\eta\|_{H_R^{s+\frac{1}{2}-}}\right) \|\nabla_w\eta\|_{H^{s-\frac{3}{2}-}} \le C\left(\|\eta\|_{H_R^{s+\frac{1}{2}-}}\right)\|\eta\|_{H^{s+\frac{1}{2}}_R},
		\end{align*}
		where in the last inequality we apply Proposition \ref{prop-para:paradiff-cal-sym} with the fact that $\eta_{jk}$ and $\left( \nabla_{u,v}G^{jk} \right)(\eta,\nabla_w\eta)$ are both in $H^{s-\frac{3}{2}-} \subset \Gamma^0_{3/2+}$.
	\end{proof}
	
	To sum up, we have proved that
	\begin{equation}\label{eq-paralin:paralin-of-H-fin}
		\begin{aligned}
			H - \frac{1}{2R} \sim& T_{\left( \nabla_{u,v}F \right)(\eta,\nabla_w\eta)}\cdot\nabla_w\eta + T_{G^{jk}(\eta,\nabla_w\eta)}\eta_{jk} + T_{\eta_{jk}\left( \nabla_{u,v}G^{jk} \right)(\eta,\nabla_w\eta)}\cdot \nabla_w\eta \\
			=& T_{\left( \nabla_{u,v}F \right)(\eta,\nabla_w\eta)\cdot i\xi}\eta - T_{G^{jk}(\eta,\nabla_w\eta)\xi_j\xi_k}\eta + T_{\eta_{jk}\left( \nabla_{u,v}G^{jk} \right)(\eta,\nabla_w\eta)\cdot i\xi}\eta,
		\end{aligned}
	\end{equation}
	and we can define $\mu^{(2)}$, $\mu^{(1)}$ as
	\begin{equation}\label{eq-paralin:mu-compute}
		\begin{aligned}
			&-G^{jk}(\eta,\nabla_w\eta)\xi_j\xi_k = \mu^{(2)}, \\
			&\left( \nabla_{u,v}F \right)(\eta,\nabla_w\eta)\cdot i\xi + \eta_{jk}\left( \nabla_{u,v}G^{jk} \right)(\eta,\nabla_w\eta)\cdot i\xi = \mu^{(1)}.
		\end{aligned}
	\end{equation}
	One can check by direct computation that $\mu^{(2)}$ is given by \eqref{eq-paralin:def-mu-2}. As for $\mu^{(1)}$, we are only interested in its imaginary part, which will be useful in Section \ref{Sect:sym}. We claim that
	\begin{lemma}\label{lem-paralin:mu-1-cal}
		The symbol $\mu = \mu^{(2)} + \mu^{(1)}$ constructed in \eqref{eq-paralin:paralin-of-H} satisfies
		\begin{equation}\label{eq-paralin:mu-1-cal}
			\Imag\mu^{(1)} = -\frac{1}{2}\partial_w\cdot\partial_\xi\mu^{(2)} - \frac{1}{2}\frac{\partial_w\eta}{\eta}\cdot\partial_\xi\mu^{(2)}.
		\end{equation}
	\end{lemma}
	\begin{proof}
		The idea of proof is similar to that of Lemma \ref{lem-paralin:lambda-0-cal}. We may assume $\eta\in H^{+\infty}_R:=\cap_{s\in\R}H^{s}_R$ and reduce the problem to show that $T_{\eta\mu}^* - T_{\eta\mu}$ is of order $1-$ (we shall prove that it is a smoothing operator).
		
		To begin with, we recall that we can represent $H$ as
		\begin{equation*}
			H-\frac{1}{2R} = F(\eta,\eta_\theta,\eta_z) + \sum_{j,k\in\{\theta,z\}} G^{jk}(\eta,\eta_\theta,\eta_z)\eta_{jk},
		\end{equation*}
		where $F(x,u,v)$ and $G^{jk}(x,u,v)$ are smooth functions defined in \eqref{eq-paralin:def-FG-paralin-H}. Then by regarding $H-\frac{1}{2R}$ as a nonlinear operator of $\eta$ (not $(\eta,\eta_\theta,\eta_z)$), one may calculate the derivative w.r.t. $\eta$ of $\eta(H-\frac{1}{2R})$. Namely, for all $\delta\eta\in H^{s}$ with $s \gg 1$, we have
		\begin{align*}
			&\left.\frac{d}{d\epsilon}\right|_{\epsilon=0} \left[ (\eta+\epsilon\delta\eta)\left( H(\eta+\epsilon\delta\eta)-\frac{1}{2R} \right) \right] \\
			=& \left.\frac{d}{d\epsilon}\right|_{\epsilon=0} \left[ (\eta+\epsilon\delta\eta)F(\eta+\epsilon\delta\eta,\partial_\theta(\eta+\epsilon\delta\eta),\partial_z(\eta+\epsilon\delta\eta)) \right] \\ &+\sum_{j,k\in\{\theta,z\}} \left.\frac{d}{d\epsilon}\right|_{\epsilon=0} \left[ (\eta+\epsilon\delta\eta)G^{jk}(\eta+\epsilon\delta\eta,\partial_\theta(\eta+\epsilon\delta\eta),\partial_z(\eta+\epsilon\delta\eta))\partial_{j}\partial_{k}(\eta+\epsilon\delta\eta) \right] \\
			=& F(\eta,\nabla_w\eta) \delta\eta + \eta F_x(\eta,\nabla_w\eta) \delta\eta + \eta\left(\nabla_{u,v}F\right)(\eta,\nabla_w\eta)\cdot\nabla_w\delta\eta + \eta_{jk}G^{jk}(\eta,\nabla_w\eta)\delta\eta  \\
			&+ \eta\eta_{jk}G^{jk}_x(\eta,\nabla_w\eta)\delta\eta + \eta\eta_{jk}\left(\nabla_{u,v}G^{jk}\right)(\eta,\nabla_w\eta)\cdot\nabla_w\delta\eta + \eta G^{jk}(\eta,\nabla_w\eta)\partial_j\partial_k\delta\eta,
		\end{align*}
		where the right hand side is a differential operator acting on $\delta\eta$. By comparing it with \eqref{eq-paralin:mu-compute}, we can see that the symbol of this differential operator equals $\eta\mu + F_0(\eta,\nabla_w\eta,\nabla_w^2\eta)$, where $F_0$ is a smooth real-valued function. 
		
		Now, for any $\delta\eta_1,\delta\eta_2\in H^s_R$ with $s\gg 1$, we define
		\begin{equation*}
			h(\epsilon_1,\epsilon_2) := A(\eta+\epsilon_1\delta\eta_1 + \epsilon_2\delta\eta_2),
		\end{equation*}
		where $A$ is the normalized area of interface $\Sigma(t)$, which is defined in \eqref{eq-pre:def-energy-pot}. By \eqref{eq-pre:var-energy-pot}, we have
		\begin{equation*}
			\partial_{\epsilon_1} h|_{\epsilon_1=0} = \int (\eta+\epsilon_2\delta\eta_2)\left(H(\eta+\epsilon_2\delta\eta_2)-\frac{1}{2R}\right)\delta\eta_1 dw,
		\end{equation*}
		which, together with the calculus above, yields
		\begin{equation*}
			\partial_{\epsilon_2}\partial_{\epsilon_1} h|_{\epsilon_1=\epsilon_2=0} = \int \Op{\eta\mu + F_0(\eta,\nabla_w\eta,\nabla_w^2\eta)}\delta\eta_2 \delta\eta_1 dw.
		\end{equation*}
		Due to the regularity of $\eta$, $h$ is a smooth function of $(\epsilon_1,\epsilon_2)$ near $(0,0)$. Therefore, the order of derivative is not important, from which one can deduce that $\Op{\eta\mu + F_0(\eta,\nabla_w\eta,\nabla_w^2\eta)}$ is self-adjoint. Since $F_0$ is real-valued, the differential operator $\Op{\eta\mu}$ is self-adjoint. Form the fact that the difference between differential operator and paradifferential operators with the same symbol is a smoothing operator, we conclude that $T_{\eta\mu}- T_{\eta\mu}^*$ is a smoothing operator, which completes the proof.
 	\end{proof}
 	
 	Note that, as a by-product, we have re-expressed $H-(2R)^{-1}$ as
 	\begin{equation}
 		\begin{aligned}
 			H - \frac{1}{2R} =& T_\mu \eta + T_{F_x(\eta,\nabla_w\eta)}\eta + T_{\eta_{jk}} T_{G^{jk}_x(\eta,\nabla_w\eta)} \eta + R\left( \eta_{jk},G^{jk}(\eta,\nabla_w\eta) \right) \\
 			& + \left(T_{\eta_{jk}} T_{\left( \nabla_{u,v}G^{jk} \right)(\eta,\nabla_w\eta)} - T_{\eta_{jk}\left( \nabla_{u,v}G^{jk} \right)(\eta,\nabla_w\eta)}\right)\cdot \nabla_w\eta + r_4
 		\end{aligned}
 	\end{equation}
	which can be seen by inserting \eqref{eq-paralin:mu-compute} in \eqref{eq-paralin:H-decomp}. That is to say, the error $r_3$ appearing in \eqref{eq-paralin:paralin-of-H} admits an explicit formulation :
	\begin{equation}\label{eq-paralin:r-3-formula}
		\begin{aligned}
			r_3 =& T_{F_x(\eta,\nabla_w\eta)}\eta + T_{\eta_{jk}} T_{G^{jk}_x(\eta,\nabla_w\eta)} \eta + R\left( \eta_{jk},G^{jk}(\eta,\nabla_w\eta) \right) \\
			& + \left(T_{\eta_{jk}} T_{\left( \nabla_{u,v}G^{jk} \right)(\eta,\nabla_w\eta)} - T_{\eta_{jk}\left( \nabla_{u,v}G^{jk} \right)(\eta,\nabla_w\eta)}\right)\cdot \nabla_w\eta + r_4.
		\end{aligned}
	\end{equation}
	Recall that $F,G$ are defined in \eqref{eq-paralin:def-FG-paralin-H} and $r_4$ is given by \eqref{eq-paralin:def-r-4}.

	\subsection{Continuity of source term}\label{subsect:contin-in-eta}
	
	In previous sections, we manage to rewire \eqref{eq-intro:WW} in paralinear form \eqref{eq-paralin:WW}. Moreover, by combining estimates in Proposition \ref{prop-paralin:paralin-DtN} and Proposition \ref{prop-paralin:paralin-of-nonlin}, we are able to decompose the source term $f$ (see \eqref{eq-paralin:def-f} and \eqref{eq-paralin:def-f-2}) as
	\begin{equation}\label{eq-paralin:f-decomp}
		f = f^{(0)} + f^{(1)} + f^{(2)},
	\end{equation}
	where
	\begin{equation}\label{eq-paralin:f-decomp-def}
		f^{(0)} = \left(\begin{array}c
			0 \\
			-\sigma r_3(\eta)
		\end{array}\right),\ \ 
		f^{(1)} = \left(\begin{array}c
			f_1(\eta,\psi) \\
			0
		\end{array}\right),\ \ 
		f^{(2)} = \left(\begin{array}c
			0 \\
			T_{B(\eta,\psi)} f_1(\eta,\psi) - r_2(\eta;\psi,\psi)
		\end{array}\right),
	\end{equation}
	are respectively independent, linear, and quadratic in $\psi$, with estimates \eqref{eq-paralin:esti-f-1}, \eqref{eq-paralin:esti-r-2}, and \eqref{eq-paralin:esti-r-3}. In these estimates, by taking $s_0=s$, we have, for all $(\eta,\psi)\in H_R^{s+\frac{1}{2}}\times H^s$ with $s>3$, $B\in H^{s-1}$ (see \eqref{eq-paralin:esti-B-V-psi}) where $s-1>1$, which, thanks to Proposition \ref{prop-para:paraprod-bound}, yields 
	\begin{equation}\label{eq-paralin:source-esti}
		\|f(\eta,\psi)\|_{H^{s+\frac{1}{2}}\times H^{s}} \le C\left(\|\eta\|_{H_R^{s+\frac{1}{2}-}},\|\psi\|_{H^{s-}}\right) \left( \|\eta\|_{H_R^{s+\frac{1}{2}}}+\|\psi\|_{H^{s}} \right).
	\end{equation}
	Meanwhile, with $\frac{3}{2}<s_0<s-\frac{3}{2}$, we can prove the local Lipschitz regularity of $f(\eta,\psi)$ in $\psi$, i.e. for all $\eta\in H_R^{s+\frac{1}{2}}$ and $\psi_1,\psi_2\in H^s$ with $s>3$, 
	\begin{equation}\label{eq-paralin:source-Lip-in-psi}
		\|f(\eta,\psi_1) - f(\eta,\psi_2) \|_{H^{s_0+\frac{1}{2}}\times H^{s_0}} \le C\left(\|\eta\|_{H_R^{s+\frac{1}{2}}},\|\psi_1\|_{H^{s_0}},\|\psi_2\|_{H^{s_0}}\right) \|\psi_1-\psi_2\|_{H^{s_0}}.
	\end{equation}
	
	The goal of this section is to show that $f(\eta,\psi)$ is also locally Lipschitzian in $\eta$, which is important to prove the convergence of approximate solutions to the paralinear system \eqref{eq-paralin:WW} and the uniqueness (see Section \ref{subsect:conv-uni}).
	\begin{proposition}\label{prop-paralin:source-Lip}
		Let $(\eta_1,\psi_1),(\eta_2,\psi_2)\in H_R^{s+\frac{1}{2}}\times H^s$ with $s>3$, such that
		\begin{equation*}
			\|(\eta_1,\psi_1)\|_{H_R^{s+\frac{1}{2}}\times H^{s}} + \|(\eta_2,\psi_2)\|_{H_R^{s+\frac{1}{2}}\times H^{s}} \le M
		\end{equation*}
		for some $M>0$. Then we have, for all $\frac{3}{2}<s_0<s-\frac{3}{2}$
		\begin{equation}\label{eq-paralin:source-Lip}
			\|f(\eta_1,\psi_1) - f(\eta_2,\psi_2) \|_{H^{s_0+\frac{1}{2}}\times H^{s_0}} \le C(M) \|(\eta_1-\eta_2,\psi_1-\psi_2)\|_{H^{s_0+\frac{1}{2}}\times H^{s_0}}.
		\end{equation}
	\end{proposition}
	
	The estimate \eqref{eq-paralin:source-Lip-in-psi} allows us to focus on the continuity in $\eta$ by taking $\psi:=\psi_1=\psi_2\in H^s$, and the desired result can be reduced to
	\begin{equation}\label{eq-paralin:source-Lip-in-eta}
		\|f(\eta_1,\psi) - f(\eta_2,\psi) \|_{H^{s_0+\frac{1}{2}}\times H^{s_0}} \le C\left(\|\eta\|_{H_R^{s+\frac{1}{2}}}, \|\psi\|_{H^{s}}\right) \|\eta_1-\eta_2\|_{H^{s_0+\frac{1}{2}}},
	\end{equation}
	which can be further reduced to the boundedness of derivation in $\eta$.
	\begin{lemma}\label{lem-paralin:source-deri-in-eta}
		Let $\eta\in H^{s+\frac{1}{2}}_R$, $\delta\eta\in H^{s+\frac{1}{2}}$, and $\psi\in H^s$ with $s>3$. Then the derivative of $f$ in $\eta$, defined by
		\begin{equation}\label{eq-paralin:def-deri-in-eta}
			\delta f := \frac{d}{d\eta}f(\eta,\psi)\cdot\delta\eta = \left.\frac{d}{d\epsilon}\right|_{\epsilon=0} f(\eta+\epsilon\delta\eta,\psi),
		\end{equation}
		satisfies, for all $\frac{3}{2}<s_0<s-\frac{3}{2}$,
		\begin{equation}\label{eq-paralin:source-deri-in-eta}
			\| \delta f(\eta,\psi) \|_{H^{s_0+\frac{1}{2}}\times H^{s_0}} \le C\left(\|\eta\|_{H_R^{s+\frac{1}{2}}}, \|\psi\|_{H^{s}}\right) \|\delta\eta\|_{H^{s_0+\frac{1}{2}}}.
		\end{equation}
	\end{lemma}
	Since $f$ can be decomposed as \eqref{eq-paralin:f-decomp} with $f^(0)$, $f^{(1)}$, and $f^{(2)}$ given by \eqref{eq-paralin:f-decomp-def}, we need to study $\delta f_1$ and $\delta f_2$ (more precisely $\delta r_2$ and $\delta r_3$, see definition \eqref{eq-paralin:def-f-2} of $f_2$). The study of $f_1$ is the most difficult one, since the only formula for $f_1$ is \eqref{eq-paralin:paralin-DtN}, meaning that it is unavoidable to calculate the derivative-in-$\eta$ of $G(\eta)\psi$ (depending implicitly in $\eta$), which has been done in Proposition \ref{prop-pre:shape-deri}. As for the study of $r_2$, we shall rewrite it as a function of $B$, $V$, and $G(\eta)\psi$ and reduce the problem again to shape derivative (Proposition \ref{prop-pre:shape-deri}). The last term $r_3$, according to its definition \eqref{eq-paralin:paralin-of-H}, depends explicitly on $\eta$ and the corresponding variational calculus is direct.
	
	\subsubsection{Derivative of $f_1$}
	In this part, we shall prove \eqref{eq-paralin:source-deri-in-eta} for $f_1$ (defined in Proposition \ref{prop-paralin:paralin-DtN}), namely
	\begin{equation}\label{eq-paralin:f-1-deri-in-eta}
		\| \delta f_1(\eta,\psi) \|_{H^{s_0+\frac{1}{2}}} \le C\left(\|\eta\|_{H_R^{s+\frac{1}{2}}}, \|\psi\|_{H^{s}}\right) \|\delta\eta\|_{H^{s_0+\frac{1}{2}}}.
	\end{equation}
	
	Before starting the proof, we introduce several technical lemmas which will be frequently used:
	\begin{lemma}\label{lem-paralin:fct-of-eta-deri-in-eta}
		Let $\eta\in H^{s+\frac{1}{2}}_R$ and $\delta\eta\in H^{s_0}$ with $s>3$ and $\frac{3}{2}<s_0\le s$. All the functions $u$ taking the form $u = F(\eta,\nabla_w\eta,\dots\nabla_w^k\eta)$ ($k=0,1,2$) belong to $H^{s+\frac{1}{2}-k}_{F(R,0)}$ with 
		\begin{equation}\label{eq-paralin:fct-of-eta-deri-in-eta}
			\|\delta u\|_{H^{s_0+\frac{1}{2}-k}} \le C\left(\|\eta\|_{H_R^{s+\frac{1}{2}}}\right) \|\delta\eta\|_{H^{s_0+\frac{1}{2}}}.
		\end{equation}
	\end{lemma}
	\begin{proof}
		The fact that $u\in H^{s+\frac{1}{2}-k}_{F(R,0)}$ is no more than a consequence of Proposition \ref{prop-para:paralin}. For the proof of \eqref{eq-paralin:fct-of-eta-deri-in-eta}, we focus on the case $k=2$, while the other cases can be proved in the same way. By definition, the derivative of $u$ in $\eta$ reads
		\begin{equation*}
			\delta u = \sum_{j=0}^k \partial_j F(\eta,\nabla_w\eta,\dots\nabla_w^k\eta) \nabla_w^j\delta\eta,
		\end{equation*}
		where, again by Proposition \ref{prop-para:paralin}, $\partial_j F(\eta,\nabla_w\eta,\dots\nabla_w^k\eta) \in H^{s+\frac{1}{2}-k}$, up to constant normalization, and $\nabla_w^j\delta\eta \in H^{s_0+\frac{1}{2}-j} \subset H^{s_0+\frac{1}{2}-k}$. Thus, Corollary \ref{cor-para:product-law} gives the desired estimate,
		\begin{align*}
			\|\delta u\|_{H^{s_0+\frac{1}{2}-k}} \le& \sum_{j=0}^k \| \partial_j F(\eta,\nabla_w\eta,\dots\nabla_w^k\eta) \nabla_j^k\delta\eta \|_{H^{s_0+\frac{1}{2}-k}} \\
			\lesssim& \sum_{j=0}^k \| \partial_j F(\eta,\nabla_w\eta,\dots\nabla_w^k\eta) \|_{H^{s+\frac{1}{2}-k}} \| \nabla_j^k\delta\eta \|_{H^{s_0+\frac{1}{2}-k}} \\
			\le& C\left(\|\eta\|_{H_R^{s+\frac{1}{2}}}\right) \|\delta\eta\|_{H^{s_0+\frac{1}{2}}}.
		\end{align*}
	\end{proof}
	
	\begin{lemma}\label{lem-paralin:BV-deri-in-eta}
		Let $\eta\in H^{s+\frac{1}{2}}_R$, $\psi\in H^s$, and $\delta\eta\in H^{s_0}$ with $s>3$ and $\frac{3}{2}<s_0\le s$. Then the following estimate holds,
		\begin{equation}\label{eq-paralin:BV-deri-in-eta}
			\|\delta B\|_{H^{s_0-1}} + \|\delta V\|_{H^{s_0-1}} \le C\left(\|\eta\|_{H_R^{s+\frac{1}{2}}}, \|\psi\|_{H^{s}}\right) \|\delta\eta\|_{H^{s_0+\frac{1}{2}}}.
		\end{equation}
	\end{lemma}
	\begin{proof}
		Recall that, since $\delta\eta\in H^{s_0+\frac{1}{2}}$, the derivative in $\eta$ of $G(\eta)\psi$ lies in $H^{s_0-1}$ due to Proposition \ref{prop-pre:shape-deri}. The formula \eqref{eq-paralin:B-formula} allows us to estimate $\delta B$. In fact, from \eqref{eq-paralin:B-formula}, we have
		\begin{equation*}
			\delta B = \frac{\delta \left(G(\eta)\psi\right)}{1+ \left|\bar{\nabla}\eta\right|^2} + \frac{\delta\left(\bar{\nabla}\psi\cdot\bar{\nabla}\eta\right)}{1+ \left|\bar{\nabla}\eta\right|^2} + \left(G(\eta)\psi + \bar{\nabla}\psi\cdot\bar{\nabla}\eta\right)\delta\left(\left(1+ \left|\bar{\nabla}\eta\right|^2\right)^{-1}\right).
		\end{equation*}
		From $\delta\left(G(\eta)\psi\right) \in H^{s_0-1}$, the first term on the right hand side belongs to $H^{s_0-1}$, while the other terms belong to the same space $H^{s_0-1}$ due to Lemma \ref{lem-paralin:fct-of-eta-deri-in-eta} and Corollary \ref{cor-para:product-law}, which completes the proof for $\delta B$. As for $\delta V$, the formula \eqref{eq-paralin:V-formula} implies
		\begin{equation*}
			\delta V = \delta\left( \bar{\nabla}\psi \right) + \delta B \bar{\nabla}\eta + B \delta\left( \bar{\nabla}\eta \right).
		\end{equation*}
		The desired estimate then follows from that of $\delta B$ and Lemma \ref{lem-paralin:fct-of-eta-deri-in-eta}.
	\end{proof}
	
	\begin{lemma}\label{lem-paralin:sym-deri-in-eta}
		Let $\eta\in H^{s+\frac{1}{2}}_R$ and $\delta\eta\in H^{s_0+\frac{1}{2}}$ with $s>3$ and $\frac{3}{2}<s_0\le s$. Let $a = a^{(m)} + a^{(m-1)}$ be a symbol with $m\in\R$ and 
		\begin{equation*}
			a^{(m)} = F(\eta,\nabla_{\theta,z}\eta;\xi),\ \ a^{(m-1)} = \sum_{|\alpha|\le 2} G_{\alpha}(\eta,\nabla_{\theta,z}\eta;\xi)\partial_{\theta,z}^\alpha\eta,
		\end{equation*}
		where $F(x,u,v;\xi)$ and $G_\alpha(x,u,v;\xi)$'s are smooth function of order $m$ and $m-1$ in $\xi$, respectively. Namely, for all $\beta\in\N^3$ and $\gamma\in\N^2$
		\begin{equation*}
			\sup_{x,u,v}\left| \partial_{x,u,v}^\beta \partial_\xi^\gamma F(x,u,v;\xi) \right| \lesssim \langle\xi\rangle^{m-|\gamma|},\ \ \sup_{x,u,v}\left| \partial_{x,u,v}^\beta \partial_\xi^\gamma G_\alpha(x,u,v;\xi) \right| \lesssim \langle\xi\rangle^{m-1-|\gamma|}.
		\end{equation*}
		Then $T_a$ and $T_{\delta a}$ are of order $m$.
	\end{lemma}
	\begin{proof}
		Clearly $a\in\Gamma^m_{3/2+} + \Gamma^{m-1}_{1/2+}$ and thus $T_a$ is of order $m$. As for $T_{\delta a}$, a direct calculus gives that
		\begin{equation*}
			\delta a = \delta \left( F(\eta,\nabla_{w}\eta;\xi) \right) + \sum_{|\alpha|\le 2} \delta\left(G_{\alpha}(\eta,\nabla_{w}\eta;\xi)\right)\partial_{w}^\alpha\eta + \sum_{|\alpha|\le 2} G_{\alpha}(\eta,\nabla_{w}\eta;\xi)\partial_{w}^\alpha\delta\eta.
		\end{equation*}
		By Lemma \ref{lem-paralin:fct-of-eta-deri-in-eta} and $s>3$, all the terms on the right hand side belong to $\Gamma^m_{0+}$ or $\Gamma^{m-1}_{0+}$ except for those $G_{\alpha}(\eta,\nabla_{w}\eta;\xi)\partial_{\theta,z}^\alpha\delta\eta$ with $|\alpha|=2$. To deal with these terms, we observe that $G_{\alpha}(\eta,\nabla_{w}\eta;\xi)$ has $H^{s-\frac{1}{2}}$ regularity in $w$ variable while the low regularity comes from $\partial_{\theta,z}^\alpha\delta\eta \in H^{s_0-\frac{3}{2}}$. By applying Corollary \ref{cor-para:product-law}, it is easy to check that, for all $|\xi|>\frac{1}{2}$ and $\gamma\in\N^2$,
		\begin{equation*}
			\|\partial_\xi^\gamma G_{\alpha}(\eta,\nabla_{w}\eta;\xi)\partial_{\theta,z}^\alpha\delta\eta\|_{H^{s_0-\frac{3}{2}}} \lesssim \|G_{\alpha}(\eta,\nabla_{w}\eta;\xi)\|_{H^{s-\frac{1}{2}}} \|\delta\eta\|_{H^{s_0+\frac{1}{2}}} \langle\xi\rangle^{m-1-|\gamma|}.
		\end{equation*}
		Then by Proposition \ref{prop-para:paradiff-symbol-class-sobo}, the associated paradifferential operators is of order $m-1$ when $s_0-\frac{3}{2}>1$, or
		\begin{equation*}
			m-1 + 1 - \left( s_0-\frac{3}{2} \right) = m - \left( s_0-\frac{3}{2} \right) < m,
		\end{equation*}
		when $s_0-\frac{3}{2}<1$ (in the critical case $s_0-\frac{3}{2}=1$, one may replace $s_0-\frac{3}{2}$ by $s_0-\frac{3}{2}-$), which completes the proof.
	\end{proof}
	
	Now we are ready to prove \eqref{eq-paralin:f-1-deri-in-eta}. As before, during this proof, we denote
	\begin{equation}\label{eq-paralin:equi-f-1-deri-in-eta}
		u \sim v \Leftrightarrow \|u-v\|_{H^{s_0+\frac{1}{2}}} \le C\left(\|\eta\|_{H_R^{s+\frac{1}{2}}}, \|\psi\|_{H^{s}}\right) \|\delta\eta\|_{H^{s_0+\frac{1}{2}}}.
	\end{equation} 
	The main idea of the proof is to use \eqref{eq-paralin:paralin-DtN} together with shape derivative formula \eqref{eq-pre:shape-deri} to obtain an explicit expression of $\delta f_1$, where most terms can be treated separately and the remaining ones (see \eqref{eq-paralin:f-1-deri-in-eta-reduce}) will be proved to admit some cancellation thanks to an application of paralinearization \eqref{eq-paralin:paralin-DtN} with $\psi$ replaced by $B\delta\eta$ and $B$.
	
	From \eqref{eq-paralin:paralin-DtN}, the derivative of $f_1$ in $\eta$ reads
	\begin{align*}
		\delta f_1 =& \delta \left( G(\eta)\psi - T_\lambda U + T_V \cdot \bar{\nabla}\eta \right) \\
		=& \delta G(\eta)\psi - T_{\delta \lambda} U - T_\lambda \delta U + T_{\delta V} \cdot \bar{\nabla}\eta + T_V \cdot \delta \left( \bar{\nabla}\eta \right) \\
		=& \delta G(\eta)\psi - T_{\delta \lambda} U - T_\lambda \delta U + T_{\delta V} \cdot \bar{\nabla}\eta + T_V \cdot \left( -\frac{\eta_\theta}{\eta^2}\delta\eta, 0 \right) + T_V \cdot \bar{\nabla}\delta\eta.
	\end{align*}
	By \eqref{eq-pre:shape-deri} and definition \eqref{eq-paralin:good-unknown} of $U$, the right hand side can be further written as
	\begin{align*}
		\delta f_1 =& - G(\eta)\left( B\delta\eta \right) - \nabla_{w} \cdot \left( \frac{V^\theta}{\eta}\delta\eta, V^z\delta\eta \right) - \frac{B\delta\eta}{\eta} - T_{\delta \lambda} U + T_\lambda \delta \left( T_B\eta \right) \\
		& + T_{\delta V} \cdot \bar{\nabla}\eta + T_V \cdot \left( -\frac{\eta_\theta}{\eta^2}\delta\eta, 0 \right) + T_V \cdot \bar{\nabla}\delta\eta \\
		=& - G(\eta)\left( B\delta\eta \right) - \bar{\nabla}\cdot(V\delta\eta) + \frac{\eta_\theta}{\eta^2}V^\theta\delta\eta - \frac{B\delta\eta}{\eta} - T_{\delta \lambda} U + T_\lambda T_{\delta B}\eta + T_\lambda T_B\delta\eta \\
		& + T_{\delta V} \cdot \bar{\nabla}\eta + T_V \cdot \left( -\frac{\eta_\theta}{\eta^2}\delta\eta, 0 \right) + T_V \cdot \bar{\nabla}\delta\eta.
	\end{align*}
	
	\begin{lemma}\label{lem-paralin:f-1-deri-in-eta-S1}
		Under the hypotheses of Lemma \ref{lem-paralin:source-deri-in-eta}, we have
		\begin{align}
			\frac{\eta_\theta}{\eta^2}V^\theta\delta\eta,\ \frac{B\delta\eta}{\eta},\  T_V \cdot \left( -\frac{\eta_\theta}{\eta^2}\delta\eta, 0 \right) & \sim 0, \label{eq-paralin:f-1-deri-in-eta-S1-1} \\
			T_\lambda T_{\delta B}\eta,\ T_{\delta V} \cdot \bar{\nabla}\eta & \sim 0, \label{eq-paralin:f-1-deri-in-eta-S1-2} \\
			T_{\delta\lambda} U & \sim 0, \label{eq-paralin:f-1-deri-in-eta-S1-3}
		\end{align}
		in the sense of \eqref{eq-paralin:equi-f-1-deri-in-eta}
	\end{lemma}
	\begin{proof}[Proof of Lemma \ref{lem-paralin:f-1-deri-in-eta-S1}]
		The proof of \eqref{eq-paralin:f-1-deri-in-eta-S1-1} is direct. Since $B,V\in H^{s-1}$, we are able to deduce the following estimates from Proposition \ref{prop-para:paraprod-bound}, Corollary \ref{cor-para:product-law}, and Proposition \ref{prop-para:paralin},
		\begin{align*}
			&\|\frac{\eta_\theta}{\eta^2}V^\theta\delta\eta\|_{H^{s_0+\frac{1}{2}}} \lesssim \|\frac{\eta_\theta}{\eta^2}V^\theta\|_{H^{s-1}} \|\delta\eta\|_{H^{s_0+\frac{1}{2}}} \le C\left(\|\eta\|_{H_R^{s+\frac{1}{2}}}, \|\psi\|_{H^{s}}\right) \|\delta\eta\|_{H^{s_0+\frac{1}{2}}}, \\
			&\|\frac{B\delta\eta}{\eta}\|_{H^{s_0+\frac{1}{2}}} \lesssim \|\frac{B}{\eta}\|_{H^{s-1}} \|\delta\eta\|_{H^{s_0+\frac{1}{2}}} \le C\left(\|\eta\|_{H_R^{s+\frac{1}{2}}}, \|\psi\|_{H^{s}}\right) \|\delta\eta\|_{H^{s_0+\frac{1}{2}}}, \\
			&\|T_V \cdot \left( -\frac{\eta_\theta}{\eta^2}\delta\eta, 0 \right)\|_{H^{s_0+\frac{1}{2}}} \lesssim \|V\|_{H^{1+}} \|\frac{\eta_\theta}{\eta^2}\delta\eta\|_{H^{s_0+\frac{1}{2}}} \lesssim \|V\|_{H^{1+}} \|\frac{\eta_\theta}{\eta^2}\|_{H^{s_0+\frac{1}{2}}} \|\delta\eta\|_{H^{s_0+\frac{1}{2}}} \\
			& \hspace{18em}\le C\left(\|\eta\|_{H_R^{s+\frac{1}{2}}}, \|\psi\|_{H^{s}}\right) \|\delta\eta\|_{H^{s_0+\frac{1}{2}}}.
		\end{align*}
		
		To prove \eqref{eq-paralin:f-1-deri-in-eta-S1-2}, we recall that, due to Lemma \ref{lem-paralin:BV-deri-in-eta}, $\delta B, \delta V\in H^{s_0-1}$ As a result,
		\begin{align*}
			\| T_\lambda T_{\delta B}\eta \|_{H^{s_0+\frac{1}{2}}} \le& C\left(\|\eta\|_{H_R^{s+\frac{1}{2}}}\right) \|T_{\delta B}\eta\|_{H^{s_0+\frac{3}{2}}} \lesssim \|\delta B\|_{H^{s_0-1}} \|\eta\|_{H^{\max(\frac{7}{2}+,s_0+\frac{3}{2})}} \\
			\le& C\left(\|\eta\|_{H_R^{s+\frac{1}{2}}}, \|\psi\|_{H^{s}}\right) \|\delta\eta\|_{H^{s_0+\frac{1}{2}}}, \\
			\| T_{\delta V} \cdot \bar{\nabla}\eta \|_{H^{s_0+\frac{1}{2}}} \lesssim& \|\delta V\|_{H^{s_0-1}} \|\bar{\nabla}\eta\|_{H^{\max(\frac{5}{2},s_0+\frac{1}{2})+}} \le C\left(\|\eta\|_{H_R^{s+\frac{1}{2}}}, \|\psi\|_{H^{s}}\right) \|\delta\eta\|_{H^{s_0+\frac{1}{2}}}.
		\end{align*}
		
		As for the last equivalence \eqref{eq-paralin:f-1-deri-in-eta-S1-3}, it is no more than a consequence of $U\in H^{s} \subset H^{s_0+\frac{3}{2}}$ and $\lambda\in \Sigma^{1}$ with Lemma \ref{lem-paralin:sym-deri-in-eta}.
	\end{proof}
	
	It remains to prove \eqref{eq-paralin:f-1-deri-in-eta}, namely that
	\begin{equation}\label{eq-paralin:f-1-deri-in-eta-reduce}
		\delta f_1 \sim - G(\eta)\left( B\delta\eta \right) - \bar{\nabla}\cdot\left( V\delta\eta \right) + T_\lambda T_B\delta\eta + T_V \cdot \bar{\nabla}\delta\eta
	\end{equation}
	is equivalent to zero in the sense of \eqref{eq-paralin:equi-f-1-deri-in-eta}. Note that $\delta\eta$ is assumed to be in $H^{s+\frac{1}{2}}$, permitting us to apply the paralinearization of Dirichlet-to-Neumann operator (Proposition \ref{prop-paralin:paralin-DtN}). More precisely, we have
	\begin{align*}
		&\left\| G(\eta)\left( B\delta\eta \right) - T_\lambda \left(B\delta\eta - T_{\mathcal{B}(\eta)(B\delta\eta)}\eta\right) + T_{\mathcal{V}(\eta)(B\delta\eta)}\cdot \bar{\nabla}\eta \right\|_{H^{s_0+\frac{1}{2}}} \\
		&\le C\left(\|\eta\|_{H^{s+\frac{1}{2}}_R}\right) \|B\delta\eta\|_{H^{s_0}} \le C\left(\|\eta\|_{H_R^{s+\frac{1}{2}}}, \|\psi\|_{H^{s}}\right) \|\delta\eta\|_{H^{s_0+\frac{1}{2}}}.
	\end{align*}
	Recall that linear operators $\mathcal{B}(\eta)$ and $\mathcal{V}(\eta)$ are defined in \eqref{eq-paralin:def-B-V-op} with estimate \eqref{eq-paralin:esti-B-V-op}. Then an application of Lemma \ref{lem-paralin:esti-B-V-psi}, together with $B\delta\eta\in H^{s_0+\frac{1}{2}}$, gives the following estimates
	\begin{align*}
		\left\| T_{\mathcal{V}(\eta)(B\delta\eta)}\cdot \bar{\nabla}\eta \right\|_{H^{s_0+\frac{1}{2}}} \le& \| \mathcal{V}(\eta)(B\delta\eta) \|_{H^{1+}} \left\| \bar{\nabla}\eta \right\|_{H^{s_0+\frac{1}{2}}} \le C\left(\|\eta\|_{H_R^{s+\frac{1}{2}}}, \|\psi\|_{H^{s}}\right) \|\delta\eta\|_{H^{s_0+\frac{1}{2}}}, \\
		\left\| T_{\mathcal{B}(\eta)(B\delta\eta)} \eta \right\|_{H^{s_0+\frac{3}{2}}} \le& \| \mathcal{B}(\eta)(B\delta\eta) \|_{H^{s_0-\frac{1}{2}}} \left\| \eta \right\|_{H^{s_0+\frac{3}{2}}} \le C\left(\|\eta\|_{H_R^{s+\frac{1}{2}}}, \|\psi\|_{H^{s}}\right) \|\delta\eta\|_{H^{s_0+\frac{1}{2}}}.
	\end{align*}
	Therefore, the following paradifferential calculus holds:
	\begin{align*}
		\delta f_1 \sim& -T_\lambda\left( B\delta\eta \right) + T_\lambda T_B\delta\eta - (\bar{\nabla}\cdot V) \delta\eta - V\cdot\bar{\nabla}\delta\eta + T_V \cdot \bar{\nabla}\delta\eta \\
		=& -T_\lambda \left( T_{\delta\eta}B + R(B,\delta\eta) \right) - (\bar{\nabla}\cdot V) \delta\eta - \left( T_{\bar{\nabla}\delta\eta}\cdot V + R(V,\bar{\nabla}\delta\eta) \right) \\
		\sim& -T_\lambda T_{\delta\eta}B + (G(\eta)B) \delta\eta,
	\end{align*}
	where the last equivalence is from Proposition \ref{prop-pre:cancellation} and the following estimate (recall that we assume $s>s_0+3/2$)
	\begin{align*}
		&\left\| R(B,\delta\eta) \right\|_{H^{s_0+\frac{3}{2}}} + \left\| T_{\bar{\nabla}\delta\eta}V \right\|_{H^{s_0+\frac{1}{2}}} + \left\| R(V,\bar{\nabla}\delta\eta) \right\|_{H^{s_0+\frac{1}{2}}} \\
		\lesssim& \|\delta\eta\|_{H^{s_0+\frac{1}{2}}} \|B\|_{H^{s-1}} +  \|\bar{\nabla}\delta\eta\|_{H^{s_0-\frac{1}{2}}} \|V\|_{H^{s-1}} \\
		\le& C\left(\|\eta\|_{H_R^{s+\frac{1}{2}}}, \|\psi\|_{H^{s}}\right) \|\delta\eta\|_{H^{s_0+\frac{1}{2}}}.
	\end{align*}
	Our problem is now reduced to
	\begin{equation*}
		 (G(\eta)B) \delta\eta - T_\lambda T_{\delta\eta}B \sim 0.
	\end{equation*}
	
	To prove this, we apply again Proposition \ref{prop-paralin:paralin-DtN} with $\psi$ replaced by $B\in H^{s-1}$, namely
	\begin{equation*}
		G(\eta)B \sim T_{\lambda}\left( B - T_{\mathcal{B}(\eta)B}\eta \right) - T_{\mathcal{V}(\eta)B} \cdot \bar{\nabla}\eta.
	\end{equation*}
	Recall that linear operators $\mathcal{B}(\eta)$ and $\mathcal{V}(\eta)$ are defined in \eqref{eq-paralin:def-B-V-op} and the estimate \eqref{eq-paralin:esti-B-V-op} ensures that $\mathcal{B}(\eta)B,\mathcal{V}(\eta)B \in H^{s-2}$. Since $\eta\in H^{s+\frac{1}{2}}_R$ and $\bar{\nabla}\eta\in H^{s-\frac{1}{2}}$ due to Corollary \ref{cor-para:product-law}, one can apply Proposition \ref{prop-para:paraprod-bound} to show that $T_\lambda T_{\mathcal{B}(\eta)B}\eta$ and $T_{\mathcal{V}(\eta)B} \cdot \bar{\nabla}\eta$ belong to $H^{s-\frac{1}{2}}$, which implies that
	\begin{equation*}
		G(\eta)B \sim T_\lambda B,
	\end{equation*}
	since $s-1/2>s_0+1/2$ (recall that in our assumption $s>s_0+3/2$). Consequently,
	\begin{equation*}
		(G(\eta)B) \delta\eta \sim (T_\lambda B) \delta\eta = T_{\delta\eta}T_\lambda B + T_{T_\lambda B} \delta\eta + R(T_\lambda B, \delta\eta) \sim T_{\delta\eta}T_\lambda B,
	\end{equation*}
	where the last equivalence is due to
	\begin{align*}
		&\left\| T_{T_\lambda B} \delta\eta \right\|_{H^{s_0+\frac{1}{2}}} + \left\| R(T_\lambda B, \delta\eta) \right\|_{H^{s_0+\frac{1}{2}}} \lesssim \|T_\lambda B\|_{H^{s-2}} \|\delta\eta\|_{H^{s_0+\frac{1}{2}}} \\ 
		&\hspace{4em}\le C\left(\|\eta\|_{H_R^{s+\frac{1}{2}}}\right) \|B\|_{H^{s-1}} \|\delta\eta\|_{H^{s_0+\frac{1}{2}}} \le C\left(\|\eta\|_{H_R^{s+\frac{1}{2}}}, \|\psi\|_{H^{s}}\right) \|\delta\eta\|_{H^{s_0+\frac{1}{2}}}.
	\end{align*}
	Thus,
	\begin{equation*}
		G(\eta)B \delta\eta - T_\lambda T_{\delta\eta}B \sim T_{\delta\eta}T_\lambda B- T_\lambda T_{\delta\eta}B = [T_{\delta\eta},T_{\lambda^{(1)}}] B + [T_{\delta\eta},T_{\lambda^{(0)}}] B.
	\end{equation*}
	Since $\delta\eta \in H^{s_0+\frac{1}{2}} \subset \Gamma_{1+}^0$, $\lambda^{(1)} \in \Gamma^1_{3/2+}$, and $\lambda^{(0)}\in\Gamma^0_{1/2+}$, the commutator estimate (Corollary \ref{cor-para:commu-esti}) implies that $[T_{\delta\eta},T_{\lambda^{(1)}}]$ and $[T_{\delta\eta},T_{\lambda^{(0)}}]$ are both of order less than $0$, and the proof of \eqref{eq-paralin:f-1-deri-in-eta} is completed by observing that $B\in H^{s-1} \subset H^{s_0+\frac{1}{2}}$.

	\subsubsection{Derivative of $f_2$}
	
	The purpose of this part is to show that the derivative in $\eta$ of $f_2$ (defined by \eqref{eq-paralin:def-f-2}) is bounded, namely
	\begin{equation}\label{eq-paralin:f-2-deri-in-eta}
		\| \delta f_2(\eta,\psi) \|_{H^{s_0}} \le C\left(\|\eta\|_{H_R^{s+\frac{1}{2}}}, \|\psi\|_{H^{s}}\right) \|\delta\eta\|_{H^{s_0+\frac{1}{2}}}.
	\end{equation}
	
	We first claim that this estimate implies the desired result \eqref{eq-paralin:source-deri-in-eta}. In fact, by \eqref{eq-paralin:f-1-deri-in-eta} and the definition \eqref{eq-paralin:def-f} of $f$, it suffices to show 
	\begin{equation}\label{eq-paralin:2nd-compo-of-f-deri-in-eta}
		\| \delta \left( T_B f_1 + f_2 \right) \|_{H^{s_0}} \le C\left(\|\eta\|_{H_R^{s+\frac{1}{2}}}, \|\psi\|_{H^{s}}\right) \|\delta\eta\|_{H^{s_0+\frac{1}{2}}},
	\end{equation}
	which can be reduced to
	\begin{equation*}
		\| \delta \left( T_B f_1\right) \|_{H^{s_0}} \le C\left(\|\eta\|_{H_R^{s+\frac{1}{2}}}, \|\psi\|_{H^{s}}\right) \|\delta\eta\|_{H^{s_0+\frac{1}{2}}},
	\end{equation*}
	provided \eqref{eq-paralin:f-2-deri-in-eta} is correct. In \eqref{eq-paralin:f-1-deri-in-eta} and Lemma \ref{lem-paralin:BV-deri-in-eta}, we have seen $\delta B\in H^{s_0-1}$ and $\delta f_1 \in H^{s_0+\frac{1}{2}}$, while $B\in H^{s-1}$ and $f_1\in H^{s+\frac{1}{2}}$. These estimates give
	\begin{align*}
		\| \delta \left( T_B f_1\right) \|_{H^{s_0}} \le& \| T_{\delta B}f_1 \|_{H^{s_0}} + \| T_B \delta f_1 \|_{H^{s_0}} \\
		\lesssim& \| \delta B \|_{H^{s_0-1}} \| f_1 \|_{H^{\max(s_0,2+)}} + \| B \|_{H^{1+}} \| \delta f_1 \|_{H^{s_0}} \\
		\lesssim& C\left(\|\eta\|_{H_R^{s+\frac{1}{2}}}, \|\psi\|_{H^{s}}\right) \|\delta\eta\|_{H^{s_0+\frac{1}{2}}},
	\end{align*}
	which completes the proof of \eqref{eq-paralin:2nd-compo-of-f-deri-in-eta}.
	
	As before, during the proof of \eqref{eq-paralin:f-2-deri-in-eta}, we denote
	\begin{equation}\label{eq-paralin:equi-f-2-deri-in-eta}
		u \sim v \Leftrightarrow \|u-v\|_{H^{s_0}} \le C\left(\|\eta\|_{H_R^{s+\frac{1}{2}}}, \|\psi\|_{H^{s}}\right) \|\delta\eta\|_{H^{s_0+\frac{1}{2}}}.
	\end{equation}
	From definition \eqref{eq-paralin:def-f-2}, it suffices to prove that $r_2$ (defined in \eqref{eq-paralin:paralin-of-N}) and $r_3$ (defined in \eqref{eq-paralin:paralin-of-H}) are equivalent to zero. We shall begin with the estimate of $r_2$.
	
	\subsubsection{Study of $r_2$.} 
	By \eqref{eq-paralin:paralin-of-N}, $r_2$ reads
	\begin{equation*}
		r_2 = N - T_V\cdot\bar{\nabla}\psi + T_B G(\eta)\psi + T_B T_V\cdot\bar{\nabla}\eta.
	\end{equation*}
	As in Section \ref{subsubsect:paralin-N}, we shall decompose it as a function of $B$, $V$, and $G(\eta)\psi$ and apply Proposition \ref{prop-pre:shape-deri}. In this part, the shape derivative formula \eqref{eq-pre:shape-deri} is not necessary, we only need the regularity of $\delta(G(\eta)\psi)$ (see also Lemma \ref{lem-paralin:BV-deri-in-eta}). 
	
	Via the same calculus as in \eqref{eq-paralin:N-first-paralin}, we can reformulate $r_2$ as 
	\begin{align*}
		r_2 =& \frac{|V|^2 + B^2}{2} - B G(\eta)\psi - T_V\cdot\left( V + B\bar{\nabla}\eta \right) + T_B G(\eta)\psi + T_B T_V\cdot\bar{\nabla}\eta \\
		=& T_B B - T_{G(\eta)\psi}B - T_V\cdot\left( B\bar{\nabla}\eta \right) + T_B T_V\cdot\bar{\nabla}\eta + R(V,V) + R(B,B) + R(B,G(\eta)\psi) \\
		=& T_{V\cdot\bar{\nabla}\eta} B - T_V\cdot\left( B\bar{\nabla}\eta \right) + T_B T_V\cdot\bar{\nabla}\eta + R(V,V) + R(B,B) + R(B,G(\eta)\psi) \\
		=& \left( T_{V\cdot\bar{\nabla}\eta} - T_{V}\cdot T_{\bar{\nabla}\eta} \right)B + [T_B,T_V]\cdot\bar{\nabla}\eta -T_V\cdot R(B,\bar{\nabla}\eta) \\
		&\hspace{12em}+ R(V,V) + R(B,B) + R(B,G(\eta)\psi).
	\end{align*}
	And the desired result $r_2\sim0$ is no more than a consequence of the following lemma:
	\begin{lemma}\label{lem-paralin:r-2-deri-in-eta}
		Under the hypotheses of Lemma \ref{lem-paralin:source-deri-in-eta}, we have the following equivalences in the sense of \eqref{eq-paralin:equi-f-2-deri-in-eta},
		\begin{align}
			&\delta\left(\left( T_{V\cdot\bar{\nabla}\eta} - T_{V}\cdot T_{\bar{\nabla}\eta} \right)B \right) \sim 0, \label{eq-paralin:r-2-deri-in-eta-1} \\
			&\delta\left( [T_B,T_V]\cdot\bar{\nabla}\eta \right) \sim 0, \label{eq-paralin:r-2-deri-in-eta-2} \\
			&\delta\left( T_V\cdot R(B,\bar{\nabla}\eta) \right) \sim 0, \label{eq-paralin:r-2-deri-in-eta-3} \\
			&\delta\left( R(V,V) + R(B,B) + R(B,G(\eta)\psi) \right) \sim 0. \label{eq-paralin:r-2-deri-in-eta-4}
		\end{align}
	\end{lemma}
	\begin{proof}
		The left hand side of \eqref{eq-paralin:r-2-deri-in-eta-1} is equal to
		\begin{align*}
			\delta\left(\left( T_{V\cdot\bar{\nabla}\eta} - T_{V}\cdot T_{\bar{\nabla}\eta} \right)B \right) 
			=& T_{\delta V \cdot \bar{\nabla}\eta}B - T_{\delta V} \cdot T_{\bar{\nabla}\eta} B + T_{V \cdot \delta\left( \bar{\nabla}\eta \right) }B - T_{V} \cdot T_{\delta\left(\bar{\nabla}\eta\right)} B \\
			& + \left( T_{V\cdot\bar{\nabla}\eta} - T_{V}\cdot T_{\bar{\nabla}\eta} \right)\delta B,
		\end{align*}
		where each term on the right hand side is equivalent to zero. In fact, since $\delta B,\delta V\in H^{s_0-1}$ (see Lemma \ref{lem-paralin:BV-deri-in-eta}) and $\delta\left( \bar{\nabla}\eta \right)\in H^{s_0-\frac{1}{2}}$ (see Lemma \ref{lem-paralin:fct-of-eta-deri-in-eta}), one may obtain from Proposition \ref{prop-para:paradiff-cal-sym}, \ref{prop-para:paraprod-bound}, and Corollary \ref{cor-para:product-law} that
		\begin{align*}
			&\| T_{\delta V \cdot \bar{\nabla}\eta}B \|_{H^{s_0}} + \| T_{\delta V} \cdot T_{\bar{\nabla}\eta} B \|_{H^{s_0}} \\
			\lesssim& \|\delta V \cdot \bar{\nabla}\eta\|_{H^{s_0-1}} \|B\|_{H^{\max(s_0,2+)}} + \|\delta V \|_{H^{s_0-1}} \| \bar{\nabla}\eta\|_{H^{1+}} \|B\|_{H^{\max(s_0,2+)}} \\
			\lesssim& \|\delta V \|_{H^{s_0-1}} \| \bar{\nabla}\eta\|_{H^{\max(1+,s_0-1)}} \|B\|_{H^{\max(s_0,2+)}} \le C\left(\|\eta\|_{H_R^{s+\frac{1}{2}}}, \|\psi\|_{H^{s}}\right) \|\delta\eta\|_{H^{s_0+\frac{1}{2}}}, \\[1ex]
			&\| T_{V \cdot \delta\left( \bar{\nabla}\eta \right) }B \|_{H^{s_0}} + \| T_{V} \cdot T_{\delta\left(\bar{\nabla}\eta\right)} B \|_{H^{s_0}} \\
			\lesssim& \|V \cdot \delta\left( \bar{\nabla}\eta \right)\|_{H^{1+}} \|B\|_{H^{s_0}} + \|V \|_{H^{1+}} \| \delta\left( \bar{\nabla}\eta \right)\|_{H^{1+}} \|B\|_{H^{s_0}}  \\
			\lesssim& \|V \|_{H^{1+}} \| \delta\left( \bar{\nabla}\eta \right)\|_{H^{1+}} \|B\|_{H^{s_0}} \le C\left(\|\eta\|_{H_R^{s+\frac{1}{2}}}, \|\psi\|_{H^{s}}\right) \|\delta\eta\|_{H^{s_0+\frac{1}{2}}}, \\[1ex]
			&\| \left( T_{V\cdot\bar{\nabla}\eta} - T_{V}\cdot T_{\bar{\nabla}\eta} \right)\delta B \|_{H^{s_0}} \\
			&\le C\left(\|\eta\|_{H_R^{s+\frac{1}{2}}}, \|\psi\|_{H^{s}}\right) \| \delta B \|_{H^{s_0-1}} \le C\left(\|\eta\|_{H_R^{s+\frac{1}{2}}}, \|\psi\|_{H^{s}}\right) \|\delta\eta\|_{H^{s_0+\frac{1}{2}}}.
		\end{align*}
		Recall that the regularity of $B$ and $V$ has been studied in Lemma \ref{lem-paralin:esti-B-V-psi}.
		
		The proof of \eqref{eq-paralin:r-2-deri-in-eta-2} is similar. We first write its left hand side as
		\begin{equation*}
			\delta\left( [T_B,T_V]\cdot\bar{\nabla}\eta \right) = [T_{\delta B},T_V]\cdot\bar{\nabla}\eta + [T_{B},T_{\delta V}]\cdot\bar{\nabla}\eta + [T_B,T_V]\cdot\delta\left( \bar{\nabla}\eta \right),
		\end{equation*}
		where we shall apply commutator estimate (Corollary \ref{cor-para:commu-esti}) only for the last term on the right hand side.
		\begin{align*}
			&\| [T_{\delta B},T_V]\cdot\bar{\nabla}\eta \|_{H^{s_0}} + \| [T_{B},T_{\delta V}]\cdot\bar{\nabla}\eta \|_{H^{s_0}} \\
			\le& \| T_{\delta B}T_V\cdot\bar{\nabla}\eta \|_{H^{s_0}} + \| T_V\cdot T_{\delta B}\bar{\nabla}\eta \|_{H^{s_0}} + \| T_{B}T_{\delta V}\cdot\bar{\nabla}\eta \|_{H^{s_0}} + \| T_{\delta V}\cdot T_{B}\bar{\nabla}\eta \|_{H^{s_0}} \\
			\lesssim& \left( \|B\|_{H^{1+}} + \|V\|_{H^{1+}} \right) \left( \|\delta B\|_{H^{s_0-1}} + \|\delta V\|_{H^{s_0-1}} \right) \|\bar{\nabla}\eta\|_{H^{\max(s_0,2+)}} \\
			\le& C\left(\|\eta\|_{H_R^{s+\frac{1}{2}}}, \|\psi\|_{H^{s}}\right) \|\delta\eta\|_{H^{s_0+\frac{1}{2}}}, \\
			&\| [T_B,T_V]\cdot\delta\left( \bar{\nabla}\eta \right) \|_{H^{s_0}} \\
			&\le C\left(\|\eta\|_{H_R^{s+\frac{1}{2}}}, \|\psi\|_{H^{s}}\right) \| \delta\left( \bar{\nabla}\eta \right) \|_{H^{s_0-\frac{1}{2}}} \le C\left(\|\eta\|_{H_R^{s+\frac{1}{2}}}, \|\psi\|_{H^{s}}\right) \|\delta\eta\|_{H^{s_0+\frac{1}{2}}}.
		\end{align*}
		
		As for the estimate \eqref{eq-paralin:r-2-deri-in-eta-3}, we shall apply Proposition \ref{prop-para:paraprod-bound}.
		\begin{align*}
			&\| \delta\left( T_V\cdot R(B,\bar{\nabla}\eta) \right) \|_{H^{s_0}} \\
			\le& \| T_{\delta V}\cdot R(B,\bar{\nabla}\eta) \|_{H^{s_0}} + \| T_V\cdot R(\delta B,\bar{\nabla}\eta) \|_{H^{s_0}} + \| T_V\cdot R\left(B,\delta\left(\bar{\nabla}\eta\right)\right) \|_{H^{s_0}} \\
			\lesssim& \|\delta V\|_{H^{s_0-1}} \|R(B,\bar{\nabla}\eta)\|_{H^{\max(s_0,2+)}} + \|V\|_{H^{1+}} \|R(\delta B,\bar{\nabla}\eta)\|_{H^{s_0}} \\
			& \hspace{20em} + \|V\|_{H^{1+}} \|R\left(B,\delta\left(\bar{\nabla}\eta\right)\right)\|_{H^{s_0}} \\
			\lesssim& \|\delta V\|_{H^{s_0-1}} \|B\|_{H^{s-1}} \|\bar{\nabla}\eta\|_{H^{s-\frac{1}{2}}} + \|V\|_{H^{1+}} \|\delta B\|_{H^{s_0-1}} \|\bar{\nabla}\eta\|_{H^{2+}} \\
			& \hspace{20em} + \|V\|_{H^{1+}} \|B\|_{H^{\frac{3}{2}+}} \|\delta\left(\bar{\nabla}\eta\right)\|_{H^{s_0-\frac{1}{2}}} \\
			\le& C\left(\|\eta\|_{H_R^{s+\frac{1}{2}}}, \|\psi\|_{H^{s}}\right) \|\delta\eta\|_{H^{s_0+\frac{1}{2}}}.
		\end{align*}
		
		The last equivalence \eqref{eq-paralin:r-2-deri-in-eta-4} is merely a consequence of Proposition \ref{prop-para:paraprod-bound}.
		\begin{align*}
			&\| \delta\left( R(V,V) + R(B,B) + R(B,G(\eta)\psi) \right) \|_{H^{s_0}} \\
			\le& \| R(\delta V, V) \|_{H^{s_0}} + \| R(\delta B,B) \|_{H^{s_0}} + \| R(\delta B, G(\eta)\psi) \|_{H^{s_0}} + \| R(B, \delta(G(\eta)\psi)) \|_{H^{s_0}} \\
			\lesssim& \|\delta V\|_{H^{s_0-1}} \|V\|_{H^{2+}} + \|\delta B\|_{H^{s_0-1}} \|B\|_{H^{2+}} + \|\delta B\|_{H^{s_0-1}} \|G(\eta)\psi\|_{H^{2+}} \\
			&\hspace{22em} + \|B\|_{H^{2+}} \|\delta(G(\eta)\psi)\|_{H^{s_0-1}} \\
			\le& C\left(\|\eta\|_{H_R^{s+\frac{1}{2}}}, \|\psi\|_{H^{s}}\right) \|\delta\eta\|_{H^{s_0+\frac{1}{2}}},
		\end{align*}
		where we use the fact that $G(\eta)\psi \in H^{s-1}$ and $\delta(G(\eta)\psi) \in H^{s_0-1}$ due to Proposition \ref{prop-pre:shape-deri}.
	\end{proof}

	\subsubsection{Study of $r_3$.} 
	In the end of Section \ref{subsubsect:paralin-H}, we have obtained an explicit formula \eqref{eq-paralin:r-3-formula} for $r_3$. Then it suffices to show that the derivative-in-$\eta$ of the right hand side of \eqref{eq-paralin:r-3-formula} is equivalent to zero in the sense of \eqref{eq-paralin:equi-f-2-deri-in-eta}, which can be reduced to the following lemma.
	\begin{lemma}\label{lem-paralin:r-3-deri-in-eta}
		Under the hypotheses of Lemma \ref{lem-paralin:source-deri-in-eta}, we have the following equivalences in the sense of \eqref{eq-paralin:equi-f-2-deri-in-eta},
		\begin{align}
			&\delta\left(F(\eta,\nabla_w\eta) - T_{\nabla_{x,u,v}F(\eta,\nabla_w\eta)}\cdot (\eta,\nabla_w\eta) \right) \sim 0, \label{eq-paralin:r-3-deri-in-eta-1}\\
			&\delta\left( T_{\eta_{jk}} \left( G^{jk}(\eta,\nabla_w\eta) - T_{\nabla_{x,u,v}G^{jk}(\eta,\nabla_w\eta)}\cdot (\eta,\nabla_w\eta) \right)\right) \sim 0, \label{eq-paralin:r-3-deri-in-eta-2}\\
			&\delta \left(T_{F_x(\eta,\nabla_w\eta)}\eta\right) \sim 0, \label{eq-paralin:r-3-deri-in-eta-3}\\
			&\delta \left(T_{\eta_{jk}} T_{G^{jk}_x(\eta,\nabla_w\eta)} \eta\right),\ \delta \left(R\left( \eta_{jk},G^{jk}(\eta,\nabla_w\eta) \right)\right) \sim 0, \label{eq-paralin:r-3-deri-in-eta-4}\\
			&\delta \left(\left(T_{\eta_{jk}} T_{\left( \nabla_{u,v}G^{jk} \right)(\eta,\nabla_w\eta)} - T_{\eta_{jk}\left( \nabla_{u,v}G^{jk} \right)(\eta,\nabla_w\eta)}\right)\cdot \nabla_w\eta\right) \sim 0. \label{eq-paralin:r-3-deri-in-eta-5}
		\end{align}
		Recall that $F,G^{jk}$ are smooth functions defined in \eqref{eq-paralin:def-FG-paralin-H}.
	\end{lemma}
	\begin{proof}[Proof of \eqref{eq-paralin:r-3-deri-in-eta-3}-\eqref{eq-paralin:r-3-deri-in-eta-5}]
		The proof of these equivalences are similar to the study of $r_2$. Due to Lemma \ref{lem-paralin:fct-of-eta-deri-in-eta}, the composition of smooth functions with $(\eta,\nabla_w\eta)$ belongs to $H^{s-\frac{1}{2}}$ and their derivative in $\eta$ lie in $H^{s_0-\frac{1}{2}}$. Then \eqref{eq-paralin:r-3-deri-in-eta-3} follows from
		\begin{align*}
			\| \delta \left(T_{F_x(\eta,\nabla_w\eta)}\eta\right) \|_{H^{s_0}} \le& \|T_{\delta\left(F_x(\eta,\nabla_w\eta)\right)}\eta\|_{H^{s_0}} + \|T_{F_x(\eta,\nabla_w\eta)}\delta \eta\|_{H^{s_0}} \\
			\lesssim& \|\delta\left(F_x(\eta,\nabla_w\eta)\right)\|_{H^{1+}} \|\eta\|_{H^{s_0}} + \|F_x(\eta,\nabla_w\eta)\|_{H^{1+}} \|\delta\eta\|_{H^{s_0}} \\
			\le& C\left(\|\eta\|_{H_R^{s+\frac{1}{2}}}\right) \|\delta\eta\|_{H^{s_0+\frac{1}{2}}},
		\end{align*}
		while \eqref{eq-paralin:r-3-deri-in-eta-4} is a consequence of 
		\begin{align*}
			&\| \delta \left(T_{\eta_{jk}} T_{G^{jk}_x(\eta,\nabla_w\eta)} \eta\right) \|_{H^{s_0}} + \| \delta \left(R\left( \eta_{jk},G^{jk}(\eta,\nabla_w\eta) \right)\right) \|_{H^{s_0}} \\
			\le& \| T_{\delta\eta_{jk}} T_{G^{jk}_x(\eta,\nabla_w\eta)} \eta \|_{H^{s_0}} + \| T_{\eta_{jk}} T_{\delta\left(G^{jk}_x(\eta,\nabla_w\eta)\right)} \eta \|_{H^{s_0}} + \| T_{\eta_{jk}} T_{G^{jk}_x(\eta,\nabla_w\eta)} \delta\eta \|_{H^{s_0}} \\
			& + \| R\left( \delta\eta_{jk},G^{jk}(\eta,\nabla_w\eta) \right) \|_{H^{s_0}} + \| R\left( \eta_{jk},\delta\left(G^{jk}(\eta,\nabla_w\eta)\right) \right) \|_{H^{s_0}} \\
			\lesssim& \|\delta\eta_{jk}\|_{H^{s_0-\frac{3}{2}}} \|G^{jk}_x(\eta,\nabla_w\eta)\|_{H^{1+}} \|\eta\|_{H^{\max(\frac{5}{2}+,s_0)}} + \|\eta_{jk}\|_{H^{1+}} \|\delta\left(G^{jk}_x(\eta,\nabla_w\eta)\right)\|_{H^{1+}} \|\eta\|_{H^{s_0}} \\
			&+ \|\eta_{jk}\|_{H^{1+}} \|G^{jk}_x(\eta,\nabla_w\eta)\|_{H^{1+}} \|\delta\eta\|_{H^{s_0}} + \|\delta\eta_{jk}\|_{H^{s_0-\frac{3}{2}}} \|G^{jk}_x(\eta,\nabla_w\eta)\|_{H^{\frac{5}{2}+}} \\
			&+ \|\eta_{jk}\|_{H^{\frac{3}{2}+}} \|\delta\left(G^{jk}_x(\eta,\nabla_w\eta)\right)\|_{H^{s_0-\frac{1}{2}}} \\
			\le& C\left(\|\eta\|_{H_R^{s+\frac{1}{2}}}\right) \|\delta\eta\|_{H^{s_0+\frac{1}{2}}}.
		\end{align*}
		
		As for \eqref{eq-paralin:r-3-deri-in-eta-5}, when the derivative $\delta$ acts on $\nabla_w\eta$, by Proposition \ref{prop-para:paradiff-cal-sym}, we have
		\begin{align*}
			&\|\left(T_{\eta_{jk}} T_{\left( \nabla_{u,v}G^{jk} \right)(\eta,\nabla_w\eta)} - T_{\eta_{jk}\left( \nabla_{u,v}G^{jk} \right)(\eta,\nabla_w\eta)}\right)\cdot \nabla_w\delta\eta\|_{H^{s_0}} \\
			\le& C\left(\|\eta\|_{H_R^{s+\frac{1}{2}}}\right) \|\delta\eta\|_{s_0-\frac{1}{2}} \le C\left(\|\eta\|_{H_R^{s+\frac{1}{2}}}\right) \|\delta\eta\|_{H^{s_0+\frac{1}{2}}},
		\end{align*}
		since the symbols $\eta_{jk}$ and $\left( \nabla_{u,v}G^{jk} \right)(\eta,\nabla_w\eta)$ belong to $H^{s-\frac{3}{2}} \subset \Gamma^0_{3/2+}$. When $\delta$ acts on the symbols, each terms are equivalent to zero,
		\begin{align*}
			\| T_{\delta\eta_{jk}} T_{\left( \nabla_{u,v}G^{jk} \right)(\eta,\nabla_w\eta)} \cdot\nabla_w\eta \|_{H^{s_0}} \lesssim& \|\delta\eta_{jk}\|_{H^{s_0-\frac{3}{2}}} \|\left( \nabla_{u,v}G^{jk} \right)(\eta,\nabla_w\eta)\|_{H^{1+}} \|\nabla_w\eta\|_{H^{\max(\frac{5}{2}+,s_0)}}, \\
			\| T_{\eta_{jk}} T_{\delta\left(\left( \nabla_{u,v}G^{jk} \right)(\eta,\nabla_w\eta)\right)} \cdot\nabla_w\eta \|_{H^{s_0}} \lesssim& \|\eta_{jk}\|_{H^{1+}} \|\delta\left(\left( \nabla_{u,v}G^{jk} \right)(\eta,\nabla_w\eta)\right)\|_{H^{s_0-\frac{1}{2}}} \|\nabla_w\eta\|_{H^{s_0}}, \\
			\| T_{\delta\eta_{jk}\left( \nabla_{u,v}G^{jk} \right)(\eta,\nabla_w\eta)}\cdot \nabla_w \eta \|_{H^{s_0}} \lesssim& \|\delta\eta_{jk}\left( \nabla_{u,v}G^{jk} \right)(\eta,\nabla_w\eta)\|_{H^{s_0-\frac{3}{2}}} \|\nabla_w \eta\|_{H^{\max(\frac{5}{2}+,s_0)}} \\
			\le& \|\delta\eta_{jk}\|_{H^{s_0-\frac{3}{2}}} \|\left( \nabla_{u,v}G^{jk} \right)(\eta,\nabla_w\eta)\|_{H^{s-1}} \|\nabla_w \eta\|_{H^{\max(\frac{5}{2}+,s_0)}}, \\
			\| T_{\eta_{jk}\delta\left(\left( \nabla_{u,v}G^{jk} \right)(\eta,\nabla_w\eta)\right)}\cdot \nabla_w \eta \|_{H^{s_0}} \lesssim& \|\eta_{jk}\delta\left(\left( \nabla_{u,v}G^{jk} \right)(\eta,\nabla_w\eta)\right)\|_{H^{s_0-\frac{1}{2}}} \|\nabla_w \eta\|_{H^{s_0}} \\
			\le& \|\eta_{jk}\|_{H^{s-\frac{3}{2}}} \|\delta\left(\left( \nabla_{u,v}G^{jk} \right)(\eta,\nabla_w\eta)\right)\|_{H^{s_0-\frac{1}{2}}} \|\nabla_w \eta\|_{H^{s_0}},
		\end{align*}
		the right hand side of which are all bounded by $C\left(\|\eta\|_{H_R^{s+\frac{1}{2}}}\right) \|\delta\eta\|_{H^{s_0+\frac{1}{2}}}$.
	\end{proof}
	
	To prove \eqref{eq-paralin:r-3-deri-in-eta-1} and \eqref{eq-paralin:r-3-deri-in-eta-2}, we need the following lemma,
	\begin{lemma}\label{lem-paralin:paralin-deri-in-eta}
		Let $\eta\in H^{s+\frac{1}{2}}_R$ and $\delta\eta\in H^{s_0}$ with $s>3$ and $\frac{3}{2}<s_0< s-\frac{3}{2}$. For all smooth function $F = F(x,u,v)$, we have
		\begin{equation}\label{eq-paralin:paralin-deri-in-eta}
			\|\delta\left( F(\eta,\nabla_w\eta) - T_{\left( \nabla_{x,u,v}F \right)(\eta,\nabla_w\eta)}\cdot(\eta,\nabla_w\eta) \right)\|_{H^{s_0}} \le C\left(\|\eta\|_{H_R^{s+\frac{1}{2}}}\right) \|\delta\eta\|_{H^{s_0+\frac{1}{2}}}.
		\end{equation}
		Recall that $\delta$ stands for the derivation in $\eta$.
	\end{lemma}
	\begin{proof}
		The function on the left hand side of \eqref{eq-paralin:paralin-deri-in-eta} reads
		\begin{equation*}
			\left( \nabla_{x,u,v}F \right)(\eta,\nabla_w\eta)\cdot(\delta\eta,\nabla_w\delta\eta) - T_{\delta\left(\left( \nabla_{x,u,v}F \right)(\eta,\nabla_w\eta)\right)}\cdot(\eta,\nabla_w\eta) - T_{\left( \nabla_{x,u,v}F \right)(\eta,\nabla_w\eta)}\cdot(\delta\eta,\nabla_w\delta\eta).
		\end{equation*}
		Recall that Lemma \ref{lem-paralin:fct-of-eta-deri-in-eta} ensures that $\delta\left( \left(\nabla_{x,u,v}F\right)(\eta,\nabla_w\eta) \right) \in H^{s_0-\frac{1}{2}}$. Then Proposition \ref{prop-para:paraprod-bound} implies
		\begin{align*}
			\|T_{\delta\left(\left( \nabla_{x,u,v}F \right)(\eta,\nabla_w\eta)\right)}\cdot(\eta,\nabla_w\eta)\|_{H^{s_0}} \lesssim& \|\delta\left(\left( \nabla_{x,u,v}F \right)(\eta,\nabla_w\eta)\right)\|_{H^{1+}} \|(\eta,\nabla_w\eta)\|_{H^{s_0}} \\
			\le& C\left(\|\eta\|_{H_R^{s+\frac{1}{2}}}\right) \|\delta\eta\|_{H^{s_0+\frac{1}{2}}}.
		\end{align*}
		The remaining term
		\begin{align*}
			&\left( \nabla_{x,u,v}F \right)(\eta,\nabla_w\eta)\cdot(\delta\eta,\nabla_w\delta\eta) - T_{\left( \nabla_{x,u,v}F \right)(\eta,\nabla_w\eta)}\cdot(\delta\eta,\nabla_w\delta\eta) \\
			=&T_{(\delta\eta,\nabla_w\delta\eta)} \cdot \left( \nabla_{x,u,v}F \right)(\eta,\nabla_w\eta) + R\left( \left( \nabla_{x,u,v}F \right)(\eta,\nabla_w\eta), (\delta\eta,\nabla_w\delta\eta) \right),
		\end{align*}
		which, due to Proposition \ref{prop-para:paraprod-bound}, can be bounded by
		\begin{align*}
			\|T_{(\delta\eta,\nabla_w\delta\eta)} \cdot \left( \nabla_{x,u,v}F \right)(\eta,\nabla_w\eta)\|_{H^{s_0}} \lesssim& \|(\delta\eta,\nabla_w\delta\eta)\|_{H^{s_0-\frac{1}{2}}} \|\left( \nabla_{x,u,v}F \right)(\eta,\nabla_w\eta)\|_{H^{s_0}} \\
			\le& C\left(\|\eta\|_{H_R^{s+\frac{1}{2}}}\right) \|\delta\eta\|_{H^{s_0+\frac{1}{2}}}, \\
			\|R\left( \left( \nabla_{x,u,v}F \right)(\eta,\nabla_w\eta), (\delta\eta,\nabla_w\delta\eta) \right)\|_{H^{s_0}} \lesssim& \|(\delta\eta,\nabla_w\delta\eta)\|_{H^{s_0-\frac{1}{2}}} \|\left( \nabla_{x,u,v}F \right)(\eta,\nabla_w\eta)\|_{H^{\frac{3}{2}+}} \\
			\le& C\left(\|\eta\|_{H_R^{s+\frac{1}{2}}}\right) \|\delta\eta\|_{H^{s_0+\frac{1}{2}}}.
		\end{align*}
	\end{proof}
	
	\begin{proof}[Proof of \eqref{eq-paralin:r-3-deri-in-eta-1} and \eqref{eq-paralin:r-3-deri-in-eta-2}]
		\eqref{eq-paralin:r-3-deri-in-eta-1} is a direct consequence of Lemma \ref{lem-paralin:paralin-deri-in-eta}. For \eqref{eq-paralin:r-3-deri-in-eta-2}, one may rewrite its left hand side as
		\begin{align*}
			&T_{\delta\eta_{jk}} \left( G^{jk}(\eta,\nabla_w\eta) - T_{\nabla_{x,u,v}G^{jk}(\eta,\nabla_w\eta)}\cdot (\eta,\nabla_w\eta) \right) \\
			&+ T_{\eta_{jk}}  \delta\left(\left( G^{jk}(\eta,\nabla_w\eta) - T_{\nabla_{x,u,v}G^{jk}(\eta,\nabla_w\eta)}\cdot (\eta,\nabla_w\eta) \right)\right),
		\end{align*}
		where the second term is equivalent to zero due to Lemma \ref{lem-paralin:paralin-deri-in-eta} and Proposition \ref{prop-para:paraprod-bound}, while the first term can be estimated via \eqref{eq-para:paraprod-bound-main} and \eqref{eq-para:paralin},
		\begin{align*}
			&\| T_{\delta\eta_{jk}} \left( G^{jk}(\eta,\nabla_w\eta) - T_{\nabla_{x,u,v}G^{jk}(\eta,\nabla_w\eta)}\cdot (\eta,\nabla_w\eta) \right) \|_{H^{s_0}} \\
			\lesssim& \|\delta\eta_{jk}\|_{H^{s_0-\frac{3}{2}}} \|G^{jk}(\eta,\nabla_w\eta) - T_{\nabla_{x,u,v}G^{jk}(\eta,\nabla_w\eta)}\cdot (\eta,\nabla_w\eta)\|_{H^{\max(\frac{5}{2}+,s_0)}} \\
			\le& C\left(\|\eta\|_{H_R^{s+\frac{1}{2}}}\right) \|\delta\eta\|_{H^{s_0+\frac{1}{2}}}.
		\end{align*}
	\end{proof}

	\section{Symmetrization of the system}\label{Sect:sym}
	
	In previous section, the main system \eqref{eq-intro:WW} has been reformulated into paralinear form \eqref{eq-paralin:WW}, which we recall here
	\begin{equation*}
		\partial_t 
		\left(\begin{array}{c}
			\eta \\
			\psi
		\end{array}\right) + \left(T_V\cdot\bar{\nabla} + 
		\mathcal{L}\right)\left(\begin{array}{c}
			\eta \\
			\psi
		\end{array}\right) = f,
	\end{equation*}
	where 
	\begin{equation*}
	\mathcal{L} = \left(\begin{array}{cc}
		I & 0 \\
		T_B & I
	\end{array}\right)
	\left(\begin{array}{cc}
		0 & -T_\lambda \\
		\sigma T_\mu & 0
	\end{array}\right)
	\left(\begin{array}{cc}
		I & 0 \\
		-T_B & I
	\end{array}\right).
	\end{equation*}
	The goal of this section is to check that this system is symmetrizable. Namely, the operator $T_V\cdot\bar{\nabla} + 
	\mathcal{L}$ will be shown to be anti-self-adjoint, up to remainders with proper regularity. Under the hypotheses of Theorem \ref{thm-intro:main}, $(\eta,\psi)\in H^{s+\frac{1}{2}}_R \times H^s$ and thus the vector $V$ and $\eta$ has Lipschitz regularity, which guarantees that $T_V\cdot\bar{\nabla}$ is anti-self-adjoint up to a remainder of order $0$ (see Lemma \ref{lem-cauchy:energy-esti-1}). The main difficulty is to prove that $\mathcal{L}$ is symmetrizable. In fact, the matrix of principal symbols associated to $\mathcal{L}$ reads
	\begin{equation*}
		\left(\begin{array}{cc}
			0 & -\lambda^{(1)} \\
			\sigma \mu^{(2)} & 0
		\end{array}\right),
	\end{equation*}
	where $\lambda^{(1)},\mu^{(2)}$ are real and elliptic with $\sigma>0$. When $\eta$ (the coefficients of $\lambda,\mu$) is regular enough, a simple calculation of matrices gives the desired symmetrization
	\begin{equation*}
		\left(\begin{array}{cc}
			\sqrt{\sigma\mu^{(2)}/\lambda^{(1)}} & 0 \\
			0 & 1
		\end{array}\right)^{-1}
		\left(\begin{array}{cc}
			0 & -\sqrt{\sigma\lambda^{(1)}\mu^{(2)}} \\
			\sqrt{\sigma\lambda^{(1)}\mu^{(2)}} & 0
		\end{array}\right)
		\left(\begin{array}{cc}
			\sqrt{\sigma\mu^{(2)}/\lambda^{(1)}} & 0 \\
			0 & 1
		\end{array}\right).
	\end{equation*}
	This heuristic calculation indicates us to try to construct symmetrizer $S$ such that
	\begin{equation*}
		S \left(\begin{array}{cc}
			0 & -T_\lambda \\
			\sigma T_\mu & 0
		\end{array}\right) S^{-1} 
	\end{equation*}
	is anti-self-adjoint up to remainders. Meanwhile, one could also see the proper order of symbols, which is stated in detail in Proposition \ref{prop-sym:main}. To rigorously construct the symmetrizer $S$, we shall repeat the method in \cite{alazard2011water} with the symbolic calculus introduced in Section \ref{subsect:paralin-pre} and Appendix \ref{subsect:paradiff}.
	
	As in previous section, for $a,b\in\Sigma^m$, we say that $T_a$ and $T_b$ are equivalent and write $T_a\approx T_b$ if their difference is of order $m-3/2-$ with operator norm bounded by $C\left( \|\eta\|_{H^{\frac{7}{2}+}} \right)$. Besides, for any symbol $a\in\Sigma^m \subset \Gamma^{m}_{3/2+} + \Gamma^{m-1}_{1/2+}$, we shall also use the notation $a^{(m)}\in \Gamma^{m}_{3/2+}$ and $a^{(m-1)}\in \Gamma^{m-1}_{1/2+}$ to present its principal and subprincipal component respectively (note that this decomposition is unique if we further assume that $a^{(m)}$ is homogeneous in $\xi$ of order $m$).
	
	\begin{proposition}\label{prop-sym:main}
		Let $(\eta,\psi)\in H_R^{s+\frac{1}{2}}\times H^s$ with $s>3$. There exist elliptic symbols $p\in\Sigma^{\frac{1}{2}}$ and $q\in\Sigma^0$ such that
		\begin{equation}\label{eq-sym:def-S}
			S:= \left(\begin{array}{cc}
				T_p & 0 \\
				0 & T_q
			\end{array}\right)
		\end{equation}\index{S@$S$ Symmtrizer of the system}
		satisfies, for some elliptic $\gamma\in\Sigma^{\frac{3}{2}}$,
		\begin{equation}\label{eq-sym:main}
			S\left(\begin{array}{cc}
				0 & -T_\lambda \\
				\sigma T_\mu & 0
			\end{array}\right)
			= \left(\begin{array}{cc}
				0 & -T_p T_\lambda \\
				\sigma T_q T_\mu & 0
			\end{array}\right)
			\approx \left(\begin{array}{cc}
				0 & -T_\gamma T_q \\
				T_\gamma^* T_p & 0
			\end{array}\right)
			=\left(\begin{array}{cc}
				0 & -T_\gamma \\
				T_\gamma^* & 0
			\end{array}\right)S
		\end{equation}
		in the sense that $T_p T_\lambda \approx T_\gamma T_q$ and $\sigma T_q T_\mu \approx T_\gamma^* T_p$.
		
		Furthermore, we have,
		\begin{equation}\label{eq-sym:main-equi}
			T_p T_\lambda \approx T_\gamma T_q,\ \ T_q T_{\sigma\mu} \approx T_\gamma T_p, \ \ T_\gamma^* \approx T_\gamma,
		\end{equation}
		which is stronger than \eqref{eq-sym:main}
	\end{proposition}

	The construction of $p,q,\gamma$ from $\lambda,\mu$ is exactly the same as in Section 4.2 of \cite{alazard2011water}, which we recall here. To begin with, we investigate the conditions on $\gamma$ such that $T_\gamma^* \sim T_\gamma$. By choosing 
	\begin{equation}\label{eq-sym:cond-im-gamma-prin}
		\Imag \gamma^{(3/2)}=0,
	\end{equation}
	one may apply Propostion \ref{prop-paralin:homo-sym-cond-aa} to deduce that
	\begin{equation}\label{eq-sym:cond-im-gamma-sub}
		\Imag \gamma^{(1/2)} = -\frac{1}{2}\partial_w\cdot\partial_\xi \gamma^{(3/2)}.
	\end{equation}
	
	We shall next determine $\Real\gamma^{(3/2)}$, $\Real\gamma^{(1/2)}$ and $q$ such that
	\begin{equation}\label{eq-sym:main-alt}
		T_q T_{\sigma\mu} T_\lambda \approx T_\gamma T_\gamma T_q,\ \ \text{and }T_p T_\lambda \approx T_\gamma T_q.
	\end{equation}
	This will provide a solution to the first two equivalences in \eqref{eq-sym:main-equi}, plugging the last formula of \eqref{eq-sym:main-alt} inside the first one and using that $\lambda$ has a parametrix, since $\lambda$ is elliptic (see Proposition \ref{prop-paralin:homo-sym-ellip-inverse}). By applying symbolic calculus \eqref{eq-paralin:homo-sym-comp}, the first equivalence of \eqref{eq-sym:main-alt} follows from
	\begin{equation*}
		q\sharp\left(\sigma\mu\sharp\lambda\right) = \left( \gamma\sharp\gamma \right) \sharp q.
	\end{equation*}
	By \eqref{eq-paralin:homo-sym-comp-formula}, we can write this equality as
	\begin{equation}\label{eq-sym:main-alt-1}
		\begin{aligned}
			&q^{(0)} \left(\sigma\mu\sharp\lambda\right) + \partial_\xi q^{(0)} \cdot D_w \left(\sigma\mu^{(2)}\lambda^{(1)}\right) + q^{(-1)}\left(\sigma\mu\sharp\lambda\right)^{(2)} \\
			=&\left( \gamma\sharp\gamma \right)q^{(0)} + \partial_\xi \left(\gamma^{(3/2)}\gamma^{(3/2)}\right)\cdot D_w q^{(0)} + \left( \gamma\sharp\gamma \right)^{(2)} q^{(-1)}.
		\end{aligned}
	\end{equation}
	By comparing the principal part of both sides (terms of degree $3$ in $\xi$), we are able to determine
	\begin{equation}\label{eq-sym:cond-gamma-prin}
		\gamma^{(3/2)} = \sqrt{\sigma\mu^{(2)}\lambda^{(1)}}.
	\end{equation}
	To deal with the subprincipal terms of \eqref{eq-sym:main-alt-1}, one may set
	\begin{equation}\label{eq-sym:cond-q-sub}
		q^{(-1)} = 0
	\end{equation}
	and \eqref{eq-sym:main-alt-1} becomes
	\begin{equation}\label{eq-sym:main-alt-2}
		\begin{aligned}
			q^{(0)} \left( \sigma\mu\sharp\lambda - \gamma\sharp\gamma \right) =&  \partial_\xi \left(\gamma^{(3/2)}\gamma^{(3/2)}\right)\cdot D_w q^{(0)} - \partial_\xi q^{(0)} \cdot D_w \left(\sigma\mu^{(2)}\lambda^{(1)}\right) \\
			=& \partial_\xi \left(\sigma\mu^{(2)}\lambda^{(1)}\right)\cdot D_w q^{(0)} - \partial_\xi q^{(0)} \cdot D_w \left(\sigma\mu^{(2)}\lambda^{(1)}\right) \\
			=&\frac{1}{i}\{ \sigma\mu^{(2)}\lambda^{(1)}, q^{(0)} \},
		\end{aligned}
	\end{equation}
	where $\{\cdot,\cdot\}$ is Poisson bracket. By comparing real and imaginary part of both sides of \eqref{eq-sym:main-alt-2} and setting
	\begin{equation}\label{eq-sym:cond-im-q-prin}
		\Imag q^{(0)} = 0,
	\end{equation}
	we have
	\begin{equation}\label{eq-sym:main-alt-3}
		\Real \left( \sigma\mu\sharp\lambda - \gamma\sharp\gamma \right) = 0,\ \ q^{(0)} \Imag \left( \sigma\mu\sharp\lambda - \gamma\sharp\gamma \right) = -\{ \sigma\mu^{(2)}\lambda^{(1)}, q^{(0)} \}.
	\end{equation}
	From the first equation in \eqref{eq-sym:main-alt-3}, we are able to solve $\Real \gamma^{(1/2)}$,
	\begin{equation}\label{eq-sym:cond-re-gamma-sub}
		2\gamma^{(3/2)} \Real\gamma^{(1/2)} = \sigma \left( \mu^{(2)}\Real\lambda^{(0)} + \Real\mu^{(1)}\lambda^{(1)} \right) = \sigma\mu^{(2)}\Real\lambda^{(0)},
	\end{equation}
	where the first equality holds since $\gamma^{(3/2)}$, $\lambda^{(1)}$, and $\mu^{(2)}$ are real with \eqref{eq-sym:cond-gamma-prin}, and the second follows from $\Real\mu^{(1)}=0$, which can be seen from the second formula in \eqref{eq-paralin:mu-compute}. Then \eqref{eq-sym:cond-re-gamma-sub}, together with \eqref{eq-sym:cond-gamma-prin} and \eqref{eq-sym:cond-im-gamma-sub}, completes the definition of $\gamma$. Meanwhile, to solve the second equation in \eqref{eq-sym:main-alt-3}, we first calculate the coefficient before $q^{(0)}$ by applying \eqref{eq-paralin:lambda-0-cal}, \eqref{eq-paralin:mu-1-cal}, and \eqref{eq-sym:cond-im-gamma-sub},
	\begin{align*}
		&\Imag\left( \sigma\mu\sharp\lambda - \gamma\sharp\gamma \right) \\
		=& - \sigma \partial_\xi\mu^{(2)}\cdot\partial_w\lambda^{(1)} + \sigma \Imag\mu^{(1)} \lambda^{(1)} + \sigma \mu^{(2)} \Imag\lambda^{(0)}  \\
		& + \partial_\xi\gamma^{(3/2)} \cdot \partial_w\gamma^{(3/2)} - 2\gamma^{(3/2)} \Imag\gamma^{(1/2)} \\
		=& - \sigma \partial_\xi\mu^{(2)}\cdot\partial_w\lambda^{(1)} - \frac{\sigma}{2} \partial_w\cdot\partial_\xi\mu^{(2)} \lambda^{(1)} - \frac{\sigma}{2} \frac{\partial_w\eta}{\eta}\cdot\partial_\xi\mu^{(2)} \lambda^{(1)} - \frac{\sigma}{2} \mu^{(2)} \partial_w\cdot\partial_\xi\lambda^{(1)}  \\
		& - \frac{\sigma}{2} \mu^{(2)} \frac{\partial_w\eta}{\eta}\cdot\partial_\xi\lambda^{(1)} + \partial_\xi\gamma^{(3/2)} \cdot \partial_w\gamma^{(3/2)} + \gamma^{(3/2)} \partial_w\cdot\partial_\xi \gamma^{(3/2)} \\
		=& - \frac{\sigma}{2} \partial_\xi\mu^{(2)}\cdot\partial_w\lambda^{(1)} - \left( \frac{\sigma}{2} \partial_\xi\mu^{(2)}\cdot\partial_w\lambda^{(1)} + \frac{\sigma}{2}\partial_w\cdot\partial_\xi\mu^{(2)} \lambda^{(1)} + \frac{\sigma}{2} \mu^{(2)} \partial_w\cdot\partial_\xi\lambda^{(1)} \right) \\
		&- \left( \frac{\sigma}{2} \frac{\partial_w\eta}{\eta}\cdot\partial_\xi\mu^{(2)} \lambda^{(1)} + \frac{\sigma}{2} \mu^{(2)} \frac{\partial_w\eta}{\eta}\cdot\partial_\xi\lambda^{(1)} \right) + \left( \partial_\xi\gamma^{(3/2)} \cdot \partial_w\gamma^{(3/2)} + \gamma^{(3/2)} \partial_w\cdot\partial_\xi \gamma^{(3/2)} \right) \\
		=& - \frac{\sigma}{2} \partial_\xi\mu^{(2)}\cdot\partial_w\lambda^{(1)} + \frac{\sigma}{2} \partial_w\mu^{(2)}\cdot\partial_\xi\lambda^{(1)} - \frac{\sigma}{2}\partial_w\cdot\partial_\xi(\mu^{(2)}\lambda^{(1)}) - \frac{\sigma}{2} \frac{\partial_w\eta}{\eta}\cdot\partial_\xi\left(\mu^{(2)} \lambda^{(1)}\right) \\
		&+ \frac{1}{2}\partial_w\cdot\partial_\xi(\gamma^{(3/2)}\gamma^{(3/2)}) \\
		=& \frac{\sigma}{2}\{\lambda^{(1)},\mu^{(2)}\} - \frac{\sigma}{2} \frac{\partial_w\eta}{\eta}\cdot\partial_\xi\left(\mu^{(2)} \lambda^{(1)}\right) + \frac{1}{2}\partial_w\cdot\partial_\xi(\gamma^{(3/2)}\gamma^{(3/2)} - \sigma\mu^{(2)}\lambda^{(1)}) \\
		=& \frac{\sigma}{2}\{\lambda^{(1)},\mu^{(2)}\} - \frac{\sigma}{2} \frac{\partial_w\eta}{\eta}\cdot\partial_\xi\left(\mu^{(2)} \lambda^{(1)}\right),
	\end{align*}
	where we apply \eqref{eq-sym:cond-gamma-prin} to obtain the last equality. The equation for $q^{(0)}$ becomes
	\begin{equation*}
		\frac{1}{2}q^{(0)} \left[\{\lambda^{(1)},\mu^{(2)}\} - \frac{\partial_w\eta}{\eta}\cdot\partial_\xi\left(\mu^{(2)} \lambda^{(1)}\right)\right] = -\{ \mu^{(2)}\lambda^{(1)}, q^{(0)} \}.
	\end{equation*}
	From \eqref{eq-paralin:def-lambda} and \eqref{eq-paralin:def-mu-2}, we have $\mu^{(2)} = a^2 \left(\lambda^{(1)}\right)^2$ with
	\begin{equation}\label{eq-sym:def-a}
		a:= \frac{1}{\sqrt{2}} \left( 1+|\bar{\nabla}\eta|^2 \right)^{-\frac{3}{4}},
	\end{equation}
	which is positive and independent of $\xi$. Furthermore, we look for $q^{(0)}$ also independent of $\xi$. Then the equation for $q^{(0)}$ can be written as
	\begin{equation*}
		\frac{\partial_wq^{(0)}}{q^{(0)}} \cdot\partial_\xi\lambda^{(1)} = \frac{1}{2}\frac{\partial_w\eta}{\eta}\cdot\partial_\xi\lambda^{(1)} - \frac{1}{3}\frac{\partial_w a}{a}\cdot\partial_\xi\lambda^{(1)},
	\end{equation*}
	an obvious solution of which reads
	\begin{equation}\label{eq-sym:cond-q-subprin}
		q^{(0)} = \eta^{\frac{1}{2}}a^{-\frac{1}{3}} = 2^{\frac{1}{6}} \sqrt{\eta}\left( 1+|\bar{\nabla}\eta|^2 \right)^{\frac{1}{4}}.
	\end{equation}
	Clearly, $q = q^{(0)}$ is elliptic.
	
	Till now, we have constructed $\gamma\in\Sigma^2$ and $q\in\Sigma^0$ satisfying \eqref{eq-sym:main-alt}, according to our discussion before, it remains to construct $p\in\Sigma^{\frac{1}{2}}$ with either $T_p T_\lambda \approx T_\gamma T_q$ or $T_q T_{\sigma\mu} \approx T_\gamma T_p$. Here we focus on the former one. In fact, since $\lambda\in\Sigma^1$ is elliptic, by Proposition \ref{prop-paralin:homo-sym-ellip-inverse}, there exists $\tilde{\lambda}\in\Sigma^{-1}$ such that $T_\lambda T_{\tilde{\lambda}} \approx T_\lambda T_{\tilde{\lambda}} \approx id$. Thus, by symbolic calculus (Proposition \ref{prop-paralin:homo-sym-cal}), one has
	\begin{equation*}
		T_p T_\lambda \approx T_\gamma T_q \Leftrightarrow T_p T_\lambda T_{\tilde{\lambda}} \approx T_\gamma T_q T_{\tilde{\lambda}} \Leftrightarrow T_p \approx T_{(\gamma\sharp q)\sharp\tilde{\lambda}},
	\end{equation*}
	which gives a possible definition of $p$:
	\begin{equation}\label{eq-sym:cond-p}
		p := (\gamma\sharp q)\sharp\tilde{\lambda} \in \Sigma^{\frac{1}{2}},
	\end{equation}
	whose principal symbol equals
	\begin{equation}\label{eq-sym:cond-p-prin}
		p^{(1/2)} = \frac{\gamma^{(3/2)}q^{(0)}}{\lambda^{(1)}}.
	\end{equation}
	Hence, $p$ is elliptic. 
	
	From Proposition \ref{prop-paralin:homo-sym-ellip-inverse}, the ellipticity of $p,q$ implies that $S$ has a parametrix $\tilde{S}$, in the sense that
	\begin{equation}\label{eq-sym:inverse-S}
		S\tilde{S} \approx \tilde{S}S \approx id,
	\end{equation}
	where
	\begin{equation}\label{eq-sym:def-tilde-S}
		\tilde{S} := \left(\begin{array}{cc}
			T_{\tilde{p}} & 0 \\
			0 & T_{\tilde{q}}
		\end{array}\right)
	\end{equation}
	with $\tilde{p}\in\Sigma^{-1/2}$ and $\tilde{q}\in\Sigma^{0}$ the elliptic symbols constructed by Proposition \ref{prop-paralin:homo-sym-ellip-inverse}.

	To end this section, we introduce some commutator estimates involving $S$ (see also Section 2.4 of \cite{alazard2014cauchy} for the case of planar water-wave).
	\begin{proposition}\label{prop-sym:commu-S-partial-t}
		Let $(\eta,\psi)\in H_R^{s+\frac{1}{2}}\times H^s$ with $s>3$. For all symbol $a\in\Sigma^m$ with $m\in\R$, the commutator $[T_a,\partial_t]$ is of order $m$, with
		\begin{equation}\label{eq-sym:commu-S-partial-t}
			\|[T_a,\partial_t]\|_{\mathcal{L}(H^r;H^{r-m})} \le C\left( \|\eta\|_{H^{s+\frac{1}{2}-}_R} \right),\ \ \forall r\in\R.
		\end{equation}
		In particular, the commutator between $S$ and $\partial_t$ is equal to
		\begin{equation*}
			\left[ S,\partial_t \right] = -\left(\begin{array}{cc}
				T_{\partial_tp} & 0 \\
				0 & T_{\partial_t q}
			\end{array}\right),
		\end{equation*}
		where $T_{\partial_t p}$ is of order $\frac{1}{2}$ and $T_{\partial_t q}$ is of order $0$.
	\end{proposition}
	\begin{proof}
		We intend to apply Lemma \ref{lem-paralin:sym-deri-in-eta}, which gives the invariance of order of $T_a$ under derivative in $\eta$ in a weaker space $H^{s_0+\frac{1}{2}}$ with $\frac{3}{2}<s_0\le s$. By Leibniz rule, it suffices to check that $\eta_t$ belongs to this space for some $s_0 \in ]\frac{3}{2},s]$, which is a consequence of the first equation of \eqref{eq-intro:WW} together with Corollary \ref{cor-pre:bound-DtN} (with $s_0=s$), namely
		\begin{equation*}
			\eta_t = G(\eta)\psi \in H^{s-1}.
		\end{equation*}
	\end{proof}

	\begin{proposition}\label{prop-sym:commu-S-T-V-nabla}
		Let $(\eta,\psi)\in H_R^{s+\frac{1}{2}}\times H^s$ with $s>3$. Let $a = a^{(m)} + a^{(m-1)}$ be a symbol in $\Gamma^m_{3/2+} + \Gamma^{m-1}_{1/2+}$ with $m\in\R$ and
		\begin{equation*}
			\| \partial_\xi^\alpha a^{(m)}(\xi) \|_{C^{\frac{3}{2}+}} \langle\xi\rangle^{-m+|\alpha|} + \| \partial_\xi^\alpha a^{(m-1)}(\xi) \|_{C^{\frac{1}{2}+}} \langle\xi\rangle^{-m-1+|\alpha|} \lesssim C\left( \|\eta\|_{H^{s+\frac{1}{2}-}_R} \right),
		\end{equation*}
		for all $\alpha\in\N^2$ and $\xi\in\R^2$. Then we have, for all $r\le \min(s+m+\frac{1}{2},s+\frac{3}{2}-)$,
		\begin{equation}\label{eq-sym:commu-S-T-V-nabla}
			\left\|\left[T_a, T_V\cdot\bar{\nabla}\right]\right\|_{\mathcal{L}(H^{r};H^{r-m})} \le C\left( \|\eta\|_{H^{s+\frac{1}{2}-}_R}, \|\psi\|_{H^{s-}} \right).
		\end{equation}
		In particular, we have, for all $r\le s+\frac{1}{2}$,
		\begin{align}
			&\left\|\left[T_p, T_V\cdot\bar{\nabla}\right]\right\|_{\mathcal{L}(H^{r};H^{r-\frac{1}{2}})} \le C\left( \|\eta\|_{H^{s+\frac{1}{2}-}_R}, \|\psi\|_{H^{s-}} \right),\label{eq-sym:commu-T-p-T-V-nabla}\\
			&\left\|\left[T_q, T_V\cdot\bar{\nabla}\right]\right\|_{\mathcal{L}(H^{r};H^{r})} \le C\left( \|\eta\|_{H^{s+\frac{1}{2}-}_R}, \|\psi\|_{H^{s-}} \right).\label{eq-sym:commu-T-q-T-V-nabla}
		\end{align}
	\end{proposition}
	\begin{proof}
		Let $u$ be any function in $H^{r}$ with $r\le s+m+\frac{1}{2}$. One may write
		\begin{align*}
			\left[T_a, T_V\cdot\bar{\nabla}\right]u =& \left[T_a, T_{V^z i\xi_z}\right]u + T_a T_{V^\theta} \left( T_{\eta^{-1}}\partial_\theta u + T_{\partial_\theta u}\eta^{-1} + R(\eta^{-1},\partial_\theta u) \right) \\
			&- T_{V^\theta} \left( T_{\eta^{-1}}\partial_\theta (T_au) + T_{\partial_\theta (T_au)}\eta^{-1} + R(\eta^{-1},\partial_\theta (T_au)) \right) \\
			=& \left[T_a, T_{V^z i\xi_z}\right]u + [T_a, T_{V^\theta}T_{\eta^{-1}}]\partial_\theta u + T_a T_{V^\theta} \left( T_{\partial_\theta u}\eta^{-1} + R(\eta^{-1},\partial_\theta u) \right) \\
			&- T_{V^\theta} \left( T_{\eta^{-1}}T_{\partial_\theta a}u + T_{\partial_\theta (T_au)}\eta^{-1} + R(\eta^{-1},\partial_\theta (T_au)) \right).
		\end{align*}
		We shall check that each term on the right hand side lies in $H^{r-m}$. The first term $\left[T_a, T_{V^z i\xi_z}\right]u \in H^{r-m}$ is no more than a consequence of commutator estimate (Corollary \ref{cor-para:commu-esti}). For the second term $[T_a, T_{V^\theta}T_{\eta^{-1}}]\partial_\theta u$, one may replace $T_{V^\theta}T_{\eta^{-1}}$ by $T_{\eta^{-1}V^\theta}$, since their difference is of order $-1-$ leading to an error term in $H^{r-m+}$, and apply again Corollary \ref{cor-para:commu-esti}. The third term can be estimated via Proposition \ref{prop-para:paraprod-bound},
		\begin{align*}
			&\| T_a T_{V^\theta} \left( T_{\partial_\theta u}\eta^{-1} + R(\eta^{-1},\partial_\theta u) \right) \|_{H^{r-m}} \\
			\le& C\left( \|\eta\|_{H^{s+\frac{1}{2}-}_R}, \|\psi\|_{H^{s-}} \right) \| T_{\partial_\theta u}\eta^{-1} + R(\eta^{-1},\partial_\theta u) \|_{H^{r}} \\
			\le& C\left( \|\eta\|_{H^{s+\frac{1}{2}-}_R}, \|\psi\|_{H^{s-}} \right) \|u\|_{H^{r}} \|\eta^{-1}\|_{H^{s+\frac{1}{2}-}_{R^{-1}}} \le C\left( \|\eta\|_{H^{s+\frac{1}{2}-}_R}, \|\psi\|_{H^{s-}} \right) \|u\|_{H^{r}},
		\end{align*}
		which is bounded due to Proposition \ref{prop-para:paralin}. To deal with the last term, we observe that $T_au\in H^{r-m}$ ,and that $T_{\partial_\theta a}$ has the same order as $T_a$ due to Proposition \ref{prop-para:paradiff-symbol-class-sobo} and the fact that $\eta_\theta\in H^{s-\frac{1}{2}}$. Then, again by Proposition \ref{prop-para:paraprod-bound} and \ref{prop-para:paralin}, we have
		\begin{align*}
			&\| T_{V^\theta} \left( T_{\eta^{-1}}T_{\partial_\theta a}u + T_{\partial_\theta (T_au)}\eta^{-1} + R(\eta^{-1},\partial_\theta (T_au)) \right) \|_{H^{r-m}} \\
			\le& C\left( \|\eta\|_{H^{s+\frac{1}{2}-}_R}, \|\psi\|_{H^{s-}} \right) \| T_{\eta^{-1}}T_{\partial_\theta a}u + T_{\partial_\theta (T_au)}\eta^{-1} + R(\eta^{-1},\partial_\theta (T_au)) \|_{H^{r-m}} \\
			\le& C\left( \|\eta\|_{H^{s+\frac{1}{2}-}_R}, \|\psi\|_{H^{s-}} \right) \left( \|u\|_{H^{r}} + \|\partial_\theta (T_au)\|_{H^{r-m-1}} \|\eta^{-1}\|_{H^{s+\frac{1}{2}-}_{R^{-1}}} \right) \\
			\le& C\left( \|\eta\|_{H^{s+\frac{1}{2}-}_R}, \|\psi\|_{H^{s-}} \right) \|u\|_{H^{r}}.
		\end{align*}
		Note that in the second inequality, we need the assumption $r\le \min(s+m+\frac{1}{2},s+\frac{3}{2}-)$ to apply Proposition \ref{prop-para:paraprod-bound}.
	\end{proof}

	\section{Cauchy problem}\label{Sect:cauchy}
	
	In this section, we shall prove Theorem \ref{thm-intro:main} and \ref{thm-intro:flow-map}. Recall that in Section \ref{Sect:paralin}, we have managed to reformulate the equation \eqref{eq-intro:WW} as \eqref{eq-paralin:WW},
	\begin{equation*}
		\left(\partial_t + T_V\cdot\bar{\nabla}\right)
		\left(\begin{array}{c}
			\eta \\
			\psi
		\end{array}\right) + 
		\mathcal{L}\left(\begin{array}{c}
			\eta \\
			\psi
		\end{array}\right) = f,
	\end{equation*}
	where the linear operator $\mathcal{L}$ admits a symmetrizer $S$ defined by \eqref{eq-sym:def-S} (see Proposition \ref{prop-sym:main}). Moreover, provided that $(\eta,\psi)\in H_R^{s+\frac{1}{2}}\times H^s$ with $s>3$, $f =f(\eta,\psi)$ belongs to $H^{s+\frac{1}{2}}\times H^s$ and has local Lipschitz regularity from $H_R^{s_0+\frac{1}{2}}\times H^{s_0}$ to $H^{s_0+\frac{1}{2}}\times H^{s_0}$, for all $\frac{3}{2}<s_0<s-\frac{3}{2}$ (see Proposition \ref{prop-paralin:source-Lip}).
	
	Based on the definition \eqref{eq-paralin:def-L} of $\mathcal{L}$, we shall introduced a mollified version $\mathcal{L}^\epsilon$ by inserting proper smoothing operators (see \eqref{eq-cauchy:def-L-eps}). The resulting approximate system \eqref{eq-cauchy:WW-eps}, which is nothing else than an ODE, has a unique solution $(\eta^\epsilon,\psi^\epsilon)$ on the time interval $[0,T_\epsilon[$ due to Cauchy-Lipschitz Theorem. Energy estimates (Proposition \ref{prop-cauchy:energy-esti}) guarantee that the lifespan $T_\epsilon$ admits a uniform lower bound $T$ (see Corollary \ref{cor-cauchy:low-bd-lifespan}), and the approximate solution $(\eta^\epsilon,\psi^\epsilon)$ is bounded in $H_R^{s+\frac{1}{2}}\times H^{s}$ with $s>3$, and converges on $[0,T]$ in a weaker norm $H_R^{s_0+\frac{1}{2}}\times H^{s_0}$ where $\frac{3}{2}<s_0<s-\frac{3}{2}$. Consequently, $(\eta^\epsilon,\psi^\epsilon)$ tends to a solution $(\eta,\psi)$ to \eqref{eq-intro:WW} and the uniqueness follows from a similar argument used in the convergence of approximate solution. Note that, by classical interpolation argument, the limit of approximate solution has only $L^\infty$ regularity in time, which can be optimized to $C^0$ via nonlinear interpolation introduced in \cite{alazard2024nonlinear}, which completes the proof of Theorem \ref{thm-intro:main}. As a by-product, we can also deduce the continuity of flow map, i.e. Theorem \ref{thm-intro:flow-map}.
	
	Before starting the proof, we introduce some notations and conventions used in this section,
	\begin{definition}\label{def-cauchy:fct-space}
		For $s\in\R$, $R>0$, and $T>0$, $\mathcal{X}^s_R$ denotes the collection of all $(\eta,\psi)$ such that $\psi\in H^{s}$ and $\eta-R \in H^{s+\frac{1}{2}}$, where we further assume that $\eta$ satisfy hypotheses \eqref{hyp-intro:bounds} and \eqref{hyp-intro:perturb} (or \eqref{hyp-intro:period}). For all $(\eta_1,\psi_1),(\eta_2,\psi_2)\in\mathcal{X}^s$, we denote
		\begin{equation}\label{eq-cauchy:norm-dist}
			d_{\mathcal{X}^s_R}\left( (\eta_1,\psi_1),(\eta_2,\psi_2) \right) := \|\eta_1-\eta_2\|_{H^{s+\frac{1}{2}}} + \|\psi_1-\psi_2\|_{H^{s}}.
		\end{equation}
		It is clear that $(\mathcal{X}^s_R,d_{\mathcal{X}^s_R})$ is a complete metric space. Note that in periodic case (i.e. with hypothesis \eqref{hyp-intro:period}), the normalization $R$ should be omitted and $(\mathcal{X}^s,d_{\mathcal{X}^s})$ becomes a Banach space. Moreover, if $(\eta,\psi)$ depends on time $t\in[0,T[$, we denote by $L^\infty_T\mathcal{X}^s_R$ the space of $L^\infty$-functions with value in $\mathcal{X}^s_R$ such that
		\begin{equation*}
			\sup_{t\in[0,T[} \|(\eta(t)-R,\psi(t))\|_{H^{s+\frac{1}{2}}\times H^s} < +\infty.
		\end{equation*}
		Furthermore, we denote by $C_T\mathcal{X}^s$ the subspace of $L^\infty_T\mathcal{X}^s$ where $(\eta(t),\psi(t))$ is continuous in time w.r.t. the distance \eqref{eq-cauchy:norm-dist}.
	\end{definition}\index{X@$\mathcal{X}^s_R$ Data space} \index{L@$L^\infty_T\mathcal{X}^s_R$ Solution space (bounded in time)} \index{C@$C_T\mathcal{X}^s_R$ Solution space (continuous in time)}
	
	In the sequel, we focus on the perturbative case (with hypothesis \eqref{hyp-intro:perturb}) and the proof for periodic case can be obtained simply by deleting all the normalization.
	
	As in previous sections, for linear operators $A,B$ of order $m\in\R$, we write
	\begin{equation}\label{eq-cauchy:equi-paradiff}
		A\approx B \Leftrightarrow
		\begin{aligned}
			& A-B\text{ is of order }m-\frac{3}{2}-, \\
			&\text{with operator norm bounded by }C\left(\|(\eta,\psi)\|_{H^{s+\frac{1}{2}}_R\times H^{s}}\right).
		\end{aligned}
	\end{equation}

	\subsection{Construction of approximate solutions}\label{subsect:appro}
	
	Recall that $\mathcal{L}$ is defined by \eqref{eq-paralin:def-L},
	\begin{equation*}
		\mathcal{L} := \left(\begin{array}{cc}
			I & 0 \\
			T_B & I
		\end{array}\right)
		\left(\begin{array}{cc}
			0 & -T_\lambda \\
			\sigma T_\mu & 0
		\end{array}\right)
		\left(\begin{array}{cc}
			I & 0 \\
			-T_B & I
		\end{array}\right).
	\end{equation*}
	As in \cite{alazard2011water}, we introduce the mollifier $J_\epsilon := T_{j_\epsilon}$ ($0\le\epsilon\ll 1$), where the symbol $j_\epsilon$ is given by
	\begin{equation}\label{eq-cauchy:def-j-eps}
		j_\epsilon = j^{(0)}_\epsilon + j^{(-1)}_\epsilon,\ \ j^{(0)}_\epsilon := \exp\left(-\epsilon\gamma^{(3/2)}\right),\ \ j^{(-1)}_\epsilon:= -\frac{i}{2}\partial_w\cdot\partial_\xi j^{(0)}_\epsilon.
	\end{equation}
	Recall that $\gamma^{(3/2)}$ is constructed in Proposition \ref{prop-sym:main}. As a consequence of Corollary \ref{cor-paralin:homo-sym-commu} and \ref{prop-paralin:homo-sym-cond-aa}, all the symbols constructed as above satisfy the following property.
	\begin{lemma}\label{lem-cauchy:symbol-fct-of-gamma}
		Given $m\in\R$, we denote by $\mathcal{G}^{m}$ the collection of real functions $F = F(\epsilon,\rho)$ which are smooth on $\{\rho>0\}$ with parameter $\epsilon\in]0,1]$, satisfying that, for all $k\in\N$ and $\rho>0$,
		\begin{equation}\label{eq-cauchy:symbol-fct-of-gamma-cond}
			\sup_{\epsilon\in]0,1]} \left| \partial^k_\rho F(\epsilon,\rho) \right| \lesssim \rho^{m-k}.
		\end{equation}
		Let $\eta\in L^\infty_TH^{s+\frac{1}{2}-}_R$ with $s>3$ and $T>0$. For all $F\in\mathcal{G}^m$, we define symbol $a_{F,\epsilon}$ as
		\begin{equation}\label{eq-cauchy:symbol-fct-of-gamma-def}
			a_{F,\epsilon} := a_{F,\epsilon}^{(3m/2)} + a_{F,\epsilon}^{(3m/2-1)},\ a_{F,\epsilon}^{(3m/2)} = F(\epsilon,\gamma^{(3/2)}),\ a_{F,\epsilon}^{(3m/2-1)} = -\frac{i}{2}\partial_w\cdot\partial_\xi a_{F,\epsilon}^{(3m/2)}.
		\end{equation}
		Then $a_{F,\epsilon}$ belongs to $\Gamma^{3m/2}_{3/2+} + \Gamma^{3m/2-1}_{1/2+}$ uniformly in $\epsilon$ and $t\in[0,T[$ with 
		\begin{equation}\label{eq-cauchy:symbol-fct-of-gamma-aa}
			T_{a_{F,\epsilon}} \approx T_{a_{F,\epsilon}}^*.
		\end{equation}
		Moreover, for any $F,G\in\mathcal{G}^m$, we have
		\begin{equation}\label{eq-cauchy:symbol-fct-of-gamma-sommu}
			T_{a_{F,\epsilon}} T_{a_{G,\epsilon}} \approx T_{a_{G,\epsilon}} T_{a_{F,\epsilon}}.
		\end{equation}
	\end{lemma}
	\begin{proof}
		The uniform-in-$(\epsilon,t)$ boundedness of $a_{F\epsilon}^{(3m/2)}$ and $a_{F\epsilon}^{(3m/2-1)}$ in $\Gamma^{3m/2}_{3/2+}$ and $\Gamma^{3m/2-1}_{1/2+}$, respectively, can be directly checked from the definition \eqref{eq-sym:cond-gamma-prin} of $\gamma^{(3/2)}$ and the condition \eqref{eq-cauchy:symbol-fct-of-gamma-cond} of $F$. Since $F$ is real-valued, $a_{F\epsilon}^{(3m/2)}$ is also real. Then an application of Proposition \ref{prop-paralin:homo-sym-cond-aa} gives \eqref{eq-cauchy:symbol-fct-of-gamma-aa}. To prove the last equivalence \eqref{eq-cauchy:symbol-fct-of-gamma-sommu}, we first reduce it to $a_{F,\epsilon}\sharp a_{G,\epsilon} = a_{G,\epsilon}\sharp a_{F,\epsilon}$, thanks to \eqref{eq-paralin:homo-sym-comp}. In fact, by formula \eqref{eq-paralin:homo-sym-comp-formula} and definition \eqref{eq-cauchy:symbol-fct-of-gamma-def}, we have
		\begin{align*}
			& a_{F,\epsilon}\sharp a_{G,\epsilon} - a_{G,\epsilon}\sharp a_{F,\epsilon} \\
			=& \partial_\xi a_{F,\epsilon}^{(3m/2)} \cdot D_w a_{G,\epsilon}^{(3m/2)} - \partial_\xi a_{G,\epsilon}^{(3m/2)} \cdot D_w a_{F,\epsilon}^{(3m/2)} \\
			=& F'(\epsilon,\gamma^{(3/2)}) G'(\epsilon,\gamma^{(3/2)}) \partial_\xi \gamma^{(3/2)}\cdot D_w\gamma^{(3/2)} - G'(\epsilon,\gamma^{(3/2)}) F'(\epsilon,\gamma^{(3/2)}) \partial_\xi \gamma^{(3/2)}\cdot D_w\gamma^{(3/2)} \\
			=&0.
		\end{align*}
	\end{proof}
	
	In particular, by choosing $F(\epsilon,\rho) = \exp(-\epsilon\rho)$ and $G(\epsilon,\rho)=\rho$, we have the following property on $j_\epsilon$ constructed in \eqref{eq-cauchy:def-j-eps},
	\begin{lemma}\label{lem-cauchy:esti-j-eps}
		Let $\eta\in L^\infty_TH^{s+\frac{1}{2}-}_R$ with $s>3$ and $T>0$. Then the symbol $j_\epsilon$ defined above is elliptic and belong to $\Gamma^0_{3/2+} + \Gamma^{-1}_{1/2+}$ uniformly in $\epsilon\ge 0$ and $t\in[0,T[$. Moreover, we have
		\begin{equation}\label{eq-cauchy:commu-j-eps}
			J_\epsilon T_\gamma \approx T_\gamma J_\epsilon,\ \ J_\epsilon^* \approx J_\epsilon,
		\end{equation}
		uniformly in $\epsilon,t$.\index{J@$J_\epsilon$ Smoothing opertaor}\index{j@$j_\epsilon$ Symbol of $J_\epsilon$}
	\end{lemma}
	Note that when $\epsilon>0$, $j_\epsilon\in \Gamma^m_{3/2+} + \Gamma^{m-1}_{1/2+}$ for all $m<0$, and as a consequence, $J_\epsilon$ serves as a smoothing operator. Nevertheless, the estimates of $j_\epsilon$ are uniform in $\epsilon\ge 0$ only for $m\ge 0$. In the sequel, we shall regard $j_\epsilon$ as an element in $\Gamma^0_{3/2+} + \Gamma^{-1}_{1/2+}$ in symbolic calculus.
	
	Now, we define the mollification of $\mathcal{L}$ as
	\begin{equation}\label{eq-cauchy:def-L-eps}
		\mathcal{L}^\epsilon := \left(\begin{array}{cc}
			I & 0 \\
			T_B & I
		\end{array}\right)
		\left(\begin{array}{cc}
			0 & -T_\lambda \\
			\sigma T_\mu & 0
		\end{array}\right) \tilde{S}J_\epsilon S
		\left(\begin{array}{cc}
			I & 0 \\
			-T_B & I
		\end{array}\right),
	\end{equation}\index{L@$\mathcal{L}^\epsilon$ Approximate principal operator}
	where $S$ is the symmetrizer of $\mathcal{L}$ defined by \eqref{eq-sym:def-S} and $\tilde{S}$ is an inverse of $S$ (in the sense of \eqref{eq-sym:inverse-S}), which can be taken as \eqref{eq-sym:def-tilde-S}. The resulting approximate system reads
	\begin{equation}\label{eq-cauchy:WW-eps}
		\left\{
		\begin{aligned}
			&\left(\partial_t + T_V\cdot\bar{\nabla}J_\epsilon + \mathcal{L}^\epsilon\right)
			\left(\begin{array}{c}
				\eta \\
				\psi
			\end{array}\right) = f(J_\epsilon\eta,J_\epsilon\psi), \\
			&(\eta,\psi)|_{t=0} = (\eta_0,\psi_0) \in H^{s+\frac{1}{2}}_R\times H^s,
		\end{aligned}
		\right.
	\end{equation}
	where $s>3$.
	
	\begin{proposition}\label{prop-cauchy:WP-apprx-sys}
		Let $s>3$ and $\epsilon>0$. For all $(\eta_0,\psi_0) \in \mathcal{X}_R^s$, there exists $T'_\epsilon>0$ and a unique solution $(\eta^\epsilon,\psi^\epsilon)\in C_{T'_\epsilon}\mathcal{X}_R^s$ to the Cauchy problem \eqref{eq-cauchy:WW-eps}. Moreover, if $T_\epsilon$ denotes the supremum of all such $T'_\epsilon$, we have,
		\begin{equation}\label{eq-cauchy:blow-up-criteria}
			T_\epsilon=+\infty,\ \ \text{or}\ \ T_\epsilon<+\infty\ \ \text{with }\limsup_{t\rightarrow T_\epsilon-} \left( \|\eta^\epsilon(t)-R\|_{H^{s+\frac{1}{2}}} + \|\psi^\epsilon(t)\|_{H^s} \right) = +\infty.
		\end{equation}
	\end{proposition}
	
	In the following, we still denote the approximate solution by $(\eta,\psi)$ if there is no confusion.
	
	\begin{proof}
		To prove Proposition \ref{prop-cauchy:WP-apprx-sys}, we attempt to apply Cauchy-Lipschitz theorem, requiring the following quantities to be locally Lipschitizian on $\mathcal{X}_R^s$:
		\begin{equation*}
			T_V\cdot\bar{\nabla}J_\epsilon
			\left(\begin{array}{c}
				\eta \\
				\psi
			\end{array}\right),\ \ 
			\mathcal{L}^\epsilon
			\left(\begin{array}{c}
				\eta \\
				\psi
			\end{array}\right),\ \ f(J_\epsilon\eta,J_\epsilon\psi).
		\end{equation*}
		As in Section \ref{subsect:contin-in-eta}, it suffices to check that the derivative-in-$(\eta,\psi)$ of these quantities are bounded in $\mathcal{X}_R^s$ (in this case, we can assume $\delta\eta\in H^{s+\frac{1}{2}}$). This boundedness is no more than a consequence of Lemma \ref{lem-paralin:BV-deri-in-eta}, \ref{lem-paralin:sym-deri-in-eta}, and Proposition \ref{prop-paralin:source-Lip}, \ref{prop-para:paraprod-bound}. The only term that needs to be treated carefully is the $\eta^{-1}$ arising from the definition \eqref{eq-paralin:def-deri} of $\bar{\nabla}$, namely
		\begin{equation*}
			T_{V^\theta}\left(\frac{1}{\eta}\partial_\theta J_\epsilon
			\left(\begin{array}{c}
				\eta \\
				\psi
			\end{array}\right) \right).
		\end{equation*}
		After applying derivatives in $\eta$ or $\psi$, $\epsilon>0$ ensures that $\partial_\theta J_\epsilon\eta$, $\partial_\theta J_\epsilon\psi$, and their derivatives in $(\eta,\psi)$ lie in $H^{+\infty}:= \cap_{k\in\Z}H^k$. Therefore, their product with $\eta^{-1}$ or $\delta(\eta^{-1})$ has regularity $H^{s+\frac{1}{2}}$ due to Corollary \ref{cor-para:product-law}. Since $T_{V^\theta}$ is of order $0$, the contribution of these terms is still in $\mathcal{X}_R^s$.
	\end{proof}

	\subsection{Energy estimates}\label{subsect:energy-esti}
	
	In this part, we shall make energy estimates for approximate solutions constructed in Proposition \ref{prop-cauchy:WP-apprx-sys}, which also hold for the original system \eqref{eq-paralin:WW}, namely in the case $\epsilon=0$. One should keep in mind that all the estimates in this section are uniform for $\epsilon\in[0,1]$. 
	
	\begin{proposition}\label{prop-cauchy:energy-esti}
		Let $(\eta_0,\psi_0) \in \mathcal{X}_R^s$ with $s>3$. The approximate solution $(\eta^\epsilon,\psi^\epsilon)$ obtained in Proposition \ref{prop-cauchy:WP-apprx-sys} satisfies, for some $T_0$ depending only on $H^{s+\frac{1}{2}}_R\times H^{s}$-norm of initial data $(\eta_0,\psi_0)$,
		\begin{equation}\label{eq-cauchy:energy-esti}
			M_{T,\epsilon} \le C(M_{0,\epsilon})M_{0,\epsilon},\ \ \forall T\in[0,\min(T_0,T_\epsilon)[,\ \ \forall 0<\epsilon\le 1,
		\end{equation}
		where $C$ is a positive increasing function independent of $\epsilon$, $T_\epsilon>0$ is the lifespan of the approximate solution, and
		\begin{equation}\label{eq-cauchy:def-M-eps}
			M_{T,\epsilon} := \sup_{t\in[0,T[} \left( \|\eta^\epsilon(t)-R\|_{H^{s+\frac{1}{2}}} + \|\psi^\epsilon(t)\|_{H^s} \right).
		\end{equation}
	\end{proposition}
	
	As a corollary, the blow-up criteria \eqref{eq-cauchy:blow-up-criteria} implies
	\begin{corollary}\label{cor-cauchy:low-bd-lifespan}
		Under the hypotheses of Proposition \ref{prop-cauchy:WP-apprx-sys}, there exists $T_0>0$ depending only on $H^{s+\frac{1}{2}}_R\times H^{s}$-norm of initial data $(\eta_0,\psi_0)$, such that the lifespan $T_\epsilon$ satisfies,
		\begin{equation}\label{eq-cauchy:low-bd-lifespan}
			T_\epsilon \ge T_0,\ \ \forall \epsilon>0.
		\end{equation}
		As a result, the energy estimate \eqref{eq-cauchy:energy-esti} becomes
		\begin{equation}\label{eq-cauchy:energy-esti-main}
			M_{T,\epsilon} \le C(M_{0,\epsilon})M_{0,\epsilon},\ \ \forall T\in[0,T_0[,\ \ \forall 0<\epsilon\le 1.
		\end{equation}
	\end{corollary}
	\begin{remark}\label{rmk-cauchy:unif-bound}
		Note that, once we manage to prove the convergence of approximate solutions $(\eta^\epsilon,\psi^\epsilon)$ in $L^\infty_T\mathcal{X}_R^s$ or weaker space $L^\infty_T\mathcal{X}_R^{s_0}$ ($s_0<s$), the limit $(\eta,\psi)$ will inherit this uniform bound, namely,
		\begin{equation*}
			\|(\eta,\psi)\|_{L^\infty_T(H^{s+\frac{1}{2}}_R\times H^s)} \le C\left( \|(\eta_0,\psi_0)\|_{H^{s+\frac{1}{2}}_R\times H^s} \right) \|(\eta_0,\psi_0)\|_{H^{s+\frac{1}{2}}_R\times H^s},\ \ \forall T\in[0,T_0[.
		\end{equation*}
	\end{remark}
	
	Instead of proving Proposition \ref{prop-cauchy:energy-esti}, we shall consider a more general case. 
	\begin{proposition}\label{prop-cauchy:energy-esti-gen}
		Let $(\eta'_0,\psi'_0) \in H^{s+\frac{1}{2}}_{R'}\times H^{s}$ and $(\eta,\psi) \in L^\infty_T\mathcal{X}_R^s$ with $R'\in\R$, $R>0$, $T>0$, and $s>3$. Given $s_0\in ]\frac{3}{2},s]$, we assume that
		
		(1) The system
		\begin{equation}\label{eq-cauchy:WW-eps-gen}
			\left\{
			\begin{aligned}
				&\left(\partial_t + T_V\cdot\bar{\nabla}J_\epsilon + \mathcal{L}^\epsilon\right)
				\left(\begin{array}{c}
					\eta' \\
					\psi'
				\end{array}\right) = f, \\
				&(\eta',\psi')|_{t=0} = (\eta'_0,\psi'_0),
			\end{aligned}
			\right.
		\end{equation}
		admits a solution $(\eta',\psi')\in L^\infty_T(H^{s+\frac{1}{2}}_{R'}\times H^s)$, where the operators $T_V\cdot\bar{\nabla}J_\epsilon$ and $\mathcal{L}^\epsilon$ are associated to $(\eta,\psi)$;
		
		(2) The source term $f$ verifies, for all $t\in[0,T[$,
		\begin{equation}\label{eq-cauchy:energy-esti-gen-f-esti}
			\|f\|_{L^\infty_t(H^{s_1+\frac{1}{2}}\times H^{s_1})} \le C(N^{s-}_t) \left( M^{s_1}_t + c_f \right),
		\end{equation}
		where $c_f\in[0,1]$ is a constant,
		\begin{equation}\label{eq-cauchy:def-M-T}
			M_T^r := \sup_{t\in[0,T[} \left( \|\eta'(t)-R'\|_{H^{r+\frac{1}{2}}} + \|\psi'(t)\|_{H^r} \right),
		\end{equation}\index{M@$M_T^r$ Sobolev energy}
		and $N^r_T$ is defined as
		\begin{equation}\label{eq-cauchy:def-N-T}
			N_T^r := \sup_{t\in[0,T[} \left( \|\eta(t)-R\|_{H^{r+\frac{1}{2}}} + \|\psi(t)\|_{H^r} \right) + M^r_T,
		\end{equation}\index{N@$N_T^r$ Auxiliary Sobolev energy}
		
		(3) There exists $\frac{3}{2} < s_1 < \min(s-\frac{3}{2},s_0)$ such that $(\eta,\psi)$ satisfies for all $t\in[0,T[$,
		\begin{equation}\label{eq-cauchy:energy-esti-gen-deri-in-t}
			\|\eta_t\|_{H^{s_1+\frac{1}{2}}} + \|\psi_t\|_{H^{s_1}} \le C(N^{s-}_t)N^{s_1+\frac{3}{2}}_t.
		\end{equation}
		
		Then there exists $0<T_0 \le T$ depending only on $\|(\eta'_0,\psi'_0)\|_{H^{s+\frac{1}{2}}_{R'}\times H^{s}}$ and $c_f\in[0,1]$, such that for all $\epsilon\in[0,1]$ and $t\in[0,T_0]$,
		\begin{equation}\label{eq-cauchy:energy-esti-gen}
			(M^{s_0}_t)^2 \le C(N^{s-}_0)(M^{s_0}_0)^2 + t C(N^{s-}_t)\left( N^{s_0}_{t} + 1 \right) \left( (N^{s_0}_{t})^2 + c_f^2 \right),
		\end{equation}
		where the smooth increasing function $C>0$ is independent of $\epsilon\in[0,1]$.
		
		In particular, if $s_0<s$, the estimate above can be written as
		\begin{equation}\label{eq-cauchy:energy-esti-gen-weak}
			(M^{s_0}_t)^2 \le C(N^{s-}_0)(M^{s_0}_0)^2 + t C(N^{s-}_t)\left( (M^{s_0}_{t})^2 + c_f^2 \right),\ \ \forall t\in[0,T_0].
		\end{equation}
		Meanwhile, if $s_0\ge s-$, $c_f=0$, $R'=R$, and $(\eta',\psi') = (\eta,\psi)\in C_T\mathcal{X}^{s}_R$, we have $N^r_t=2M^r_t$ as well as
		\begin{equation}\label{eq-cauchy:energy-esti-gen-strong}
			N^{s_0}_t \le C(N^{s-}_0)N^{s_0}_0,\ \ \forall t\in[0,T_0].
		\end{equation}
	\end{proposition}
	
	\begin{remark}
		In application of this proposition, we always have that $(\eta,\psi)$ solves \eqref{eq-paralin:WW} or the approximate system \eqref{eq-cauchy:WW-eps}. Therefore, condition (3) is trivial. Besides, $(\eta',\psi')$ is equal to either $(\eta,\psi)$ or the difference of two solutions to \eqref{eq-paralin:WW} or \eqref{eq-cauchy:WW-eps}. In both cases, we can easily derive the equation \eqref{eq-cauchy:WW-eps-gen} in (1).
	\end{remark}
	
	When Proposition \ref{prop-cauchy:energy-esti-gen} is proved, Proposition \ref{prop-cauchy:energy-esti} follows from \eqref{eq-paralin:source-esti} and \eqref{eq-cauchy:energy-esti-gen-strong} with $s_0=s$, $c_f=0$, and $(\eta',\psi')=(\eta,\psi)$. In order to prove Proposition \ref{prop-cauchy:energy-esti-gen}, we begin with the following two lemmas.
	\begin{lemma}\label{lem-cauchy:energy-esti-commu-1}
		Let $(\eta,\psi)\in L^\infty_T\mathcal{X}_R^s$ with $s>3$ and $T>0$. Then the following estimate holds uniformly in $\epsilon\in[0,1]$ and $t\in[0,T[$,
		\begin{equation}\label{eq-cauchy:energy-esti-commu-1}
			S\left(\begin{array}{cc}
				I & 0 \\
				-T_B & I
			\end{array}\right) \mathcal{L}^\epsilon \approx
			\left(\begin{array}{cc}
				0 & -T_\gamma J_\epsilon \\
				(T_\gamma J_\epsilon)^* & 0
			\end{array}\right)S\left(\begin{array}{cc}
			I & 0 \\
			-T_B & I
			\end{array}\right).
		\end{equation}
		The meaning of this equivalence between matrix of paralinear operators is that the corresponding entries are equivalent in the sense of \eqref{eq-cauchy:equi-paradiff}.
	\end{lemma}
	\begin{proof}
		By definition \eqref{eq-cauchy:def-L-eps} of $\mathcal{L}^\epsilon$, we have
		\begin{equation*}
			S\left(\begin{array}{cc}
				I & 0 \\
				-T_B & I
			\end{array}\right) \mathcal{L}^\epsilon = 
			S\left(\begin{array}{cc}
				0 & -T_\lambda \\
				\sigma T_\mu & 0
			\end{array}\right) \tilde{S}J_\epsilon S
			\left(\begin{array}{cc}
				I & 0 \\
				-T_B & I
			\end{array}\right).
		\end{equation*}
		Thanks to Proposition \ref{prop-sym:main}, the right hand side is equivalent to
		\begin{equation*}
			\left(\begin{array}{cc}
				0 & -T_\gamma \\
				T_\gamma^* & 0
			\end{array}\right)S\tilde{S}J_\epsilon S
			\left(\begin{array}{cc}
			I & 0 \\
			-T_B & I
			\end{array}\right) \approx
			\left(\begin{array}{cc}
				0 & -T_\gamma \\
				T_\gamma^* & 0
			\end{array}\right)J_\epsilon S
			\left(\begin{array}{cc}
				I & 0 \\
				-T_B & I
			\end{array}\right),
		\end{equation*}
		where we use $S\tilde{S}\approx id$ (see \eqref{eq-sym:inverse-S}). It remains to check that $T_\gamma^*J_\epsilon \approx (T_\gamma J_\epsilon)^*$. In fact, by \eqref{eq-sym:main-equi} and Lemma \ref{lem-cauchy:esti-j-eps},
		\begin{equation*}
			(T_\gamma J_\epsilon)^* = J_\epsilon^* T_\gamma^* \approx J_\epsilon T_\gamma^* \approx J_\epsilon T_\gamma \approx T_\gamma J_\epsilon \approx T_\gamma^*J_\epsilon.
		\end{equation*}
	\end{proof}
	
	\begin{lemma}\label{lem-cauchy:energy-esti-commu-2}
		Let $(\eta,\psi)\in L^\infty_T\mathcal{X}_R^s$ with $s>3$ and $T>0$. Then the following estimate holds uniformly in $\epsilon\in[0,1]$ and $t\in[0,T[$,
		\begin{equation}\label{eq-cauchy:energy-esti-commu-2}
			\| [T_B, \partial_t] \|_{\mathcal{L}(H^r;H^{r-\frac{1}{2}+})} + \| [T_B, T_V\cdot\bar{\nabla}] \|_{\mathcal{L}(H^r;H^{r-\frac{1}{2}+})} \le C(N^{s-}_T),\ \ \forall r\le s+\frac{1}{2}.
		\end{equation}
	\end{lemma}
	\begin{proof}
		The commutator $T_B$ and $\partial_t$ reads
		\begin{equation*}
			[T_B, \partial_t] = T_{\partial_t B}.
		\end{equation*}
		By Proposition \ref{prop-para:paradiff-symbol-class-sobo}, it suffices to prove that $\partial_tB\in H^{\frac{1}{2}+}$. Thanks to formula \eqref{eq-paralin:B-formula} of $B$, one may regard $B$ as a smooth function of $G(\eta)\psi$, $(\eta,\nabla_w\eta)$, and $(\psi,\nabla_w\psi)$. By Proposition \ref{prop-para:paralin} and Corollary \ref{cor-para:product-law}, it suffices to check that the derivative in time of these quantities belongs to $H^{\frac{1}{2}+}$. From equation \eqref{eq-intro:WW}, we have
		\begin{equation*}
			\partial_t\eta = G(\eta)\psi\in H^{s-1},\ \ \partial_t\psi = -\sigma\left(H-\frac{1}{2R}\right) - N \in H^{s-\frac{3}{2}},
		\end{equation*}
		which is a consequence of Corollary \ref{cor-pre:bound-DtN}, \ref{cor-para:product-law}, and Proposition \ref{prop-para:paralin} (recall that $B,V\in H^{s-1}$ due to Lemma \ref{lem-paralin:esti-B-V-psi}). This observation implies that $\partial_t(\eta,\nabla_w\eta)$ and $\partial_t(\psi,\nabla_w\psi)$ have regularity $H^{\frac{1}{2}+}$ since $s>3$. As for the time derivative of $G(\eta)\psi$, we have
		\begin{equation*}
			\partial_t \left( G(\eta)\psi \right) = \frac{d}{d\eta}G(\eta)\psi \cdot \partial_t\eta + G(\eta)\partial_t\psi \in H^{s-\frac{5}{2}} \subset H^{\frac{1}{2}+},
		\end{equation*}
		thanks to Proposition \ref{prop-pre:shape-deri} and \ref{cor-pre:bound-DtN}.
		
		It remains to study the commutator $[T_B, T_V\cdot\bar{\nabla}]$, whose estimate follows from Proposition \ref{prop-sym:commu-S-T-V-nabla} with $a=B \in\Gamma^0_{2+}$ (see also \eqref{eq-paralin:esti-B-V-psi}).
	\end{proof}
	
	Now, we are ready to prove Proposition \ref{prop-cauchy:energy-esti-gen}. Recall that $(\eta',\psi')$ solves \eqref{eq-cauchy:WW-eps-gen}, namely
	\begin{equation*}
		\left(\partial_t + T_V\cdot\bar{\nabla}J_\epsilon + \mathcal{L}^\epsilon\right)
		\left(\begin{array}{c}
			\eta' \\
			\psi'
		\end{array}\right) = f\in L^\infty_T(H^{s+\frac{1}{2}}\times H^s).
	\end{equation*}
	By applying operator
	\begin{equation}\label{eq-cauchy:op-mod-main}
		S\left(\begin{array}{cc}
			I & 0 \\
			-T_B & I
		\end{array}\right)
	\end{equation}
	on both sides (recall that $S$ is defined in \eqref{eq-sym:def-S}), one has
	\begin{equation}\label{eq-cauchy:R-1-def}
		\begin{aligned}
			&S\left(\begin{array}{cc}
				I & 0 \\
				-T_B & I
			\end{array}\right)\left(\partial_t + T_V\cdot\bar{\nabla}J_\epsilon + \mathcal{L}^\epsilon\right)  \\
			=& \left(\partial_t + T_V\cdot\bar{\nabla}J_\epsilon \right)S\left(\begin{array}{cc}
				I & 0 \\
				-T_B & I
			\end{array}\right) + \left(\begin{array}{cc}
				0 & -T_\gamma J_\epsilon \\
				(T_\gamma J_\epsilon)^* & 0
			\end{array}\right)S\left(\begin{array}{cc}
				I & 0 \\
				T_B & I
			\end{array}\right) + R_1.
		\end{aligned}
	\end{equation}
	We claim that the remainder $R_1$ satisfies the following property.
	\begin{lemma}\label{lem-cauchy:energy-esti-commu-main}
		Under the assumption of Proposition \ref{prop-cauchy:energy-esti-gen}, $R_1$ defined by \eqref{eq-cauchy:R-1-def} satisfies, for all $r\in[0,s]$,
		\begin{equation}
			\|R_1\|_{\mathcal{L}(H^{r+\frac{1}{2}}_{R'}\times H^{r};H^{r}\times H^{r})} \le C(N^{s-}_t).
		\end{equation}
	\end{lemma}
	\begin{proof}
		We observe from the definition \eqref{eq-sym:def-S} of $S$, that
		\begin{align*}
			R_1 =& \left[ S\left(\begin{array}{cc}
				I & 0 \\
				-T_B & I
			\end{array}\right), \partial_t + T_V\cdot\bar{\nabla}J_\epsilon \right] \\
			&+ S\left(\begin{array}{cc}
				I & 0 \\
				-T_B & I
			\end{array}\right)\mathcal{L}^\epsilon - \left(\begin{array}{cc}
				0 & -T_\gamma J_\epsilon \\
				(T_\gamma J_\epsilon)^* & 0
			\end{array}\right)S\left(\begin{array}{cc}
				I & 0 \\
				T_B & I
			\end{array}\right) \\
			=& \left[ \left(\begin{array}{cc}
				T_p & 0 \\
				-T_q T_B & T_q
			\end{array}\right), \partial_t + T_V\cdot\bar{\nabla}J_\epsilon \right] + R_2'' = R_2' + R_2'',
		\end{align*}
		where
		\begin{align*}
			R_2' =& \left(\begin{array}{cc}
				\left[T_p, \partial_t + T_V\cdot\bar{\nabla}J_\epsilon\right] & 0 \\
				\left[-T_q T_B, \partial_t + T_V\cdot\bar{\nabla}J_\epsilon\right] & \left[T_q, \partial_t + T_V\cdot\bar{\nabla}J_\epsilon\right]
			\end{array}\right), \\
			R_2'' =& S\left(\begin{array}{cc}
				I & 0 \\
				-T_B & I
			\end{array}\right)\mathcal{L}^\epsilon - \left(\begin{array}{cc}
				0 & -T_\gamma J_\epsilon \\
				(T_\gamma J_\epsilon)^* & 0
			\end{array}\right)S\left(\begin{array}{cc}
				I & 0 \\
				T_B & I
			\end{array}\right).
		\end{align*}
		From Lemma \ref{lem-cauchy:energy-esti-commu-1}, $R_2''$ satisfies the desired estimate, so as the commutator between $T_p$ (and $T_q$) and $\partial_t + T_V\cdot\bar{\nabla}J_\epsilon$ (see Proposition \ref{prop-sym:commu-S-partial-t} and \ref{prop-sym:commu-S-T-V-nabla}). Thus, it remains to show that the commutator $\left[T_q T_B, \partial_t + T_V\cdot\bar{\nabla}J_\epsilon\right]$ is bounded from $H^{r+\frac{1}{2}}_{R'}$ to $H^r$, with operator norm bounded by $C(N^{s-}_t)$. In fact, we have
		\begin{align*}
			\left[T_q T_B, \partial_t + T_V\cdot\bar{\nabla}J_\epsilon\right] =& T_q \left[T_B, \partial_t + T_V\cdot\bar{\nabla}J_\epsilon\right] + \left[T_q, \partial_t + T_V\cdot\bar{\nabla}J_\epsilon\right]T_B \\
			=& T_q \left[T_B, \partial_t \right] + T_q [T_B,T_V\cdot\bar{\nabla}]J_\epsilon + T_qT_V\cdot\bar{\nabla}[T_B,J_\epsilon] \\
			&+ [T_q,\partial_t]T_B + [T_q, T_V\cdot\bar{\nabla}]J_\epsilon T_B + T_V\cdot\bar{\nabla}[T_q,J_\epsilon]T_B.
		\end{align*}
		Since $T_q$, $T_B$, and $J_\epsilon$ are of order zero, thanks to Proposition \ref{prop-sym:commu-S-partial-t}, \ref{prop-sym:commu-S-T-V-nabla}, and Lemma \ref{lem-cauchy:energy-esti-commu-2}, it suffices to study the term $T_qT_V\cdot\bar{\nabla}[T_B,J_\epsilon]$ and $T_V\cdot\bar{\nabla}[T_q,J_\epsilon]T_B$, which are bounded from $H^{r+\frac{1}{2}}_{R'}$ to $H^{r+\frac{1}{2}}$ since the commutators $[T_B,J_\epsilon]$, $[T_q,J_\epsilon]$ are of order $(-1)$ due to Corollary \ref{cor-para:commu-esti} and \ref{cor-paralin:homo-sym-commu}.
	\end{proof}
	
	As a consequence of Lemma \ref{lem-cauchy:energy-esti-commu-main}, the quantity
	\begin{equation}\label{eq-cauchy:def-Y}
		Y:= S\left(\begin{array}{cc}
			I & 0 \\
			T_B & I
		\end{array}\right)\left(\begin{array}{c}
			\eta' \\
			\psi'
		\end{array}\right) \in C_T(H^s\times H^s),
	\end{equation}
	verifying the direct estimate
	\begin{equation}\label{eq-cauchy:Y-esti}
		\|Y(t)\|_{H^{s_0}\times H^{s_0}} \le C(N^{s-}_t) M^{s_0}_t,\ \ \forall s_0\in[0,s],\ t\in[0,T[,
	\end{equation}
	solves the equation
	\begin{equation}\label{eq-cauchy:eq-Y}
		\partial_t Y + T_V\cdot\bar{\nabla}J_\epsilon Y + \left(\begin{array}{cc}
			0 & -T_\gamma J_\epsilon \\
			(T_\gamma J_\epsilon)^* & 0
		\end{array}\right)Y = F_1,
	\end{equation}
	with
	\begin{equation}\label{eq-cauchy:def-F-1}
		F_1 = S\left(\begin{array}{cc}
			I & 0 \\
			T_B & I
		\end{array}\right)f - R_1\left(\begin{array}{c}
		\eta' \\
		\psi'
		\end{array}\right).
	\end{equation}
	Hence, the new source term $F_1$ verifies, for all $0\le s_0\le s$ and $t\in[0,T[$,
	\begin{equation}\label{eq-cauchy:F-1-esti}
		\|F_1(t)\|_{C_T(H^{s_0}\times H^{s_0})} \le C(N^{s-}_t) \left( \|f(t)\|_{H^{s_0+\frac{1}{2}}\times H^{s_0}} + M^{s_0}_t \right).
	\end{equation}
	
	To obtain the $H^{s_0+\frac{1}{2}}_{R'} \times H^{s_0}$ estimate of $(\eta',\psi')$ (namely \eqref{eq-cauchy:energy-esti-gen}), we introduce a new symbol $\beta:= \beta^{(s_0)} + \beta^{(s_0-1)}\in\Sigma^s$, defined by
	\begin{equation}\label{eq-cauchy:def-beta}
		\beta^{(s_0)}:= \left( \gamma^{(3/2)} \right)^{\frac{2s_0}{3}},\ \ \beta^{(s_0-1)}=-\frac{i}{2}\partial_w\cdot\partial_\xi \beta^{(s_0)},
	\end{equation}
	where $\gamma^{(3/2)}$ is constructed in Proposition \ref{prop-sym:main}. By applying Lemma \ref{lem-cauchy:symbol-fct-of-gamma} with $F(\epsilon,\rho)=\rho^{2s_0/3}$, $G_1(\epsilon,\rho)=\exp(-\epsilon\rho)$, and $G_2(\epsilon,\rho)=\rho$, we obtain the following properties of $\beta$.
	\begin{lemma}\label{lem-cauchy:esti-beta}
		Let $\eta\in L^\infty_TH^{s+\frac{1}{2}}_R$ with $s>3$ and $T>0$. Then the symbol $\beta$ defined above is elliptic and belong to $\Gamma^{s_0}_{3/2+} + \Gamma^{s_0-1}_{1/2+}$ uniformly in $t\in[0,T[$. Moreover, we have
		\begin{equation}\label{eq-cauchy:commu-beta}
			T_\beta T_\gamma \approx T_\gamma T_\beta,\ \ T_\beta J_\epsilon \approx J_\epsilon T_\beta,
		\end{equation}
		uniformly in $\epsilon,t$, in the sense of \eqref{eq-cauchy:equi-paradiff}.
	\end{lemma}
	As a result of Lemma \ref{lem-cauchy:esti-beta} together with \eqref{eq-sym:main-equi} and \eqref{eq-cauchy:commu-j-eps}, we have
	\begin{equation*}
		T_\beta \left(\begin{array}{cc}
			0 & -T_\gamma J_\epsilon \\
			(T_\gamma J_\epsilon)^* & 0
		\end{array}\right) \approx
		\left(\begin{array}{cc}
			0 & -T_\gamma J_\epsilon \\
			(T_\gamma J_\epsilon)^* & 0
		\end{array}\right) T_\beta.
	\end{equation*}
	Moreover, thanks to Proposition \ref{prop-sym:commu-S-partial-t}, \ref{prop-sym:commu-S-T-V-nabla}, and Corollary \ref{cor-para:commu-esti}, the commutator between $T_\beta$ and $(\partial_t + T_V\cdot\bar{\nabla}J_\epsilon)$ is bounded from $H^{r}$ to $H^{r-s_0}$ for all $r\in[0,s]$, whose operator norm is bounded by $C(N^{s-}_t)$. Now, we apply $T_\beta$ to \eqref{eq-cauchy:eq-Y} and write the resulting equation as
	\begin{equation}\label{eq-cauchy:eq-T-beta-Y}
		\partial_t T_\beta Y + T_V\cdot\bar{\nabla}J_\epsilon T_\beta Y + \left(\begin{array}{cc}
			0 & -T_\gamma J_\epsilon \\
			(T_\gamma J_\epsilon)^* & 0
		\end{array}\right)T_\beta Y = F_2,
	\end{equation}
	with
	\begin{equation}\label{eq-cauchy:def-F-2}
		F_2 = T_\beta F_1 + R_2 Y,
	\end{equation}
	where the remainder $R_2$ is bounded from $H^{r}\times H^{r}$ to $H^{r-s_0}\times H^{r-s_0}$ for all $r\in[0,s]$, whose operator norm is bounded by $C(N^{s-}_t)$, which, combined with \eqref{eq-cauchy:Y-esti} and \eqref{eq-cauchy:F-1-esti}, yields the following estimate for $F_2$,
	\begin{equation}\label{eq-cauchy:F-2-esti}
		\|F_2(t)\|_{C_T(L^2\times L^2)} \le C(N^{s-}_t) \left( \|f(t)\|_{H^{s_0+\frac{1}{2}}\times H^{s_0}} + M^{s_0}_t \right),\ \ \forall t\in[0,T[.
	\end{equation}
	
	The $L^2$-scalar product of \eqref{eq-cauchy:eq-T-beta-Y} with $T_\beta Y$ gives
	\begin{equation*}
		\frac{1}{2}\frac{d}{dt}\|T_\beta Y\|_{L^2\times L^2}^2 = -\Real\langle T_V\cdot\bar{\nabla}J_\epsilon T_\beta Y, T_\beta Y \rangle + \Real\langle F_2, T_{\beta}Y \rangle.
	\end{equation*}
	We claim that
	\begin{lemma}\label{lem-cauchy:energy-esti-1}
		Let $(\eta,\psi)\in L^\infty_T\mathcal{X}_R^s$ with $s>3$ and $T>0$. Then we have
		\begin{equation}\label{eq-cauchy:energy-esti-1}
			\left| \Real\langle T_V\cdot\bar{\nabla}J_\epsilon T_\beta Y, T_\beta Y \rangle \right| \le C(N^{s-}_t)\|T_\beta Y\|_{L^2\times L^2}^2,\ \ \forall t\in[0,T[,
		\end{equation}
		where $C>0$ is an $\epsilon$-independent smooth increasing function.
	\end{lemma}
	Once Lemma \ref{lem-cauchy:energy-esti-1} is proved (the proof is deferred to the end of this section), we can conclude using \eqref{eq-cauchy:Y-esti} and \eqref{eq-cauchy:F-2-esti} that, for all $t\in[0,T[$,
	\begin{align*}
		&\|T_{\beta(t)} Y(t)\|_{L^2\times L^2}^2 \\
		\le& \|T_{\beta} Y(0)\|_{L^2\times L^2}^2 + \int_0^t \left( C(N^{s-}_{t'})\|T_\beta Y(t')\|_{L^2\times L^2}^2 + \|T_\beta Y(t')\|_{L^2\times L^2} \|F_2(t')\|_{L^2\times L^2} \right) dt' \\
		\le& C(N^{s-}_0)\|Y(0)\|_{H^{s_0}\times H^{s_0}}^2 + \int_0^t C(N^{s-}_{t'})\|Y(t')\|_{H^{s_0}\times H^{s_0}}\left( \|Y(t')\|_{H^{s_0}\times H^{s_0}} + \|F_2(t')\|_{L^2\times L^2} \right) dt' \\
		\le& C(N^{s-}_0)(M^{s_0}_0)^2 + \int_0^t C(N^{s-}_{t'})M^{s_0}_{t'} \left( M^{s_0}_{t'} + \|f(t')\|_{H^{s_0+\frac{1}{2}}\times H^{s_0}} \right) dt' \\
		\le& C(N^{s-}_0)(M^{s_0}_0)^2 + t C(N^{s-}_t) M^{s_0}_{t} \left( M^{s_0}_{t} + \|f\|_{L_t^\infty(H^{s_0+\frac{1}{2}}\times H^{s_0})} \right).
	\end{align*}
	An application of \eqref{eq-cauchy:energy-esti-gen-f-esti} leads to the following estimate of $T_\beta Y$,
	\begin{equation}\label{eq-cauchy:esti-T-beta-Y}
		\|T_{\beta(t)} Y(t)\|_{L^2\times L^2}^2 \le C(N^{s-}_0)(M^{s_0}_0)^2 + t C(N^{s-}_t) M^{s_0}_{t} \left( M^{s_0}_{t} + c_f \right).
	\end{equation}
	
	\begin{proof}[Proof of \eqref{eq-cauchy:energy-esti-gen}]
		Compared with the desired estimate \eqref{eq-cauchy:energy-esti}, we need to pass from the estimate of $T_\beta Y$ to that of $(\eta',\psi')$. Notice that $\beta$, $p$, and $q$ are elliptic symbols. Then, by Proposition \ref{prop-paralin:homo-sym-ellip-esti}, we have
		\begin{equation*}
			\|(\eta',\psi')(t)\|_{H^{s_0+\frac{1}{2}}_{R'}\times H^{s_0}} \le C(N^{\frac{3}{2}+}_t) \left( \|T_{\beta(t)} Y(t)\|_{L^2\times L^2} + \|(\eta',\psi')(t)\|_{H^{s_0-1}_{R'}\times H^{s_0-\frac{3}{2}}} \right).
		\end{equation*}
		Recall that, $\mathcal{L}^\epsilon$ (defined in \eqref{eq-cauchy:def-L-eps}) is bounded from $H^{r+\frac{1}{2}}_{R'}\times H^{r}$ to $H^{r-1}_{R'}\times H^{r-\frac{3}{2}}$ for all $r\in\R$ with operator norm controlled by $C(N^{s-}_t)$. Then, by \eqref{eq-cauchy:WW-eps-gen}, one can easily see that
		\begin{equation*}
			\|\eta'_t\|_{H^{s'-1}} + \|\psi'_t\|_{H^{s'-\frac{3}{2}}} \le C(N^{s-}_t)M^{s'}_t + \|f(t)\|_{H^{s'-1}\times H^{s'-\frac{3}{2}}},\ \ \forall t\in[0,T[, s'\in]\frac{3}{2},s-\frac{3}{2}[,
		\end{equation*}
		which yields, for all $t\in[0,T[$,
		\begin{align*}
			&\|T_\beta Y(t)\|_{L^2\times L^2}^2 + \|(\eta',\psi')(t)\|_{H^{s_0-1}_{R'}\times H^{s_0-\frac{3}{2}}}^2 \\
			\le& \|T_\beta Y(t)\|_{L^2\times L^2}^2 + \|(\eta'_0,\psi'_0)\|_{H^{s_0-1}_{R'}\times H^{s_0-\frac{3}{2}}}^2 \\
			&\hspace{8em} +2\int_0^t \|(\eta',\psi')(t')\|_{H^{s_0-1}_{R'}\times H^{s_0-\frac{3}{2}}} \left( \|\eta'_t\|_{H^{s_0-1}} + \|\psi'_t\|_{H^{s_0-\frac{3}{2}}} \right) dt' \\
			\le& \|T_\beta Y(t)\|_{L^2\times L^2}^2 + \|(\eta'_0,\psi'_0)\|_{H^{s_0+\frac{1}{2}}_{R'}\times H^{s_0}}^2 \\
			&\hspace{4em}+ 2\int_0^t \|(\eta',\psi')(t')\|_{H^{s_0-1}_{R'}\times H^{s_0-\frac{3}{2}}} \left( C(N^{s-}_{t'})M^{s_0}_{t'} + \|f(t')\|_{H^{s_0-1}\times H^{s_0-\frac{3}{2}}} \right) dt' \\
			\le& \|T_\beta Y(t)\|_{L^2\times L^2}^2 + (M^{s_0}_0)^2 + 2\int_0^t M^{s_0}_{t'} \left( C(N^{s-}_{t'})M^{s_0}_{t'} + \|f(t')\|_{H^{s_0-1}\times H^{s_0-\frac{3}{2}}} \right) dt' \\
			\le& \|T_\beta Y(t)\|_{L^2\times L^2}^2 + (M^{s_0}_0)^2 + t C(N^{s-}_{t}) M^{s_0}_{t} \left(M^{s_0}_{t} + \|f\|_{L^\infty_t(H^{s_0-1}\times H^{s_0-\frac{3}{2}})} \right) \\
			\le& \|T_\beta Y(t)\|_{L^2\times L^2}^2 + (M^{s_0}_0)^2 + t C(N^{s-}_{t}) M^{s_0}_{t} \left(M^{s_0}_{t} + c_f \right) \\
			\le& C(N^{s-}_0)(M^{s_0}_0)^2 + t C(N^{s-}_t) M^{s_0}_{t} \left( M^{s_0}_{t} + c_f \right),
		\end{align*}
		where we used again the assumption \eqref{eq-cauchy:energy-esti-gen-f-esti} and estimate \eqref{eq-cauchy:esti-T-beta-Y} of $T_\beta Y$.
		
		To sum up, we have proved that, for all $t\in[0,T[$,
		\begin{equation}\label{eq-cauchy:energy-esti-gen-alt}
			\|(\eta,\psi)(t)\|_{H^{s_0+\frac{1}{2}}_R\times H^{s_0}}^2 \le C\left( N^{\frac{3}{2}+}_t \right) \left( C(N^{s-}_0)(M^{s_0}_0)^2 + t C(N^{s-}_t) M^{s_0}_{t} \left( M^{s_0}_{t} + c_f \right) \right).
		\end{equation}
		Notice that, it is harmless to replace $C\left( N^{\frac{3}{2}+}_t \right)$ by $C\left( \left(N^{s_1}_t \right)^2\right)$ (recall that $\frac{3}{2}<s_1<\min(s-\frac{3}{2},s_0)$), where $C>0$ is smooth increasing function. From the definition \eqref{eq-cauchy:def-N-T} of $N^{s_1}_t$, equation \eqref{eq-cauchy:WW-eps-gen}, and assumption \eqref{eq-cauchy:energy-esti-gen-deri-in-t}, it is clear that, for all $t\in[0,T[$,
		\begin{align*}
			C\left( \left(N^{s_1}_t \right)^2\right) \le& C\left( \left(N^{s_1}_0 \right)^2\right) + \int_0^t C'\left( \left(N^{s_1}_{t'} \right)^2\right)\|(\eta,\psi)(t')\|_{H^{s_1+\frac{1}{2}}_R\times H^{s_1}} \|(\eta_t,\psi_t)(t')\|_{H^{s_1+\frac{1}{2}}\times H^{s_1}} dt' \\
			&+ \int_0^t C'\left( \left(N^{s_1}_{t'} \right)^2\right)\|(\eta',\psi')(t')\|_{H^{s_1+\frac{1}{2}}_{R'}\times H^{s_1}} \|(\eta'_t,\psi'_t)(t')\|_{H^{s_1+\frac{1}{2}}\times H^{s_1}} dt' \\
			\le& C\left( N^{s_1}_0 \right) + \int_0^t C\left( N^{s_1}_{t'}\right) N^{s_1}_{t'} C(N^{s-}_{t'}) N^{s_1+\frac{3}{2}}_{t'} dt' \\
			\le& C\left( N^{s-}_0 \right) + \int_0^t C\left( N^{s-}_{t'}\right) N^{s_0}_{t'} dt' \\
			\le& C\left( N^{s-}_0 \right) + t C\left( N^{s-}_{t}\right) N^{s_0}_{t},
		\end{align*}
		and the right hand side of \eqref{eq-cauchy:energy-esti-gen-alt} can be bounded by 
		\begin{align*}
			\|(\eta,\psi)(t)\|_{H^{s_0+\frac{1}{2}}_R\times H^{s_0}}^2 \le& \left( C\left( N^{s-}_0 \right) + t C\left( N^{s-}_{t}\right) N^{s_0}_{t}  \right) \\
			&\hspace{4em}\times \left( C(N^{s-}_0)(M^{s_0}_0)^2 + t C(N^{s-}_t) M^{s_0}_{t} \left( M^{s_0}_{t} + c_f \right) \right) \\
			\le& C(N^{s-}_0)(M^{s_0}_0)^2 + t C\left( N^{s-}_{t}\right) N^{s_0}_{t}(M^{s_0}_t)^2 + t C(N^{s-}_t) \left( (M^{s_0}_{t})^2 + c_f^2 \right) \\
			&+  t^2 C(N^{s-}_t) N^{s_0}_{t} \left( (M^{s_0}_{t})^2 + c_f^2 \right) \\
			\le& C(N^{s-}_0)(M^{s_0}_0)^2 + (t+t^2) C(N^{s-}_t)\left( N^{s_0}_{t} + 1 \right) \left( (M^{s_0}_{t})^2 + c_f^2 \right),
		\end{align*}
		which completes the proof of \eqref{eq-cauchy:energy-esti-gen} by taking $T_0<1$. 
	\end{proof}
	
	\begin{proof}[Proof of \eqref{eq-cauchy:energy-esti-gen-weak} and \eqref{eq-cauchy:energy-esti-gen-strong}]
		When $s_0<s$, \eqref{eq-cauchy:energy-esti-gen-weak} can be obtained by the trivial inequality $N^{s_0}_t \le N^{s-}_t$. Now, we assume $s_0\ge s-$, $c_f=0$, and prove \eqref{eq-cauchy:energy-esti-gen-strong} via a bootstrap argument. Let 
		\begin{equation*}
			B_\nu := \sqrt{2C(N^{s-}_0)(N^{s_0}_0)^2 + \nu^2} \ge \nu>0,
		\end{equation*}
		where $\nu>0$ is a small parameter. We fix $T_{0,\nu}\in]0,T[$ such that
		\begin{equation*}
			C(N^{s-}_0)(N^{s_0}_0)^2 + T_{0,\nu} C(B_\nu)(B_\nu+1)B_\nu^2 < B_\nu^2,
		\end{equation*}
		or equivalently,
		\begin{equation*}
			T_{0,\nu} C(B_\nu)(B_\nu+1)B_\nu^2 < \frac{B_\nu^2+\nu^2}{2}.
		\end{equation*}
		One may take $T_{0,\nu}$ as
		\begin{equation*}
			T_{0,\nu} = \frac{B_\nu^2+\nu^2}{4C(B_\nu)(B_\nu+1)B_\nu^2},
		\end{equation*}
		which depends only on initial data $(\eta,\psi)(0)$ and parameter $0<\nu \ll 1$. We claim that $N^{s_0}_{T_{0,\nu}}\le B_\nu$. Otherwise, there exists $t_0\in]0,T_{0,\nu}]$, such that
		\begin{equation*}
			t_0 = \sup \{ t\in]0,T_{0,\nu}[: N^{s_0}_t\le B_\nu \}.
		\end{equation*}
		The continuity-in-time of $(\eta,\psi)$ guarantees that $N^s_{t_0}=B_\nu$. As a result, \eqref{eq-cauchy:energy-esti-gen} gives
		\begin{equation*}
			B_\nu^2 = (N^{s_0}_{t_0})^2 \le C(N^{s-}_0)(N^{s_0}_0)^2 + t_0 C(B_\nu)\left( B_\nu + 1 \right) B_\nu^2 < B_\nu^2,
		\end{equation*}
		which is a contradiction. Now, by passing to the limit $\nu\rightarrow 0+0$, one obtains \eqref{eq-cauchy:energy-esti-gen-strong}. Note that $T_{0,\nu}$ tends to $T_0$, which is a positive smooth function of $N^{s_0}_0$ an $N^{s-}_0$.
	\end{proof}

	\begin{proof}[Proof of Lemma \ref{lem-cauchy:energy-esti-1}]
		We shall prove a general result covering the inequality \eqref{eq-cauchy:energy-esti-1}: for all $u\in L^2$ and $t\in[0,T[$,
		\begin{equation}\label{lem-cauchy-energy-esti-1.1}
			\left| \Real\langle T_V\cdot\bar{\nabla}J_\epsilon u, u \rangle \right| \lesssim C(N^{s-}_t) \|u\|_{L^2}^2.
		\end{equation}
		To begin with, a simple calculus gives
		\begin{align*}
			T_V\cdot\bar{\nabla}J_\epsilon u =& T_{V^z i\xi_z} T_{j_\epsilon} u + T_{V^\theta} T_{\eta^{-1}}\partial_\theta(J_\epsilon u) + T_{V^\theta} \left( T_{\partial_\theta(J_\epsilon u)}\eta^{-1} + R(\eta^{-1}, \partial_\theta(J_\epsilon u)) \right) \\
			=& T_{V^z i\xi_z} T_{j_\epsilon} u + T_{V^\theta} T_{\eta^{-1}} T_{\partial_\theta j_\epsilon} u + T_{V^\theta} T_{\eta^{-1}} T_{j_\epsilon i\xi_\theta} u \\
			& + T_{V^\theta} \left( T_{\partial_\theta(J_\epsilon u)}\eta^{-1} + R(\eta^{-1}, \partial_\theta(J_\epsilon u)) \right),
		\end{align*}
		where the last term can be omitted since, due to Proposition \ref{prop-para:paraprod-bound},
		\begin{align*}
			\| T_{V^\theta} \left( T_{\partial_\theta(J_\epsilon u)}\eta^{-1} + R(\eta^{-1}, \partial_\theta(J_\epsilon u)) \right) \|_{L^2} \lesssim& \|  T_{\partial_\theta(J_\epsilon u)}\eta^{-1} + R(\eta^{-1}, \partial_\theta(J_\epsilon u)) \|_{L^2} \\
			\lesssim& \|\partial_\theta(J_\epsilon u)\|_{H^{-1}} \|\eta^{-1}\|_{H^{2+}_{R^{-1}}} \lesssim C(N^{s-}_t) \|u\|_{L^2}.
		\end{align*}
		Moreover, by Lemma \ref{lem-paralin:sym-deri-in-eta} and the fact that $\eta_\theta\in H^{s-\frac{1}{2}}$, $T_{\partial_\theta j_\epsilon}$ is of order $\le 0$, and thus
		\begin{equation*}
			\| T_{V^\theta} T_{\eta^{-1}} T_{\partial_\theta j_\epsilon} u \|_{L^2} \lesssim C(N^{s-}_t) \|u\|_{L^2}.
		\end{equation*}
		This estimate is independent of $\epsilon$ since $j_\epsilon\in\Sigma^0$ uniformly in $\epsilon \ge 0$ (that is, all the involved estimates are uniform in $\epsilon$). Furthermore, by symbolic calculus (Proposition \ref{prop-para:paradiff-cal-sym}),
		\begin{align*}
			&\| T_{V^z i\xi_z} T_{j_\epsilon} u - T_{V^z j_\epsilon i\xi_z} u \|_{L^2} \lesssim C(N^{s-}_t) \|u\|_{L^2}, \\
			&\| T_{V^\theta} T_{\eta^{-1}} T_{j_\epsilon i\xi_\theta} u - T_{\eta^{-1}V^\theta j_\epsilon i\xi_\theta} u \|_{L^2} \lesssim C(N^{s-}_t) \|u\|_{L^2},
		\end{align*}
		which reduces \label{eq-cauchy-energy-esti-1.1} to 
		\begin{equation*}
			\left| \Real\langle T_{h_\epsilon} u, u \rangle \right| \lesssim C(N^{s-}_t) \|u\|_{L^2}^2,\ \ h_\epsilon = j_\epsilon \left( V^z i\xi_z + \eta^{-1}V^\theta i\xi_\theta \right) \in \Gamma^1_{1+}.
		\end{equation*}
		By Proposition \ref{prop-para:paradiff-cal-sym}, $T_{h_\epsilon}^* = T_{h_\epsilon^*}$, up to some operators of order $0-$, with
		\begin{equation*}
			h_\epsilon^* = \overline{h_\epsilon} + \partial_\xi\cdot D_w\overline{h_\epsilon} = -h_\epsilon + \partial_\xi\cdot D_w\overline{h_\epsilon}.
		\end{equation*}
		Consequently, 
		\begin{align*}
			\Real\langle T_{h_\epsilon} u, u \rangle = \Real\left\langle T_{\frac{h_\epsilon+h_\epsilon^*}{2}}u, u \right\rangle + \frac{1}{2}\Real \left\langle u, (T_{h_\epsilon}^*-T_{h_\epsilon^*}) \right\rangle,
		\end{align*}
		where both terms on the right hand side are bounded by $C(N^{s-}_t) \|u\|_{L^2}^2$.
	\end{proof}

	\subsection{Convergence of approximate solutions and uniqueness}\label{subsect:conv-uni}
	
	In previous sections, we have constructed a sequence of approximate solution $(\eta^\epsilon,\psi^\epsilon)$ solving the approximate system \eqref{eq-cauchy:WW-eps}, and that, for all $\epsilon>0$, $(\eta^\epsilon,\psi^\epsilon)$ is well-defined on time interval $[0,T_0]$ and uniformly (in $\epsilon>0$) bounded. The goal of this section is to check that these approximate solutions form a Cauchy sequence in a weaker topology $L^\infty_{T_1}\mathcal{X}_R^{s_0}$, where $s_0<s-\frac{3}{2}$ and $0<T_1 \le T_0$, which proves the existence in Theorem \ref{thm-intro:main}. As a by-product, one will see that the same argument allows us to compare two different solutions to the system \eqref{eq-paralin:WW} and deduce the uniqueness part of Theorem \ref{thm-intro:main}.
	
	\begin{proposition}\label{prop-cauchy:converg-apprx-sol}
		Under the hypotheses of Proposition \ref{prop-cauchy:WP-apprx-sys}, there exists $0<T_1\le T_0$, where $T_0$ is defined in Corollary \ref{cor-cauchy:low-bd-lifespan}, such that the sequence $\{(\eta^\epsilon,\psi^\epsilon)\}_{\epsilon\in]0,1[}$ is Cauchy in $C_{T_1}\mathcal{X}_R^{s_0}$ with $3/2<s_0<s-3/2$. More precisely, we have
		\begin{equation}\label{eq-cauchy:cauchy-apprx-sol}
			\lim_{\epsilon_2\rightarrow 0+} \sup_{0<\epsilon_1<\epsilon_2<1} \sup_{t\in [0,T_1[} \left( \|\eta^{\epsilon_1}(t)-\eta^{\epsilon_2}(t)\|_{H^{s_0+\frac{1}{2}}} + \|\psi^{\epsilon_1}(t)- \psi^{\epsilon_2}(t)\|_{H^{s_0}} \right) = 0.
		\end{equation}
	\end{proposition}
	
	During the proof of Proposition \eqref{prop-cauchy:converg-apprx-sol}, we emphasize that the appearance of mollifier $J_\epsilon$ has no impact in the estimates and one may delete all the $J_\epsilon$'s to obtain the proof of uniqueness in Theorem \ref{thm-intro:main}.
	
	Recall that, with $0<\epsilon_1<\epsilon_2<1$, the approximate solutions $(\eta^{\epsilon_j},\psi^{\epsilon_j})$ ($j=1,2$) satisfy \eqref{eq-cauchy:WW-eps}, namely
	\begin{equation*}
		\left\{
		\begin{aligned}
			&\left(\partial_t + T_{V^{\epsilon_j}}\cdot\bar{\nabla}^{\epsilon_j}J_{\epsilon_j} + \mathcal{L}^{\epsilon_j}\right)
			\left(\begin{array}{c}
				\eta^{\epsilon_j} \\
				\psi^{\epsilon_j}
			\end{array}\right) = f(J_{\epsilon_j}\eta^{\epsilon_j},J_{\epsilon_j}\psi^{\epsilon_j}), \\
			&(\eta^{\epsilon_j},\psi^{\epsilon_j})|_{t=0} = (\eta_0,\psi_0) \in H^{s+\frac{1}{2}}_R\times H^s,
		\end{aligned}
		\right.
	\end{equation*}
	where $V^{\epsilon_j}$ and $\bar{\nabla^{\epsilon_j}}$ stand for the $V$ and $\bar{\nabla}$ associated to $(\eta^{\epsilon_j},\psi^{\epsilon_j})$, respectively. The difference between approximate solutions
	\begin{equation}\label{eq-cauchy:def-diff-app-sol}
		(\delta\eta,\delta\psi) := (\eta^{\epsilon_2}-\eta^{\epsilon_1}, \psi^{\epsilon_2}-\psi^{\epsilon_1}) \in C_{T_0}(H^{s+\frac{1}{2}}\times H^s)
	\end{equation}
	satisfies
	\begin{equation}\label{eq-cauchy:WW-diff}
		\left\{
		\begin{aligned}
			&\left(\partial_t + T_{V^{\epsilon_2}}\cdot\bar{\nabla}^{\epsilon_2}J_{\epsilon_2} + \mathcal{L}^{\epsilon_2}\right)
			\left(\begin{array}{c}
				\delta\eta \\
				\delta\psi
			\end{array}\right) = h, \\
			&(\delta\eta,\delta\psi)|_{t=0} = 0,
		\end{aligned}
		\right.
	\end{equation}
	where $h$ equals
	\begin{align*}
		h=& h_1 + h_2 - h_3 - h_4 - h_5, \\
		h_1=& f(J_{\epsilon_2}\eta^{\epsilon_2},J_{\epsilon_2}\psi^{\epsilon_2}) - f(J_{\epsilon_2}\eta^{\epsilon_1},J_{\epsilon_2}\psi^{\epsilon_1}), \\
		h_2=& f(J_{\epsilon_2}\eta^{\epsilon_1},J_{\epsilon_2}\psi^{\epsilon_1}) - f(J_{\epsilon_1}\eta^{\epsilon_1},J_{\epsilon_1}\psi^{\epsilon_1}), \\
		h_3=& \left(T_{V^{\epsilon_2}}\cdot\bar{\nabla}^{\epsilon_2} - T_{V^{\epsilon_1}}\cdot\bar{\nabla}^{\epsilon_1}\right)J_{\epsilon_2}
		\left(\begin{array}{c}
			\eta^{\epsilon_1} \\
			\psi^{\epsilon_1}
		\end{array}\right), \\
		h_4=& \left(T_{V^{\epsilon_1}}\cdot\bar{\nabla}^{\epsilon_1}\right)(J_{\epsilon_2}-J_{\epsilon_1})
		\left(\begin{array}{c}
			\eta^{\epsilon_1} \\
			\psi^{\epsilon_1}
		\end{array}\right), \\
		h_5=& \left(\mathcal{L}^{\epsilon_2} - \mathcal{L}^{\epsilon_1}\right)
		\left(\begin{array}{c}
			\eta^{\epsilon_1} \\
			\psi^{\epsilon_1}
		\end{array}\right).
	\end{align*}
	Let us denote by $M$ the upper bound of $L^\infty_{T_0}(H^{s+\frac{1}{2}}_R\times H^{s})$-norm of approximate solutions $(\eta^{\epsilon},\psi^{\epsilon})$, namely
	\begin{equation}\label{eq-cauchy:def-M}
		M := \sup_{\epsilon\in]0,1]} \|(\eta^{\epsilon},\psi^{\epsilon})\|_{L^\infty_{T_0}(H^{s+\frac{1}{2}}_R\times H^{s})},
	\end{equation}
	which is finite due to Corollary \ref{cor-cauchy:low-bd-lifespan}. We claim that
	\begin{lemma}\label{lem-cauchy:lip-sys}
		Under the assumption of Proposition \ref{prop-cauchy:converg-apprx-sol}, we have the following estimates,
		\begin{align}
			&\|h_k(t)\|_{H^{s_0+\frac{1}{2}}\times H^{s_0}} \le C(M)\|(\delta\eta,\delta\psi)(t)\|_{H^{s_0+\frac{1}{2}}\times H^{s_0}},\ \ \forall t\in[0,T_0[,\ k=1,3,5,\label{eq-cauchy:lip-sys-1} \\
			&\|h_k(t)\|_{H^{s_0+\frac{1}{2}}\times H^{s_0}} \le C(M) \left( \|(\delta\eta,\delta\psi)(t)\|_{H^{s_0+\frac{1}{2}}\times H^{s_0}} + \epsilon_2^{\nu} \right),\ \ \forall t\in[0,T_0[,\ k=2,4,\label{eq-cauchy:lip-sys-2}
		\end{align}
		where $0<\nu\ll 1$ is a constant and $C>0$ is an increasing smooth function that does not depend on time $t$ and $\epsilon_j$'s.
	\end{lemma}
	
	Once Lemma \ref{lem-cauchy:lip-sys} is true, one may apply energy estimate \eqref{eq-cauchy:energy-esti-gen-weak}, with $\frac{3}{2}<s_0<s-\frac{3}{2}$ and $c_f = \epsilon_2^\nu\in[0,1]$, to conclude that, for all $t\in[0,T_0[$,
	\begin{equation*}
		\begin{aligned}
			\|(\delta\eta,\delta\psi)(t)\|_{H^{s_0+\frac{1}{2}}\times H^{s_0}}^2 \le& C\left( \|(\eta_0,\psi_0)\|_{H^{s+\frac{1}{2}}_R\times H^{s}} \right) \|(\delta\eta,\delta\psi)(0)\|_{H^{s+\frac{1}{2}\times H^s}}^2 \\
			&+ tC(M)\left(\|(\delta\eta,\delta\psi)\|_{L^\infty_t(H^{s_0+\frac{1}{2}}\times H^{s_0})}^2 + \epsilon_2^{2\nu}\right) \\
			=&tC(M)\left(\|(\delta\eta,\delta\psi)\|_{L^\infty_t(H^{s_0+\frac{1}{2}}\times H^{s_0})}^2 + \epsilon_2^{2\nu}\right),
		\end{aligned}
	\end{equation*}
	where we used the fact that $(\delta\eta,\delta\psi)|_{t=0}=0$. Then, by choosing $T_1>0$ such that $T_1 C(M)<\frac{1}{2}$, we have
	\begin{equation*}
		\|(\delta\eta,\delta\psi)\|_{L^\infty_{T_1}(H^{s_0+\frac{1}{2}}\times H^{s_0})} \le \epsilon_2^{\nu} \rightarrow 0,\ \ \text{as }\epsilon_2\rightarrow 0,
	\end{equation*}
	which completes the proof of \eqref{eq-cauchy:cauchy-apprx-sol}. 
	
	\begin{remark}\label{rmk-cauchy:uniqueness}
		In Lemma \ref{lem-cauchy:lip-sys}, $\epsilon_2^\nu$ appears solely in the estimates concerning $J_{\epsilon_2}-J_{\epsilon_1}$. Thus, in the proof of uniqueness, there is no $\epsilon_2^\nu$ and the above inequality becomes $(\delta\eta,\delta\psi)=0$ for all $t\in[0,T_1[$, which proves the uniqueness (see Proposition \ref{prop-cauchy:contraction}).
	\end{remark}
	
	By Proposition \ref{prop-cauchy:converg-apprx-sol}, the approximate solutions converges to $(\eta,\psi)\in L^\infty_{T_1}\mathcal{X}_R^{s_0}$ which is the unique solution to \eqref{eq-intro:WW}. Besides, from the uniform bound of approximate solutions in $L^\infty_{T_0}\mathcal{X}_R^s$, one may apply an interpolation argument to conclude that this solution $(\eta,\psi)$ has regularity $L^\infty_{T_1}\mathcal{X}_R^{s}$ and $C_{T_1}\mathcal{X}_R^{s-}$. The continuity in time and the continuous dependence in initial data $(\eta_0,\psi_0)$ (w.r.t. the topology of $\mathcal{X}_R^s$) will be given in the next section.
	
	\begin{proof}[Proof of \eqref{eq-cauchy:lip-sys-1}]
		In the term $h_1$, $h_3$, and $h_5$, the difference comes from that of $(\eta^{\epsilon_j},\psi^{\epsilon_j})$. Thus, this estimate is equivalent to boundedness of derivatives in $\eta,\psi$ of
		\begin{equation*}
			f(J_\epsilon\eta,J_\epsilon\psi),\ \ 
			T_{V^{\epsilon}}\cdot\bar{\nabla}^{\epsilon}J_{\epsilon}
			\left(\begin{array}{c}
				\eta^{\epsilon} \\
				\psi^{\epsilon}
			\end{array}\right),\ \ 
			\mathcal{L}^{\epsilon}
			\left(\begin{array}{c}
				\eta^{\epsilon} \\
				\psi^{\epsilon}
			\end{array}\right)
		\end{equation*}
		in $H^{s_0+\frac{1}{2}} \times H^{s_0}$. For the derivative of the first term, if $\epsilon=0$, the desired result has been given in Proposition \ref{prop-paralin:source-Lip}. Otherwise, the following extra terms will appear
		\begin{equation*}
			(\frac{\delta}{\delta\eta}f)(J_\epsilon\eta,J_\epsilon\psi) T_{\delta j_\epsilon}\eta,\ (\frac{\delta}{\delta\psi}f)(J_\epsilon\eta,J_\epsilon\psi) T_{\delta j_\epsilon}\psi,
		\end{equation*}
		where the extra factors $T_{\delta j_\epsilon}\eta$ and $T_{\delta j_\epsilon}\psi$ do not change the desired estimate thanks to Lemma \ref{lem-paralin:source-deri-in-eta}. As for the other two terms, we observe that they are written in terms of finitely many paradifferential operators (with symbol in $\Sigma^m$ for some $m\in\R$ or equal to $B$, $V$) acting on $\eta^\epsilon$ or $\psi^\epsilon$. Thus the desired estimates are no more than a consequence of Lemma \ref{lem-paralin:sym-deri-in-eta} and \ref{lem-paralin:BV-deri-in-eta}. The only case uncovered is when the derivative in $\eta$ acts on $(\eta^\epsilon)^{-1}$ from $\bar{\nabla}^\epsilon$, whose contribution reads
		\begin{equation*}
			-T_{\left(V^{\epsilon}\right)^\theta} \left[ \frac{\delta\eta}{(\eta^\epsilon)^2} \partial_\theta\left( J_{\epsilon}
			\left(\begin{array}{c}
				\eta^{\epsilon} \\
				\psi^{\epsilon}
			\end{array}\right) \right) \right].
		\end{equation*}
		The estimate of this term simply follows from Proposition \ref{prop-para:paraprod-bound} and Corollary \ref{cor-para:product-law}.
	\end{proof}
	
	By Proposition \ref{prop-paralin:source-Lip}, \ref{prop-para:paraprod-bound} and Corollary \ref{cor-para:product-law}, \eqref{eq-cauchy:lip-sys-2} can be reduced to the following lemma
	\begin{lemma}\label{lem-cauchy:lip-molli}
		Under the hypotheses of Proposition \ref{prop-cauchy:converg-apprx-sol}, for all $r\in\R$ and $0<\nu\ll 1$,
		\begin{equation}\label{eq-cauchy:lip-molli}
			\| J_{\epsilon_2} - J_{\epsilon_1} \|_{\mathcal{L}(H^r;H^{r-\frac{3}{2}\nu})} \le C(M) \left( \|(\delta\eta,\delta\psi)\|_{H^{s_0+\frac{1}{2}}\times H^{s_0}} + \epsilon_2^{\nu} \right),
		\end{equation}
		where $C>0$ is an increasing smooth function that does not depend on time $t$ and $\epsilon_j$'s.
	\end{lemma}
	\begin{proof}
		By definition \eqref{eq-cauchy:def-j-eps} of $j_\epsilon$, the principal symbol of $J_{\epsilon_2} - J_{\epsilon_1}$ reads (the subprincipal part can be treated in the same way)
		\begin{equation*}
			\exp(-\epsilon_2 \gamma_2^{(3/2)}) - \exp(-\epsilon_1 \gamma_1^{(3/2)}),
		\end{equation*}
		where $\gamma^{(3/2)}_j$ stands for the $\gamma$ (defined in Proposition \ref{prop-sym:main}) associated to $\eta^{\epsilon_j}$. One may decompose this symbol as
		\begin{equation*}
			\left(\exp(-\epsilon_2 \gamma_2^{(3/2)}) - \exp(-\epsilon_2 \gamma_1^{(3/2)})\right) + \exp(-\epsilon_1 \gamma_1^{(3/2)}) \left(\exp(-(\epsilon_2-\epsilon_1) \gamma_1^{(3/2)}) - 1\right).
		\end{equation*}
		The contribution of the first part is of order $0$ due to Lemma \ref{lem-paralin:sym-deri-in-eta}. For the second part, since $\gamma_1^{(3/2)}\in\Gamma^{\frac{3}{2}}_{3/2+}$, we have, 
		$$ \exp(-\epsilon_1 \gamma_1^{(3/2)}) \frac{\exp(-(\epsilon_2-\epsilon_1) \gamma_1^{(3/2)}) - 1}{(\epsilon_2-\epsilon_1)^{\nu} |\xi|^{\frac{3}{2}\nu}} \in \Gamma^{0}_{3/2+}, $$
		and consequently, 
		$$ (\epsilon_2-\epsilon_1)^{-\nu}\exp(-\epsilon_1 \gamma_1^{(3/2)}) \left(\exp(-(\epsilon_2-\epsilon_1) \gamma_1^{(3/2)}) - 1\right) \in \Gamma^{\frac{3}{2}\nu}_{3/2+}, $$
		which implies the desired estimate \eqref{eq-cauchy:lip-molli}.
	\end{proof}

	\subsection{Continuity in time and initial data}\label{subsect:contin}
	
	In this section, we finish the proof of Theorem \ref{thm-intro:main} and \ref{thm-intro:flow-map} by showing that the unique solution constructed in previous section is continuous in time $t$ and initial data $(\eta_0,\psi_0)$. To achieve this, we shall apply the nonlinear interpolation theorem recently proved in \cite{alazard2024nonlinear}. Recall that, in the statement of Theorem \ref{thm-intro:flow-map}, we have defined a small ball in $H^{s+\frac{1}{2}}\times H^s$, which we recall here
	\begin{equation*}
		B_s(\eta_0,\psi_0;r) := \{v_0=(\tilde{\zeta}_0,\tilde{\psi}_0) \in H^{s+\frac{1}{2}}\times H^s : \|(\eta_0 -R)- \tilde{\zeta}_0\|_{H^{s+\frac{1}{2}}} + \|\psi_0 - \tilde{\psi}_0\|_{H^{s}} < r \}.
	\end{equation*}
	And in previous section, we have proved that the following flow map (already defined in \eqref{eq-intro:def-flow-map}) is well-defined for all $s>3$,
	\begin{equation*}
		\begin{array}{lccc}
			\mathfrak{F} : & B_s\left(\eta_0,\psi_0;r\right) & \rightarrow & L^\infty_T(H^{s+\frac{1}{2}}(\T\times\R)\times H^s(\T\times\R)) \\[0.5ex]
			&(\tilde{\zeta}_0,\tilde{\psi}_0) & \mapsto & \left( \tilde{\eta}(t)-R,\tilde{\psi}(t) \right),
		\end{array}
	\end{equation*}
	where $(\tilde{\eta}(t),\tilde{\psi}(t))$ is the unique solution to \eqref{eq-intro:WW} with initial data $(\tilde{\zeta}_0+R,\tilde{\psi}_0)$. Note that, when $r>0$ is small enough, the life span of these solutions admits a minimum which is denoted by $T$ here. To prove this, it suffices to combine the energy estimate \eqref{eq-cauchy:energy-esti} and the same argument in Corollary \ref{cor-cauchy:low-bd-lifespan}.
	
	Then Theorem 1 of \cite{alazard2024nonlinear} can be stated as,
	\begin{theorem}\label{thm-cauchy:nonlin-interpol}
		Let $s_0,s,s_1,r\in\R$ with $s_0<s<s_1$, $r>0$, and $0<T\ll 1$. Assume that the flow map $\mathfrak{F}$ satisfies 
		
		(1)(contraction) for all $v_0,v_0'\in B_s(\eta_0,\psi_0;r) \cap H^{+\infty}\times H^{+\infty}$,
		\begin{equation}\label{eq-cauchy:nonlin-interpol-contraction}
			\| \mathfrak{F}(v_0) - \mathfrak{F}(v_0') \|_{L^\infty_T(H^{s_0+\frac{1}{2}}\times H^{s_0})} \le C\left( \|v_0\|_{H^{s+\frac{1}{2}}\times H^{s}}+\|v_0'\|_{H^{s+\frac{1}{2}}\times H^{s}}\right) \|v_0-v_0'\|_{H^{s_0+\frac{1}{2}}\times H^{s_0}};
		\end{equation}
		
		(2)(tame estimate) for all $v_0\in B_s(\eta_0,\psi_0;r) \cap H^{+\infty}\times H^{+\infty}$,
		\begin{equation}\label{eq-cauchy:nonlin-interpol-tame}
			\| \mathfrak{F}(v_0)\|_{L^\infty_T(H^{s_1+\frac{1}{2}}\times H^{s_1})} \le C\left( \|v_0\|_{H^{s+\frac{1}{2}}\times H^{s}} \right) \|v_0\|_{H^{s_1+\frac{1}{2}}\times H^{s_1}}.
		\end{equation}
		
		Then $\mathfrak{F}$ is continuous on the ball $B_s(\eta_0,\psi_0;r)$ and $\mathfrak{F}(v_0)\in C_T(H^{s+\frac{1}{2}}\times H^{s})$ for all $v_0\in B_s(\eta_0,\psi_0;r)$.
	\end{theorem}
	\begin{remark}
		This theorem can be generalized to more general case, for instant, Besov spaces. One may refer to Theorem 18 of \cite{alazard2024nonlinear} for the setting.
	\end{remark}

	\begin{proof}[Proof of tame estimate]
		The tame estimate \eqref{eq-cauchy:nonlin-interpol-tame} is a direct result of energy estimate \eqref{eq-cauchy:energy-esti-gen} (replacing $(s_0,s-)$ by $(s_1,s)$). In fact, let us consider any initial data $(\eta_0,\psi_0)$ such that $v_0=(\eta_0-R,\psi_0)\in B_{s_1}(\eta_0,\psi_0;r)$. The energy of the resulting solution $(\eta,\psi)$ is still denoted as
		\begin{equation*}
			N_T^{s_1} = \sup_{t\in[0,T[} \left( \|\eta(t)-R\|_{H^{s_1+\frac{1}{2}}} + \|\psi(t)\|_{H^{s_1}} \right) = \| \mathfrak{F}(v_0)\|_{L^\infty_T(H^{s_1+\frac{1}{2}}\times H^{s_1})}.
		\end{equation*}
		We recall that \eqref{eq-paralin:source-esti} gives
		\begin{equation*}
			\|f\|_{L^\infty_t(H^{s_1+\frac{1}{2}}\times H^{s_1})} \le C\left( \|(\eta,\psi)\|_{L^\infty_t(H^{s+\frac{1}{2}}_R\times H^{s})} \right) \|(\eta,\psi)\|_{L^\infty_t(H^{s_1+\frac{1}{2}}_R\times H^{s_1})},
		\end{equation*}
		which, combined with energy estimates \eqref{eq-cauchy:energy-esti-gen-strong} yields the desired estimate.
	\end{proof}
	
	The contraction condition \eqref{eq-cauchy:nonlin-interpol-contraction} can be obtained directly from the following proposition,
	\begin{proposition}\label{prop-cauchy:contraction}
		Given $s>3$, we take arbitrary $(\eta_0,\psi_0),(\eta_0',\psi_0')\in\mathcal{X}_R^s$. Let $(\eta,\psi),(\eta',\psi')\in L^\infty_T\mathcal{X}_R^s$ be the solutions to \eqref{eq-intro:WW} with initial data $(\eta_0,\psi_0),(\eta_0',\psi_0')$, respectively. Then, when $T\ll 1$, the following estimate holds:
		\begin{equation}\label{eq-cauchy:contraction}
			\|(\eta-\eta',\psi-\psi')\|_{L^\infty_T(H^{s_0+\frac{1}{2}}\times H^{s_0})} \le C\left( M \right) \|(\eta_0-\eta'_0,\psi_0-\psi'_0)\|_{H^{s_0+\frac{1}{2}}\times H^{s_0}},
		\end{equation}
		where $\frac{3}{2}<s_0<s-\frac{3}{2}$, and $M>0$ is the maximum of $\|(\eta_0,\psi_0)\|_{L^\infty_T(H^{s+\frac{1}{2}}_R\times H^{s})}$ and $\|(\eta'_0,\psi'_0)\|_{L^\infty_T(H^{s+\frac{1}{2}}_R\times H^{s})}$.
	\end{proposition}
	\begin{proof}
		We shall only give the sketch of the proof, since the argument is the same as Proposition \ref{prop-cauchy:converg-apprx-sol}. By calculating the difference between the equations for $(\eta,\psi)$ and $(\eta',\psi')$, we are able to write the equation for
		\begin{equation*}
			(\delta\eta,\delta\psi) := (\eta-\eta',\psi-\psi') \in L^\infty_T(H^{s_0+\frac{1}{2}}\times H^{s_0}).
		\end{equation*}
		More precisely, 
		\begin{equation*}
			\left\{
			\begin{aligned}
				&\left(\partial_t + T_{V}\cdot\bar{\nabla} + \mathcal{L}\right)
				\left(\begin{array}{c}
					\delta\eta \\
					\delta\psi
				\end{array}\right) = g, \\
				&(\delta\eta,\delta\psi)|_{t=0} = (\eta_0-\eta_0',\psi_0-\psi_0') \in H^{s+\frac{1}{2}}\times H^s,
			\end{aligned}
			\right.
		\end{equation*}
		where the source term $g= g_1 - g_2 - g_3$ is equal to
		\begin{equation*}
			g_1= f(\eta,\psi) - f(\eta',\psi'),\ \ g_2= \left(T_{V}\cdot\bar{\nabla} - T_{V'}\cdot\bar{\nabla}'\right)
			\left(\begin{array}{c}
				\eta' \\
				\psi'
			\end{array}\right),\ \ g_3= \left(\mathcal{L} - \mathcal{L}'\right)
			\left(\begin{array}{c}
				\eta' \\
				\psi'
			\end{array}\right),
		\end{equation*}
		and $V'$, $\bar{\nabla}'$, $\mathcal{L}'$ are associated to $(\eta',\psi')$. By using the same proof as for \eqref{eq-cauchy:lip-sys-1}, one can show that,
		\begin{equation*}
			\begin{aligned}
				\|g\|_{L^\infty_t(H^{s_0+\frac{1}{2}}\times H^{s_0})} \le& C\left(\|(\eta,\psi)\|_{L^\infty_t(H^{s+\frac{1}{2}}_R\times H^{s})}+\|(\eta',\psi')\|_{L^\infty_t(H^{s+\frac{1}{2}}_R\times H^{s})}\right) \\
				&\hspace{14em}\times \|(\delta\eta,\delta\psi)(t)\|_{H^{s_0+\frac{1}{2}}\times H^{s_0}}.
			\end{aligned}
		\end{equation*}
		Then one may apply energy estimate \eqref{eq-cauchy:energy-esti-gen} with $s_0$ and obtain
		\begin{equation*}
			\|(\delta\eta,\delta\psi)\|_{L^\infty_T(H^{s_0+\frac{1}{2}}\times H^{s_0})} \le C(M) \|(\delta\eta,\delta\psi)(0)\|_{H^{s_0+\frac{1}{2}}\times H^{s_0}} + TC(M) \|(\delta\eta,\delta\psi)\|_{L^\infty_T(H^{s_0+\frac{1}{2}}\times H^{s_0})}.
		\end{equation*}
		Then the desired result follows by choosing $TC(M)<\frac{1}{2}$.
	\end{proof}

	\appendix
	
	\section{Fractional Sobolev spaces in regular domains}\label{App:sobo}
	
	In this appendix, we give a brief review of the fractional Sobolev spaces defined on regular domain $\Omega$. We focus on four types of domain: (1) Euclidean spaces, torus, and their products, namely $\mathbb{T}^{d_1}\times\mathbb{R}^{d_2}$ with $d_1,d_2\in\N$; (2) half spaces $\R^d_+:=\{(x',d_d)\in\R^{d-1}\times\R,x_d>0\}$; (3) bounded subdomains of $\mathbb{T}^{d_1}\times\mathbb{R}^{d_2}$ with smooth boundary; (4) cylindrical domains $\Omega_0\times\R$, where $\Omega_0$ is a bounded subdomain of $\mathbb{T}^{d_1}\times\mathbb{R}^{d_2}$ with smooth boundary. For simplicity, the properties below will be stated only for $\R^d$ and its subdomains, while the same results holds for $\mathbb{T}^{d_1}\times\mathbb{R}^{d_2}$ and its subdomains. We mainly refer to \cite{lions2012non,demengel2012functional,adams2003sobolev} for the detailed demonstrations.
	
	\begin{definition}\label{def-sobo:basic}
		Let $s\in\mathbb{R}$. When $s\in\mathbb{N}$, $H^s(\Omega)$ is the collection of distributions $u$ with
		\begin{equation*}
			\|u\|_{H^s(\Omega)}^2 := \sum_{|\alpha|\leqslant s} \|\partial^\alpha u\|_{L^2(\Omega)}^2 < +\infty.
		\end{equation*}
		When $s = n+\sigma$ with $n\in\mathbb{N}$ and $\sigma\in]0,1[$, the space $H^s(\Omega)$ is defined by interpolation (see Section 2.1 of \cite{lions2012non} for the definition and \cite{lions1960construction,lions1961proprietes} for general theory of interpolation spaces)
		\begin{equation*}
			H^s(\Omega) = [H^{n}(\Omega), H^{n+1}(\Omega)]_\sigma.
		\end{equation*}
		It is clear that the spaces defined above are Hilbert spaces, which allows us to define $H^s(\Omega)$ with $s<0$ by duality
		\begin{equation*}
			H^{s}(\Omega) = \left(H^{|s|}_0(\Omega)\right)',
		\end{equation*}
		where $H^{|s|}_0(\Omega)$ is the closure of $C^\infty_c(\Omega)$ w.r.t. the $H^{|s|}(\Omega)$-norm.
	\end{definition}
	
	In the case of full space $\R^d$, the Sobolev space can be characterized by Fourier transform.
	\begin{proposition}\label{prop-sobo:def-fourier}
		If $\Omega = \mathbb{R}^d$, for all $s\in\R$, we have the equivalence
		\begin{equation*}
			\|u\|_{H^s(\Omega)}^2 \sim \int_{\mathbb{R}^d} \langle\xi\rangle^{2s} \left|\hat{u}(\xi) \right|^2 d\xi.
		\end{equation*}
	\end{proposition}
	
	From Definition \ref{def-sobo:basic}, one can deduce the boundedness of restriction operator.
	\begin{proposition}\label{prop-sobo:restri}
		For all $s\in\R$, the restriction operator
		\begin{equation*}
			\begin{array}{cccc}
				\mathcal{R} : & H^s(\R^d) & \rightarrow & H^s(\Omega) \\
				& u & \mapsto & u|_{\Omega}
			\end{array}
		\end{equation*}
		is bounded.
	\end{proposition}
	In the mean time, any function in $H^s(\Omega)$ can be continuously extended to $\R^d$.
	\begin{proposition}\label{prop-sobo:extension}
		There exists an extension operator
		\begin{equation*}
			\begin{array}{cccc}
				\mathcal{E} : & H^s(\Omega) & \rightarrow & H^s(\R^d) \\
				& u & \mapsto & \tilde{u},
			\end{array}
		\end{equation*}
		such that $\tilde{u}|_{\Omega} = u$ and $\mathcal{E}\in\mathcal{L}(H^s(\Omega);H^s(\R^d))$ for all $s\in\R$.
	\end{proposition}
	The construction of operator $\mathcal{E}$ is not unique and we refer to Section 5.1 of \cite{agranovich2015sobolev} for one possible construction (see also Chapter 5 of \cite{stein1970singular} for non-negative index $s$). Note that, in these references, the extension operator is constructed for bounded domains and half-space. As for the case of cylindrical domains $\Omega=\Omega_0\times\R \subset\R^d$ (where $\Omega_0$ is bounded domain), it suffices to fix a finite cover $\{U_j\}$ of $\Omega_0$ such that $U_j\cap\Omega_0$ is diffeomorphic to some subset of half-space $\R^{d-1}_+$. Thus $U_j\times\R \cap\Omega$ is diffeomorphic to half-space $\R^d_+$ and the problem is reduced to the case of half-space. The boundedness of restriction and extension operator guarantees the following equivalence
	\begin{proposition}\label{prop-sobo:def-extension}
		For all $s\in\mathbb{R}$, we have the equivalence
		\begin{equation*}
			\|u\|_{H^s(\Omega)} \sim \inf_{\tilde{u}|_{\Omega}=u} \|\tilde{u}\|_{H^s(\mathbb{R}^d)}.
		\end{equation*}
	\end{proposition}
	As a corollary, Sobolev space $H^s(\Omega)$ is invariant by multiplication with $C_b^\infty(\Omega)$ functions.
	\begin{corollary}\label{cor-sobo:prod-smooth-fct}
		Let $W\in C_b^\infty(\Omega)$. Then for all $s\in\R$, we have the following estimate,
		\begin{equation}\label{eq-sobo:prod-smooth-fct}
			\|Wu\|_{H^s(\Omega)} \le C \|W\|_{C^M(\Omega)} \|u\|_{H^s(\Omega)},
		\end{equation}
		where $C>0$ and $M\in\N$ are constants depending on $s$ and dimension $d$ and
		\begin{equation*}
			\|W\|_{C^M(\Omega)} = \sup_{|\alpha|\le M} \|\partial^\alpha W\|_{L^\infty(\Omega)}.
		\end{equation*}
	\end{corollary}

	\begin{proposition}\label{prop-sobo:chgt-of-var}
		Let $s\in\mathbb{R}$ and $\Omega,\Omega'\subset\mathbb{R}^d$ be two subdomains. $\chi:\Omega'\rightarrow\Omega$ is a smooth diffeomorphism such that $\chi$ and $\chi^{-1}$ are bounded as well as their derivatives. Then for all $u\in H^s(\Omega)$,
		\begin{equation*}
			\|u\|_{H^s(\Omega)} \sim \|u\circ\chi\|_{H^s(\Omega')}.
		\end{equation*}
	\end{proposition}
	We refer to Section 12.9 of \cite{lions2012non} for a proof of this invariance by change of coordinate (see also Chapter 3 of \cite{adams2003sobolev}). More precisely, we have the following refined version.
	\begin{proposition}\label{prop-sobo:chgt-of-var-refine}
		Let $s\in\mathbb{R}$ with $s>\frac{d}{2}+1$. $\chi:\R^d\rightarrow\R^d$ is a diffeomorphism such that $(\chi - id)$ or $(\chi^{-1}-id)\in H^{s}$. Then for all $u\in H^{s_0}(\R^d)$, $0\le s_0 \le s$,
		\begin{equation*}
			\|u\circ\chi\|_{H^{s_0}(\R^d)} \sim \|u\|_{H^{s_0}(\R^d)}.
		\end{equation*}
	\end{proposition}
	We refer to Theorem 1.1 of \cite{inci2013regularity} for the proof of this proposition (see also \cite{bruverus2017completeness}). Local result can also be obtained via paracomposition (see Proposition \ref{prop-para:paracomp-exist}).

	\section{Proof of elliptic regularity}\label{App:ellip}
	
	This appendix serves as a proof of Proposition \ref{prop-pre:ellip-reg} and Lemma \ref{lem-paralin:reg-of-pot}. The former one is no more than a special case of the following proposition,
	\begin{proposition}\label{prop-ellip:ellip-alt-reg}
		Let $\eta\in H^{s+\frac{1}{2}-}_R(\T\times\R)$ and $F\in H^{s_0-\frac{3}{2}}(\D\times\R)$ with s>$\frac{3}{2}$ and $\frac{1}{2}\le s_0 \le s$. Then the equation
		\begin{equation}\label{eq-ellip:ellip-alt}
			\left\{\begin{array}{l}
				\Delta_g\varphi = F \in H^{s_0-\frac{3}{2}}, \\[0.5ex]
				\varphi|_{\rho=1} = 0,
			\end{array}\right.
		\end{equation}
		admits a unique solution $\varphi$ in $H^{s_0+\frac{1}{2}}(\D\times\R)$ with
		\begin{equation}\label{eq-ellip:ellip-alt-reg}
			\|\varphi\|_{H^{s_0+\frac{1}{2}}(\D\times\R)} \le C\left(\|\eta\|_{H_R^{s+\frac{1}{2}-}(\T\times\R)}\right) \|F\|_{H^{s_0-\frac{3}{2}}(\D\times\R)}.
		\end{equation}
		where $C>0$ is an increasing smooth function.
	\end{proposition}
	
	The idea of the proof of Proposition \ref{prop-ellip:ellip-alt-reg} comes from \cite{lannes2013water}, Chapter 2, where the author focus on the water-wave with a bottom, the regularity of which is the same as interface. In this case, the boundary condition at bottom should also be considered, which does not exist in our problem. For axis-symmetric jets, one may find a similar proof in \cite{huang2023wellposedness}.
	
	Note that the $s_0=\frac{1}{2}$ case can be proved by classical arguments via Lions–Lax–Milgram theorem (see Section \ref{subsect:ellip-reg}), which gives a unique solution. Therefore, for general case $s_0>\frac{1}{2}$, it remains to show that this solution is regular in the sense of \eqref{eq-ellip:ellip-alt-reg}. We shall apply an iteration in $s_0$. More precisely, we claim that
	\begin{equation}\label{eq-ellip:claim-iter}
		\text{If \eqref{eq-ellip:ellip-alt-reg} is true for all $s_0 \in\left[\frac{1}{2},s_1\right]$ with $s_1\le s-\delta$, it also holds for $s_0=s_1+\delta$.}
	\end{equation}
	Here $\delta>0$ is a small constant depending only on $s$, the choice of which will be precised during the proof. From now on, we assume the assumption in \eqref{eq-ellip:claim-iter} to be correct, and turn to prove \eqref{eq-ellip:ellip-alt-reg} with $s_0=s_1+\delta$.
	
	To do so, we first show that it is possible to add a cut-off before $\varphi$ via a commutator estimate, which allows us to study the equation \eqref{eq-ellip:ellip-alt} in interior part $\{\rho<\frac{1}{2}\}$ and boundary part $\{\frac{1}{2}<\rho<1\}$. 
	
	\subsection{Localization}\label{subsect:ellip-localization}
	Let $\chi\in C_b^\infty(\R^3)$ be any smooth truncation. Then \eqref{eq-ellip:ellip-alt} can be localized as
	\begin{equation}\label{eq-ellip:ellip-alt-loc}
		\left\{\begin{array}{l}
			\Delta_g(\chi\varphi) = F_1  :=\chi F + [\Delta_g,\chi]\varphi, \\[0.5ex]
			(\chi\varphi)|_{\rho=1} = 0.
		\end{array}\right.
	\end{equation}
	\begin{lemma}\label{lem-ellip:commu-loc}
		If $\chi\in C_b^\infty(\R^3)$ and $\eta\in H_R^{s+\frac{1}{2}-}(\T\times\R)$ with $s>\frac{3}{2}$, for all  $0\le \sigma\le s$, we have
		\begin{equation}\label{eq-ellip:commu-loc}
			\|[\Delta_g,\chi]w\|_{H^{\sigma-1}(\D\times\R)} \leqslant C\left(\|\eta\|_{H_R^{s+\frac{1}{2}-}(\T\times\R)}\right) \|w\|_{H^{\sigma}(\D\times\R)},\ \ \forall w\in H^{\sigma}(\D\times\R).
		\end{equation}
	\end{lemma}
	\begin{proof}
		A simple calculus gives that
		\begin{equation*}
			[\Delta_g,\chi]w = \partial_\alpha\chi g^{\alpha\beta} \partial_{\beta}w + \frac{1}{\sqrt{g}}\partial_\alpha\left( g^{\alpha\beta}\sqrt{g} \partial_\beta\chi w \right).
		\end{equation*}
		Recall that, from the construction of $(g^{\alpha\beta}) = J^{-1}J^{-T}$ and $g$ by \eqref{eq-pre:jacobian-mat-inverse} and \eqref{eq-pre:jacobian-det}, respectively, we have $g^{\alpha\beta},g\in H^{s-}(\D\times\R)$ up to $C_b^\infty(\D\times\R)$ normalizations, provided that $\eta\in H_R^{s+\frac{1}{2}-}(\T\times\R)$ (see Proposition \ref{prop-pre:chgt-of-var}). As a result, $\partial_\alpha\chi g^{\alpha\beta}, \frac{1}{\sqrt{g}}, g^{\alpha\beta}\sqrt{g} \partial_\beta\chi$ belong to $H^{s-}(\D\times\R)$ up to $C_b^\infty(\D\times\R)$ normalizations, while $w\in H^{\sigma}(\D\times\R)$ and $w\in H^{\sigma}(\D\times\R)$. Since $s>3/2$ and $\sigma\le s$, by Corollary \ref{cor-sobo:prod-smooth-fct} and \ref{cor-para:product-law}, each term on the right hand side is linear in $w$ and belongs to $H^{\sigma-1}(\D\times\R)$. Note that the estimate for product, Corollary \ref{cor-para:product-law} is stated for Euclidean spaces, which also holds for domains such as $\D\times\R$ via Proposition \ref{prop-sobo:def-extension}.
	\end{proof}
	As a result, under the hypotheses of Proposition \ref{prop-ellip:ellip-alt-reg}, the new source term $F_1$ verifies
	\begin{equation}\label{eq-ellip:commu-source}
		\begin{aligned}
			\|F_1\|_{H^{s_0-\frac{3}{2}}(\D\times\R)} \lesssim& \|F\|_{H^{s_0-\frac{3}{2}}(\D\times\R)} + \|\varphi\|_{H^{s_0-\frac{1}{2}}(\D\times\R)} \\
			\lesssim& \|F\|_{H^{s_0-\frac{3}{2}}(\D\times\R)} + \|\varphi\|_{H^{s_1}(\D\times\R)} \le C\left(\|\eta\|_{H_R^{s+\frac{1}{2}-}(\T\times\R)}\right) \|F\|_{H^{s_0-\frac{3}{2}}(\T\times\R)},
		\end{aligned}
	\end{equation}
	once $\delta>0$ is chosen to be strictly smaller than $1/2$, i.e. $s_0-1/2 = s_1+\delta-1/2 \le s_1$.
	
	In the following, we shall take $\chi$ to be smooth truncations near $\{\rho=0\}$ and $\{\rho=1\}$, respectively, and show \eqref{eq-ellip:ellip-alt-reg} with $s_0=s_1+\delta$ and $\varphi$ replaced by $\chi\varphi$.

	\subsection{Interior regularity}\label{subsect:ellip-int-reg}
	Let us fix $\chi\in C_b^\infty(\R^3)$ a smooth truncation near $\{\rho<1/2\}$. The goal of this paragraph is to prove the regularity of $\chi\varphi$ and deduce \eqref{eq-ellip:ellip-alt-reg} with $s_0=s_1+\delta \in ]1/2,s]$ and $\varphi$ replaced by $\chi\varphi$. With such localization, it is harmless to extend the boundary value problem \eqref{eq-ellip:ellip-alt-loc} to $\R^3$ by zero extension away from $\D\times\R$, since $\chi\varphi$ vanishes in a neighborhood of $\{\rho<1/2\}$.
	
	Let $\Lambda = \langle D_{y,z} \rangle$ be the Fourier multiplier of symbol $\langle \tilde{\xi} \rangle$, where $\tilde{\xi}\in\R^3$ is the Fourier variable associated to $(y,z)\in\R^3$. Now, we apply $\Lambda^\delta$ to both sides of \eqref{eq-ellip:ellip-alt-loc} and obtain
	\begin{equation}\label{eq-ellip:ellip-alt-loc-int}
		\left\{\begin{array}{l}
			\Delta_g \Lambda^\delta(\chi\varphi) = F_2  := \Lambda^\delta F_1 - [\Lambda^\delta,\Delta_g](\chi\varphi), \\[0.5ex]
			\Lambda^\delta(\chi\varphi)|_{\rho=1} = 0.
		\end{array}\right.
	\end{equation}
	\begin{lemma}\label{lem-ellip:commu-reg-int}
		Let $\eta\in H_R^{s+\frac{1}{2}-}(\T\times\R)$ and $\varphi\in H^{s_1+\frac{1}{2}}(\D\times\R)$ with $\frac{1}{2}\le s_1 \le s-\delta$ and $s>\frac{3}{2}$. Then, when $\delta>0$ is small enough, we have
		\begin{equation}\label{eq-ellip:commu-reg-int}
			\|[\Lambda^\delta,\Delta_g](\chi\varphi)\|_{H^{s_1-\frac{3}{2}}(\R^3)} \le C\left(\|\eta\|_{H_R^{s+\frac{1}{2}-}(\T\times\R)}\right) \|\varphi\|_{H^{s_1+\frac{1}{2}}(\D\times\R)}.
		\end{equation}
		Note that $\chi\in C^\infty_c(\R^3)$ is supported near $\{\rho<1/2\}$ and thus $\chi\varphi$ can be extended to $\R^3$ by taking zero value outside $\D\times\R$.
	\end{lemma}
	
	Once this lemma holds true, by using the assumption of iteration (\eqref{eq-ellip:ellip-alt-reg} holds for $s_0\in[1/2,s_1]$ and $F\in H^{s_0-\frac{3}{2}}(\D\times\R)$ with $s_0=s_1+\delta$), we have
	\begin{align*}
		\|F_2\|_{H^{s_1-\frac{3}{2}}(\D\times\R)} \le& \|\Lambda^\delta F_1\|_{H^{s_1-\frac{3}{2}}(\D\times\R)} + \|[\Lambda^\delta,\Delta_g](\chi\varphi)\|_{H^{s_1-\frac{3}{2}}(\D\times\R)} \\
		\lesssim& \|F_1\|_{H^{s_0-\frac{3}{2}}(\D\times\R)} + C\left(\|\eta\|_{H_R^{s+\frac{1}{2}-}(\T\times\R)}\right) \|\varphi\|_{H^{s_1+\frac{1}{2}}(\D\times\R)} \\
		\lesssim& C\left(\|\eta\|_{H_R^{s+\frac{1}{2}-}(\T\times\R)}\right) \|F\|_{H^{s_0-\frac{3}{2}}(\D\times\R)},
	\end{align*}
	where the last inequality can be seen from estimate \eqref{eq-ellip:commu-source} as well as \eqref{eq-ellip:ellip-alt-reg} with $s_0$ replaced by $s_1$. Then, for equation \eqref{eq-ellip:ellip-alt-loc-int}, the assumption of iteration ensures that $\Lambda^\delta(\chi\varphi)\in H^{s_1+\frac{3}{2}}(\D\times\R)$. Furthermore, the desires result $\chi\varphi\in H^{s_0+\frac{3}{2}}(\D\times\R)$ follows from the elliptic regularity of $\Lambda^\delta$ in $\R^3$ and Proposition \ref{prop-sobo:def-extension}.
	\begin{proof}[Proof of Lemma \ref{lem-ellip:commu-reg-int}]
		By definition, $\Lambda^\delta$ is commutable with derivatives, then
		\begin{equation*}
			[\Lambda^\delta,\Delta_g](\chi\varphi) = \left[\Lambda^\delta,\frac{1}{\sqrt{g}}\right]\partial_\alpha\left(g^{\alpha\beta}\sqrt{g}\partial_\beta(\chi\varphi)\right) + \frac{1}{\sqrt{g}}\partial_\alpha\left(\left[\Lambda^\delta,g^{\alpha\beta}\sqrt{g}\right]\partial_\beta(\chi\varphi)\right).
		\end{equation*}
		From $\varphi\in H^{s_1+\frac{1}{2}}(\D\times\R)$, where $s_1 \le s-\delta$, and $g^{\alpha\beta},\sqrt{g} \in H^{s-}(\D\times\R)$ up to $C^\infty_b$ normalization, it is easy to see that $\partial_\alpha\left(g^{\alpha\beta}\sqrt{g}\partial_\beta(\chi\varphi)\right) \in H^{s_1-\frac{3}{2}}(\R^3)$ with $s_1-3/2 \le (s-)-\delta$ due to Corollary \ref{cor-para:product-law}. Thus by Corollary \ref{cor-para:paraprod-commu-bound} and the fact that $1/\sqrt{g} \in H^{s-}(\D\times\R)$, we have
		\begin{equation*}
			\left[\Lambda^\delta,\frac{1}{\sqrt{g}}\right]\partial_\alpha\left(g^{\alpha\beta}\sqrt{g}\partial_\beta(\chi\varphi)\right) \in H^{s_1-\frac{3}{2}}(\R^3),
		\end{equation*}
		for some small $\delta>0$. Similarly, by noticing that $\partial_\beta(\chi\varphi) \in H^{s_1-\frac{1}{2}}$ where $s_1-\frac{1}{2} \le s-1-\delta <(s-)-\delta$ and that $g^{\alpha\beta}\sqrt{g} \in H^{s-}(\D\times\R)$ up to $C^\infty_b$ normalization, we apply again Corollary \ref{cor-para:paraprod-commu-bound} and obtain that
		\begin{equation*}
			\left[\Lambda^\delta,g^{\alpha\beta}\sqrt{g}\right]\partial_\beta(\chi\varphi) \in H^{s_1-\frac{1}{2}}(\R^3)
		\end{equation*}
		and the desired estimate follows due to Corollary \ref{cor-para:product-law}.
	\end{proof}

	\subsection{Boundary regularity}\label{subsect:ellip-bdy-reg}
	The proof of the case where $\chi$ is a smooth truncation near $\{\rho=1\}$ is similar to the interior case. The main difference is that, to maintain the boundary condition, it is impossible to apply $\Lambda^\delta$, which involves normal derivatives at boundary. Instead, we introduce tangential multiplier $\Lambda_0^\delta$. Let $\langle D_{\theta,z} \rangle$ be a Fourier multiplier on $\T\times\R$ with symbol $\langle \xi \rangle$, where $\xi$ is the Fourier variable associated to $(\theta,z)\in\T\times\R$. Then, with polar coordinate $\kappa:(\rho,\theta,z)\mapsto(y,z)$, we can define $\Lambda_0^\delta$ as
	\begin{equation*}
		\Lambda_0^\delta f := \left( \langle D_{\theta,z} \rangle^\delta f\circ\kappa\right) \circ \kappa^{-1},
	\end{equation*}
	for functions $f$ defined on $\D\times\R$. Note that, for any $0<\delta\ll 1$, $\varphi|_{\rho=1}=0$ implies $\Lambda_0^\delta\varphi|_{\rho=1}=0$, and thus \eqref{eq-ellip:ellip-alt-loc-int} becomes
	\begin{equation}\label{eq-ellip:ellip-alt-loc-bdy}
		\left\{\begin{array}{l}
			\Delta_g \Lambda_0^\delta(\chi\varphi) = F_3  := \Lambda_0^\delta F_1 - [\Lambda_0^\delta,\Delta_g](\chi\varphi), \\[0.5ex]
			\Lambda_0^\delta(\chi\varphi)|_{\rho=1} = 0,
		\end{array}\right.
	\end{equation}
	where the source term $F_3$ can be estimated as $F_2$ in interior case,
	\begin{align*}
		\|F_3\|_{H^{s_1-\frac{3}{2}}(\D\times\R)} \le& \|\Lambda_0^\delta F_1\|_{H^{s_1-\frac{3}{2}}(\D\times\R)} + \|[\Lambda_0^\delta,\Delta_g](\chi\varphi)\|_{H^{s_1-\frac{3}{2}}(\D\times\R)} \\
		\lesssim& \|F_1\|_{H^{s_0-\frac{3}{2}}(\D\times\R)} + C\left(\|\eta\|_{H_R^{s+\frac{1}{2}-}(\T\times\R)}\right) \|\varphi\|_{H^{s_1+\frac{1}{2}}(\D\times\R)} \\
		\lesssim& C\left(\|\eta\|_{H_R^{s+\frac{1}{2}-}(\T\times\R)}\right) \|F\|_{H^{s_0-\frac{3}{2}}(\D\times\R)},
	\end{align*}
	where the second inequality follows from Lemma \ref{lem-ellip:commu-reg-int} with $\Lambda$ replaced by $\Lambda_0$, which is possible since Corollary \ref{cor-para:paraprod-commu-bound}, the main step of Lemma \ref{lem-ellip:commu-reg-int}, is also stated for multipliers independent of the normal variable. As a result, the assumption of iteration ensures that $\Lambda_0^\delta(\chi\varphi) \in H^{s_1+\frac{1}{2}}$.
	
	Now we prove the regularity in normal direction. Recall that the coefficients of $\Delta_g$ lie in $H^{s-}(\D\times\R)$ (2-order terms) or $H^{s-1-}(\D\times\R)$ (1-order terms), up to $C^\infty_b(\D\times\R)$ normalizations. Then, in polar coordinate, $\Delta_g$ can be written as
	\begin{equation*}
		\alpha \partial_\rho^2 + \beta\cdot\nabla_{\theta,z}\partial_\rho + \gamma\partial_\rho + R_2 + R_1,
	\end{equation*}
	where $R_j$ is a $j$-order differential operator in $\theta,z$ with coefficients in $H^{s+j-2-}(\D\times\R)$, $\alpha,\beta\in H^{s-}(\D\times\R)$, and $\gamma\in H^{s-1-}(\D\times\R)$, up to $C^\infty_b(\D\times\R)$ normalizations. Moreover, $\alpha$ is strictly positive and its lower bound depends only on $c_0,C_0$ from \eqref{hyp-intro:bounds} and $\|\eta\|_{H_R^{s+\frac{1}{2}-}}$. Therefore, Corollary \ref{cor-para:product-law} and Proposition \ref{prop-para:paralin} gives that
	\begin{equation*}
		\partial_\rho^2 (\chi\varphi) = - \frac{\beta}{\alpha}\cdot\nabla_{\theta,z}\partial_\rho(\chi\varphi) - \frac{\gamma}{\alpha}\partial_\rho(\chi\varphi) - \frac{1}{\alpha}(R_2+R_1)(\chi\varphi) \in H^{s_1+\delta-\frac{3}{2}}(\D\times\R),
	\end{equation*}
	where we use the following regularities
	\begin{equation*}
		\nabla_{\theta,z}\partial_\rho(\chi\varphi), \nabla_{\theta,z}^2(\chi\varphi) \in H^{s_1+\delta-\frac{3}{2}}(\D\times\R),\ \ \nabla_{\rho,\theta,z}(\chi\varphi) \in H^{s_1+\delta-\frac{1}{2}}(\D\times\R).
	\end{equation*}
	Till now, we have showed $\nabla_{\rho,\theta,z}\partial_\rho(\chi\varphi) \in H^{s_1+\delta-\frac{3}{2}}(\D\times\R)$, i.e. $\partial_\rho(\chi\varphi) \in H^{s_1+\delta-\frac{1}{2}}(\D\times\R)$, which gives $\chi\varphi \in H^{s_1+\delta+\frac{1}{2}}(\D\times\R) = H^{s_0+\frac{1}{2}}(\D\times\R)$ due to $\nabla_{\theta,z}(\chi\varphi) \in H^{s_1+\delta-\frac{1}{2}}(\D\times\R)$. And the proof of Proposition \ref{prop-ellip:ellip-alt-reg} is finished.
	
	\subsection{Proof of Lemma \ref{lem-paralin:reg-of-pot}}\label{subsect:ellip-alt-reg}
	
	Recall that the goal of Lemma \ref{lem-paralin:reg-of-pot} is to show that, after a change of variable 
	\begin{equation*}
		(\bar{\rho},\bar{\theta},\bar{z}) = \bar{\iota}(\rho,\theta,z) := \left( \frac{\rho\zeta(\rho\theta,z)}{\eta(\theta,z)},\theta,z \right),
	\end{equation*}
	the new scalar potential $\bar{\varphi} := \varphi\circ \bar{\iota}^{-1}$ satisfies \eqref{eq-paralin:reg-of-pot-loc}. The proof is divided into four steps. Firstly, we check that $\bar{\iota}-id$ has $H^{s+\frac{1}{2}-}$-regularity away from $\rho=0$ (so as $\bar{\iota}^{-1}-id$). Secondly, we adopt Proposition \ref{prop-sobo:chgt-of-var-refine} to prove that $\bar{\varphi}$ belongs to $H^{s_0+\frac{1}{2}}([1-\bar{\delta},1]\times\T\times\R)$ for some $0<\bar{\delta} \ll 1$ where $s_0$ can be taken as any value in $]3/2,s[$. In the third step, we complete the case of $s_0=s$ by reducing the problem to Proposition \ref{prop-ellip:ellip-alt-reg}. The last step is the application of trace theorem so that one can obtain the continuity in $\bar{\rho}$. During the proof, we identify $(y,z)\in\D\times\R$ with polar coordinate $(\rho,\theta,z)\in[0,1]\times\T\times\R$, if the information near the axis $\{\rho=0\}$ is not important.
	
	\paragraph{Step 1: Regularity of $\bar{\iota}$} 
	\begin{lemma}\label{lem-ellip:reg-chgt-of-var}
		Let $\eta\in H^{s+\frac{1}{2}-}_R(\T\times\R)$ with $s>\frac{3}{2}$. Then $\bar{\iota}$ defined in \eqref{eq-paralin:alt-var} is a diffeomorphism from $[0,1]\times\T\times\R$ to $[0,1]\times\T\times\R$. Moreover, there exists $0<\delta,\bar{\delta}\ll 1$ such that $\bar{\iota} - id \in H^{s+\frac{1}{2}-}([1-\delta,1]\times\T\times\R)$ and $\bar{\iota}^{-1} - id \in H^{s+\frac{1}{2}-}([1-\bar{\delta},1]\times\T\times\R)$.
	\end{lemma}
	\begin{proof}
		By definition \eqref{eq-paralin:alt-var} of $\bar{\iota}$, 
		\begin{equation*}
			\bar{\iota}(\rho,\theta,z) - id = \left( \rho\frac{\zeta(\rho\theta,z)-\eta(\theta,z)}{\eta(\theta,z)},0,0 \right),
		\end{equation*}
		which reduces our problem to showing that $(\zeta(\rho\theta,z)-\eta(\theta,z))/\eta(\theta,z)$ belongs to $H^{s+\frac{1}{2}-}$ near $\rho=1$. In Proposition \ref{prop-pre:chgt-of-var}, we have seen that $\zeta \in H^{s+1}_{R_\epsilon}(\D\times\R)$ with $R_\epsilon$ defined in \eqref{eq-pre:def-R-eps}, which indicates that $R_\epsilon$ equals $R$ when $\rho$ is close to $1$. Thus, by choosing $\delta>0$ small enough, we have $\zeta-R\in H^{s+1-}([1-\delta,1]\times\T\times\R)$. And $\eta\in H^{s+\frac{1}{2}-}_R(\T\times\R)$ implies  $\zeta-\eta\in H^{s+\frac{1}{2}-}([1-\delta,1]\times\T\times\R)$, from which one may conclude $\bar{\iota} - id \in H^{s+\frac{1}{2}-}([1-\delta,1]\times\T\times\R)$ by applying Proposition \ref{prop-para:paralin} and Corollary \ref{cor-para:product-law}. 
		
		To obtain the regularity of $\bar{\iota}^{-1}$ near $\{\bar{\rho}=1\}$, due to \cite{bruverus2017completeness}, Section 2, it suffices to extend $\bar{\iota}$ as a diffeomorphism on $\R\times\T\times\R$ and check that its Jacobian has strictly positive lower bound. One possible extension is as follow,
		\begin{equation*}
			\bar{\iota}_{ext}(\rho,\theta,z) := \left( f(\rho,\theta,z) , \theta, z \right),\ \ f(\rho,\theta,z)=\chi_-(\rho)f_-(\rho) + \chi_1(\rho)\frac{\rho\zeta(\rho\theta,z)}{\eta(\theta,z)} + \chi_+(\rho)f_+(\rho),
		\end{equation*} 		
		where $\chi_1$ is a smooth truncation near $\{\rho=1\}$ increasing on $\{\rho<1\}$ and decreasing on $\{\rho>1\}$, $\chi_-$ ($\chi_+$ resp.) is a smooth function supported in $\{\rho<1\}$ ($\{\rho>1\}$ resp.) with $\chi_- + \chi_1 + \chi_+=1$ for all $\rho\in\R$, and $f_\pm$ is a smooth increasing function such that $f_\pm(\rho)=\rho$ when $\rho$ is away from $1$ and $f_\pm(\rho)=l_\pm\rho$ for $\rho\in\Supp{\chi_1}$ with constants $l_\pm>0$ to be determined later. 
		
		Clearly $\bar{\iota}_{ext}$ coincides with the original one \eqref{eq-paralin:alt-var} when $\rho$ is close enough to $1$ and equals identity when $\rho$ is away from $1$. Moreover, since the extended parts are smooth and equal to identity when $\rho$ is away from $1$, we have $\bar{\iota}_{ext} - id \in H^{s-\frac{1}{2}-}(\R\times\T\times\R)$. The Jacobian of $\bar{\iota}_{ext}$ reads 
		\begin{align*}
			\partial_\rho f =& \chi_-(\rho) f_-'(\rho) - \chi_-'(\rho)\left( \frac{\rho\zeta}{\eta} - f_-(\rho) \right) + \chi_1(\rho) \frac{\partial_\rho(\rho\zeta)}{\eta} \\
			&+ \chi_+'(\rho)\left( f_+(\rho) - \frac{\rho\zeta}{\eta} \right) + \chi_+(\rho)f_+'(\rho) \\
			=& \chi_-(\rho) f_-'(\rho) - \chi_-'(\rho)\rho\left( \frac{\zeta}{\eta} - l_- \right) + \chi_1(\rho) \frac{\partial_\rho(\rho\zeta)}{\eta} \\
			&+ \chi_+'(\rho)\rho\left( l_+ - \frac{\zeta}{\eta} \right) + \chi_+(\rho)f_+'(\rho).
		\end{align*}
		Due to our construction, $f_\pm'$, $-\chi_-'$, and $\chi_+'$ are non-negative. Besides, due to Proposition \ref{prop-pre:chgt-of-var} (and estimate \eqref{eq-pre:low-bd-rhozeta}), $\zeta/\eta$ and $\partial_\rho(\rho\zeta)/\eta$ have upper and lower bounds depending only on constants $c_0,C_0$ appearing in hypothesis \eqref{hyp-intro:bounds}. Thus, by choosing $0<l_- \ll 1 \ll l_+$, the Jacobian of $\bar{\iota}_{ext}$ can be bounded from below by some positive constant.
	\end{proof}
	
	\paragraph{Step 2: The case $s_0<s$} Recall that $\bar{\varphi} = \varphi\circ\bar{\iota}^{-1}$ with $\varphi$ verifying \eqref{eq-pre:ellip-reg}. Note that, in polar coordinate $(\rho,\theta,z)\in[0,1]\times\T\times\R$, we can only deduce that $\varphi\in H^{s_0+\frac{1}{2}}([1-\delta,1]\times\T\times\R)$ for $\delta>0$. Then the condition $s_0<s$ guarantees $s_0+1/2 \le s+1/2-$, which allows us to apply Proposition \ref{prop-sobo:chgt-of-var-refine} and conclude that 
	\begin{equation}\label{eq-ellip:reg-bar-varphi-subcrit}
		\|\bar{\varphi}\|_{H^{s_0+\frac{1}{2}}([1-\bar{\delta},1]\times\T\times\R)} \le C\left( \|\eta\|_{H^{s+\frac{1}{2}-}_R(\T\times\R)} \right) \|\psi\|_{H^{s_0}(\T\times\R)},
	\end{equation}
	for some $0<\bar{\delta}\ll 1$ such that $[1-\bar{\delta},1]\times\T\times\R$ lies in $\bar{\iota}^{-1}([1-\delta,1]\times\T\times\R)$. 
	
	\paragraph{Step 3: The case $s_0=s$} Let $\epsilon>0$ be small enough (to be determined later). When $\psi\in H^{s}(\T\times\R)$, from Step 2, we have seen that $\bar{\varphi}\in H^{s+\frac{1}{2}-\epsilon}([1-\bar{\delta},1]\times\T\times\R)$. To obtain the $H^{s+\frac{1}{2}}$-regularity, we notice that $\bar{\varphi}$ is the solution to equation \eqref{eq-paralin:ellip}, which we recall here
	\begin{equation*}
		\left\{ \begin{array}{ll}
			L\bar{\varphi}:= \left(\alpha\partial_{\bar{\rho}}^2 + \beta\cdot\nabla_{\bar{\theta},\bar{z}}\partial_{\bar{\rho}} + \gamma\partial_{\bar{\rho}} + \frac{1}{\bar{\rho}^2\eta^2}\partial_{\bar{\theta}}^2 + \partial_{\bar{z}}^2 \right) \bar{\varphi} = 0, & \forall 1-\bar{\delta}\le\bar{\rho}\le 1, \\ [0.5ex]
			\varphi|_{\bar{\rho}=1} = \psi,
		\end{array}\right.
	\end{equation*}
	where $\alpha,\beta,\gamma$ are defined in \eqref{eq-paralin:def-alpha}-\eqref{eq-paralin:def-gamma}. Let $\chi_1 = \chi_1(\rho)$ be a smooth truncation near $\{\bar{\rho}=1\}$ supported in $[1-\bar{\delta},+\infty[$. By denoting $\Lambda_1:=\langle D_{\bar{\theta},\bar{z}}\rangle$, we have
	\begin{equation*}
		\left\{ \begin{array}{ll}
			L\Lambda_1^\epsilon(\chi_1\bar{\varphi}) = \bar{f} := -\Lambda_1^\epsilon[\chi_1,L]\bar{\varphi} - [\Lambda_1^\epsilon,L](\chi_1\bar{\varphi}), & \forall \bar{\rho}\in [0,1], \\ [0.5ex]
			\varphi|_{\bar{\rho}=1} = \Lambda_1^\epsilon\psi.
		\end{array}\right.
	\end{equation*}
	Note that the first equation is valid for all $\bar{\rho}\in [0,1]$ since both sides vanish when $\bar{\rho}<1-\bar{\delta}$, thanks to the truncation $\chi_1$. It is clear that $\bar{f} \in H^{s-\frac{3}{2}-\epsilon}([1-\bar{\delta},1]\times\T\times\R)$ (it suffices to apply Corollary \ref{cor-para:paraprod-commu-bound} to deal with $[\Lambda^\epsilon_1,L]$). Then, via change of variable $\bar{\iota}$, this equation can be rewritten as
	\begin{equation*}
		\left\{ \begin{array}{ll}
			\Delta_g\varphi_1 = \bar{f}\circ\bar{\iota}, & \text{on }\D\times\R, \\ [0.5ex]
			\varphi_1|_{\rho=1} = \Lambda_1^\epsilon\psi,
		\end{array}\right.
	\end{equation*}
	where $\varphi_1 := \left(\Lambda_1^\epsilon(\chi_1\bar{\varphi})\right) \circ\bar{\iota}$. Recall that we identify $(y,z)\in\D\times\R$ with polar coordinate $(\rho,\theta,z)\in[0,1]\times\T\times\R$, since all the involved functions are supported away from the axis $\{\rho=0\}$. By Proposition \ref{prop-sobo:chgt-of-var-refine} and Lemma \ref{lem-ellip:reg-chgt-of-var}, the source term $\bar{f}\circ\bar{\iota}$ belongs to $H^{s-\frac{3}{2}-\epsilon}(\D\times\R)$, while the boundary data $\Lambda_1^\epsilon\psi \in H^{s-\epsilon}(\T\times\R)$. Due to Proposition \ref{prop-ellip:ellip-alt-reg}, we have $\varphi_1 \in H^{s+\frac{1}{2}-\epsilon}(\D\times\R)$. Finally, we apply again Proposition \ref{prop-sobo:chgt-of-var-refine} and Lemma \ref{lem-ellip:reg-chgt-of-var} to deduce $\Lambda_1^\epsilon(\chi_1\bar{\varphi}) \in H^{s+\frac{1}{2}-\epsilon}([1-\bar{\delta},1]\times\T\times\R)$ and thus $\Lambda_1^\epsilon\bar{\varphi} \in H^{s+\frac{1}{2}-\epsilon}([1-\bar{\delta}',1]\times\T\times\R)$ for some $0<\bar{\delta}'<\bar{\delta}$. One may recover the regularity in $\bar{\rho}$ from elliptic operator $L$ as in the proof of Proposition \ref{prop-ellip:ellip-alt-reg} and conclude that $\bar{\varphi} \in H^{s+\frac{1}{2}}([1-\bar{\delta}',1]\times\T\times\R)$, i.e. estimate \eqref{eq-ellip:reg-bar-varphi-subcrit} also holds for $s_0=s$. Here the difference between $\bar{\delta}$ and $\bar{\delta}'$ can be ignored since we are solely interested in the behavior near $\{\bar{\rho}=1\}$.
	
	\paragraph{Step 4: Trace estimate} To complete the proof of Lemma \ref{lem-paralin:reg-of-pot}, it remains to deduce the estimate below from \eqref{eq-ellip:reg-bar-varphi-subcrit}. 
	\begin{equation*}
		\|\partial_{\bar{\rho}}^l\bar{\varphi}\|_{C^0([1-\bar{\delta},1];H^{s_0-l}(\T\times\R))} \le C\left(\|\eta\|_{H_R^{s+\frac{1}{2}-}}\right) \|\psi\|_{H^{s_0}},\ \ l=0,1,2,3.
	\end{equation*}
	Recall that in the assumption of Lemma \ref{lem-paralin:reg-of-pot}, we have $s>3$ and $3/2<s_0\le s$. Then the case $l=0,1$ follows directly from trace estimate. For $l=2$, we observe that
	\begin{equation*}
		\partial_{\bar{\rho}}^2 \bar{\varphi} = -\frac{1}{\alpha}\left( \beta\cdot\nabla_{\bar{\theta},\bar{z}}\partial_{\bar{\rho}} + \gamma\partial_{\bar{\rho}} + \frac{1}{\bar{\rho}^2\eta^2}\partial_{\bar{\theta}}^2 + \partial_{\bar{z}}^2 \right) \bar{\varphi}.
	\end{equation*}
	where $\alpha,\beta,\gamma$ are estimated in Lemma \ref{lem-paralin:esti-alpha-beta-gamma} with $\alpha>c$ for some constant $c>0$ (depending on $c_0,C_0$ appearing in hypothesis \eqref{hyp-intro:bounds}). An application of Lemma \ref{lem-paralin:esti-alpha-beta-gamma}, Proposition \ref{prop-para:paralin}, and Corollary \ref{cor-para:product-law} gives $\partial_{\bar{\rho}}^2 \bar{\varphi} \in C^0([1-\bar{\delta},1];H^{s_0-2}(\T\times\R))$. The last case $l=3$ can be obtained by applying an extra $\partial_{\bar{\rho}}$ and repeating this argument.

	\section{Paradifferential calculus}\label{App:para}
	
	In this section, we shall review the definition and properties of paradifferential operators on $\R^d$, $\T^d$, or $\T^{d_1}\times\R^{d_2}$ with $d_1+d_2=d$.
	
	\subsection{Pseudo-differential operator}\label{subsect:PDO}
	
	\begin{definition}\label{def-para:PDO}
		Let $a$ be a tempered distribution on $\R^d \times \R^d$. For all Schwartz function $u\in \mathcal{S}(\R^d)$, we define
		\begin{equation}\label{eq-para:PDO}
			\left\langle \Op{a}u, v\right\rangle_{\mathcal{S}'\times\mathcal{S}} := \frac{1}{(2\pi)^d} \iint e^{ix\cdot\xi}a(x,\xi)\hat{u}(\xi) \overline{v(x)}d\xi dx.
		\end{equation}\index{O@$\Op{a}$ Pseudo-differential operator}
		It is easy to check that $\Op{a}$ is a continuous application from Schwartz functions $\mathcal{S}(\R^d)$ to tempered distributions $\mathcal{S}'(\R^d)$. We say that $\Op{a}$ is a \textit{pseudo-differential operator} with \textit{symbol} $a$.
	\end{definition}
	
	\begin{definition}\label{def-para:PDO-symbol-class}
		Let $\rho,\delta\in [0,1]$ and $m\in\R$. The class of symbols $S^{m}_{\rho,\delta} = S^{m}_{\rho,\delta}(\R^d)$ is defined as the collection of all symbols in $C^\infty(\R^d\times\R^d)$, such that, for all $\alpha,\beta\in\N^d$,
		\begin{equation*}
			\left| \partial^\alpha_x \partial^\beta_\xi a(x,\xi) \right| \le C_{\alpha,\beta} \langle\xi\rangle^{m+\delta|\alpha|-\rho|\beta|}.
		\end{equation*}
	\end{definition}\index{S@$S^m_{\rho,\delta}$ Symbol class}
	The following results are classic, whose proof can be founded in \cite{hormander2007analysis,coifman1978operateur}.
	\begin{proposition}\label{prop-para:PDO-bound}
		Let $0\le \delta<\rho \le1$ or $0 \le \rho=\delta<1$. Then the pseudo-differential operator $\Op{a} \in \mathcal{L}(H^{s};H^{s-m})$ for all $a\in S^m_{\rho,\delta}$ and $s,m\in\R$.
	\end{proposition}
	\begin{proposition}\label{prop-para:PDO-symb-cal}
		Let $0\le \delta<\rho \le1$ and $m,m'\in\R$. For all symbols $a\in S^m_{\delta,\rho}$ and $b\in S^{m'}_{\delta,\rho}$, the composition $\Op{a}\Op{b}$ and the adjoint $\Op{a}^*$ are also pseudo-differential operators, with symbol $a\sharp b \in S^{m+m'}_{\delta,\rho}$, $a^*\in S^{m}_{\delta,\rho}$, respectively. Moreover, for all $N\in\N$, we have the following \textit{symbolic calculus},
		\begin{align}
			&a\sharp b - \sum_{|\alpha|< N} \frac{1}{\alpha!} \partial_\xi^\alpha a D_x^\alpha b \in S^{m+m'-N(\rho-\delta)}_{\rho,\delta}, \label{eq-para:PDO-symb-cal-composition} \\
			&a^* - \sum_{|\alpha|< N} \frac{1}{\alpha!} \partial_\xi^\alpha D_x^\alpha \overline{a} \in S^{m-N(\rho-\delta)}_{\rho,\delta},\label{eq-para:PDO-symb-cal-adjoint}
		\end{align}\index{c@$a\sharp b$ Composition of symbols}\index{a@$a^*$ Adjoint of symbol}
		where $D_x = -i\partial_x$.
	\end{proposition}
	
	In periodic case, there are two ways to define $\Op{a}$. One is to regard the Fourier variable $\xi$ as an element in $\N^d$, and symbols as distributions on $\T^d\times\N^d$. The only difference with $\R^d$ case is that the derivative in $\xi$ should be understood as finite difference:
	\begin{equation*}
		\partial_\xi a(x,\xi) := a(x,\xi+1) - a(x,\xi).
	\end{equation*}
	Another method is to regard functions on $\T^d$ as periodic functions on $\R^d$. Thus, for symbols $a\in \mathcal{S}(\R^d\times\R^d)$, $\Op{a}u$ is a well-defined periodic distribution, which can be viewed as a distribution on $\T^d$. Then, for all $a\in S^{m}_{\rho,\delta}$, one may define $\Op{a}$ by density arguments. Via this method, the derivatives in $\xi$ is just the same as those on Euclidean space. These two methods are equivalent in the sense that the difference of pseudo-differential operators defined by two methods is of lower order than the itself. That is to say, if $a\in S^{m}_{\delta,\rho}$, the difference of two definition should be an operator of order $m-\rho$, i.e. belonging to $\mathcal{L}(H^s;H^{s-(m-\rho)})$ for all $s\in\R$. A rigorous study of these definitions can be found in \cite{ruzhansky2010quantization,said2023paracomposition}. For simplicity, we shall focus on $\R^d$ case in the sequel.
	
	One observes that all the results above requires the symbols to be smooth in $x$ and $\xi$. However, in applications, they are usually rough in $x$. For example, the $\Delta_g$ defined by \eqref{eq-pre:def-lap-g} has only Sobolev regularity in $(y,z)$ but smooth in Fourier variables. To overcome this difficulty, we turn to a refined version known as \textit{paradifferential calculus} developped by Bony \cite{bony1981calcul} in 1980s.
	
	\subsection{Paraproduct}\label{subsect:paraprod}
	
	To begin with, we consider the simplest case, multiplication operators. We fix a dyadic decomposition
	\begin{equation*}
		1 = \chi(\xi) + \sum_{j=0}^\infty \varphi\left(\frac{\xi}{2^j}\right),
	\end{equation*}
	where $\chi$, $\varphi$ are radial positive smooth truncations near $\{|\xi|\le 1\}$ and $\{1\le|\xi|\le 2\}$, respectively, such that
	\begin{equation*}
		\chi\left(\frac{\xi}{2^j}\right) = \chi\left(\xi\right) + \sum_{k=0}^j \varphi\left(\frac{\xi}{2^k}\right),\ \ \forall \xi\in\R^d,j\in\N.
	\end{equation*}
	Then we can define the following multipliers:
	\begin{equation*}
		\Delta_j := \varphi\left(\frac{D_x}{2^j}\right),\ \ S_j:=\chi\left(\frac{D_x}{2^j}\right).
	\end{equation*}
	\begin{definition}[Bony's decomposition]\label{def-para:paraprod}
		For all functions $a,b\in\mathcal{S}(\R^d)$,
		\begin{align*}
			T_a b &:= \sum_{j=2}^\infty S_{j-2} a \Delta_j b, \\
			R(a,b) &:= ab - T_a b - T_b a.
		\end{align*}\index{T@$T_a$ Paralinear operator (paraproduct)} \index{R@$R(a,b)$ Remainder in Bony's decomposition}
		The linear operator $T_a$ is known as \textit{paraproduct}.
	\end{definition}
	\begin{remark}\label{rmk-para:add-constant}
		In the definition of $T_a b$, the low frequency part of $b$ is eliminated. Thus, it makes no difference to replace $b$ by $b$ plus any function whose Fourier transform is supported near zero. In particular, we have
		\begin{equation*}
			T_a b = T_a (b-R),\ \ \forall R\in\R.
		\end{equation*}
	\end{remark}
	
	Now, we review the boundedness of $T_a b$ and $R(a,b)$ in Sobolev spaces,
	\begin{proposition}\label{prop-para:paraprod-bound-basic}
		Let $a\in H^s$, $b\in H^{s'}$ with $s,s'\in\R$. Then we have,
		\begin{align}
			&\|T_a b\|_{H^{\min(s',s+s'-\frac{d}{2})}} \lesssim \|a\|_{H^{s}} \|b\|_{H^{s'}}, \ \ \text{if } s\neq\frac{d}{2}; \label{eq-para:paraprod-bound-sobo-main} \\
			&\|R(a,b)\|_{H^{s+s'-\frac{d}{2}}} \lesssim \|a\|_{H^{s}} \|b\|_{H^{s'}},\ \ \text{if } s+s'>0. \label{eq-para:paraprod-bound-sobo-remainder}
		\end{align}
		
		Moreover, when $a\in C^\infty_b$ and $b\in H^{s'}$ with $s'\in\R$, we have
		\begin{align}
			&\|T_a b\|_{H^{s'}} \lesssim \|a\|_{L^\infty} \|b\|_{H^{s'}}, \label{eq-para:paraprod-bound-holder-main} \\
			&\|R(a,b)\|_{H^{s'}} \lesssim \|a\|_{C^M} \|b\|_{H^{s'}}, \label{eq-para:paraprod-bound-holder-remainder}
		\end{align}
		for some $M\in\N$ depending on $s'$ and dimension $d$. Recall that $\|a\|_{C^M} := \sup_{|\alpha|\le M} \|\partial^\alpha a\|_{L^\infty}$. And when $a\in H^s$ and $b\in C^\infty_b$, the estimate \eqref{eq-para:paraprod-bound-holder-main} above becomes
		\begin{equation}\label{eq-para:paraprod-bound-holder-alt-main}
			\|T_a b\|_{H^{s}} \lesssim \|a\|_{H^s} \|b\|_{C^N},
		\end{equation}
		where $N\in\N$ depends on $s$ and dimension $d$.
	\end{proposition}
	These results are classical, one may find a proof in \cite{bony1981calcul} or Chapter 2 of \cite{bahouri2011fourier}, where the results are generalized to Besov spaces. By combining Proposition \ref{prop-para:paraprod-bound-basic} and Remark \ref{rmk-para:add-constant}, we have the following result to be used frequently in this paper.
	\begin{proposition}\label{prop-para:paraprod-bound}
		Let $a\in H^s_W$ and $b\in H^{s'}_R$ for some $s,s'\in\R$, with $W\in C^\infty_b$ and $R\in\R$. Then, we have, for all $r\in\R$,
		\begin{align}
			&\|T_a b\|_{H^{r}} \le C (\|a\|_{H^{s}_W}+\|W\|_{L^\infty}) \|b\|_{H^{\max(r,r+\frac{d}{2}-s)+}_R}, \ \ \text{if } \max(r,r+\frac{d}{2}-s)<s'; \label{eq-para:paraprod-bound-main} \\
			&\|R(a,b)\|_{H^{r}} \le C (\|a\|_{H^{s}_W}+\|W\|_{C^M}) \|b\|_{H^{r+\frac{d}{2}-s}_R},\ \ \text{if } r>-\frac{d}{2}, \label{eq-para:paraprod-bound-remainder}
		\end{align}
		where constants $C>0$ and $M\in\N$ rely solely on $r,s,s'$ and dimension $d$.
	\end{proposition}
	
	If the normalization of $b$ is not constant but a $C^\infty_b$ function, the following corollary follows from Proposition \ref{prop-para:paraprod-bound} with estimates \eqref{eq-para:paraprod-bound-holder-alt-main} and \eqref{eq-para:paraprod-bound-holder-remainder}.
	\begin{corollary}\label{cor-para:product-law}
		Given $r>-\frac{d}{2}$, we assume that $a\in H^s_{R_1}$, $b\in H^{s'}_{R_2}$ with $s,s'\in\R$ and $R_1,R_2\in C^\infty_b$. If $s'>\max(r,r+\frac{d}{2}-s)$, there exists constants $C>0$ and $M\in\N$ depending on $r,s,s'$ and dimension $d$, such that
		\begin{equation}\label{eq-para:product-law}
			\|ab\|_{H^{r}_{R_1R_2}} \le C (\|a\|_{H^{s}_{R_1}}+\|R_1\|_{C^M}) (\|b\|_{H^{s'}_{R_2}}+\|R_2\|_{C^M}).
		\end{equation}
	\end{corollary}

	\subsection{Paradifferential operators}\label{subsect:paradiff}
	
	The paraproduct operator $T_a$ defined in previous section can be regarded as a refinement of multiplication operator, which is equal to the pseudo-differential operator $\Op{a}$. This inspires us to study $\Op{a}$, where $a=a(x,\xi)$ is a symbol with limited regularity in $x$, by turning to the \textit{paradifferential operator} $T_a$ defined below, which is no more than a generalization of Definition \ref{def-para:paraprod}.
	
	\begin{definition}\label{def-para:paradiff}
		Let $a = a(x,\xi)$ be a symbol smooth in $\xi\neq 0$ with Sobolev or H{\"o}lder regularity in $x$. Then the paradifferential operator of symbol $a$ is defined as
		\begin{equation}\label{eq-para:paradiff}
			T_a u (x) := \frac{1}{(2\pi)^d} \int e^{ix\cdot\xi} \sum_{j=2}^\infty S_{j-2}a(x,\xi) \varphi\left(\frac{\xi}{2^j}\right) \hat{u}(\xi) d\xi,
		\end{equation}
		where $S_{j-2}$ acts on $x$ variable.
	\end{definition}
	\begin{remark}
		By construction, the low-frequency information $|\xi| \ll 1$ of $u$ is eliminated, which means that paradifferential operators are never bijective. Nevertheless, for elliptic symbols (see Definition \ref{def-paralin:homo-sym-ellip}), it is possible to construct left and right inverse, up to some reasonable remainders, which is known as a \textit{parametrix}.
	\end{remark}
	\begin{remark}
		The definition above is a special case of the general one introduced by Bony \cite{bony1981calcul},
		\begin{equation*}
			T_a := \Op{\tilde{a}},\ \ \tilde{a}(x,\xi) = \tilde{\chi}(D_x,\xi)a(\cdot,\xi),
		\end{equation*}
		where $\tilde{\chi} = \tilde{\chi}(\eta,\xi)$ is a smooth truncation near $\{|\eta|<\epsilon(1+|\xi|)\}$, such that $\tilde{\chi}=0$ on $\{|\eta|>\epsilon'(1+|\xi|)\}$ with $0<\epsilon<\epsilon'<1$ and for all $\alpha,\beta\in\N^d$,
		\begin{equation*}
			|\partial^\alpha_\eta \partial^\beta_\xi \tilde{\chi}(\eta,\xi)| \lesssim_{\alpha,\beta} \langle\xi\rangle^{-|\alpha|-|\beta|}.
		\end{equation*}
		In fact, with different choice of $\tilde{\chi}$, the resulting paradifferential operators are equivalent in the sense that their difference is a smoothing operator (see \cite{metivier2008para} for more details).
	\end{remark}
	\begin{example}
		If $a = a(x)$, the definition \ref{def-para:paraprod} and \ref{def-para:paradiff} coincide. For $a = a(\xi)$,
		\begin{equation*}
			T_a = a(D_x)-\chi\left(\frac{D_x}{4}\right)a(D_x),
		\end{equation*}
		i.e. $T_a$ equals the Fourier multiplier $a(D_x)$ up to a smoothing operator. Furthermore, for general symbol $a=a(x,\xi)$ and multiplier $b=b(\xi)$, we have
		\begin{equation*}
			T_a\circ b(D_x) = T_{ab}.
		\end{equation*}
	\end{example}
	
	\begin{definition}\label{def-para:paradiff-symbol-class}
		Let $\rho\ge0$ and $m\in\R$. The symbol class $\Gamma_\rho^m$ is defined as the collection of symbols $a=a(x,\xi)$ H{\"o}lder in $x$ and smooth in $\xi$, such that for all $\alpha\in\N^d$,
		\begin{equation*}
			\|\partial_\xi^\alpha a(\cdot,\xi)\|_{C^\rho} \le C_{\alpha} \langle\xi\rangle^{m-|\alpha|},\ \ \forall\ |\xi|>\frac{1}{2},
		\end{equation*}
		where $C^\rho$ is the class of H{\"o}lder function for non-integer $\rho$ and the usual Sobolev space $W^{\rho,\infty}$ for $\rho\in\N$.
	\end{definition}\index{G@$\Gamma^m_\rho$ Symbol class for paralinear operators}
	\begin{remark}
		If $a = a(x,\xi)$ is homogeneous in $\xi$ of degree $m$, then it belongs to $\Gamma^m_\rho$ if and only if
		\begin{equation*}
			\sup_{|\xi|=1} \|\partial_\xi^\alpha a(\cdot,\xi)\|_{C^\rho} \le C_{\alpha},\ \ \forall\alpha\in\N^d.
		\end{equation*}
	\end{remark}
	
	In this paper, we are interested in symbols with Sobolev regularity in $x$. An application of Bernstein's Lemma (see, for example, Lemma 2.1 of \cite{bahouri2011fourier}) implies that
	\begin{proposition}\label{prop-para:paradiff-symbol-class-sobo}
		For all $m\in\R$ and $a\in\Gamma^m_r$ with $r<0$, $T_a$ is of order $m-r$. In particular, let $s<\frac{d}{2}$ be a real number. We assume that $a= a(x,\xi)$ is smooth in $\xi$ with Sobolev regularity in $x$, namely
		\begin{equation*}
			\|\partial_\xi^\alpha a(\cdot,\xi)\|_{H^s} \le C_{\alpha} \langle\xi\rangle^{m-|\alpha|},\ \ \forall\ |\xi|>\frac{1}{2}.
		\end{equation*}
		Then $T_a$ is of order $m-s+\frac{d}{2}$.
	\end{proposition}
	
	Now, we are able to generalize Proposition \ref{prop-para:PDO-bound} and \ref{prop-para:PDO-symb-cal}. The proof of following results can be found in Chapter 5 of \cite{metivier2008para}.
	\begin{proposition}\label{prop-para:paradiff-bound}
		Given $m\in\R$, for all $a\in\Gamma^m_0$, the paradifferential operator $T_a$ is of order $m$, namely
		\begin{equation}\label{eq-para:paradiff-bound}
			T_a \in \mathcal{L}\left(H^{s};H^{s-m}\right),\ \ \forall s\in\R.
		\end{equation}
	\end{proposition}
	
	\begin{proposition}\label{prop-para:paradiff-cal-sym}
		Let $a\in\Gamma^m_\rho$ and $b\in\Gamma^{m'}_\rho$ with $m,m'\in\R$ and $\rho> 0$. Then the composition $T_a T_b$ and adjoint $T_a^*$ are both paradifferential operators, such that $T_aT_b - T_{a\sharp_\rho b}$ is of order $m+m'-\rho$ and $T_a^* - T_{a^*}$ is of order $m-\rho$, where
		\begin{align}
			a \sharp_\rho b &= \sum_{|\alpha|<\rho} \frac{1}{\alpha!} \partial_\xi^\alpha a D_x^\alpha b, \label{eq-para:paradiff-cal-sym-composition} \\
			a^* &= \sum_{|\alpha|<\rho} \frac{1}{\alpha!} \partial_\xi^\alpha D_x^\alpha \overline{a}. \label{eq-para:paradiff-cal-sym-adjoint}
		\end{align}
	\end{proposition}
	As a corollary, we have the following commutator estimate:
	\begin{corollary}\label{cor-para:commu-esti}
		If $a\in\Gamma^m_\rho$, $b\in\Gamma^{m'}_\rho$ with $m,m'\in\R$ and $\rho>0$, the commutator $[T_a,T_b]$ is of order $m+m'-\min(\rho,1)$.
	\end{corollary}
	
	The following estimate concerning commutator will also be used,
	\begin{corollary}\label{cor-para:paraprod-commu-bound}
		Let $a\in H^s_W$ with $s>\frac{d}{2}$ and $W\in C^\infty_b$. There exists $0<\delta_0 \ll 1$, such that, for all symbol $\lambda = \lambda(\xi)$ in the class $S^\delta_{1,0}(\R^d)$ (see Definition \ref{def-para:PDO-symbol-class}) with $\delta \le \delta_0$, we have
		\begin{equation}\label{eq-para:paraprod-commu-bound}
			\left\| \left[ \Op{\lambda}, a \right] \right\|_{\mathcal{L}(H^r)} \le C,\ \ \forall r\le s-\delta,
		\end{equation}
		where $C>0$ and $\delta_0$ depends only on $s,r$. 
		
		Moreover, the same result holds true if the symbol $\lambda = \lambda(\xi')$ depends only on $\xi'$ (we write $\xi = (\xi',\xi_d)\in\R^{d-1}\times\R$) and belongs to the class $S^\delta_{1,0}(\R^{d-1})$.
	\end{corollary}
	\begin{proof}
		For all $u\in H^r$, we decompose,
		\begin{equation*}
			\left[ \Op{\lambda}, a \right] u = \left[ \Op{\lambda}, T_a \right] u + \Op{\lambda} \left( T_u a + R(a,u) \right) - \left( T_{\Op{\lambda} u} a + R(a, \Op{\lambda} u) \right).
		\end{equation*}
		For the last two terms, we will use estimates in Proposition \ref{prop-para:paraprod-bound} and estimate \eqref{eq-para:paraprod-bound-holder-alt-main}. In fact, $u\in H^r$ implies
		\begin{equation*}
			T_u a \in H^{\min(s,s+r-\frac{d}{2})-} \subset H^{r+\delta},\ \ R(a,u)\in H^{s+r-\frac{d}{2}} \subset H^{r+\delta},
		\end{equation*}
		whenever $s \ge r+\delta_0$ and $s-d/2>\delta_0$. Consequently, $\Op{\lambda} \left( T_u a + R(a,u) \right)$ lies in $H^r$. With this choice of $\delta_0$, from $\Op{\lambda} u \in H^{r-\delta}$ one may also deduce that
		\begin{equation*}
			T_{\Op{\lambda} u} a \in H^{\min(s,s+r-\delta-\frac{d}{2})-} \subset H^{r},\ \ R(a,\Op{\lambda} u)\in H^{s+r-\delta-\frac{d}{2}} \subset H^{r}.
		\end{equation*}
		
		It remains to deal with the principal part $\left[ \Op{\lambda}, T_a \right] u$. If $\lambda \in S^0_{1,0}(\R^d)$, it is harmless to replace $\Op{\lambda}$ by $T_\lambda$ since their difference is a smoothing operator. Then by Corollary \ref{cor-para:commu-esti}, it is easy to see that $\left[ T_\lambda, T_a \right] u$ belongs to $H^r$ since $\lambda \in \Gamma^\delta_{\delta}$ and $a\in \Gamma^0_{\delta}$ by choosing $s-d/2 > \delta_0$. The proof for the case of full space ($\lambda\in S^0_{1,0}(\R^d)$) is completed.
		
		If $\lambda$ depends only on $\xi'$ and belongs to $S^0_{1,0}(\R^{d-1})$, we observe from Definition \ref{def-para:paraprod} that
		\begin{equation*}
			\left[ \Op{\lambda}, T_a \right] u = \sum_{j\ge 2} [\Op{\lambda},S_ja]\Delta_j u,
		\end{equation*}
		It is easy to see that, for each $j\ge 2$, the Fourier transform of $[\Op{\lambda},S_ja]\Delta_j u$ is supported in $2^j \mathcal{C}$, where $\mathcal{C}$ is an annulus. Thus, by almost orthogonality, it suffices to check that $[\Op{\lambda},S_ja]$ is bounded from $L^2(\R^d)$ to $L^2(\R^d)$ uniformly in $j$, and the problem can be reduced to the following commutator estimate
		\begin{equation}\label{eq-para:paraprod-commu-bound-tangent}
			\left\| [\Op{\lambda},b] \right\|_{\mathcal{L}(L^2(\R^d))} \lesssim \|b\|_{H^s}.
		\end{equation}
		For any $u\in L^2(\R^d)$, we fix arbitrary $x_d\in\R$ and apply the result proved above (full space case) on $\R^{d-1}$ (recall that we write $x=(x',x_d)\in\R^{d-1}\times\R$).
		\begin{equation*}
			\left\| [\Op{\lambda},b(\cdot,x_d)] u(\cdot,x_d) \right\|_{L^2_{x'}} \lesssim \|b\|_{L^\infty_{x_d}H_{x'}^{s-\frac{1}{2}}} \|u(\cdot,x_d)\|_{L^2_{x'}} \lesssim \|b\|_{H^s(\R^d)} \|u(\cdot,x_d)\|_{L^2_{x'}},
		\end{equation*}
		where the second inequality is due to the classical trace theorem. The desired estimate \eqref{eq-para:paraprod-commu-bound-tangent} follows by taking $L^2_{x_d}$ norm on both sides.
	\end{proof}
	
	Another useful corollary is
	\begin{corollary}\label{cor-para:quadra-prod}
		Let $a\in H^{s}_R$ and $u\in H^{s'}$ with $\frac{d}{2}<s' \le s$ and $R\in\R$ constant. Then there exists $C>0$ such that
		\begin{equation}\label{eq-para:quadra-prod}
			\|au^2 - 2T_{au}u\|_{H^{\min(s,2s'-\frac{d}{2})}} \le C\left(\|a\|_{H^s_R}\right) \|u\|_{H^{s'}}^2.
		\end{equation}
	\end{corollary}
	\begin{proof}
		We may write $au^2$ as
		\begin{align*}
			au^2 = T_{au}u + T_u au + R(au,u) =& T_{au}u + T_u\left( T_au + T_u a + R(a,u) \right) + R(au,u) \\
			=& 2T_{au}u + \left( T_uT_a - T_{au} \right)u + T_u^2a + T_uR(a,u) + R(au,u).
		\end{align*}
		Since $a,u$ can be regarded as symbols in $\Gamma^0_{s'-d/2}$, respectively, the operator $T_uT_a - T_{au}$ is of order $-(s'-d/2)$ and thus $\left( T_uT_a - T_{au} \right)u\in H^{s+s'-\frac{d}{2}}$ (note that when $a,u$ are independent of $\xi$, $a\sharp_\rho u=au$ for all $\rho>0$). By applying Proposition \ref{prop-para:paraprod-bound} and Corollary \ref{cor-para:product-law}, we have $T_u^2a\in H^{s}$, $T_uR(a,u)\in H^{s+s'-\frac{d}{2}}$, and $R(au,u)\in H^{2s'-\frac{d}{2}}$, which completes the proof. 
	\end{proof}

	\subsection{Composition}\label{subsect:compo}
	
	To end this section, we review the paradifferential calculus concerning composition. 
	
	\begin{proposition}\label{prop-para:paralin}
		Let $F$ be a smooth function and $u\in H^s$ with $s>\frac{d}{2}$. Then $F(u)\in H^{s}_{F(0)}$ with
		\begin{equation}\label{eq-para:paralin}
			\|F(u) - F(0) - T_{F'(u)}u\|_{H^{2s-\frac{d}{2}}} \le C\left(\|u\|_{H^s}\right)\|u\|_{H^s}.
		\end{equation}
		
		Consequently, when $u\in H^s_R$ with $s>\frac{d}{2}$ and $R\in C^\infty_b$, we have
		\begin{equation}\label{eq-para:paralin-cor}
			\|F(u)\|_{H^{s}_{F(R)}} \leqslant C\left(\|u\|_{H^s_R}\right)\|u\|_{H^s_R},
		\end{equation}
		where $C>0$ is a smooth increasing function depending on $F$ and $R$.
	\end{proposition}
	The proof of this proposition can be found in \cite{bony1981calcul} or \cite{bahouri2011fourier}, Chapter 2.
	
	Now we focus on the study of $u\circ\chi$, where $u$ and $\chi$ both have limited regularity. In this case, the singularity concentrates on two terms, $T_{u'\circ\chi}\chi$ and the \textit{paracomposition} $X^*u$, which is firstly studied by Alinhac in \cite{alinhac1986paracomposition}. In the sequel, we assume that $\chi:\Omega'\rightarrow\Omega$ is a diffeomorphism, where $\Omega$, $\Omega'$ are bounded domains in $\R^d$ with smooth boundary, and $u$ is a function on $\Omega$ with Sobolev regularity. The following results can be found in \cite{alinhac1986paracomposition} (see also \cite{alinhac1985paracomposition}), while the case of $\R^d$ is studied in \cite{taylor2000tools} (see also \cite{said2023paracomposition}).
	\begin{proposition}\label{prop-para:paracomp-exist}
		Let $\chi\in H_{loc}^{1+\frac{d}{2}+\sigma}$ and $\sigma\in\R_+\backslash\N$. Then there exists a paracomposition operator $X^*$, such that, for all $u\in H^{s}_{loc}$, $s>\frac{d}{2}+1$,
		\begin{equation}\label{eq-para:paracomp-main}
			u\circ\chi = T_{u'\circ\chi}\chi + X^*u + Ru,
		\end{equation}
		where the remainder $R\in H^{1+\sigma+\min(1+\sigma+\frac{d}{2},s-1)}_{loc}$.\index{X@$X^*$ Paracomposition operator}
	\end{proposition}
	An important property of paracomposition operator $\chi^*$ is the following conjugation formula,
	\begin{proposition}\label{prop-para:paracomp-conj}
		Let $\chi\in H_{loc}^{1+\frac{d}{2}+\sigma}$ be as in Proposition \ref{prop-para:paracomp-exist} and $a\in\Gamma^m_r$ with $m\in\R$ and $r\ge 0$. The $X^*$ defined in Proposition \ref{prop-para:paracomp-exist} satisfies
		\begin{equation}\label{eq-para:paracomp-conj}
			X^*T_a = T_{\chi^*a}X^* + R,
		\end{equation}
		where $\chi^*$ is the pull-back by $\chi$ and the remainder $R$ is also a paradifferential operator with symbol in $\Gamma_0^{m-r}$.
	\end{proposition}
	
	\begin{remark}
		The rigorous definition of the pull-back $\chi^*a$ depends on a delicate study of pseudo-differential operators on domains (or smooth manifolds), which will not be precised here (refer to \cite{hormander2007analysis}, Theorem 18.1.17). In the case where $a$ is a differential operator $a(s,\xi) = \sum_{\alpha}a_{\alpha}(x)\xi^{\alpha}$ (Laplacian operator, for example), $\chi^*a$ equals the symbol of the pull-back of $\Op{a}$, which is clearly a differential operator and can be calculated simply by usual change of variable.
	\end{remark}
	
	\begin{remark}
		An explicit definition of paracomposition operator $X^*$ is 
		\begin{equation*}
			X^*u := \sum_{j}\Delta_j\left(\Delta_ju\circ\chi\right),
		\end{equation*}
		which is hard to use in applications. Usually, we may define $X^*u$ via \eqref{eq-para:paracomp-main} when the remainder is not important.
	\end{remark}

	\section{Variation in metric}\label{App:var-metric}
	
	This section is devoted to the proof of identity \eqref{eq-pre:var-in-metric}, which can be obtained via a direct but complicated calculus.

	\begin{lemma}\label{lem-tech:variation-of-metric}
		Under the hypotheses of Proposition \ref{prop-pre:Hamiltonian-formulation}, we have
		\begin{equation}\label{eq-tech:variation-of-metric}
			\begin{aligned}
				\frac{1}{2}\delta(\sqrt{g}g^{\alpha\beta})\partial_\alpha\varphi\partial_\beta\varphi =& \frac{1}{2} \nabla_{y,z}^T \left( J^{-1}Y \left|\nabla_g\varphi\right|^2 \sqrt{g}\delta\zeta \right) - \nabla_{y,z}^T\left( \sqrt{g}J^{-1}\nabla_g\varphi\ Y^T\nabla_g\varphi \delta\zeta \right) \\
				& + \sqrt{g}\Delta_g\varphi Y\cdot\nabla_g\varphi \delta\zeta.
			\end{aligned}
		\end{equation}
		Recall that, in the coordinate $(y_1,y_2,z)$, Greek letters are indices for $\{1,2,z\}$, while Latin letters correspond to $\{1,2\}$.
	\end{lemma}
	
	\begin{proof}
		To begin with, by using $(g^{\alpha\beta}) = J^{-1}J^{-T}$, one could right the left hand side as
		\begin{align*}
			\frac{1}{2}\delta(\sqrt{g}g^{\alpha\beta})\partial_\alpha\varphi\partial_\beta\varphi =& \frac{1}{2}\delta\left(\sqrt{g}\right)g^{\alpha\beta}\varphi_\alpha\varphi_\beta + \frac{1}{2} \nabla_{y,z}^T\varphi\delta\left(J^{-1}J^{-T}\right)\nabla_{y,z}\varphi\sqrt{g} \\
			=& \frac{1}{2} \delta\left(\det J\right)|\nabla_g\varphi|^2 - \nabla_{y,z}^T\varphi J^{-1}\delta J J^{-1} J^{-T} \nabla_{y,z}\varphi \sqrt{g} \\
			=& \frac{1}{2} \tr{J^{-1}\delta J}|\nabla_g\varphi|^2 \sqrt{g} - \nabla_g^T\varphi \delta J J^{-1} \nabla_g\varphi \sqrt{g},
		\end{align*}
		where we use formulas \eqref{eq-pre:jacobian-det}, \eqref{eq-pre:def-nabla-g}, and
		\begin{equation*}
			\delta\left(\det J\right) = \tr{J^{-1}\delta J}\det J\ \ \text{and}\ \ \delta\left(J^{-1}\right) = -J^{-1}\delta J J^{-1}.
		\end{equation*}
		From the definition \eqref{eq-pre:jacobian-mat}, one may compute
		\begin{equation*}
			J^{-1} \delta J = 
			\left(\begin{array}{cc}
				J_0^{-1}\delta\zeta & 0 \\[0.5ex]
				0 & 0
			\end{array}\right)+
			\left(\begin{array}{cc}
				J_0^{-1}y \nabla_y^T\delta\zeta & J_0^{-1}y \delta\zeta_z \\[0.5ex]
				0 & 0
			\end{array}\right),
		\end{equation*}
		with $J_0 = (\zeta\delta_{ij} + y_i\zeta_j)_{ij}$. As a result, 
		\begin{align*}
			&\frac{1}{2} \tr{J^{-1}\delta J}|\nabla_g\varphi|^2 \sqrt{g} \\
			=& \frac{1}{2} \tr{J_0^{-1}}|\nabla_g\varphi|^2 \sqrt{g}\delta\zeta + \frac{1}{2} \tr{J_0^{-1}y\nabla_y^T\delta\zeta}|\nabla_g\varphi|^2 \sqrt{g} \\
			=& \frac{1}{2} \tr{J_0^{-1}}|\nabla_g\varphi|^2 \sqrt{g}\delta\zeta + \frac{1}{2} J_0^{-1}y\cdot \nabla_y\delta\zeta |\nabla_g\varphi|^2 \sqrt{g} \\
			=& \frac{1}{2} \tr{J_0^{-1}}|\nabla_g\varphi|^2 \sqrt{g}\delta\zeta + \frac{1}{2} \nabla_y \cdot\left( J_0^{-1}y \delta\zeta |\nabla_g\varphi|^2 \sqrt{g} \right) - \frac{1}{2} \nabla_y \cdot\left( J_0^{-1}y  |\nabla_g\varphi|^2 \sqrt{g} \right)\delta\zeta \\
			=& \frac{1}{2} \tr{J_0^{-1}}|\nabla_g\varphi|^2 \sqrt{g}\delta\zeta + \frac{1}{2} \nabla_y \cdot\left( J_0^{-1}y \delta\zeta |\nabla_g\varphi|^2 \sqrt{g} \right) - \frac{1}{2} \nabla_y \cdot\left( J_0^{-1}y  |\nabla_g\varphi|^2 \right)\sqrt{g}\delta\zeta \\
			& - \frac{1}{2} |\nabla_g\varphi|^2 J_0^{-1}y\cdot\left(\nabla_y\sqrt{g}\right)\delta\zeta \\
			=& \frac{1}{2} \tr{J_0^{-1}}|\nabla_g\varphi|^2 \sqrt{g}\delta\zeta + \frac{1}{2} \nabla_y \cdot\left( J_0^{-1}y \delta\zeta |\nabla_g\varphi|^2 \sqrt{g} \right) - \frac{1}{2} \nabla_y \cdot\left( J_0^{-1}y  |\nabla_g\varphi|^2 \right)\sqrt{g}\delta\zeta \\
			& - \frac{1}{2} |\nabla_g\varphi|^2 J_0^{-1}y\cdot\tr{J_0^{-1}\nabla_y J_0} \sqrt{g}\delta\zeta.
		\end{align*}
		Note that the last equality is a consequence of 
		\begin{equation*}
			\nabla_y \sqrt{g} = \nabla_y \det J_0 = \tr{J_0^{-1}\nabla_y J_0} \det J_0 = \tr{J_0^{-1}\nabla_y J_0} \sqrt{g}.
		\end{equation*}
		Therefore, we obtain the identity
		\begin{equation*}
			\frac{1}{2} \tr{J^{-1}\delta J}|\nabla_g\varphi|^2 \sqrt{g} = r_1 \sqrt{g} \delta\zeta + \frac{1}{2} \nabla_{y,z}^T\left( J^{-1}Y \delta\zeta |\nabla_g\varphi|^2 \sqrt{g} \right),
		\end{equation*}
		where
		\begin{equation}\label{eq-tech:var-of-met-r-1}
			r_1 = \frac{1}{2} \tr{J_0^{-1}}|\nabla_g\varphi|^2  - \frac{1}{2} \nabla_y \cdot\left( J_0^{-1}y  |\nabla_g\varphi|^2 \right) - \frac{1}{2} |\nabla_g\varphi|^2 J_0^{-1}y\cdot\tr{J_0^{-1}\nabla_y J_0}.
		\end{equation}
		
		In the mean time, it is clear that 
		\begin{equation*}
			\delta J = \left(\begin{array}{cc}
				\delta\zeta + y\nabla_y^T\delta\zeta & y\delta\zeta_z \\[0.5ex]
				0 & 0
			\end{array}\right) = 
			\left(\begin{array}{cc}
				\delta\zeta  & 0 \\[0.5ex]
				0 & 0
			\end{array}\right) + 
			\left(\begin{array}{cc}
				 y\nabla_y^T\delta\zeta & y\delta\zeta_z \\[0.5ex]
				0 & 0
			\end{array}\right),
		\end{equation*}
		which yields
		\begin{align*}
			&\nabla_g^T\varphi \delta J J^{-1} \nabla_g\varphi \sqrt{g} \\
			=& \nabla_g^i\varphi \left(J^{-1}\nabla_g\varphi\right)^i \sqrt{g} \delta\zeta +  y_i\nabla_g^i\varphi \delta\zeta_\alpha \left( J^{-1}\nabla_g\varphi \right)^{\alpha} \sqrt{g} \\
			=& \nabla_g^i\varphi \left(J^{-1}\nabla_g\varphi\right)^i \sqrt{g} \delta\zeta + \nabla_{y,z}\cdot\left( J^{-1}\nabla_g\varphi y_i\nabla_g^i\varphi \delta\zeta  \sqrt{g}\right) - \nabla_{y,z}\cdot\left( \sqrt{g} J^{-1}\nabla_g\varphi y_i\nabla_g^i\varphi \right)\delta\zeta \\
			=& \nabla_g^i\varphi \left(J^{-1}\nabla_g\varphi\right)^i \sqrt{g} \delta\zeta + \nabla_{y,z}^T\left( J^{-1}\nabla_g\varphi Y^T\nabla_g\varphi \delta\zeta  \sqrt{g}\right) - \sqrt{g} J^{-1}\nabla_g\varphi \cdot\nabla_{y,z}\left( y_i\nabla_g^i\varphi \right)\delta\zeta \\
			& - \sqrt{g}\Delta_g\varphi Y^T\nabla_g\varphi\delta\zeta,
		\end{align*}
		where the last equality is due to \eqref{eq-pre:def-nabla-g} and \eqref{eq-pre:def-lap-g}. Then the following identity holds true,
		\begin{equation*}
			\nabla_g^T\varphi \delta J J^{-1} \nabla_g\varphi \sqrt{g} = r_2 \sqrt{g} \delta\zeta + \nabla_{y,z}^T\left( J^{-1}\nabla_g\varphi Y^T\nabla_g\varphi \delta\zeta\right) - \sqrt{g}\Delta_g\varphi Y^T\nabla_g\varphi\delta\zeta,
		\end{equation*}
		with
		\begin{equation}\label{eq-tech:var-of-met-r-2}
			r_2 = \nabla_g^i\varphi \left(J^{-1}\nabla_g\varphi\right)^i  -  \left(J^{-1}\nabla_g\varphi\cdot\nabla_{y,z}\right)\left( y_i\nabla_g^i\varphi \right).
		\end{equation}
		It follows that the desired result \eqref{eq-tech:variation-of-metric} is equivalent to 
		\begin{equation*}
			r_1 - r_2 = 0.
		\end{equation*}
		
		We notice that
		\begin{equation*}
			r_1 = - \left(\left(J_0^{-1}y\cdot \nabla_y\right)\nabla_g^\alpha\varphi\right) \nabla_g^\alpha\varphi - \frac{1}{2}\partial_ia^{ij}y_j |\nabla_g\varphi|^2 - \frac{1}{2} J_0^{-1}y\cdot\tr{J_0^{-1}\nabla_y J_0} |\nabla_g\varphi|^2
		\end{equation*}
		(recall that $(a^{ij})$ is defined in \eqref{eq-pre:jacobian-mat-inverse} with $i,j$ taken in $\{y_1,y_2\}$), and
		\begin{equation*}
			r_2 = -   y_i\left(J^{-1}\nabla_g\varphi\cdot\nabla_{y,z}\right)\nabla_g^i\varphi.
		\end{equation*}
		Then the problem is reduced the following identities
		\begin{align}
			-\partial_ia^{ij}y_j &= J_0^{-1}y\cdot\tr{J_0^{-1}\nabla_y J_0}, \label{eq-tech:var-of-met-claim-1} \\
			\left(\left(J_0^{-1}y\cdot \nabla_y\right)\nabla_g^\alpha\varphi\right) \nabla_g^\alpha\varphi &= y_i\left(J^{-1}\nabla_g\varphi\cdot\nabla_{y,z}\right)\nabla_g^i\varphi. \label{eq-tech:var-of-met-claim-2}
		\end{align}
		To prove these, we shall use the following formulas which can be checked easily from \eqref{eq-pre:jacobian-mat},
		\begin{align}
			&\partial_\alpha a^{\beta\gamma} = - a^{\beta\beta'}\partial_\alpha a_{\beta'\gamma'} a^{\gamma'\gamma}, \label{eq-tech:var-of-met-deri-J-1} \\
			& \partial_\alpha a_{z\beta} = 0, \label{eq-tech:var-of-met-deri-z-J} \\
			& \partial_{\alpha} a_{i\beta} = \delta_{i\beta}\zeta_\alpha + \delta_{\alpha i} \zeta_\beta + y_{i}\zeta_{\alpha\beta}. \label{eq-tech:var-of-met-deri-y-J}
		\end{align}
		Note that we write the version for $J$, while the same formulas hold for $J_0$, which means one could replace the Greek letters by Latin ones.
		
		We first check \eqref{eq-tech:var-of-met-claim-1}. By \eqref{eq-tech:var-of-met-deri-y-J}, the left hand side is equal to
		\begin{align*}
			-\partial_ia^{ij}y_j =& a^{ik} \partial_i a_{kl} a^{lj} y_j \\
			=& a^{ik} \left( \delta_{kl}\zeta_i + \delta_{ik}\zeta_l + y_k \zeta_{il} \right) a^{lj} y_j \\
			=& \zeta_i a^{ik}a^{kj} y_j + \zeta_l a^{lj} y_j a^{ii} + a^{ik}y_k a^{lj}y_j \zeta_{il},
		\end{align*}
		while the right hand side reads
		\begin{align*}
			J_0^{-1}y\cdot\tr{J_0^{-1}\nabla_y J_0}
			=& a^{ij}y_j a^{kl}\partial_{i}a_{lk} \\
			=& a^{ij}y_j a^{kl} (\delta_{lk}\zeta_i + \delta_{il}\zeta_k + y_l\zeta_{ik}) \\
			=& \zeta_i a^{ij}y_j a^{kk} + \zeta_ka^{ki}a^{ij}y_j + a^{ij}y_j a^{kl}y_l \zeta_{ik},
		\end{align*}
		which proves \eqref{eq-tech:var-of-met-claim-1}.
		
		As for \eqref{eq-tech:var-of-met-claim-2}, the left hand side equals
		\begin{align*}
			&a^{ij}y_j \partial_i(a^{\beta\alpha}\varphi_\beta) \nabla_g^\alpha\varphi \\
			=& a^{ij}y_j \partial_ia^{\beta\alpha}\varphi_\beta \nabla_g^\alpha\varphi + a^{ij}y_ja^{\beta\alpha}\varphi_{i\beta}  \nabla_g^\alpha\varphi \\
			=& - a^{ij}y_j a^{\beta\lambda} \partial_i a_{\lambda\mu} a^{\mu\alpha}\varphi_\beta \nabla_g^\alpha\varphi + a^{ij}y_ja^{\beta\alpha}\varphi_{i\beta}  \nabla_g^\alpha\varphi \\
			=& - a^{ij}y_j a^{\beta k} \partial_ia_{k\mu} a^{\mu\alpha}\varphi_\beta \nabla_g^\alpha\varphi + a^{ij}y_ja^{\beta\alpha}\varphi_{i\beta}  \nabla_g^\alpha\varphi \\
			=& - a^{ij}y_j a^{\beta k} ( \delta_{k\mu}\zeta_i + \delta_{ik}\zeta_\mu + y_k \zeta_{i\mu} ) a^{\mu\alpha}\varphi_\beta \nabla_g^\alpha\varphi + a^{ij}y_ja^{\beta\alpha}\varphi_{i\beta}  \nabla_g^\alpha\varphi \\
			=& - \zeta_i a^{ij}y_j \varphi_\beta a^{\beta k} a^{k\alpha} \nabla_g^\alpha\varphi - \varphi_\beta a^{\beta i} a^{ij}y_j \zeta_\mu a^{\mu\alpha} \nabla_g^\alpha\varphi - \varphi_\beta a^{\beta k} y_k a^{ij}y_j \zeta_{i\mu} a^{\mu\alpha} \nabla_g^\alpha\varphi \\
			& + a^{ij}y_ja^{\beta\alpha}\varphi_{i\beta}  \nabla_g^\alpha\varphi \\
			=& - \zeta_i a^{ij}y_j \varphi_l a^{l k} a^{k\alpha} \nabla_g^\alpha\varphi - \varphi_l a^{l i} a^{ij}y_j \zeta_\mu a^{\mu\alpha} \nabla_g^\alpha\varphi - \varphi_l a^{lk} y_k a^{ij}y_j \zeta_{i\mu} a^{\mu\alpha} \nabla_g^\alpha\varphi \\
			& + a^{ij}y_ja^{\beta\alpha}\varphi_{i\beta}  \nabla_g^\alpha\varphi,
		\end{align*}
		while the right hand side can be written as
		\begin{align*}
			&y_i a^{\beta\alpha}\nabla_g^\alpha\varphi \partial_\beta (a^{\gamma i}\varphi_\gamma) \\
			=& y_i a^{\beta\alpha}\nabla_g^\alpha\varphi \partial_\beta (a^{j i}\varphi_j) \\
			=&  y_i a^{\beta\alpha}\partial_\beta a^{ji} \varphi_j\nabla_g^\alpha\varphi + a^{ji}y_i  a^{\beta\alpha} \varphi_{j\beta}\nabla_g^\alpha\varphi \\
			=& -y_i a^{\beta\alpha}a^{jk}\partial_\beta a_{kl} a^{li} \varphi_j\nabla_g^\alpha\varphi + a^{ji}y_i  a^{\beta\alpha} \varphi_{j\beta}\nabla_g^\alpha\varphi \\
			=&-y_i a^{\beta\alpha}a^{jk} ( \delta_{kl}\zeta_\beta + \delta_{k\beta}\zeta_l + y_k \zeta_{\beta l} ) a^{li} \varphi_j\nabla_g^\alpha\varphi + a^{ji}y_i  a^{\beta\alpha} \varphi_{j\beta}\nabla_g^\alpha\varphi \\
			=&-\varphi_ja^{jk}a^{ki}y_i \zeta_\beta a^{\beta\alpha} \nabla_g^\alpha\varphi - \zeta_l a^{li} y_i \varphi_j a^{jk}a^{k\alpha} \nabla_g^\alpha\varphi - \varphi_j a^{jk} y_k a^{li} y_i \zeta_{\beta l} a^{\beta\alpha} \nabla_g^\alpha\varphi  \\
			&+ a^{ji}y_i  a^{\beta\alpha} \varphi_{j\beta}\nabla_g^\alpha\varphi,
		\end{align*}
		which gives \eqref{eq-tech:var-of-met-claim-2}.
	\end{proof}

	\printindex
	
	\printbibliography[heading=bibliography,title=References]
\end{document}